\documentclass[a4paper,10pt
]{article}%

\usepackage[dvipsnames]{xcolor}
\usepackage{ifdraft}[pos=0.8]
\ifdraft{
  \color[RGB]{63,63,63}
  \usepackage[notref]{showkeys}
  \usepackage{todonotes}
}
{
  \usepackage[disable]{todonotes}
}

\usepackage{amsmath, amsthm}

\usepackage{imakeidx}
\makeindex[columns=2, title=Glossary of recurrent notation, intoc]

\usepackage[latin1]{inputenc}%
\usepackage{MnSymbol}
\usepackage[
cal = cm,
bb = ams,
frak = euler,
scr = rsfs
]{mathalpha}

\usepackage{upgreek}

\usepackage[normalem]{ulem} 

\usepackage{paralist} 
\usepackage[inline,shortlabels]{enumitem}%
\setenumerate{label=(\roman*)}

\usepackage[final]{microtype}
\usepackage{relsize}
\usepackage{geometry}

\usepackage{tikz}%
\usetikzlibrary{matrix,arrows,decorations.pathmorphing,
cd,patterns,calc}
\tikzset{%
  dummy/.style    = {circle,draw,inner sep=0pt,minimum size=2mm}%
}%

\usetikzlibrary[decorations.pathreplacing]

\makeatletter

\def\@testdef #1#2#3{%
  \def\reserved@a{#3}\expandafter \ifx \csname #1@#2\endcsname
  \reserved@a  \else
  \typeout{^^Jlabel #2 changed:^^J%
    \meaning\reserved@a^^J%
    \expandafter\meaning\csname #1@#2\endcsname^^J}%
  \@tempswatrue \fi}

\makeatother

\usepackage[pagebackref, colorlinks, citecolor=PineGreen, linkcolor=PineGreen]{hyperref}
\hypersetup{
  final,
  pdftitle={Genuine equivariant operads},
  pdfauthor={Bonventre, P. and Pereira, L. A.},
  linktoc=page,
  pdfpagemode=UseNone,
  pdfpagelayout=SinglePage,
}

\numberwithin{equation}{section} 
\numberwithin{figure}{section}

\usepackage{mathtools}

\newtheorem{theorem}[equation]{Theorem}%
\newtheorem*{theorem*}{Theorem}%
\newtheorem{lemma}[equation]{Lemma}%
\newtheorem{proposition}[equation]{Proposition}%
\newtheorem{corollary}[equation]{Corollary}%
\newtheorem{conjecture}[equation]{Conjecture}%
\newtheorem*{conjecture*}{Conjecture}%
\newtheorem{innercustomgeneric}{\customgenericname}
\providecommand{\customgenericname}{}
\newcommand{\newcustomtheorem}[2]{%
  \newenvironment{#1}[1]
  {%
   \renewcommand\customgenericname{#2}%
   \renewcommand\theinnercustomgeneric{##1}%
   \innercustomgeneric
  }
  {\endinnercustomgeneric}
}

\newcustomtheorem{customthm}{Theorem}
\newcustomtheorem{customcor}{Corollary}

\theoremstyle{definition} 
\newtheorem{definition}[equation]{Definition}%
\newtheorem*{definition*}{Definition}%
\newtheorem{example}[equation]{Example}%
\newtheorem{remark}[equation]{Remark}%
\newtheorem{notation}[equation]{Notation}%
\newtheorem{convention}[equation]{Convention}%
\newcommand{\set}[1]{\left\{#1\right\}}%

\usepackage{harpoon}

\newcommand{\Sym}{\ensuremath{\mathsf{Sym}}}%
\newcommand{\Fin}{\mathsf{F}}%

\newcommand{\sSet}{\ensuremath{\mathsf{sSet}}}%

\newcommand{\Op}{\mathsf{Op}}%
\newcommand{\sOp}{\ensuremath{\mathsf{sOp}}}%

\newcommand{\Alg}{\mathsf{Alg}}

\DeclareMathOperator{\colim}{colim}%
\DeclareMathOperator{\Lan}{Lan}%
\newcommand{\F}{\ensuremath{\mathcal F}}
\newcommand{\V}{\ensuremath{\mathcal V}}

\renewcommand{\O}{\ensuremath{\mathcal O}}
\renewcommand{\P}{\ensuremath{\mathcal P}}
\newcommand{\C}{\ensuremath{\mathcal C}}

\newcommand{\D}{\mathcal D}
\newcommand{\mcI}{\ensuremath{\mathcal{I}}}%

\author{Peter Bonventre, Lu\'is A. Pereira}%
\title{Genuine equivariant operads}%

\begin{document}

\maketitle%

\abstract{We build new algebraic structures, which we call genuine equivariant operads and which can be thought of as a hybrid between operads and coefficient systems.
We then prove an Elmendorf-Piacenza type theorem stating that equivariant operads, with their graph model structure, are equivalent to genuine equivariant operads, with their projective model structure.

As an application, we build explicit models for the $N_{\infty}$-operads of Blumberg and Hill.}

\tableofcontents

\section{Introduction}

A surprising feature of topological algebra is that
the category of (connected) topological commutative monoids is quite small,
consisting only of products of Eilenberg-MacLane spaces (e.g. \cite[4K.6]{Hatcher}).
Instead, the more interesting structures are
those monoids which are commutative and associative only up to homotopy and, moreover, up to ``all higher homotopies''.
To capture these more subtle algebraic notions, Boardman-Vogt \cite{BV73} and May \cite{May72} developed
the theory of \textit{operads} \index{categories!of operads/symmetric sequences!Operad@$\Op(\V)$}.
Informally, a (simplicial) operad $\O \in \sOp = \Op(\sSet)$
\index{categories!of operads/symmetric sequences!operadsi@$\sOp = \Op(\sSet)$}
consists of a sequence of
spaces $\O(n) \in \sSet$
of ``$n$-ary operations'' carrying a $\Sigma_n$-action
(recording ``reordering of the inputs of the operations''),
and a suitable notion of ``composition of operations''.
The purpose of the theory is then the study of 
``objects $X$ with operations indexed by $\O$'',
referred to as \textit{algebras}, with the notions of monoid, commutative monoid, Lie algebra, algebra with a module, and more,
all being recovered as algebras over some fixed operad in an appropriate category.
Of special importance are the 
$E_\infty$-operads, introduced by 
May in \cite{May72}, which are 
homotopical replacements for the commutative operad
 and encode the aforementioned
 ``commutative monoids up to homotopy''.
In particular, while an
 $E_\infty$-algebra structure on $X$ does not specify unique maps $X^n \to X$, 
 it nonetheless specifies such maps ``uniquely up to homotopy''.

$E_\infty$-operads are characterized by the homotopy type of their levels $\O(n)$: $\O$ is $E_\infty$ if and only if each $\O(n)$ is $\Sigma_n$-free and contractible. That is, for each subgroup $\Gamma \leq \Sigma_n$,
the homotopy type of the $\Gamma$-fixed points is
\[
\O(n)^\Gamma \sim
\begin{cases}
        * & \text{if } \Gamma = \{*\} , \\
        \varnothing \qquad & \text{otherwise.}
\end{cases}
\]
Notably, when studying the homotopy theory of operads in topological spaces
the preferred notion of weak equivalence is 
usually that of ``naive equivalence'',
with a map of operads 
$\O \to \O'$
deemed a weak equivalence if each of the maps
$\O(n) \to \O'(n)$
is a weak equivalence of spaces
upon forgetting the $\Sigma_n$-actions (e.g. \cite[3.2]{BM03}).
In this context, $E_\infty$-operads are then equivalent to the commutative operad $\mathsf{Com}$
and, moreover, 
any cofibrant replacement of $\mathsf{Com}$
is $E_{\infty}$.
These naive equivalences differ from the equivalences in ``genuine equivariant homotopy theory'',
where (for $G$ a group) a map of $G$-spaces $X \to Y$ is deemed a $G$-equivalence only if
the induced fix point maps $X^H \to Y^H$ are weak equivalences for all $H \leq G$.
This contrast hints at a number of novel subtleties that appear in the study of equivariant operads, which we now discuss.

First, note that for a finite group $G$ and $G$-operad $\O$ (i.e. an operad $\O$ together with a $G$-action commuting with all the structure),
the $n$-th level $\O(n)$ has a $G \times \Sigma_n$-action.
As such, one might guess that a map of $G$-operads
$\O \to \O'$
should be called a weak equivalence if each of the maps
$\O(n) \to \O'(n)$
is a $G$-equivalence after forgetting the $\Sigma_n$-actions, 
i.e. if the maps
\begin{equation}\label{NAIVEOPEQ EQ}
	\O(n)^H \xrightarrow{\sim} \O'(n)^H, \qquad H \leq G \leq G\times \Sigma_n,
\end{equation}
are weak equivalences of spaces. 
However, the notion of equivalence suggested in \eqref{NAIVEOPEQ EQ} turns out to not be ``genuine enough''.
To see why, we consider a homotopical replacement for $\mathsf{Com}$ using this theory: 
if one simply equips an $E_{\infty}$-operad $\O$ with a trivial $G$-action, the resulting $G$-operad has fixed points for each subgroup $\Gamma \leq G \times \Sigma_n$ described by
\begin{equation}\label{NAIVEGEINFTY EQ}
	\O(n)^{\Gamma} \sim 
\begin{cases}
	\** & \text{if } \Gamma \leq G,
\\
	\emptyset & \text{otherwise}.
\end{cases}
\end{equation}
However, as first noted by Costenoble-Waner \cite{CW91} in their study of equivariant infinite loop spaces,
the \textit{$G$-trivial $E_\infty$-operads} of \eqref{NAIVEGEINFTY EQ} do not provide 
the correct replacement of $\mathsf{Com}$
 in the $G$-equivariant context. 
Rather, that replacement is provided instead by the 
\textit{$G$-$E_{\infty}$-operads}, characterized by the fixed point conditions
\begin{equation}\label{GENGEINFTY EQ}
	\O(n)^{\Gamma} \sim 
\begin{cases}
	\** & \text{if } \Gamma \cap \Sigma_n = \{\**\},
\\
	\emptyset & \text{otherwise}.
\end{cases}
\end{equation}
In contrasting 
\eqref{NAIVEGEINFTY EQ} and \eqref{GENGEINFTY EQ},
we note first that the subgroups $\Gamma \leq G \times \Sigma_n$ such that $\Gamma \cap \Sigma_n = \{\**\}$
are characterized as being the 
\emph{graph subgroups},
i.e. the subgroups of the form
\begin{equation}\label{GRAPHSUBIN EQ}
	\Gamma = \left\{(h,\phi(h)) \in G \times \Sigma_n
	| h \in H\right\}
\end{equation}
for some subgroup $H \leq G$ and homomorphism
$\phi \colon H \to \Sigma_n$.
On the other hand, $\Gamma \leq G$ if and only if $\Gamma$ is the graph subgroup \eqref{GRAPHSUBIN EQ} for $\phi$ a trivial homomorphism. 
As it turns out,
the notion of weak equivalence described in \eqref{NAIVEOPEQ EQ} fails to distinguish
\eqref{NAIVEGEINFTY EQ} and \eqref{GENGEINFTY EQ}, 
and indeed it is possible
to build maps $\O \to \O'$ where
$\O$ is a $G$-trivial $E_{\infty}$-operad (as in \eqref{NAIVEGEINFTY EQ})
and $\O'$ is a $G$-$E_{\infty}$-operad 
(as in \eqref{GENGEINFTY EQ}).
Therefore, in order to differentiate such operads, one needs to replace the notion of weak equivalence in \eqref{NAIVEOPEQ EQ} 
with the finer notion of \textit{graph equivalence}, 
so that $\O \to \O'$ is considered a weak equivalence only if
the maps
\begin{equation}\label{GENEOPEQ EQ}
	\O(n)^{\Gamma} \xrightarrow{\sim} \O'(n)^{\Gamma}, \qquad
	\Gamma \leq G\times \Sigma_n, \Gamma \cap \Sigma_n = \{\**\}.
\end{equation}
are all weak equivalences.

As mentioned above, the original evidence \cite{CW91}
that \eqref{GENGEINFTY EQ}, 
rather than \eqref{NAIVEGEINFTY EQ}, 
provides the best up-to-homotopy replacement for $\mathsf{Com}$ in the equivariant context comes from the study of equivariant infinite loop spaces. 
For our purposes, however, we instead focus on the perspective of Blumberg-Hill in \cite{BH15},
which concerns the Hill-Hopkins-Ravenel norm maps featured in the solution of the Kervaire Invariant One Problem \cite{HHR}.

Given a $G$-spectrum $R$ and finite $G$-set $X$ with $n$ elements, 
the corresponding \textit{norm} is another $G$-spectrum $N^X \! R$,
whose underlying spectrum is 
$R^{\wedge X} \simeq R^{\wedge n}$,
but equipped with a ``mixed $G$-action''
which both permutes wedge factors via the action on $X$ 
and acts diagonally on each factor
(alternatively, $N^X \! R$ can be described via graph subgroups; see the next paragraph).
Moreover, for any $\mathsf{Com}$-algebra $R$, i.e.
any strictly commutative $G$-ring spectrum, 
ring multiplication further induces
$G$-equivariant \textit{norm maps}
\begin{equation}\label{NORMMAPS EQ}
	N^X \! R \to R.
\end{equation}
Furthermore, by restricting the structure on $R$,
the maps \eqref{NORMMAPS EQ} are also defined when $X$ is only an $H$-set for some subgroup $H \leq G$, and the maps \eqref{NORMMAPS EQ}
then satisfy a number of 
natural equivariance and associativity conditions.
Crucially,  we note that the more interesting of these associativity conditions involve $H$-sets for various $H$ simultaneously
(for an example packaged in operadic language,
see \eqref{INTFIXPTCOMP EQ} below).

The key observation at the source of the work in 
\cite{BH15} is then that, operadically, 
norm maps are encoded by the graph fixed points 
appearing in \eqref{GENEOPEQ EQ}.
More explicitly, noting that, for $H \leq G$, an $H$-set $X$ with $n$ elements 
is encoded by a
homomorphism 
$H \to \Sigma_n$,
one obtains an associated graph subgroup 
$\Gamma_X \leq G \times \Sigma_n$,
well-defined up to conjugation\index{f@graph subgroups!graphsub@$\Gamma_X \leq G \times \Sigma_n$}.
Next,
using the natural 
$(G \times \Sigma_n)$-action on $R^{\wedge n}$,
the $H$-action on $N^X \! R \simeq R^{\wedge n}$
is obtained via the obvious identification
$H \simeq \Gamma_X$, cf. \eqref{GRAPHSUBIN EQ}.
It then follows that,
for any $\mathcal{O}$-algebra $R$,
maps of the form \eqref{NORMMAPS EQ}
are parametrized by the fixed point space
$\mathcal{O}(n)^{\Gamma_X}$.
The flaw of the $G$-trivial $E_{\infty}$-operads
described in \eqref{NAIVEGEINFTY EQ} is then that
they lack all norms maps other than those for $H$-trivial $X$, thus lacking 
some of the data encoded by $\mathsf{Com}$.
Further, from this perspective one may regard the more naive notion of weak equivalence in \eqref{NAIVEOPEQ EQ},
according to which \eqref{NAIVEGEINFTY EQ} and \eqref{GENGEINFTY EQ} are equivalent,
as studying ``operads without norm maps''
(in the sense that equivalences ignore norm maps), 
while the equivalences \eqref{GENEOPEQ EQ}
study ``operads with norm maps''.

Our first main result, Theorem \ref{MAINEXIST1 THM}, 
establishes the existence of a model structure on $G$-operads with weak equivalences the graph equivalences of \eqref{GENEOPEQ EQ},
though our analysis goes significantly further, again guided by Blumberg and Hill's work in \cite{BH15}.

The main novelty of \cite{BH15} is the definition, for each finite group $G$, of a finite lattice of new types of equivariant operads, which they dub $N_{\infty}$-operads.
The minimal type of $N_{\infty}$-operads is that of the 
$G$-trivial $E_{\infty}$-operads in \eqref{NAIVEGEINFTY EQ} 
while the maximal type is that of the $G$-$E_{\infty}$-operads in \eqref{GENGEINFTY EQ}.
The remaining types, which interpolate between 
the two,
can hence be thought of as encoding varying degrees of ``equivariant commutativity up to homotopy''.
More concretely, each type of $N_{\infty}$-operad is determined by a collection
$\mathcal{F} = \{\mathcal{F}_n\}_{n \geq 0}$
\index{indexing systems!F@$\mathcal F = \set{\mathcal F_n}_{n \geq 0}$}
where each $\mathcal{F}_n$ is itself a collection of graph subgroups of $G \times \Sigma_n$,
with an operad $\O$ being called a
\textit{$N \mathcal{F}$-operad}
\index{indexing systems!NF@$N \mathcal F$-operad}
if it satisfies the fixed point condition
\begin{equation}\label{NFINFTY EQ}
	\O(n)^{\Gamma} \sim 
\begin{cases}
	\** & \text{if } \Gamma \in \mathcal{F}_n,
\\
	\emptyset & \text{otherwise}.
\end{cases}
\end{equation}
Such collections $\mathcal{F}$ are, however, far from arbitrary, with much of the work in \cite[\S 3]{BH15} spent cataloging a number of closure conditions that these $\mathcal{F}$ must satisfy.
The simplest of these conditions
state that each $\mathcal{F}_n$ is a \textit{family}, i.e. closed under subgroups and conjugation. These first two conditions, 
which are ubiquitous in equivariant homotopy theory,
are a simple consequence of each $\O(n)$ being a space.
However, the remaining conditions, 
all of which involve $\mathcal{F}_n$ for various $n$ simultaneously and are a consequence of operadic multiplication,
are both novel and subtle.
In loose terms, these conditions, 
which are more easily described in terms of the 
$H$-sets $X$ associated to the graph subgroups,
concern closure of those under 
disjoint union, cartesian product, subobjects,
and an entirely new key condition called \textit{self-induction}.
The precise conditions are collected in
\cite[Def. 3.22]{BH15},
which also introduces the term \textit{indexing system} for 
an $\mathcal{F}$ satisfying all of those conditions.
A main result of \cite[\S 4]{BH15} is then that,
whenever an $N\mathcal F$-operad $\O$ as in \eqref{NFINFTY EQ} exists,
the associated collection $\mathcal{F}$ must be an indexing system.
However, the converse statement, that given any indexing system $\mathcal{F}$ such an $\O$ can be produced, was left as a conjecture.

One of the key motivating goals of the present work was to verify this conjecture of Blumberg-Hill, which we obtain in
Corollary \ref{NINFTY_REAL_COR_MAIN}.
We note here that this conjecture has also been concurrently verified
by Guti\'{e}rrez-White in \cite{GW18} 
and by Rubin in \cite{Rub17}, 
with each of their approaches having different advantages:
Guti\'{e}rrez-White's
model for $N \mathcal{F}$
is cofibrant 
while Rubin's model is explicit.
Our model, which emerges from a broader framework, satisfies both of these desiderata.


To motivate our approach, we first recall the solution of a closely related but simpler problem: that of building universal spaces for families of subgroups. 
Given a family $\mathcal{F}$ of subgroups of $G$
(i.e. a collection closed under conjugation and subgroups), 
a \textit{universal space} $X$ for $\mathcal{F}$, 
also called an \textit{$E \mathcal{F}$-space},
is a space with fixed points $X^H$ characterized  just as in \eqref{NFINFTY EQ}.
In particular, whenever $\O$ is an $N \mathcal{F}$-operad, 
each $\O(n)$ is necessarily an $E \mathcal{F}_n$-space.
The existence of $E \mathcal{F}$-spaces for any
choice of the family $\mathcal{F}$ is 
best understood in light of Elmendorf's classical result from \cite{Elm83}
(modernized by Piacenza in \cite{Pia91})
stating that there is a Quillen equivalence
(recall that $\mathsf{O}_G$ is the \textit{orbit} category, formed by the transitive $G$-sets $G/H$)
\begin{equation}\label{COFADJINT EQ}
\begin{tikzcd}[column sep =5em,row sep=0.3em]
	\mathsf{Top}^{\mathsf{O}_G^{op}}
	\ar[shift left=1]{r}{\iota^{\**}} 
&
	\mathsf{Top}^G
	\ar[shift left=1]{l}{\iota_{\**}}
\\
	\left(G/H \mapsto Y(G/H)\right)  \ar[mapsto]{r}&
	Y(G)
\\
	(G/H \mapsto X^H) &
	X \ar[mapsto]{l}
\end{tikzcd}
\end{equation}
where the weak equivalences (and fibrations)
in (topological) $G$-spaces $\mathsf{Top}^G$
are detected on all fixed points and
the weak equivalences (and fibrations)
on the category $\mathsf{Top}^{\mathsf{O}_G^{op}}$ of 
\textit{coefficient systems}
are detected at each level of the presheaves.
Noting that the fixed point characterization of $E \mathcal{F}$-spaces defines a natural object 
$\delta_{\mathcal{F}} \in \mathsf{Top}^{\mathsf{O}_G^{op}}$ by
\index{indexing systems!DeltaF@$\delta_{\mathcal F}$}
$\delta_{\mathcal{F}}(G/H)=\**$ if $H \in \mathcal{F}$ and
$\delta_{\mathcal{F}}(G/H)=\emptyset$ otherwise, 
$E \mathcal{F}$-spaces can then be built as
$\iota^{\**}(C \delta_{\mathcal{F}}) = 
C \delta_{\mathcal{F}}(G)$, where $C$ denotes cofibrant replacement in $\mathsf{Top}^{\mathsf{O}_G^{op}}$.
Moreover, we note that, as in \cite[\S 3]{Elm83}, these cofibrant replacements can be built via explicit simplicial realizations.

The overarching goal of this paper is then that of proving the analogue of Elmendorf-Piacenza's Theorem \eqref{COFADJINT EQ}
in the context of operads with norm maps (i.e. with equivalences as in \eqref{GENEOPEQ EQ}),
which we state as our main result, Theorem \ref{MAINQUILLENEQUIV THM}.
However, in trying to formulate such a result one immediately runs into a fundamental issue: 
it is unclear which category should take the role of the coefficient systems $\mathsf{Top}^{\mathsf{O}_G^{op}}$ in this context.
This last remark likely requires justification. 
Indeed, it may at first seem tempting to simply 
employ one of the known formal generalizations of Elmendorf-Piacenza's result (see, e.g. \cite[Thm. 3.17]{Ste16}) which simply replace
$\mathsf{Top}$ on either side of $\eqref{COFADJINT EQ}$
with a more general model category $\mathcal{V}$.
However, if one applies such a result when 
$\mathcal{V}=\mathsf{sOp}$
to establish a Quillen equivalence
$\mathsf{sOp}^{\mathsf{O}_G^{op}}
\rightleftarrows
\mathsf{sOp}^G$
(the existence of this equivalence is due to
upcoming work of Bergner-Guti\'{e}rrez), 
the fact that 
the levels of each 
$\mathcal{P} \in \mathsf{sOp}^{\mathsf{O}_G^{op}}$
correspond only to those fixed-point spaces appearing in \eqref{NAIVEOPEQ EQ}
would require working in the context of operads \textit{without} norm maps,
and thereby forgo the ability to distinguish 
the many types of $N \mathcal{F}$-operads.

As such, to obtain an Elemendorf-Piacenza Theorem
in the context of operads with norm maps, we will need to replace 
$\mathsf{Top}^{\mathsf{O}_G^{op}}$
with a category
$\mathsf{sOp}_G$
\index{categories!of operads/symmetric sequences!OpG@$\Op_G(\V)$}
of new algebraic objects we dub 
\textit{genuine equivariant operads}
(as opposed to (regular) equivariant operads
$\mathsf{sOp}^G$).
Each genuine equivariant operad 
$\mathcal{P} \in \mathsf{sOp}_G$
will consist of a list of spaces,
indexed in the same way as in 
\eqref{GENEOPEQ EQ},
along with obvious restriction maps and, more importantly, 
suitable \textit{composition maps}. Precisely identifying the required composition maps is one of the main challenges of this theory, and again we turn to \cite{BH15} for motivation.

Analyzing the proofs of the results in 
\cite[\S 4]{BH15}
concerning the closure properties for indexing systems $\mathcal F$,
a common motif emerges:
when performing an operadic composition
\begin{equation}\label{OPMULT EQ}
\begin{tikzcd}[row sep=0]
	\O(n) \times \O(m_1) \times \cdots \times \O(m_n) \ar{r} &
	\O(m_1 + \cdots + m_n),
\\
	(f,g_1,\cdots,g_n) \ar[mapsto]{r} &
	f(g_1,\cdots,g_n)
\end{tikzcd}
\end{equation}
careful choices of fixed point conditions on the operations $f,g_1,\cdots,g_n$ 
yield a fixed point condition on the composite operation
$f(g_1,\cdots,g_n)$.
The desired multiplication maps for a genuine equivariant operad
$\mathcal{P} \in \mathsf{sOp}_G$
will then abstract such interactions between multiplication and fixed points for an equivariant operad 
$\mathcal{O} \in \mathsf{sOp}^G$.
However, these interactions can be challenging to write down explicitly and indeed, 
the arguments in \cite[\S 4]{BH15}
do not quite provide the sort of unified conceptual approach
to these interactions needed for our purposes.
The cornerstone of the current work was then the 
joint discovery by the authors of
such a conceptual framework: equivariant trees.

Non-equivariantly, it has long been known that
the combinatorics of operadic composition is best visualized by means of tree diagrams. For instance, 
the tree $T$ on the right below
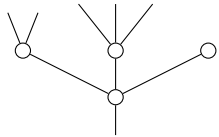
\[%
	\begin{tikzpicture}[auto,grow=up, every node/.style = {font=\footnotesize},level distance = 1.9em]%
	\tikzstyle{level 2}=[sibling distance=5.5em]%
	\tikzstyle{level 3}=[sibling distance=2em]%
		\node [font=\normalsize] {$T$}%
			child{node [dummy] {}%
				child{node [dummy] {}
				edge from parent node [swap] 
				{$h$}}%
				child{node [dummy] {}%
					child{
					edge from parent node [swap,very near end] 
					{$e$}}%
					child{
					edge from parent node [swap,very near end] 
					{$d$}}%
					child{
					edge from parent node [very near end] 
					{$c$}}%
				edge from parent node [swap] 
				{$g$}}%
				child{node [dummy] {}%
					child{
					edge from parent node [swap,very near end] 
					{$b$}}%
					child{
					edge from parent node [very near end] 
					{$\phantom{b}a$}}%
				edge from parent node 
				{$f$}}%
			edge from parent node [swap] 
			{$i$}};%
	\begin{scope}[xshift=-20em]
	\tikzstyle{level 2}=[sibling distance=2.5em]
		\node [font=\normalsize] {$C$}%
	child{node [dummy] {}%
		child{
		edge from parent node [swap,very near end] {$e$}}%
		child[level distance = 2.9em]{
		edge from parent node [swap,very near end] {$d$}}%
		child[level distance = 3.3em]{
		edge from parent node [swap,very near end] {$c$}}%
		child[level distance = 2.9em]{
		edge from parent node [very near end] {$b$}}%
		child{
		edge from parent node [very near end] {$\phantom{b}a$}}%
	edge from parent node [swap] {$i$}};%
	\end{scope}
	\end{tikzpicture}%
\]
encodes the operadic composition
\begin{equation}\label{COMPEX EQ}
	\O(3) \times \O(2) \times \O(3) \times \O(0) \to \O(5)
\end{equation}
where the inputs $\O(3), \O(2), \O(3), \O(0)$ correspond to the nodes/vertices (i.e. circles) in the tree $T$, with arity given by number of incoming edges (i.e. edges immediately above),
and the arity of the output $\O(5)$ is given by counting leaves
(i.e. edges at the top, not capped by a node).
Before recalling equivariant trees, 
it is worth making the connection between 
$T$ and \eqref{COMPEX EQ}
more precise.
Recall \cite[\S 3]{MW07} that $T$ gives rise to 
a colored operad\footnote{
Recall that colored operads, 
also known as multicategories,
are a generalization of the notion of category
where each arrow/operation
$(c_1,\cdots,c_n) \to c_0$
has multiple inputs but a single output.}
$\Omega(T)$,
as follows.
The colors/objects of $\Omega(T)$
are the edges $a,b,c,\cdots,i$
while the generating operations,
determined by the nodes,
are
$(a,b) \to f$, $(c,d,e) \to g$, $() \to h$, $(f,g,h) \to i$
(i.e., for each node, incoming edges are viewed as inputs 
and the outgoing edge as an output).
Let $C$ be the corolla (i.e. tree with a single node) above,
which is formed by the leaves and root of $T$. 
There is then a natural map of colored operads
$\Omega(C) \to \Omega(T)$
so that,
writing $\mathsf{Op}_{\bullet}$ for the category of colored operads (of sets),
\eqref{COMPEX EQ}
is the induced map of mapping sets
$\mathsf{Op}_{\bullet}
\left(\Omega(T),\O\right)
\to 
\mathsf{Op}_{\bullet}
\left(\Omega(C),\O\right)$.
Indeed,
$\mathsf{Op}_{\bullet}
\left(\Omega(T),\O\right)
\simeq
\O(3) \times \O(2) \times \O(3) \times \O(0)$
and
$\mathsf{Op}_{\bullet}\left(\Omega(C),\O\right) 
\simeq
\O(5)$
since maps 
$\Omega(T) \to \O$
and 
$\Omega(C) \to \O$
are determined by the image of the generating operations.

Analogously, the role of equivariant trees is,
in the context of equivariant operads,
to encode operadic compositions as in \eqref{COMPEX EQ}
together with fixed point compatibilities.
Briefly, 
a $G$-tree \cite[Def. 5.44]{Pe16b}
is a forest diagram 
(i.e. a collection of trees)
together with a $G$-action
that is transitive on tree components.
A detailed introduction to (and motivation for) 
equivariant trees can be found in \cite[\S 4]{Pe17}, where the second author develops the theory of equivariant dendroidal sets 
(a parallel approach to equivariant operads), 
though here we include only a single representative example.

Let $G = \{ \pm 1, \pm i, \pm j, \pm k\}$ be the group of quaternionic units 
and $G \geq H \geq K \geq L$ be the subgroups %
$H = \langle j \rangle$, %
$K = \langle -1 \rangle$, %
$L = \{1\}$.
One has a $G$-tree $T$ with 
\textit{expanded representation} given by the two leftmost trees 
below and
\textit{orbital representation} given by the rightmost tree.
\begin{equation}\label{D6SMALLER EQ}
	\begin{tikzpicture}[auto,grow=up, level distance = 2.2em,
	every node/.style={font=\scriptsize,inner sep = 2pt}]%
		\tikzstyle{level 2}=[sibling distance=7em]%
		\tikzstyle{level 3}=[sibling distance=2.25em]%
			\node at (4.75,0){}%
				child{node [dummy] {}%
					child{node [dummy] {}%
						child{node {}%
						edge from parent node [swap,very near end] {$-k a$}}%
						child[level distance = 2.4em]{node {}%
						edge from parent node [swap,near end] {$k b$}}%
						child{node {}%
						edge from parent node [very near end] {$k a$}}%
					edge from parent node [swap] {$k c$}}%
					child{node [dummy] {}%
						child{node {}%
						edge from parent node [swap,very near end] {$-i a$}}%
						child[level distance = 2.4em]{node {}%
						edge from parent node [swap,near end] {$i b$}}%
						child{node {}%
						edge from parent node [very near end] {$i a$}}%
					edge from parent node  {$i c$}}%
				edge from parent node [swap] {$i d$}};%
			\node at (0,0){}%
				child{node [dummy] {}%
					child{node [dummy] {}%
						child{node {}%
						edge from parent node [swap,very near end] {$-j a$}}%
						child[level distance = 2.4em]{node {}%
						edge from parent node [swap,near end] {$j b$}}%
						child{node {}%
						edge from parent node [very near end] {$j a$}}%
					edge from parent node [swap] {$j c$}}%
					child{node [dummy] {}%
						child{node {}%
						edge from parent node [swap,very near end] {$-a\phantom{j}$}}%
						child[level distance = 2.4em]{node {}%
						edge from parent node [swap,near end] {$b\phantom{j}$}}%
						child{node {}%
						edge from parent node [very near end] {$\phantom{-j}a$}}%
					edge from parent node  {$\phantom{j}c$}}%
				edge from parent node [swap] {$d$}};%
		\begin{scope}[every node/.style={font=\footnotesize}]%
			\node at (9.15,0){}%
				child{node [dummy] {}%
					child{node [dummy] {}%
						child{node {}%
						edge from parent node [swap,very near end] {$(G/K) \cdot b$}}%
						child{node {}%
						edge from parent node [very near end] {$(G/L) \cdot a$}}%
					edge from parent node [right] {$(G/K) \cdot c$}}%
				edge from parent node [right] {$(G/H) \cdot d$}};%
		\end{scope}%
		\draw[decorate,decoration={brace,amplitude=2.5pt}] (4.85,0) -- (-0.1,0) node[midway,inner sep=4pt,font = \normalsize]{$T$}; %
		\node at (9.15,-0.15) [font = \normalsize] {$T$};
	\end{tikzpicture}%
\end{equation}%
Here, the expanded representation of $T$ is just
a forest with edge labels that indicate the $G$-action.
Note that all edges are conjugate to one of the 
edges $a,b,c,d$ which have, respectively, stabilizers $L, K, K, H$.
For example, the labels of $T$ imply that 
$\pm j d = \pm d$ and $\pm i d = \pm k d$.
Given the expanded representation,
the orbital representation is  
obtained by collapsing each edge orbit into a single edge,
which is labeled by the corresponding orbit set of edges in the expanded representation
(one may also reverse this process, though we will not need to do so).
We note that orbital representations always ``look like a tree''.

As explained in \cite[Example 4.9]{Pe17},
the $G$-tree $T$ encodes the fact that,
for $\O \in \mathsf{sOp}^G$ a $G$-operad,
the composition 
$\mathcal{O}(2) \times \mathcal{O}(3)^{\times 2} \to 
\mathcal{O}(6)$ restricts to a fixed point composition
\begin{equation}\label{INTFIXPTCOMP EQ}
\O(H/K)^{H} \times \O(K/L \amalg K/K)^{K} \to
\O(H/L \amalg H/K)^{H}
\end{equation}
(we discuss how \eqref{INTFIXPTCOMP EQ}
is obtained in the next paragraph)
where $\O(X)$ for $X$ an $H$-set denotes $\O(|X|)$
with the $H$-action
given by the identification $H \simeq \Gamma_X$
(the graph subgroup $\Gamma_X \leq G \times \Sigma_n$ is as discussed after 
\eqref{NORMMAPS EQ}),
and likewise for $K$-sets.
In particular, $\O(X)^H \simeq \O(|X|)^{\Gamma_X}$.

We recall the exact connection between $T$ and \eqref{INTFIXPTCOMP EQ}. 
Let $\mathsf{Op}_{\bullet}^G$
be the category of 
$G$-objects in colored operads (of sets).
As in the non-equivariant case,
one builds 
$\Omega(T)$ in $\mathsf{Op}_{\bullet}^G$
and a map
$\Omega(C) \to \Omega(T)$
in $\mathsf{Op}_{\bullet}^G$,
where $C$ is the $G$-corolla (i.e. $G$-tree composed of corollas)
formed by the leaves and roots of $T$.
The composition
\eqref{INTFIXPTCOMP EQ}
is then the induced map
$\mathsf{Op}_{\bullet}^G
\left(\Omega(T),\mathcal{O}\right)
\to
\mathsf{Op}_{\bullet}^G
\left(\Omega(C),\mathcal{O}\right)$.
The implicit claim  
$\mathsf{Op}_{\bullet}^G
\left(\Omega(T),\mathcal{O}\right)
\simeq
\O(H/K)^{H} \times \O(K/L \amalg K/K)^{K}
$
follows since:
by equivariance, a $G$-map
$\phi \colon \Omega(T) \to \O$
is determined by the images of the operations
$(a,b,-a) \to c$ and
$(c,jc) \to d$;
the operation 
$\phi\left((a,b,-a) \to c\right)$
must be in $\O(K/L \amalg K/K)^{K}$,
since $K$ is the isotropy of $c$ and
$\{a,b,-a\} \simeq K/L \amalg K/K$
as $K$-sets; 
likewise $\phi\left((c,jc) \to d\right)$
must be in $\O(H/K)^{H}$.
The claim 
$\mathsf{Op}_{\bullet}^G
\left(\Omega(C),\mathcal{O}\right)
\simeq 
\O(H/L \amalg H/K)^{H}$
is similar.

We note that the two inputs 
$\O(H/K)^{H}$, $\O(K/L \amalg K/K)^{K}$ in
\eqref{INTFIXPTCOMP EQ}
correspond to the two nodes of the orbital representation
in \eqref{D6SMALLER EQ}.
Notice that now the arity (i.e. the associated ``type of input'')
of such a node does not just count incoming edge orbits,
but depends on the labels of both incoming and outgoing edge orbits
(in particular, the fixed point condition depends on the latter).
Similarly, the output 
$\O(H/L \amalg H/K)^{H}$
is determined by both the leaf and root edge orbits.
The existence of maps of the form \eqref{INTFIXPTCOMP EQ} is essentially tantamount to the subtlest 
closure property for indexing systems $\mathcal{F}$,
self-induction (cf. \cite[Def. 3.20]{BH15}),
and similar tree descriptions exist for all other closure properties, as detailed in 
\cite[\S 9]{Pe17}.

We can now at last give a full informal description of the category $\mathsf{Op}_G$ featured in 
our main result, Theorem \ref{MAINQUILLENEQUIV THM}.
A genuine equivariant operad
$\mathcal{P} \in \mathsf{sOp}_G$
has levels $\mathcal{P}(X)$ for each $H$-set $X$, $H\leq G$, 
that mimic the role of the fixed points $\O(X)^H \simeq \O(|X|)^{\Gamma_X}$ for 
$\mathcal{O} \in \mathsf{Op}^G$.
More explicitly, there are restriction maps 
$\mathcal{P}(X) \to \mathcal{P}(X|_{K})$ for $K \leq H$,
isomorphisms
$\mathcal{P}(X)\simeq \mathcal{P}(g X)$
where $gX$ denotes the conjugate $gHg^{-1}$-set,
and composition maps given by
\[
\P(H/K) \times \P(K/L \amalg K/K) \to \P(H/L \amalg H/K)
\]
in the case of the abstraction of \eqref{INTFIXPTCOMP EQ}, and more generally by
\begin{equation}\label{GENGENMULT EQ}
\begin{tikzcd}
  \mathcal{P}(H/K_1 \amalg \cdots \amalg H/K_n)
  \times
  \mathcal{P}(K_1 / L_{11} \amalg \cdots \amalg K_1/L_{1 m_1})
  \times \cdots \times
  \mathcal{P}(K_n / L_{n1} \amalg \cdots \amalg K_n / L_{n m_n})
  \arrow[d]
  \\
  \mathcal{P}(H / L_{11} \amalg \cdots \amalg H/L_{1 m_1}
  \amalg \cdots \amalg
  H / L_{n1} \amalg \cdots \amalg H/L_{n m_n}
  ).
\end{tikzcd}
\end{equation}
Lastly, these composition maps 
must satisfy associativity, unitality, compatibility with restriction maps, and equivariance conditions, as encoded by the theory of $G$-trees. 
Rather than making such compatibilities explicit, however, we will find it preferable for our purposes to simply define genuine equivariant operads intrinsically in terms of $G$-trees.

We end this introduction with an alternative perspective
(further expounded in \S \ref{CONTEX SEC})
on the role of genuine operads.
The Elmendorf-Piacenza theorem in 
\eqref{COFADJINT EQ}
is ultimately a strengthening of the basic observation that the homotopy groups
$\pi_n(X)$ of a $G$-space $X$ are coefficient systems rather than just $G$-objects.
Similarly, the generalized 
Elmendorf-Piacenza result \cite[Thm. 3.17]{Ste16}
applied to the category 
$\mathcal{V}=\mathsf{sCat}$
of simplicial categories strengthens the observation that, 
for a $G$-simplicial category $\C$,
the associated homotopy category
$\text{ho}(\mathcal{C})$
is a coefficient system of categories rather than just a $G$-category.
Likewise, Theorem \ref{MAINQUILLENEQUIV THM}
strengthens the (not so basic) observation that, for a $G$-simplicial operad $\O$, the associated homotopy operad 
$\text{ho}(\mathcal{O})$
is neither just a $G$-operad nor just a coefficient system of operads,
but rather the richer algebraic structure that we refer to as a ``genuine equivariant operad''.

\subsection{Main results}

We now discuss our main results.

Fixing a finite group $G$, we
recall that 
$\mathsf{Op}^G(\mathcal{V})
=
\left(\mathsf{Op}(\mathcal{V})\right)^G$
denotes $G$-objects in 
$\mathsf{Op}(\mathcal{V})$.




\begin{customthm}{I}\label{MAINEXIST1 THM}
Let $(\mathcal{V},\otimes)$
denote either 
$(\mathsf{sSet}, \times)$
or
$(\mathsf{sSet}_{\**}, \wedge)$.

Then there exists a model category structure on 
$\mathsf{Op}^G(\mathcal{V})$ such that 
$\O \to \O'$ is a weak equivalence (resp. fibration)
if all the maps
\begin{equation}\label{GENEOPEQMT EQ}
	\O(n)^{\Gamma} \to \O'(n)^{\Gamma}
\end{equation}
for 
$\Gamma \leq G\times \Sigma_n, \Gamma \cap \Sigma_n = \{\**\}$, 
are weak equivalences (fibrations) in $\mathcal{V}$.

More generally, for $\mathcal{F} = \{\mathcal{F}_n\}_{n \geq 0}$ with $\mathcal{F}_n$ an arbitrary collection of subgroups of $G \times \Sigma_n$ there exists a model category structure on 
$\mathsf{Op}^G(\mathcal{V})$,
which we denote
$\mathsf{Op}^G_{\mathcal{F}}(\mathcal{V})$,
with weak equivalences (resp. fibrations)
determined by \eqref{GENEOPEQMT EQ} for $\Gamma \in \mathcal{F}_n$.

Lastly, 
analogous 
semi-model category structures
$\mathsf{Op}^G(\mathcal{V})$,
$\mathsf{Op}^G_{\mathcal{F}}(\mathcal{V})$
exist
provided that
$(\mathcal{V},\otimes)$:
\begin{enumerate*}
\item[(i)] is a cofibrantly generated model category;
\item [(ii)] is a closed monoidal model category with cofibrant unit;
\item[(iii)] has cellular fixed points;
\item[(iv)] has cofibrant symmetric pushout powers.
\end{enumerate*}
\end{customthm}

We note that a similar result has also been proven by Guti\'{e}rrez-White in \cite{GW18}.

Theorem \ref{MAINEXIST1 THM}
is proven in \S \ref{MAINEXIST SEC}.
Condition (i) can be found in 
\cite[Def. 2.1.17]{Ho98},
while (ii) can be found in \cite[Def. 4.2.6]{Ho98}.
The additional conditions (iii) and (iv),
which are less standard, are discussed in 
\S \ref{FAMILY_SEC} and
\S \ref{PUSHPOW SEC}, respectively.
Further, by \textit{semi-model category}
we mean the notion in
\cite[Def. 2.1.1]{BW}\footnote{
  This agrees with the original notion of
  a \textit{$J$-semi-model category (over $\**$)} from \cite[Def. 1]{Spi01},
  as well as, e.g. 
  a \textit{semi-model category (over $\**$)} from \cite[Def. 2.2.1]{WY18}.  
  In practice, the purpose of choosing some $\mathcal M$ distinct from $\**$ in these definitions
  is that the existence of the semi-model structure on $\mathcal D$ therein
  is typically established via transfer from $\mathcal{M}$;
  however, the more powerful of these transfer theorems (e.g. \cite[Thm. 2.2.2]{WY18} and \cite[Thm. 2.2.2]{BW})
  cannot be applied to the context in this paper.
  As a final note
  we caution that, when $\mathcal{M} \neq \**$,
  the notion in \cite{Spi01}
  is more demanding than that in \cite{WY18}.
 },
which relaxes the definition of model structure by requiring that some of the axioms need only apply
if the domains of certain cofibrations are cofibrant. 

Our next result concerns the model structure
on the new category 
$\mathsf{Op}_G (\mathcal{V})$ of genuine equivariant operads
introduced in this paper. Before stating the result, we must first outline how 
$\mathsf{Op}_G (\mathcal{V})$
itself is built.
Firstly, the levels of each 
$\mathcal{P} \in \mathsf{Op}_G(\mathcal{V})$,
i.e. the $H$-sets in \eqref{GENGENMULT EQ},
are encoded by a category $\Sigma_G$ of 
\textit{$G$-corollas}, introduced in \S \ref{LRVERT SEC},
which generalizes the usual category 
$\Sigma$ of finite sets and isomorphisms.
We then define $G$-symmetric sequences by
$\mathsf{Sym}_G(\mathcal{V})=
\mathcal{V}^{\Sigma_G^{op}}$
\index{categories!of operads/symmetric sequences!SymG@$\Sym_G(\mathcal V) = \V^{\Sigma_G^{op}}$}
and,
whenever $\mathcal{V}$ is a closed symmetric monoidal category with diagonals 
(cf. Remark \ref{FINSURJ REM}),
we define in \S \ref{FGMON SEC}
a \textit{free genuine equivariant operad} monad 
$\mathbb{F}_G$ on
\index{monads!genopmonad@$\mathbb{F}_G$}
$\mathsf{Sym}_G(\mathcal{V})$
whose algebras form the desired category 
$\mathsf{Op}_G(\mathcal{V})$.

Moreover, inspired by the analogues
$\mathsf{Top}^{\mathsf{O}_{\mathcal{F}}^{op}}
	\rightleftarrows 
\mathsf{Top}^G_{\mathcal{F}}$
of the Elmendorf-Piacenza equivalence
where 
$\mathsf{Top}^{\mathsf{O}_{\mathcal{F}}^{op}}$
are partial coefficient systems determined by a family $\mathcal{F}$, 
we show in \S \ref{INDEXING_SECTION}
that (a slight generalization of)
Blumberg-Hill's indexing systems $\mathcal{F}$
give rise to sieves 
$\Sigma_{\mathcal{F}} \hookrightarrow \Sigma_G$
\index{indexing systems!SigmaF@$\Sigma_{\mathcal F} \hookrightarrow \Sigma_G$}
and partial $G$-symmetric sequences
$\mathsf{Sym_{\mathcal{F}}}(\mathcal{V})
=
\mathcal{V}^{\Sigma_{\mathcal{F}}^{op}}$
\index{categories!of operads/symmetric sequences!SymF@$\mathsf{Sym_{\mathcal{F}}}(\mathcal{F})=\mathcal{V}^{\Sigma_{\mathcal{F}}^{op}}$}.
Further, these 
$\mathsf{Sym_{\mathcal{F}}}(\mathcal{V})$
are suitably compatible with the monad
$\mathbb{F}_G$,
thus giving rise to categories
$\mathsf{Op}_{\mathcal{F}}(\mathcal{V})$
of \textit{partial genuine equivariant operads}
\index{categories!of operads/symmetric sequences!OpF@$\mathsf{Op}_{\mathcal{F}}(\mathcal{V})$}.

\begin{customthm}{II}\label{MAINEXIST2 THM}
Let $(\mathcal{V},\otimes)$
denote either 
$(\mathsf{sSet}, \times)$
or
$(\mathsf{sSet}_{\**}, \wedge)$.
Then the projective model structure on 
$\mathsf{Op}_G(\mathcal{V})$ exists. Explicitly,
a map $\mathcal{P} \to \mathcal{P}'$ is a weak equivalence (resp. fibration) if all maps
\begin{equation}\label{GENEQTHM EQ}
\mathcal{P}(C)
	\to
\mathcal{P}'(C)
\end{equation}
are weak equivalences (fibrations) in $\mathcal{V}$
for each $C \in \Sigma_G$.

More generally, for $\mathcal{F}$ a weak indexing system, the projective model structure on 
$\mathsf{Op}_{\mathcal{F}}(\mathcal{V})$ exists. Explicitly, weak equivalences (resp. fibrations) are
determined by \eqref{GENEQTHM EQ}
for $C \in \Sigma_{\mathcal{F}}$.

Lastly, 
analogous 
semi-model structures on
$\mathsf{Op}_G(\mathcal{V})$,
$\mathsf{Op}_{\mathcal{F}}(\mathcal{V})$
exist
provided that
$(\mathcal{V},\otimes)$:
\begin{enumerate*}
\item[(i)] is a cofibrantly generated model category;
\item [(ii)] is a closed monoidal model category with cofibrant unit;
\item[(iii)] has cellular fixed points;
\item[(iv)] has cofibrant symmetric pushout powers;
\item[(v)] has diagonals.
\end{enumerate*}
\end{customthm}

Theorem \ref{MAINEXIST2 THM} is proven in 
\S \ref{MAINEXIST SEC}
in parallel with Theorem \ref{MAINEXIST1 THM}.
We note that the condition (v) that 
$(\mathcal{V},\otimes)$ has diagonals (cf. Remark \ref{FINSURJ REM}), which is not needed in 
Theorem \ref{MAINEXIST1 THM}, is required to build
the monad $\mathbb F_G$, and hence the categories
$\mathsf{Op}_G(\mathcal{V})$,
$\mathsf{Op}_{\mathcal{F}}(\mathcal{V})$.

The following is our main result.
The explicit formulas for
the functors 
$\iota^{\**},\iota_{\**}$
are found in \eqref{IOTAFUNS EQ}
(also, see Corollaries \ref{TWOADJOINTSOP_COR}
and \ref{TWOADJOINTSOPF COR}).

\begin{customthm}{III}\label{MAINQUILLENEQUIV THM}
Let $(\mathcal{V},\otimes)$
denote either 
$(\mathsf{sSet}, \times)$
or
$(\mathsf{sSet}_{\**}, \wedge)$.

Then the adjunctions, 
where in the more general rightmost case 
$\mathcal{F}$ is a weak indexing system,
\begin{equation}
\begin{tikzcd}[column sep =4em]
	\mathsf{Op}_G(\mathcal{V}) \ar[shift left=1.5]{r}{\iota^{\**}} 
	&
	\mathsf{Op}^G(\mathcal{V})
	\ar[shift left=1.5]{l}{\iota_{\**}},
&
	\mathsf{Op}_{\mathcal F}(\mathcal{V}) 
	\ar[shift left=1.5]{r}{\iota^{\**}} 
	&
	\mathsf{Op}^G_{\mathcal{F}}(\mathcal{V})
	\ar[shift left=1.5]{l}{\iota_{\**}}.
\end{tikzcd}
\end{equation}
are Quillen equivalences.

Morover, 
analogous Quillen equivalences of
semi-model structures\footnote{
	An adjunction 
	$L \colon \mathcal{C} 
	\rightleftarrows 
	\mathcal{D} \colon R$
	between semi-model categories
	is a Quillen equivalence if 
	$R$ preserves (trivial) fibrations and,
	for $A\in \mathcal{C}$ cofibrant and
	$X \in \mathcal{D}$ fibrant,
	$LA \to X$ is a weak equivalence iff 
	$A \to RX$ is.
}
$\mathsf{Op}_{\mathcal F}(\mathcal{V}) \simeq
\mathsf{Op}^G_{\mathcal{F}}(\mathcal{V})$
exist
provided that
$(\mathcal{V},\otimes)$:
\begin{enumerate*}
\item[(i)] is a cofibrantly generated model category;
\item [(ii)] is a closed monoidal model category with cofibrant unit;
\item[(iii)] has cellular fixed points;
\item[(iv)] has cofibrant symmetric pushout powers;
\item[(v)] has diagonals;
\item[(vi)] has cartesian fixed points.
\end{enumerate*}
\end{customthm}

Theorem \ref{MAINQUILLENEQUIV THM}
is proven in 
\S \ref{MAINTHM_PROOF_SECTION}.
Condition (vi), 
which is not needed in either of
Theorems \ref{MAINEXIST1 THM},\ref{MAINEXIST2 THM},
is discussed in 
\S \ref{PUSHPOW SEC}.

Lastly, our techniques also verify 
the main conjecture of \cite{BH15},
which we discuss in \S \ref{NINFTY_SECTION}.
Moreover, we note that our models for
$N \mathcal{F}$-operads are given by explicit bar constructions.

\begin{customcor}{IV}\label{NINFTY_REAL_COR_MAIN}
For $\V = \sSet$ or $\mathsf{Top}$ and 
$\mathcal{F} = \{\mathcal{F}_n\}_{n \geq 0}$
any weak indexing system,
$N \mathcal{F}$-operads exist. That is, there exist explicit operads $\O$
such that
\begin{equation}\label{NFINFTY2 EQ}
	\O(n)^{\Gamma} \sim 
\begin{cases}
	\** & \text{if } \Gamma \in \mathcal{F}_n
\\
	\emptyset & \text{otherwise}.
\end{cases}
\end{equation}  
In particular, the map $\mathrm{Ho}(N_\infty$-$\Op) \to \mathcal I$
in \cite[Cor. 5.6]{BH15}
 is an equivalence of categories.
 
 Moreover, if $\mathcal{O}'$ has fixed points as in 
 \eqref{NFINFTY2 EQ} for some collection of graph subgroups 
 $\mathcal{F} = \{\mathcal{F}_n\}_{n \geq 0}$, then 
 $\mathcal{F}$ must be a weak indexing system.
\end{customcor}


\subsection{Context, applications and future work}
\label{CONTEX SEC}

\subsubsection*{Models for equivariant operads with norm maps}

This article is closely linked to the authors project in
\cite{Pe17,BP_edss,BP_HGOP,BP_TAS},
which culminates in the existence of a Quillen equivalence
\cite[Thm. I]{BP_TAS}
\begin{equation}\label{DSETSOP_EQ}
\begin{tikzcd}[column sep =4em]
	\mathsf{dSet}^G \ar[shift left=1.5]{r} 
	&
	\mathsf{sOp}_{\bullet}^G.
\ar[shift left=1.5]{l}
\end{tikzcd}
\end{equation} 
Here $\mathsf{sOp}_{\bullet}^G = \Op_\bullet(\sSet)^G$
is the model category of $G$-equivariant colored simplicial operads
with norm maps (in $\mathsf{sSet}$)
given by \cite[Thm. III]{BP_HGOP},
which is the colored extension of Theorem \ref{MAINEXIST1 THM},
while 
$\mathsf{dSet}^G = \mathsf{Set}^{\Omega^{op} \times G}$
(for $\Omega$ the category of trees)
is the model category of
$G$-$\infty$-operads \cite[Thm. 2.1]{Pe17},
whose model structure is defined using the 
category $\Omega_G$ of $G$-trees.

The equivalence \eqref{DSETSOP_EQ}
generalizes the 
equivalence
$\mathsf{dSet} \rightleftarrows \mathsf{sOp}_{\bullet}$
\cite[Thm. 8.15]{CM13b},
which culminates the project
of Cisinski, Moerdijk, Weiss in \cite{MW07,MW08,CM11,CM13a,CM13b}.
Crucially, we note that, 
while the underlying categories in \eqref{DSETSOP_EQ}
are obtained from those in 
\cite[Thm. 8.15]{CM13b}
by taking $G$-objects,
the model structures in \eqref{DSETSOP_EQ}
are more subtle, needing the use of $G$-trees. 

As a result, when generalizing the arguments and constructions in
\cite{MW07,MW08,CM11,CM13a,CM13b}
one must often think in terms of genuine operads.
For example, 
in \cite[\S 8.2]{Pe17},
to understand the homotopy theory of $G$-$\infty$-operads
in $\mathsf{dSet}^G$,
which are ``$G$-operads with norm maps up to homotopy'',
one considers objects \cite[Not. 8.11]{Pe17} in 
$\mathsf{dSet}_G = \mathsf{Set}^{\Omega_G^{op}}$
that are ``genuine operads up to homotopy''.
Similarly, in
\cite[\S 5]{BP_edss} one must consider the
``homotopy genuine operad of a Segal space'' 
\cite[Def. 5.8]{BP_edss}
(see Remark \ref{COLEREDGEN REM}).
Heuristically, this need for genuine operads comes from the 
observation that taking fixed points does not commute with taking 
homotopy/isomorphism classes
(compare with \eqref{ALGTREECOM EQ} below).
As such, while the categories in \eqref{DSETSOP_EQ}
are ``described in terms of $\Omega^{op} \times G$'',
any construction in those categories involving homotopy classes
needs to be ``described in terms of $\Omega_G^{op}$'',
so as to account for norm maps.

\begin{remark}\label{COLEREDGEN REM}
	To simplify our discussion, 
	this paper focuses only on the theory of 
	\emph{single colored} (genuine) equivariant operads.
	This is due to technical subtleties that emerge in the colored context, 
	such as the fact that colored genuine operads
	have a \emph{coefficient system} of objects 
	rather than just a $G-$set of objects.
	
	In \cite{BP_edss} we give an alternative description 
	of genuine operads (of sets) \cite[Def. 3.35]{BP_edss},
	which includes the colored case, 
	as the objects of $\mathsf{dSet}_G$
	satisfying a strict Segal condition.
	The connection between the two descriptions
	is given by the nerve functor in Theorem \ref{NERVE THM}.
\end{remark}

\subsubsection*{Algebras over genuine operads}

Just like usual operads, genuine operads admit a notion of algebra.
The full formal definition of such algebras is forthcoming, 
but the following example illustrates the main idea.

\begin{example}\label{GENALG EX}
	Let $G = \mathbb{Z}_{/2}$,
	$\mathcal{O} \in \mathsf{sOp}^G$
	be a $G$-operad
	and $X \in \mathsf{sSet}^G$
	be an $\mathcal{O}$-algebra\footnote{Note that  
		the algebra structure is $G$-equivariant.
		That is,
		for $p \in \mathcal{O}(n)$,
		$x_i \in X$, $g\in G$,
		one has
		$(gp)(gx_1,\cdots,gx_n)
		=g\left(p(x_1,\cdots,x_n)\right)
		$.}.
	It is immediate that 
	$\pi_0(X)$ is a $\pi_0(\mathcal{O})$-algebra
	while 
	$\pi_0(X^G)$ is a $\pi_0(\mathcal{O}^G)$-algebra.
	However, the sets $\pi_0(X)$ and $\pi_0(X^G)$ admit additional structure.
	Consider the following $G$-tree
\begin{equation}\label{ALGTREEEX EQ}
\begin{tikzpicture}
[grow=up, auto, every node/.style = {font=\footnotesize},level distance = 2.5em]%
\tikzstyle{level 2}=[sibling distance=3.75em]%
\tikzstyle{level 3}=[sibling distance=1.5em]%
\node at (10,0) [font=\normalsize]  {$T$}%
child{node [dummy] {}
	child{node [dummy,fill=black] {}
		edge from parent node [swap] {$G \cdot a$}
	}
	edge from parent node [swap] {$(G/G)\cdot r$}
};
\tikzstyle{level 2}=[sibling distance=3em]%

\node at (6,0) [font=\normalsize] {$T$}
child{node [dummy] {}
	child{node [dummy,fill=black] {}
		edge from parent node [swap, near end] {$a+1$}
	}
	child{node [dummy,fill=black] {}
		edge from parent node [near end] {$\phantom{-}a$}
	}
	edge from parent node [swap] {$r$}
};
\end{tikzpicture}
\end{equation}
where we regard white (resp. black) nodes as corresponding to $\mathcal{O}$ (resp. $X$).
Then, as in \eqref{INTFIXPTCOMP EQ},
the tree in \eqref{ALGTREEEX EQ}
encodes a multiplication
(note that $\Gamma_G \leq G \times \Sigma_2$ is the diagonal)
\begin{equation}\label{ALGTREECOM EQ}
	\begin{tikzcd}[row sep=0]
	\pi_0\left(\mathcal{O}(2)^
	{\Gamma_{G}}\right)
	\times
	\pi_0(X)
	\ar{r}
	&
	\pi_0\left(X^{G}\right)
\\
	{([p],[x])}
	\ar[mapsto]{r}
	&
	{[p(x,x+1)]}
	\end{tikzcd}
\end{equation}
More generally, analogues of
\eqref{ALGTREECOM EQ} are obtained for any 
$G$-tree which, as in \eqref{ALGTREEEX EQ},
has a single white node topped by black $0$-ary nodes.

Writing 
$\pi_0 (\iota_{\**} X) \in \mathsf{Set}^{\mathsf{O}_G^{op}}$
for the coefficient system
$H \mapsto \pi_0 (X^H)$,
the multiplications \eqref{ALGTREECOM EQ} for such $G$-trees  
describe the algebra structure of 
$\pi_0 (\iota_{\**} X)$
over the genuine operad
$\pi_0 \left(\iota_{\**} \mathcal{O} \right) \in
\mathsf{Op}_G(\mathsf{Set})$,
where $\iota_{\**}$ is now as in Theorem \ref{MAINQUILLENEQUIV THM}.
\end{example}

\begin{remark}
	One way to formalize algebras over genuine operads
	is to adapt the composition product $\circ$ on symmetric sequences
	\cite[Def. 1.4]{GJ94}
	to a product on $G$-symmetric sequences $\mathsf{Sym}_G(\mathcal{V}) = 
	\mathcal{V}^{\Sigma_G^{op}}$
	(here $\Sigma_G$ is the category of $G$-corollas, 
	cf. Definition \ref{GSYMCAT DEF}).
	Loosely, this $G$-composition product is encoded by $G$-trees as in \eqref{ALGTREEEX EQ} but where black nodes need not be $0$-ary.
	
	The composition product approach has some advantages over the
	``free genuine operad monad on $\mathsf{Sym}_G(\mathcal{V})$''
	approach described in \S \ref{FGMON SEC}.
	Namely, one can both define genuine operads as ``algebras over $\circ$''
	and algebras over genuine operads as ``left modules in arity $0$''.
	However, it is hard to describe 
	\emph{free} (genuine) operads using the composition product, 
	making such an approach poorly suited 
	for proving Theorems \ref{MAINEXIST1 THM} and \ref{MAINEXIST2 THM}.
\end{remark}

Based on Theorem \ref{MAINQUILLENEQUIV THM}
and the Elmendorf-Piacenza theorem in \eqref{COFADJINT EQ},
we conjecture the following.

\begin{conjecture}\label{EPGEN CONJ}
	Let $\mathcal{V}$ be as in 
	Theorem \ref{MAINQUILLENEQUIV THM},
	$\mathcal{O} \in \mathsf{Op}^G(\mathcal{V})$
	be a suitably cofibrant $G$-operad,
	and $\iota_{\**}\mathcal{O} \in \mathsf{Op}_G(\mathcal{V})$
	be the associated genuine operad. Then the adjunction 
	\eqref{COFADJINT EQ} lifts to a Quillen adjunction
\[
\begin{tikzcd}[column sep =4em]
	\mathsf{Alg}_{\iota_{\**} \mathcal{O}}
	\ar[shift left=1.5]{r}{\iota^{\**}} 
&
	\mathsf{Alg}_{\mathcal{O}}
	\ar[shift left=1.5]{l}{\iota_{\**}}.
\end{tikzcd}
\]
\end{conjecture}

The key ingredient needed to establish Conjecture \ref{EPGEN CONJ}
is an analogue of Lemma \ref{MAINLEM LEM}, which we believe holds by a similar analysis.

\begin{remark}\label{ALGELMCOL REM}
	We also expect Conjecture \ref{EPGEN CONJ}
	to hold for $\mathcal{O}$
	a $G$-equivariant colored operad,
	although, in light of Remark \ref{COLEREDGEN REM},
	defining the genuine colored operad $\iota_{\**} \mathcal{O}$
	requires care.
	
	In fact, it turns out that 
	Theorem \ref{MAINQUILLENEQUIV THM} 
	is essentially a particular case of Conjecture \ref{EPGEN CONJ}
	in the colored case, as follows.
	First, and non-equivariantly, single colored operads
	$\mathsf{Op}(\mathcal{V})$
	are the algebras over a colored operad $\mathcal{T}$,
	described by trees
	($\mathcal{T}$ specializes the operad $S^C$ in \cite[\S 3.2]{GV12}
	for the set of colors $C= \{\**\}$).
	Adapting the construction in \cite[\S 3.2]{GV12},
	one obtains a colored genuine operad 
	$\mathcal{T}_G$, described by $G$-trees, 
	whose algebras are $\mathsf{Op}_G(\mathcal{V})$.
	Moreover,
	$\mathcal{T}_G = \iota_{\**} \mathcal{T}^{\mathsf{fr}}_G$,
	where $\mathcal{T}^{\mathsf{fr}}_G$ is
	a $G$-equivariant colored operad, 
	described by free $G$-trees, 
	again adapting \cite{GV12}.
	In addition, 
	giving $\mathcal{T}$ the trivial $G$-action,
	there is a $G$-equivariant map 
	$\mathcal{T} \to \mathcal{T}^{\mathsf{fr}}_G$ 
	which, while not an isomorphism, 
	induces an equivalence of categories of algebras.
	
	Summarizing the above, 
	Theorem \ref{MAINQUILLENEQUIV THM} 
	hence verifies Conjecture \eqref{EPGEN CONJ}
	for the operad $\mathcal{T}^{\mathsf{fr}}_G$. 
	
	For some extra detail, 
	including a more detailed description of $\mathcal{T}^{\mathsf{fr}}_G$ ,
	see Remark \ref{SOMEMOREDET REM},
	which discusses the role that the map
	$\mathcal{T} \to \mathcal{T}^{\mathsf{fr}}_G$ 
	plays ``behind the scenes''  in \S \ref{COMPARISON_REGULAR_SECTION}.
\end{remark}

\subsubsection*{Comparison with parametrized $G$-$\infty$-operads}

The comparison between simplicial $G$-operads $\sOp^G$ and the parametrized $G$-$\infty$-operads of \cite{BDGNS}
factors most naturally through the category of genuine $G$-operads $\sOp_G$.
Non-equivariantly, this comparison is given by
the operadic nerve functor $N^\otimes \colon \sOp \to \Op_\infty$ \cite[Def. 2.1.1.3]{Lu09}.
This construction first converts a simplicial operad $\O$ into
a simplicial category
$\O^\otimes \to \Fin_{\**}$
equipped with a functor to the category of pointed finite sets, which behaves like a fibration over a certain wide subcategory,
and then takes the homotopy coherent nerve
$hcN(\O^\otimes) = N^\otimes(\O)$.
This process motivates Lurie's definition of an $\infty$-operad in $\sSet$.

In \cite{BNerve}, the first author generalizes this process,
by first building, from a genuine equivariant operad $\P$, 
a simplicial category
$\P^{\otimes} \to \underline{\Fin}_{\**,G}$
equipped with a partial fibration to the coefficient system of pointed finite $G$-sets (e.g. \cite[Def. 3.3]{BNerve}),
and then showing that the homotopy coherent nerve $hcN(\P^{\otimes}) = N^{\otimes \P}$
yields a $G$-$\infty$-operad in the sense of \cite{BDGNS}.
Moreover, this transformation induces a functor on the categories of algebras $\Alg(\O) \to \Alg(N^{\otimes}\O)$.
This has been applied by Horev in \cite{Hor} when $\O = \mathcal D_V$ is the equivariant little disks operad over a $G$-representation $V$.
Specifically, he shows \cite[\S 3.9]{Hor} that $N^\otimes(\mathcal D_V)$ is equivalent to the
$G$-$\infty$-operad of $V$-framed representations,
which allows for $\mathcal D_V$-algebras to be used as input into his genuine equivariant factorization homology machinery,
in particular producing new notions of equivariant topological Hochschild homology.


\subsection{Outline}

This paper is comprised of two major halves, 
with 
\S \ref{PLANAR_SECTION},\S \ref{GENUINE_OP_MONAD_SECTION}
addressing the definition of the novel structure of
genuine equivariant operads,
and 
\S \ref{FREE_EXTENSIONS_SECTION},\S \ref{COFIB SEC}
addressing the proofs of the main results,
Theorems \ref{MAINEXIST1 THM},\ref{MAINEXIST2 THM},\ref{MAINQUILLENEQUIV THM}.
A more detailed outline follows.

\S \ref{PRELIM_SECTION}
discusses some preliminary notions and notation that will be used throughout.
Of particular importance are the notions of 
split Grothendieck fibrations,
which we recall in \S \ref{GROTHFIB REF},
and the categorical wreath product defined in 
\S \ref{WREATH SEC}, which we use to define
symmetric monoidal categories with diagonals
(Remark \ref{FINSURJ REM}).

\S \ref{PLANAR_SECTION} lays the groundwork for the definition of genuine equivariant operads in 
\S \ref{GENUINE_OP_MONAD_SECTION} by discussing the concept of node substitution (which is at the core of the definition of free operads)
in the context of equivariant trees. The key idea, which is captured in diagram
\eqref{SUBSDATUMTREES EQ} and Proposition \ref{SUBSASPULL PROP},
is that such substitution data are encoded by special maps of $G$-trees that we call planar tall maps. The bulk of the section is spent studying these types of maps, culminating in the concept of planar strings in \S \ref{PLANARSTRING SEC}, which encode iterated substitution.

\S \ref{GENUINE_OP_MONAD_SECTION} then uses planar strings to provide the formal definition of the category 
$\mathsf{Op}_G(\mathcal{V})$
of genuine equivariant operads in a 
two step process in 
\S \ref{MONSPAN SEC} and \S \ref{FGMON SEC}.
\S \ref{COMPARISON_REGULAR_SECTION} compares 
the genuine equivariant operad category
$\mathsf{Op}_G(\mathcal{V})$
with the usual equivariant operad category
$\mathsf{Op}^G(\mathcal{V})$,
establishing the necessary adjunction to formulate
Theorem \ref{MAINQUILLENEQUIV THM}.
\S \ref{INDEXING_SECTION}
discusses the notion of partial genuine equivariant operads, which are very closely related to the indexing systems of Blumberg-Hill.

\S \ref{FREE_EXTENSIONS_SECTION} 
proves 
Theorems \ref{MAINEXIST1 THM} and \ref{MAINEXIST2 THM}.
As is often the case when proving existence of projective model structures,
the key to this section is a careful analysis of  the free extensions in $\mathsf{Op}_G$ as in diagram
\eqref{FREE_FG_EXT_EQ},
with 
\S \ref{LABELSTRI SEC}, 
\S \ref{EXTTREE SEC},
\S \ref{FILTRATION_SECTION}
dedicated to providing a suitable filtration of such free extensions,
and \S \ref{MAINEXIST SEC} concluding the proofs.

\S \ref{COFIB SEC} proves our main result,
Theorem \ref{MAINQUILLENEQUIV THM}.
The core of the technical analysis 
is given in 
\S \ref{FAMILY_SEC},
\S \ref{PUSHPOW SEC}
and \S \ref{G_GRAPH_SECTION}, 
which carefully study the interplay between families of subgroups, fixed points, and pushout products,
and provide the necessary ingredients for
the characterization of the cofibrant objects
in $\mathsf{Op}_G (\mathcal{V})$
given in Lemma \ref{MAINLEM LEM},
and from which 
Theorem \ref{MAINQUILLENEQUIV THM}
easily follows.
\S \ref{NINFTY_SECTION}
then establishes Corollary 
\ref{NINFTY_REAL_COR_MAIN}
by using the theory of genuine equivariant operads
to build explicit cofibrant models for 
$N \mathcal{F}$-operads.

Appendix \ref{TRANSKAN AP}
provides the proof of a lengthy technical result needed when establishing the filtrations in \S \ref{FREE_EXTENSIONS_SECTION}.

Lastly, Appendix \ref{NERVE AP} proves Theorem \ref{NERVE THM},
which compares the description of genuine operads used in this paper 
with the description used in \cite{BP_edss}.

\section{Preliminaries}
\label{PRELIM_SECTION}


\subsection{Grothendieck fibrations}\label{GROTHFIB REF}

Recall that a functor 
$\pi \colon \mathcal{E} \to \mathcal{B}$
is called a \textit{Grothendieck fibration} \cite[\S 8.1]{Bo94}
\index{Grothendieck fibrations!AGrothendieckFibration@$\pi \colon \mathcal{E} \to \mathcal{B}$}
if, for every arrow 
$f \colon b' \to b$ in $\mathcal{B}$
and $e \in \mathcal{E}$ such that $\pi(e)=b$,
there exists a \emph{cartesian arrow}
$f^{\**} e \to e$ lifting $f$,
\index{Grothendieck fibrations!Cartesian@$f^{\**} e \to e$}
i.e. an arrow such that for any choice of horizontal arrows 
\[
\begin{tikzcd}[row sep=7pt]
	e'' \ar{rr} \ar[dashed]{rd}[swap]{\exists !} & & 
	e
&
	b'' \ar{rr} \ar{rd} & & 
	b
\\
	& f^{\**} e \ar{ru} &
&
	& b' \ar{ru}[swap]{f} &
\end{tikzcd}
\]
for which the rightmost diagram commutes and 
$e'' \to e$ lifts $b'' \to b$,
there exists a unique dashed arrow
$e'' \to f^{\**} e$ lifting $b'' \to b'$ and making the leftmost diagram commute.

In most contexts the cartesian arrows $f^{\**} e \to e$ are assumed to be defined only up to unique isomorphism.
However, in all examples considered in this paper,
we will be able to identify preferred choices of cartesian arrows, and we will refer to those preferred choices as \textit{pullbacks}.
Moreover, pullbacks will be compatible with composition and units in the obvious way, i.e. $g^{\**} f^{\**} e = (fg)^{\**}e$ and $id_b^{\**}e=e$.
On a terminological note, 
the data of a Grothendieck fibration together with 
such choices of pullbacks is sometimes called a 
\textit{split fibration}, but we will have no need to distinguish the two concepts outside
of the present discussion.

A map of Grothendieck fibrations (resp. split fibrations) is then a commutative diagram
\begin{equation}
\begin{tikzcd}\label{GROTHFIBMAP EQ}
	\mathcal{E} \ar{rr}{\delta} \ar{rd}[swap]{\pi} &&
	\bar{\mathcal{E}} \ar{dl}{\bar{\pi}}
\\
	& \mathcal{B}
\end{tikzcd}
\end{equation}
such that $\delta$ preserves cartesian arrows (resp. pullbacks).

There is a well known equivalence between Grothendieck fibrations over $\mathcal{B}$ and contravariant pseudo-functors
$\mathcal{B}^{op} \to \mathsf{Cat}$,
with split fibrations corresponding to (regular) contravariant functors. We recall how this works in the split case, starting with the covariant version.

\begin{definition}\label{GROTHCONS DEF}
Given a small category $\mathcal{B}$ and functor $\mathcal{E}_{\bullet}$
\begin{equation}
\begin{tikzcd}[row sep = 0em]
	\mathcal{B} \ar{r}{\mathcal{E}_{\bullet}} & \mathsf{Cat} \\
	b \ar[mapsto]{r} & \mathcal{E}_b
\end{tikzcd}
\end{equation}
the \textit{covariant Grothendieck construction}
$\mathcal{B} \ltimes \mathcal{E}_{\bullet}$ (over $\mathcal B$)
\index{Grothendieck fibrations!GrothendieckConstruction@$\mathcal{B} \ltimes \mathcal{E}_{\bullet}$}
has objects pairs $(b,e)$ with $b \in \mathcal{B}$,
$e \in \mathcal{E}_b$ and 
arrows $(b,e) \to (b',e')$ given by pairs
\[(f\colon b \to b', g \colon f_{\**}(e) \to e'),\]
where $f_{\**}\colon \mathcal{E}_b \to \mathcal{E}_{b'}$ is a shorthand for the functor $\mathcal{E}_{\bullet}(f)$.

Note that the chosen pushforward of $(b,e)$ along 
$f \colon b \to b'$ is then $(b',f_{\**} e)$.

Further, for a contravariant functor
$\mathcal{E}_{\bullet} \colon
\mathcal{B}^{op} \to \mathsf{Cat}$,
the \textit{contravariant Grothendieck construction} is
$(\mathcal{B}^{op} \ltimes 
\mathcal{E}_{\bullet}^{op})^{op}$
(over $\mathcal B$).
\end{definition}

One useful property of Grothendieck fibrations
$\pi \colon \mathcal{E} \to \mathcal{B}$
is that right Kan extensions can be computed using fibers, i.e., 
given a functor $F \colon \mathcal{E} \to \mathcal{V}$ into a complete category $\mathcal{V}$ one has
\begin{equation}\label{FIBERKAN EQ}
	\mathsf{Ran}_{\pi}F (b)
\simeq
	\mathsf{lim} F{|_{b \downarrow \mathcal{E}}}
\simeq
	\mathsf{lim} F|_{\mathcal{E}_b}
\end{equation}
where the first identification is the usual pointwise formula for Kan extensions (cf. \cite[X.3 Thm. 1]{McL}),
and the second identification follows by noting that,
due to the existence of cartesian arrows,
the fibers $\mathcal{E}_b$ are initial (in the sense of \cite[IX.3]{McL})
in the undercategories $b \downarrow \mathcal{E}$.
In fact, a little more is true: a choice of cartesian arrows 
yields a right adjoint to the inclusion
$\mathcal{E}_b \hookrightarrow b \downarrow \mathcal{E}$, so that $\mathcal{E}_b$ is a coreflexive subcategory of 
$b \downarrow \mathcal{E}$,
a well known sufficient condition for initiality.
In practice, we will also need a generalization of the Kan extension formula \eqref{FIBERKAN EQ} for maps of Grothendieck fibrations as in \eqref{GROTHFIBMAP EQ}.
Keeping the notation therein, 
given $\bar{e} \in \bar{\mathcal{E}}$ we write 
$\bar{e} \downarrow_{\mathcal{B}} \mathcal{E} \hookrightarrow
\bar{e} \downarrow \mathcal{E}$
\index{Grothendieck fibrations!undercategory@$\bar{e} \downarrow_{\mathcal{B}} \mathcal{E} \hookrightarrow \bar{e} \downarrow \mathcal{E}$}
for the full subcategory of those pairs 
$\left(e,f \colon \bar{e} \to \delta(e)\right)$
such that $\bar{\pi}(f) = id_{\bar{\pi}(\bar{e})}$.

\begin{proposition}\label{FIBERKANMAP PROP}
	Given a map of Grothendieck fibrations
	as in \eqref{GROTHFIBMAP EQ},
	each subcategory $\bar{e} \downarrow_{\mathcal{B}} \mathcal{E}$
	for $\bar{e} \in \bar{\mathcal{E}}$
	is an initial subcategory of $\bar{e} \downarrow \mathcal{E}$
	and hence for each functor 
	$\mathcal{E} \to \mathcal{V}$
	with $\mathcal{V}$ complete one has
\begin{equation}\label{FIBERKANMAP EQ}
	\mathsf{Ran}_{\delta}F (\bar{e})
\simeq
	\mathsf{lim} F{|_{\bar{e} \downarrow \mathcal{E}}}
\simeq
	\mathsf{lim} F|_{\bar{e} \downarrow_{\mathcal{B}} \mathcal{E}}.
\end{equation}	
\end{proposition}

\begin{proof}
One readily checks that the assignment
$
	(e,f\colon\bar{e} \to \delta(e))
\mapsto
	\left(
	\pi(f)^{\**}e, \bar{e} \to \delta \pi(f)^{\**}(e)
	\right)
$
(where $\delta \pi(f)^{\**} = \bar{\pi}^{\**}(f) \delta$) is  right adjoint to the inclusion
$\bar{e} \downarrow_{\mathcal{B}} \mathcal{E} \hookrightarrow
\bar{e} \downarrow \mathcal{E}$, so that the claim follows by coreflexivity 
(note that if we are not in the split case, pullbacks may be chosen arbitrarily).
\end{proof}



We also record the following, the proof of which is straightforward.

\begin{proposition}\label{GROTHSTAB PROP}
	Suppose that $\mathcal{E} \to \mathcal{B}$ is a (split) Grothendieck fibration. Then so is the map of functor categories 
	$\mathcal{E}^{\mathcal{C}} \to \mathcal{B}^{\mathcal{C}}$ for any category $\mathcal{C}$,
        as well as the map 
	$\bar{\mathcal{E}} \to \bar{\mathcal{B}}$ in any pullback of categories
        \[
              \begin{tikzcd}
                    \bar{\mathcal{E}} \ar{r} \ar{d} & \mathcal{E} \ar{d}
                    \\
                    \bar{\mathcal{B}} \ar{r} & \mathcal{B}.
              \end{tikzcd}
        \]	
\end{proposition}

\subsection{Wreath product over finite sets}
\label{WREATH SEC}

Throughout we will let $\Fin$\index{wreath products!FIN@$\mathsf{F}$}
denote the usual skeleton of the category of (ordered) finite sets and all set maps. Explicitly, its objects are the finite sets $\{1,2,\cdots,n\}$ for $n\geq 0$.

\begin{definition}\index{wreath products!FIN@$\mathsf{F}\wr (-)$}
	For a category $\C$, we write 
	$\Fin \wr \C = (\Fin^{op} \ltimes (\C^{op})^{\times \bullet})^{op}$ 
	for the contravariant Grothendieck construction (cf. Definition \ref{GROTHCONS DEF}) of the functor
\[
\begin{tikzcd}[row sep=0pt]
	\Fin^{op} \ar{r} & \mathsf{Cat}
\\
	I \ar[r,mapsto] & \C^{\times I}
\end{tikzcd}	
 \]
Explicitly, the objects of $\Fin \wr \C$ are tuples $(c_i)_{i \in I}$ and a map 
$(c_i)_{i \in I} \to (d_j)_{j \in J}$ consists of a pair 
\[(\phi \colon I \to J, (f_i\colon c_i \to d_{\phi(i)})_{i\in I}),\]
 henceforth abbreviated as $(\phi,(f_i))$.
\end{definition}

\begin{remark}\label{WREATHFIXED REM}
Let $(c_i)_{i \in I} \in \Fin \wr \mathcal{C}$
and write $\lambda$ for the partition 
$I = \lambda_1 \amalg \cdots \amalg \lambda_k$
such that $1 \leq i_1, i_2 \leq n$ are in the same class iff
$c_{i_1}, c_{i_2} \in \mathcal{C}$ are isomorphic.
Writing 
$\Sigma_{\lambda} = \Sigma_{\lambda_1} \times \cdots \times
\Sigma_{\lambda_k}$
and picking representatives $i_j \in \lambda_j$,
the automorphism group of  
$(c_i)_{i \in I}$ is given by
\begin{equation}
	\mathsf{Aut}\left( (c_i)_{i \in I} \right)
\simeq
	\Sigma_{\lambda} \wr \prod_{i} \mathsf{Aut}(c_i)
\simeq
	\Sigma_{|\lambda_1|} \wr 
	\mathsf{Aut}(c_{i_1})
		\times \cdots \times	
	\Sigma_{|\lambda_k|} \wr 
	\mathsf{Aut}(c_{i_k}).
\end{equation}
\end{remark}

\begin{notation}
\index{wreath products!DELTAI@$\delta^i$}
\index{wreath products!SIGMAI@$\sigma^i$}
      \label{FIN_COA_COS_NOT}
Using the coproduct functor $\Fin^{\wr 2} = \Fin^{\wr \{{0,1\}}} =\Fin \wr \Fin \xrightarrow{\amalg} \Fin$ (where $\coprod_{i\in I} J_i$ is ordered lexicographically) and the singleton $\{1\} \in \Fin$,
one can regard the collection of categories 
$\Fin^{\wr n+1 }\wr \C = \Fin^{\wr \{0,\cdots,n\}} \wr \C $ for $n \geq -1$
 as a coaugmented cosimplicial object in $\mathsf{Cat}$.
As such, we will denote by
\[
	\delta^i\colon \Fin^{\wr n } \wr \C \to \Fin^{n+1} \wr \C, \qquad 0 \leq i \leq n
\]
the cofaces obtained by inserting singletons $\{1\} \in \Fin$ and by 
\[
	\sigma^i \colon \Fin^{n+2} \wr \C \to \Fin^{n+1} \wr \C, \qquad 0 \leq i \leq n
\]
the codegeneracies obtained by applying the coproduct 
$\Fin^{\wr 2} \xrightarrow{\amalg} \Fin$ to adjacent 
$\Fin$ coordinates.

Further, note that there are identifications
$\Fin \wr \delta^{i} = \delta^{i+1}$, 
$\Fin \wr \sigma^{i} = \sigma^{i+1}$.
\end{notation}

\begin{remark}
	If $\mathcal{V}$ has all finite coproducts then injections and fold maps assemble into a functor as on the left below.
Dually, if $\mathcal{V}$ has all finite products then projections and diagonals assemble into a functor as on the right.
\begin{equation}\label{WREATHPROD EQ}
\begin{tikzcd}[row sep=0pt]
	\Fin \wr \mathcal{V} \ar{r}{\coprod} & \mathcal{V} & &
	(\Fin \wr \mathcal{V}^{op})^{op} \ar{r}{\prod} & \mathcal{V}
\\
	(v_i)_{i \in I} \ar[mapsto]{r} & \coprod_{i \in I}{v_i} & &
	(v_i)_{i \in I} \ar[mapsto]{r} & \prod_{i \in I}{v_i}
\end{tikzcd}
\end{equation}

Moreover, these functors satisfy a number of additional coherence conditions.
Firstly, there is a natural isomorphism $\alpha$ as on the left below
\begin{equation}\label{COHER EQ}
\begin{tikzcd}
	\Fin^{\wr 2} \wr \mathcal{V} 
	\ar{r}{\Fin \wr \coprod} \ar{d}[swap]{\sigma^0}&
	\Fin \wr \mathcal{V} \ar{r}{\coprod} &
	|[alias=Vt]|
	\mathcal{V} \ar[equal]{d}
& &
	\mathcal{V} \ar{d}[swap]{\delta^0} \ar[equal]{rd}
\\
	|[alias=FV]|
	\Fin \wr \mathcal{V} \ar{rr}[swap]{\coprod} &&
	\mathcal{V}
& &
	\Fin \wr \mathcal{V} \ar{r}[swap]{\coprod} &
	\mathcal{V}	
\arrow[Leftrightarrow, from=Vt, to=FV,shorten <=0.10cm,shorten >=0.10cm,"\alpha"]
\end{tikzcd}
\end{equation}
that encodes both reparenthesizing of coproducts and removal of initial objects 
(note that the empty tuple $()_{i \in \emptyset}\in \Fin \wr \mathcal{V}$ is mapped under $\coprod$ to an initial object of $\mathcal{V}$). Additionally, we are free to assume that the triangle on the right of \eqref{COHER EQ} strictly commutes, i.e. 
that ``unary coproducts'' of singletons $(v)$ are given simply by $v$ itself.
The transformation $\alpha$ is then associative in the sense that the composite natural isomorphisms between the two functors
$\Fin^{\wr 3} \wr \mathcal{V} \to \mathcal{V}$
in the diagrams below coincide.
\begin{equation}\label{COHER2 EQ}
\begin{tikzcd}
	\Fin^{\wr 3} \wr \mathcal{V} \ar{d}[swap]{\sigma^0} 
	\ar{r}{\Fin^{\wr 2} \wr \coprod} \ar{d}[swap]{\sigma^0}&
	\Fin^{\wr 2} \wr \mathcal{V} \ar{r}{\Fin \wr \coprod}
	\ar{d}[swap]{\sigma^0}&
	\Fin \wr \mathcal{V} \ar{r}{\coprod} &
	|[alias=Vtt]|
	\mathcal{V} \ar[equal]{d}
&
	\Fin^{\wr 3} \wr \mathcal{V} \ar{d}[swap]{\sigma^0} 
	\ar{r}{\Fin^{\wr 2} \wr \coprod} &
	\Fin^{\wr 2} \wr \mathcal{V} \ar{r}{\Fin \wr \coprod}&
	|[alias=Vtt2]|
	\Fin \wr \mathcal{V} \ar{r}{\coprod} \ar[equal]{d}&
	\mathcal{V} \ar[equal]{d}
\\
	\Fin^{\wr 2} \wr \mathcal{V} 
	\ar{r}{\Fin \wr \coprod} \ar{d}[swap]{\sigma^1}&
	|[alias=FVt]|
	\Fin \wr \mathcal{V} \ar{rr}{\coprod} &&
	|[alias=Vt]|
	\mathcal{V} \ar[equal]{d}
&
	|[alias=FFV]|	
	\Fin^{\wr 2} \wr \mathcal{V} 
	\ar{rr}{\Fin \wr \coprod} \ar{d}[swap]{\sigma^0}&&
	\Fin \wr \mathcal{V} \ar{r}{\coprod} &
	|[alias=Vt2]|
	\mathcal{V} \ar[equal]{d}
\\
	|[alias=FV]|
	\Fin \wr \mathcal{V} \ar{rrr}[swap]{\coprod} &&&
	\mathcal{V}
&
	|[alias=FV2]|
	\Fin \wr \mathcal{V} \ar{rrr}[swap]{\coprod} &&&
	\mathcal{V}
\arrow[Leftrightarrow, from=Vt, to=FV,shorten <=0.10cm,shorten >=0.10cm,"\alpha"]
\arrow[Leftrightarrow, from=Vtt, to=FVt,shorten <=0.10cm,shorten >=0.10cm,"\alpha"]
\arrow[Leftrightarrow, from=Vt2, to=FV2,shorten <=0.10cm,shorten >=0.10cm,"\alpha"]
\arrow[Leftrightarrow, from=Vtt2, to=FFV, shorten <=0.10cm,shorten >=0.10cm,"\Fin \wr \alpha"]
\end{tikzcd}
\end{equation}
Similarly, $\alpha$ is unital in the sense that the diagrams below commute or, more precisely,
the composite natural transformation in either diagram is the identity for the functor 
$\coprod \colon \Fin \wr \mathcal{V} \to \mathcal{V}$.
\begin{equation}\label{COHER3 EQ}
\begin{tikzcd}
	\Fin \wr \mathcal{V} \ar{d}[swap]{\delta^0} \ar{r}{\coprod}&
	\mathcal{V} \ar{d}[swap]{\delta^0} \ar[equal]{r}&
	\mathcal{V} \ar[equal]{d}
& &
	\Fin \wr \mathcal{V} \ar{d}[swap]{\delta^1} \ar[equal]{r}&
	\Fin \wr \mathcal{V}\ar[equal]{d} \ar{r}{\coprod}&
	\mathcal{V} \ar[equal]{d}
\\
	\Fin^{\wr 2} \wr \mathcal{V} 
	\ar{r}{\Fin \wr \coprod} \ar{d}[swap]{\sigma^0}&
	\Fin \wr \mathcal{V} \ar{r}{\coprod} &
	|[alias=Vt]|
	\mathcal{V} \ar[equal]{d}
& &
	\Fin^{\wr 2} \wr \mathcal{V} 
	\ar{r}{\Fin \wr \coprod} \ar{d}[swap]{\sigma^0}&
	\Fin \wr \mathcal{V} \ar{r}{\coprod} &
	|[alias=Vt2]|
	\mathcal{V} \ar[equal]{d}
\\
	|[alias=FV]|
	\Fin \wr \mathcal{V} \ar{rr}[swap]{\coprod} &&
	\mathcal{V}
& &
	|[alias=FV2]|
	\Fin \wr \mathcal{V} \ar{rr}[swap]{\coprod} &&
	\mathcal{V}
\arrow[Leftrightarrow, from=Vt, to=FV,shorten <=0.10cm,shorten >=0.10cm,"\alpha"]
\arrow[Leftrightarrow, from=Vt2, to=FV2,shorten <=0.10cm,shorten >=0.10cm,"\alpha"]
\end{tikzcd}
\end{equation}
\end{remark}

\begin{remark}
\label{SIGMA_WR_REM}
More generally, if $\mathcal{V}$ is an arbitrary
symmetric monoidal category, one has a functor 
$\Sigma \wr \mathcal{V} \xrightarrow{\otimes} \mathcal{V}$
(where as usual $\Sigma \hookrightarrow \Fin$ denotes the skeleton of finite sets and isomorphisms) satisfying the obvious analogues of
\eqref{COHER EQ}, \eqref{COHER2 EQ}, \eqref{COHER3 EQ},
as is readily shown using the standard coherence results for symmetric monoidal categories 
(moreover, we note that $\alpha$ itself encodes all associativity, unital and symmetry isomorphisms, with the 
right side of \eqref{COHER EQ} and \eqref{COHER3 EQ}
being mere common sense desiderata for ``unary products'').

It is likely no surprise that the converse is also true, i.e. 
that a functor 
$\Sigma \wr \mathcal{V} \xrightarrow{\otimes} \mathcal{V}$
satisfying the analogues of 
\eqref{COHER EQ}, \eqref{COHER2 EQ}, \eqref{COHER3 EQ}
endows $\mathcal{V}$ with a symmetric monoidal structure.
We will however have no direct need to use this fact, and as such include only a few pointers concerning the associativity pentagon axiom (the hardest condition to check) that the interested reader may find useful. 
Firstly, it becomes convenient to write expressions such as
$(A \otimes B) \otimes C$ instead as 
$(A \otimes B) \otimes (C)$, so as to encode notationally the fact that this is the image of 
$((A,B),(C)) \in \Sigma^{\wr 2} \wr \mathcal{V}$ under the top map in \eqref{COHER EQ}. The associativity isomorphisms are hence given by the composites
$
(A \otimes B) \otimes (C) \xrightarrow{\simeq} 
A \otimes B \otimes C \xleftarrow{\simeq}
(A) \otimes (B \otimes C)
$
obtained by combining 
$\alpha_{((A,B),(C))}$ 
and
$\alpha_{((A),(B,C))}$.
The pentagon axiom is then checked by combining \textit{six} instances of each of the squares in \eqref{COHER2 EQ} 
(i.e. twelve squares total), most of which are obvious except for the fact that the $(A\otimes B) \otimes (C \otimes D)$ vertex of the pentagon contributes two pairs of squares 
as in \eqref{COHER2 EQ} rather than just one, 
with each pair corresponding to the two alternate expressions 
$((A \otimes B)) \otimes ((C) \otimes (D))$ and 
$((A) \otimes (B)) \otimes ((C \otimes D))$.
\end{remark}

\begin{remark}\label{FINSURJ REM}
	\index{wreath products!FINS@$\mathsf{F}_s$}
	\index{wreath products!FINSWR@$\mathsf{F}_s \wr (-)$}
	In light of the two previous remarks,
	and writing $\Fin_s \hookrightarrow \Fin$ 
	for the subcategory of surjections,
	we define a 
	\textit{symmetric monoidal category with fold maps}
	as a category $\mathcal{V}$ together with a functor
	$\Fin_s \wr \mathcal{V} \xrightarrow{\otimes} \mathcal{V}$
	satisfying the analogues of  
	\eqref{COHER EQ}, \eqref{COHER2 EQ}, \eqref{COHER3 EQ}.
	Further, the dual of such $\mathcal{V}$ is called a 
	\textit{symmetric monoidal category with diagonals}\footnote{
	These have also been called \textit{relevant monoidal categories} \cite{DP07}.}.
	
	Similarly, replacing $\Fin_s$ with the subcategory
$\Fin_i \hookrightarrow \Fin$ of injections yields the notion of a \textit{symmetric monoidal category with injection maps} or, dually, \textit{symmetric monoidal category with projections}\footnote{
These are equivalent to \textit{semicartesian symmetric monoidal categories} \cite{Lei16}.}.

Finally, we note that if a symmetric monoidal category has both diagonals and projections, it must in fact be \textit{cartesian monoidal} \cite[IV.2]{EK66}.
\end{remark}

\begin{remark}
	Extending Notation \ref{FIN_COA_COS_NOT}, one sees that 
	$\Fin \wr (\minus)$, 
	$\Fin_i \wr (\minus)$,
	$\Fin_s \wr (\minus)$,
	$\Sigma \wr (\minus)$
define monads in the category of categories.
%
\end{remark}

We end this section by collecting some straightforward lemmas
that will be used in \S \ref{GENUINE_OP_MONAD_SECTION}.

\begin{lemma}\label{FWRGROTH LEM}
	If $\mathcal{E} \to \mathcal{B}$ a (split) Grothendieck fibration then so is 
	$\Fin_s \wr \mathcal{E} \to \Fin_s \wr \mathcal{B}$.

	Moreover, if 
	$\mathcal{E} \to \bar{\mathcal{E}}$ is a map of (split) Grothendieck fibrations over $\mathcal{B}$ then
	$\Fin_s \wr \mathcal{E} \to \Fin_s \wr \bar{\mathcal{E}}$ is a map of (split) Grothendieck fibrations over $\Fin_s \wr \mathcal{B}$.
\end{lemma}

\begin{proof}
Given a map 
$(\phi,(f_i)) \colon
(b'_i)_{i \in I} \to (b_j)_{j \in J}$ 
in $\Fin \wr \mathcal{B}$ and object $(e_j)_{j \in J}$,
one readily checks that its pullback can be defined by $(f^{\**}_{\phi(i)}e_{\phi(i)})_{i \in I}$.
\end{proof}

\begin{lemma}\label{FINWREATPRODLIM LEM}
Suppose that $\mathcal{V}$ is a bicomplete monoidal category with fold maps such that
the monoidal product 
commutes with limits in each variable. If the leftmost diagram
\begin{equation}\label{WRRAN EQ}
	\begin{tikzcd}[column sep = 3.5em]
	\mathcal{C} \ar{r}[swap,name=F]{}{G} \ar{d}[swap]{k} &
	\mathcal{V} & 
	& 
	\Fin_s \wr \mathcal{C} \ar{r}[swap,name=FF]{}{\Fin_s \wr G} \ar{d}[swap]{\Fin_s \wr k}&
	\Fin_s \wr \mathcal{V} \ar{r}{\otimes} &
	\mathcal{V}
		\\
	|[alias=D]|\mathcal{D} \ar{ru}[swap]{H} &
	& & 
	|[alias=FD]|\Fin_s \wr \mathcal{D} \ar{ru}[swap]{\Fin_s \wr H}
	\ar[bend right=13]{rru}[swap]{\otimes \circ \Fin_s \wr H}
	\arrow[Rightarrow, from=D, to=F,shorten <=0.10cm,"\eta"]
	\arrow[Rightarrow, from=FD, to=FF,shorten <=0.10cm,"\Fin_s \wr \eta"]
	\end{tikzcd}
\end{equation}
is a right Kan extension diagram then so is the composite of the rightmost diagram. 

Dually, if $\mathcal{V}$ has diagonals,
the monoidal product 
commutes with colimits in each variable, and the leftmost diagram
\begin{equation}\label{WRLAN EQ}
	\begin{tikzcd}[column sep = 4.5em]
	\mathcal{C}^{op} \ar{r}[swap,name=F]{}{G} \ar{d}[swap]{k^{op}} & 
	\mathcal{V} & 
	(\Fin_s \wr \mathcal{C})^{op} \ar{d}[swap]{(\Fin_s \wr k)^{op}} 
	\ar{r}[swap,name=FF]{}{(\Fin_s \wr G^{op})^{op}} & 
	(\Fin_s \wr \mathcal{V}^{op})^{op} \ar{r}{\otimes} &
	\mathcal{V}
\\
	|[alias=D]|\mathcal{D}^{op} \ar{ru}[swap]{H} &
	& 
	|[alias=FD]|(\Fin_s \wr \mathcal{D})^{op} 
	\ar{ru}[swap]{(\Fin_s \wr H^{op})^{op}}
	\ar[bend right=13]{rru}[swap]{\otimes \circ (\Fin_s \wr H^{op})^{op}}
	&
	\arrow[Leftarrow, from=D, to=F,shorten <=0.10cm,"\epsilon"]
	\arrow[Leftarrow, from=FD, to=FF,shorten <=0.10cm]
	\end{tikzcd}
\end{equation}
is a left Kan extension diagram,
then so is the composite of the rightmost diagram. 
\end{lemma}

\begin{proof}
	Unpacking definitions using the pointwise formula for Kan extensions (cf. \cite[X.3 Thm. 1]{McL} or \eqref{FIBERKAN EQ}), the claim concerning \eqref{WRRAN EQ} amounts to showing that,
	for each $(d_i) \in \Fin_s \wr \mathcal{D}$,
	one has natural isomorphisms
	\begin{equation}\label{POINTKAN EQ}
	\underset{((d_i) \to (kc_j))\in
	\left( (d_i) \downarrow \Fin_s \wr \C \right) }{\lim} {\left(\bigotimes_j{G(c_j)}\right)}
		\simeq	
	\bigotimes_i \underset{(d_i  \to kc_i) \in 
		(d_i \downarrow \C)}{\lim}
	\left(G(c_i)\right).
	\end{equation}
Proposition \ref{FIBERKANMAP PROP} now applies to 
the map $\Fin_s \wr \mathcal{C} \to \Fin_s \wr \mathcal{D}$ of Grothendieck fibrations over $\Fin_s$ and one readily checks that
$(d_i)\downarrow_{\Fin_s} \Fin_s \wr \mathcal{C} \simeq
\prod_{i}{(d_i\downarrow \mathcal{C})}
$
so that 
	\[
	\underset{((d_i) \to (kc_j))\in
	\left( (d_i)\downarrow \Fin_s \wr \C \right) }{\lim} {\left(\bigotimes_j{G(c_j)}\right)}
		\simeq	
	\underset{(d_i \to kc_i)\in
	\prod_{i} \left( {d_i \downarrow \mathcal{D}} \right)}{\lim}
	\left(\bigotimes_i{G(c_i)}\right)
	\]
and the isomorphisms \eqref{POINTKAN EQ} now follow from the assumption that the monoidal product commutes with limits in each variable.
\end{proof}

\begin{remark}
      The previous results also hold if we replace $\Fin_s$ with $\Fin$, $\Fin_i$, $\Sigma$.
\end{remark}

\subsection{Monads and adjunctions}


In \S 4 we will make use of the following straightforward results concerning the transfer of monads along adjunctions
(note that $L$ (resp. $R$) denotes the left (right) adjoint).

\begin{proposition}\label{MONADADJ1 PROP}
Let
$
L \colon \mathcal{C} \rightleftarrows \mathcal{D} \colon R
$
be an adjunction and $T$ a monad on $\mathcal{D}$.
Then:
\begin{itemize}
\item[(i)] $RTL$ is a monad and $R$ induces a functor
$R \colon \mathsf{Alg}_T(\mathcal{D}) \to \mathsf{Alg}_{RTL}(\mathcal{C})$;
\item[(ii)] if $LRTL \xrightarrow{\epsilon} TL$ is an isomorphism one further has an induced adjunction
\[
L \colon \mathsf{Alg}_{RTL}(\mathcal{C})
	\rightleftarrows
\mathsf{Alg}_{T}(\mathcal{D}) \colon R.
\]
\end{itemize}
\end{proposition}

\begin{proposition}\label{MONADADJ PROP}
Let
$
L \colon \mathcal{C} \rightleftarrows \mathcal{D} \colon R
$
be an adjunction, $T$ a monad on $\mathcal{C}$, and suppose further that
\[
	LR \xrightarrow{\epsilon} id_{\mathcal{D}}, 
\qquad
	LT \xrightarrow{\eta} LTRL
\]
are natural isomorphisms 
(so that in particular $\mathcal{D}$ is a reflexive subcategory of $\mathcal{C}$).
Then:
\begin{itemize}
\item[(i)] $LTR$ is a monad, with multiplication and unit given by
\[LTRLTR \xrightarrow{\eta^{-1}} LTTR \to LTR,\qquad
id_{\mathcal{D}} \xrightarrow{\epsilon^{-1}} LR \to LTR;
\]
\item[(ii)]
$d \in \mathcal{D}$ is a $LTR$-algebra iff $Rd$ is a $T$-algebra;
\item[(iii)] there is an induced adjunction
\[
L \colon \mathsf{Alg}_{T}(\mathcal{C})
	\rightleftarrows
\mathsf{Alg}_{LTR}(\mathcal{D}) \colon R.
\]
\end{itemize}
\end{proposition}

Any monad $T$ on $\C$ induces obvious monads $T^{\times l}$ on $\C^{\times l}$.
More generally, and 
letting $I$ denote the identity monad,
a partition 
$\{1,\cdots,l\} = \lambda_a \amalg \lambda_i$,
which we denote by $\lambda$,
determines a monad 
$T^{\times \lambda} = T^{\times \lambda_a} \times I^{\times \lambda_i}$ on $\mathcal{C}^{\times l}$.
Here ``$a$'' stands for ``active'' and ``$i$'' for ``inert''.

Such monads satisfy a number of compatibility conditions. 
First, if $\lambda'_a \subseteq \lambda_a$
there is a monad map
$T^{\times \lambda'} \Rightarrow T^{\times \lambda}$,
and we write $\lambda' \leq \lambda$.
Next, writing $\alpha^{\**} \colon 
\C^{\times l} \to \C^{\times m}$
for the forgetful functor induced by
a map $\alpha \colon \{1,\cdots,m\} \to \{1,\cdots,l\}$,
one has an equality
$T^{\times \alpha^{\**}\lambda} \alpha^{\**} =
\alpha^{\**} T^{\times \lambda}$,
where $\alpha^{\**}\lambda$ is the pullback partition
$\alpha^{-1} \lambda_a \amalg \alpha^{-1} \lambda_i$.
The following is straightforward.

\begin{proposition}\label{MONADICFUN PROP}
	Suppose $\C$ has finite coproducts.
	Let $T$ be a monad on $\C$,
	$\alpha \colon \{1,\cdots,m\} \to \{1,\cdots,l\}$ be a map of sets,
	and $\lambda$ be a partition $\{1,\cdots,l\} = \lambda_a \amalg \lambda_i$.
	Write $\alpha_{!} \colon \mathcal{C}^{\times m} \to 
	\mathcal{C}^{\times l}$
	for the left adjoint to 
	$\alpha^{\**} \colon \mathcal{C}^{\times l} \to 
	\mathcal{C}^{\times m}$.
	Then the map
\begin{equation}\label{MONADFUNCTORALPHA EQ}
	T^{\times \alpha^{\**} \lambda} \Rightarrow \alpha^{\**} T^{\times \lambda} \alpha_{!}
\end{equation}
adjoint to the identity 
$T^{\times \alpha^{\**}\lambda} \alpha^{\**} =
\alpha^{\**} T^{\times \lambda}$
is a map of monads on $\C^{\times m}$.

Hence, since $T^{\times \lambda} \alpha_!$ is a right 
$\alpha^{\**} T^{\times \lambda}\alpha_{!}$-module\footnote{
	Recall that a right (resp. left) module
	over a monad $T$ on $\mathcal{C}$
	is a functor $M \colon \mathcal{C} \to \mathcal{D}$
	(resp. $N \colon \mathcal{D} \to \mathcal{C}$)
	together with an action map $M \circ T \Rightarrow M$
	(resp. $T \circ N \Rightarrow N$)
	that is suitably associative and unital.}, 
it is also a right $T^{\times \lambda'}$-module
whenever
$\lambda' \leq \alpha^{\**} \lambda$.
Finally, the natural map 
\begin{equation}\label{MONADFUNCTORALPHADOU EQ}
	\alpha_{!} T^{\times \alpha^{\**} \lambda} \Rightarrow  T^{\times \lambda} \alpha_{!}
\end{equation}
is a map of right $T^{\times \alpha^{\**} \lambda}$-modules, 
and thus also a map of right 
$T^{\times \lambda'}$-modules
whenever $\lambda' \leq \alpha^{\**} \lambda$.
\end{proposition}

\begin{remark}\label{PRECOMPPOSTCOMP REM}
We unpack 
\eqref{MONADFUNCTORALPHADOU EQ} for 
$\alpha \colon \{1,\cdots,m\} \to \**$ a map
to the singleton $\**$
with the partition making $\**$ active.
We can then write
$\alpha_{!} = \coprod$,
so that 
\eqref{MONADFUNCTORALPHADOU EQ} becomes
$\coprod \circ T^{\times m} 
\Rightarrow 
T \circ \coprod$.
For $\lambda$ any partition 
$\{1,\cdots,m\} = \lambda_a \amalg \lambda_i$
we thus have a map 
$\coprod \circ T^{\times \lambda} 
\Rightarrow 
T \circ \coprod$
between right $T^{\times \lambda}$-modules
which,
for each collection $\left( A_j \right)_{1\leq j \leq m}$ in $\mathcal{C}^{\times m}$,
gives commutative diagrams
\begin{equation}\label{RIGHTMODULETMAPAUX EQ}
\begin{tikzcd}
	\coprod_{j \in \lambda_a} TT A_j \amalg \coprod_{j \in \lambda_i} A_j
	\ar{r} \ar{d} &
	T\left( \coprod_{j \in \lambda_a} T A_j \amalg \coprod_{j \in \lambda_i} A_j \right) \ar{d}
\\
	\coprod_{j \in \lambda_a} T A_j \amalg \coprod_{j \in \lambda_i} A_j
	\ar{r} &
		T\left( \coprod_{j \in \lambda_a} A_j \amalg \coprod_{j \in \lambda_i} A_j \right)
\end{tikzcd}
\end{equation}
where the vertical maps
come from the right $T^{\times \lambda}$-module structure.
Writing $\mathbin{\check{\amalg}}$ for the coproduct of $T$-algebras and recalling the canonical identifications 
$\mathbin{\check{\coprod}}_{k \in K} (T A_k) \simeq T
\left( \coprod_{k \in K} A_k \right)$, 
\eqref{RIGHTMODULETMAPAUX EQ} shows that the 
right $T^{\times \lambda}$-module structure on $T \circ \coprod$
codifies the multiplication maps
\[
\mathbin{\check{\coprod}}_{j \in \lambda_a} TT A_j 
	\mathbin{\check{\amalg}} 
\mathbin{\check{\coprod}}_{j \in \lambda_i} T A_j
	\to
\mathbin{\check{\coprod}}_{j \in \lambda_a} T A_j 
	\mathbin{\check{\amalg}}  
\mathbin{\check{\coprod}}_{j \in \lambda_i} T A_j.
\] 
\end{remark}

\section{Planar and tall maps, and substitution}\label{PLANAR_SECTION}

Throughout, we will assume that the reader is familiar with the category $\Omega$ of trees.
A good introduction to $\Omega$ is given by 
\cite[\S 3]{MW07}, where arrows are described both via 
the ``colored operad generated by a tree''  and by identifying explicit generating arrows, called faces and degeneracies.
Alternatively, $\Omega$ can also be described 
using the algebraic model of 
\textit{broad posets}
introduced by Weiss in \cite{We12} and further worked out by the second author in \cite[\S 5]{Pe17}.
This latter approach will be our ``official'' model,
though a detailed understanding of broad posets is needed only
to follow our formal discussion of planar structures in \S \ref{PLASTR SEC}.
Otherwise, the reader willing to accept the results of \S \ref{PLASTR SEC} should 
need only an intuitive grasp of the notations 
$\underline{e} \leq e$,
$f \leq_d e$ and $e^{\uparrow}$
to read the remainder of the paper.
Such understanding can be obtained from 
\cite[Example 5.10]{Pe17}
and Example \ref{PLANAREX EX} below 
(see also Example \ref{OUTERTREE EX}).

Given a finite group $G$, there is also a category $\Omega_G$
\index{categories!of trees!Gtrees1@$\Omega_G$}
of $G$-trees, jointly discovered by the authors and first discussed by the second author in 
\cite[\S 4.3,\S 5.3]{Pe17}, which we now recall.
Firstly, we let $\Phi$
\index{categories!of trees!GtreesForest@$\Phi = \Fin \wr \Omega$}
denote the category of forests, 
i.e. ``formal coproducts of trees''.
A broad poset description of $\Phi$ is found in \cite[\S 5.2]{Pe17},
but here we prefer the alternative definition $\Phi = \Fin \wr \Omega$.
The category of $G$-forests is then 
$\Phi^G$, i.e. the category of $G$-objects in $\Phi$. 
Similarly, writing
$\Fin^G$ for the category of $G$-objects in $\Fin$ and
identifying the $G$-orbit category as the subcategory
$\mathsf{O}_G \hookrightarrow \Fin^G$
\index{categories!other!OrbitG@$\mathsf O_G$}
of those sets with transitive actions, $\Omega_G$ can be described
by the pullback of categories
\begin{equation}\label{OGDEF EQ}
\begin{tikzcd}
	\Omega_G \ar{r} \ar{d}[swap]{\mathsf{r}} & 
	\Phi^G \arrow{d}{\mathsf{r}}
\\
	\mathsf{O}_G \ar{r} & \Fin^G
\end{tikzcd}
\end{equation}
(where $\mathsf{r} \colon \Phi \to \Fin$ is the \emph{root functor},
\index{key functors!Gtreesleafroot0@$\mathsf{r}$}
sending a forest to its set of roots),
which is a repackaging of \cite[Def. 5.44]{Pe17}.
Explicitly, a $G$-tree $T$ is then a tuple 
$T = (T_x)_{x \in X}$ with $X \in \mathsf{O}_G$
together with isomorphisms
$T_x \to T_{g x}$ that are suitably associative and unital.

\subsection{Planar structures}\label{PLASTR SEC}

The specific model for the orbit category $\mathsf{O}_G$
used in \eqref{OGDEF EQ} has extra structure not found in the usual model (i.e. that of the $G$-sets $G/H$ for $H \leq G$),
namely the fact that each $X \in \mathsf{O}_G$
comes with a canonical total order 
(the underlying set of $X$ being one of the sets $\{1,\cdots,n\}$).

We will find it convenient to use a model of $\Omega$ with similar extra structure, given by planar structures on trees.
Intuitively, a planar structure on a tree is the data of a planar representation of the tree, and 
definitions of \textit{planar trees} along those lines
are found throughout the literature.
However, to allow for precise proofs of some key results 
concerning the interaction of planar structures with the maps in $\Omega$ 
(namely Propositions \ref{PLANARPULL PROP},  \ref{SUBDATAUNDERPLAN PROP})
we will instead use a combinatorial definition 
of planar structures in the context of broad posets.

In what follows a tree will be a 
\textit{dendroidally ordered broad poset}
as in \cite{We12}, \cite[Def. 5.9]{Pe17}.

\begin{definition}\label{PLANARIZE DEF}
      Let $T \in \Omega$ be a tree. A \textit{planar structure} of $T$ is an extension of the descendancy partial order $\leq_d$ to a total order $\leq_p$ such that:
      \index{structure on trees!relations!Gtreesorderplanar@$\leq_p$}
      \index{structure on trees!relations!Gtreesorderdescendancy@$\leq_d$}
	\begin{itemize}
		\item \textit{Planar}: if $e \leq_p f$ and $e \nleq_d f$ then 
		$g \leq_d f$ implies $e \leq_p g$.
          \end{itemize}          
\end{definition}

\begin{example}\label{PLANAREX EX}
An example of a planar structure on a tree $T$ follows, with $\leq_p$ encoded by the hexadecimal number labels
(so that $9<a<b<c<d$).
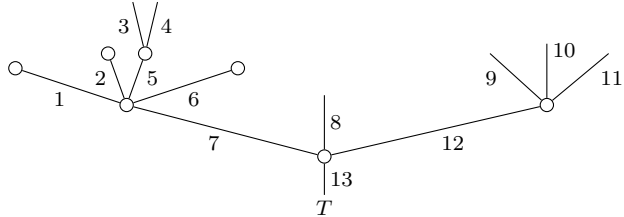
\[
	\begin{tikzpicture}[grow=up,auto,level distance=2.1em,
	every node/.style = {font=\footnotesize,inner sep=2pt},
	dummy/.style={circle,draw,inner sep=0pt,minimum size=1.75mm}]
		\node[font=\normalsize] at (0,0) {$T$}
			child{node [dummy] {}
				child[sibling distance = 9em]{node [dummy] {}
					child[sibling distance = 2.5em]{
					edge from parent node [near end,swap] {$b$}}
					child[level distance=2.5em]{
					edge from parent node [very near end,swap] {$a$}}				
					child[sibling distance = 2.3em]{
					edge from parent node [near end] {$9$}}
				edge from parent node [swap] {c}}
				child[level distance =2.5em]{
				edge from parent node [swap] {$8$}}
				child[sibling distance = 8em]{node [dummy] {}
					child[sibling distance =3em, level distance = 1.5 em]{node [dummy] {}
					edge from parent node [swap] {$6$}}
					child[sibling distance = 1.5em]{node [dummy] {}
						child[sibling distance =1em]{
						edge from parent node [swap,near end] {$4$}}
						child[sibling distance =1em]{
						edge from parent node [near end] {$3$}}
					edge from parent node [very near end,swap] {$5$}}
					child[sibling distance =1.5em]{node [dummy] {}
					edge from parent node [very near end] {$2$}}
					child[sibling distance =3em,level distance =1.5em]{node [dummy] {}
					edge from parent node {$1$}}
				edge from parent node {$7$}}
			edge from parent node [swap] {$d$}};
	\end{tikzpicture}
\]
Intuitively, given a planar depiction of a tree $T$, $e \leq_d f$ holds when the downward path from $e$ passes through $f$.
For example, $3 \leq_d 7$ but $7 \not \leq_d 9$.
On the other hand, $e \leq_p f$ holds if either
$e \leq_d f$ or if the downward path from $e$ is to the left of the downward path from $f$ (as measured at the node where the paths intersect).

For each edge $e$ topped by a vertex, the notation $e^{\uparrow}$ denotes the tuple of edges immediately above $e$.
In our example, 
$d^{\uparrow} = 78c$,
$7^{\uparrow} = 1256$,
$2^{\uparrow} = \epsilon$ 
(where $\epsilon$ is the empty tuple),
and $9^{\uparrow}$ is undefined.
The vertex above $e$ is then encoded by the \emph{broad relation}
$e^{\uparrow} \leq e$\index{structure on trees!relations!Broadvert@$e^{\uparrow} \leq e$}.

The broad relation notation is meant to suggest a form of 
\emph{broad associativity}. For example,
$78c \leq d$ and $1256 \leq 7$ combine to yield
$12568c \leq d$,
which
in turn combines with $\epsilon \leq 2$
to yield $1568c \leq d$.
The broad relations of $T$ are those relations that are obtained from the vertex relations $e^{\uparrow} \leq e$
via broad transitivity, together with reflexive relations
$e \leq e$.
Pictorially, a relation 
$\underline{e} \leq e$\index{structure on trees!relations!Broadvert@$\underline{e} \leq e$} holds 
if there is an outer subtree
(i.e. a tree subdiagram 
which contains all edges of $T$ adjacent to its vertices;
see \S \ref{OUTTALL SEC} for a rigorous discussion)
with leaf tuple 
$\underline{e}$ and root $e$.
For an illustration, see Example \ref{OUTERTREE EX}.
\end{example}

It is visually clear that a planar depiction of a tree amounts to choosing a total order for each of the sets of \textit{input edges} of each node (i.e. those edges immediately above that node).

While we will not need to make this last statement precise, we will nonetheless find it convenient to show that our Definition \ref{PLANARIZE DEF} of planarity is equivalent to such choices of total orders for each of the sets of input edges.
To do so, we first introduce some notation.

\begin{notation}\label{INPUTPATH NOT}
	Let $T \in \Omega$ be a tree and $e \in T$ an edge. We will denote
	\[ I(e) =\{f \in T \colon e \leq_d f \} \]
        and refer to this poset as the \textit{input path of $e$}.
        \index{structure on trees!other!inputpath@$I(e)$}
\end{notation}

We will repeatedly use the following, which is a consequence of \cite[Cor. 5.25]{Pe17}.

\begin{lemma}\label{INCOMPNOTOP}
If $e \leq_d f$, $e \leq_d f'$, then $f,f'$ are $\leq_d$-comparable. 
\end{lemma}

\begin{proposition}\label{INPUTPATHS PROP}
	Let $T \in \Omega$ be a tree. Then
	\begin{itemize}
		\item[(a)] for any $e \in T$ the finite poset $I(e)$ is totally ordered;
		\item[(b)] the poset $(T,\leq_d)$ has all joins, denoted $\vee$. In fact, $\bigvee_{i} e_i = \min (\bigcap_{i} I(e_i))$.
	\end{itemize}
\end{proposition}

\begin{proof}
	(a) is immediate from Lemma \ref{INCOMPNOTOP}.
	To prove (b) we note that
	the root edge is in every input path.
	Hence $\min (\bigcap_{i} I(e_i))$ exists by (a), 
	and this is clearly the join $\bigvee_i {e_i}$.
\end{proof}

\begin{notation}
	Let $T \in \Omega$ be a tree and suppose that $e <_d b$. We will denote by $b^{\uparrow}_e \in T$ the predecessor of $b$ in $I(e)$.
\end{notation}

\begin{proposition}\label{INPUTPREDECESSORPROP PROP}
Suppose $e,f$ are $\leq_d$-incomparable edges of $T$ and write $b= e \vee f$. Then
\begin{itemize}
\item [(a)] $e <_d b$, $f<_d b$ and $b^{\uparrow}_e \neq b^{\uparrow}_f$;
\item [(b)] $b^{\uparrow}_e, b^{\uparrow}_f \in b^{\uparrow}$. In fact $\{b^{\uparrow}_e\} = I(e) \cap b^{\uparrow}$,
$\{b^{\uparrow}_f\} = I(f) \cap b^{\uparrow}$;
\item[(c)] if $e' \leq_d e$, $f' \leq_d f$ then 
$b = e' \vee f'$ and $b^{\uparrow}_{e'} = b^{\uparrow}_{e}$, $b^{\uparrow}_{f'} = b^{\uparrow}_{f}$.
\end{itemize}
\end{proposition}

\begin{proof}
(a) is immediate: the condition $e = b$ (resp. $f = b$) would imply $f \leq_d e$ (resp. $e \leq_d f$)
while the condition $b^{\uparrow}_e = b^{\uparrow}_f$ would provide a predecessor of $b$ in $I(e) \cap I(f)$. 

For (b), note that any relation $a <_d b$ factors as 
$a \leq_d b^{\star}_a <_d b$ for some unique $b^{\**}_a \in b^{\uparrow}$, where uniqueness follows from Lemma \ref{INCOMPNOTOP}. Choosing $a=e$ implies $I(e) \cap b^{\uparrow} = \{b^{\**}_e\}$ and letting $a$ range over edges such that $e \leq_d a <_d b$ shows that $b^{\**}_e$ is in fact the predecessor of $b$ in $I(e)$.

To prove (c) one reduces to the case $e'=e$, in which case it suffices to check $I(e) \cap I(f') = I(e) \cap I(f)$. But if it were otherwise there would exist an edge $a$ satisfying
$f' \leq_d a <_d f$ and $e \leq_d a$, and this would imply $e \leq_d f$, contradicting our hypothesis.
\end{proof}

\begin{proposition}
\label{TERNARYJOIN PROP}
Let $c = e_1 \vee e_2 \vee e_3$.
Then $c = e_i \vee e_j$ iff $c^{\uparrow}_{e_i} \neq c^{\uparrow}_{e_j}$.

Therefore, all ternary joins in $(T,\leq_d)$ are binary, i.e.
\begin{equation}\label{TERNJOIN EQ}
	c = e_1 \vee e_2 \vee e_3 = e_i \vee e_j
\end{equation}
for some $1\leq i <j \leq 3$, and
\eqref{TERNJOIN EQ} fails for 
 at most one choice of $1\leq i <j \leq 3$.
\end{proposition}

\begin{proof}
If $c^{\uparrow}_{e_i} \neq c^{\uparrow}_{e_j}$ then
$c = \min\left(I(e_i) \cap I(e_j)\right) = e_i \vee e_j$, whereas the converse follows from Proposition \ref{INPUTPREDECESSORPROP PROP}(a).
The ``therefore'' part follows by noting that 
$c^{\uparrow}_{e_1}$, $c^{\uparrow}_{e_2}$, $c^{\uparrow}_{e_3}$
can not all coincide, or else $c$ would not be the minimum of
$I(e_1) \cap I(e_2) \cap I(e_3)$. 
\end{proof}

\begin{example} In the following example $b = e \vee f$, $c = e \vee f \vee g$, $c^{\uparrow}_e= c^{\uparrow}_f =b$.
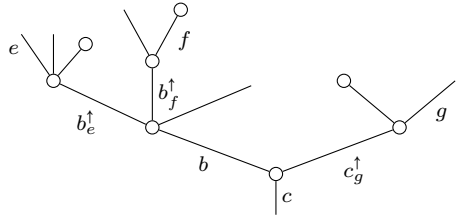
\[
	\begin{tikzpicture}[grow=up,auto,level distance=1.9em,
	every node/.style = {font=\footnotesize,inner sep=2pt},
	dummy/.style={circle,draw,inner sep=0pt,minimum size=1.75mm}]
		\node at (0,0) {}
			child{node [dummy] {}
				child[sibling distance = 10em]{node [dummy] {}
					child[sibling distance = 4.5em]{
					edge from parent node [swap] {$g$}}
					child[sibling distance = 4.5em]{node [dummy] {}}
				edge from parent node [swap] {$c_g^{\uparrow}$}}
				child[sibling distance = 10em]{node [dummy] {}
					child[sibling distance = 4em,level distance=1.7em]
					child[sibling distance = 1.5em,level distance=2.7em]{node [dummy] {}
						child[level distance=2.1em,sibling distance = 2.3em]{node [dummy] {}
						edge from parent node [near end,swap] {$f$}}		
						child[level distance=2.1em,sibling distance = 2.3em]{
						edge from parent node [near end] {}}
					edge from parent node [swap] {$b^{\uparrow}_{f}$}}
					child[sibling distance = 4em]{node [dummy] {}
						child[sibling distance =1.3em, level distance = 1.5 em]{node [dummy]  {}
						edge from parent node [swap] {}}
						child[sibling distance = 1.3em]{
						edge from parent node [very near end,swap] {}}
						child[sibling distance =1.3em]{
						edge from parent node [very near end] {$e$}}
					edge from parent node {$b^{\uparrow}_e$}}
				edge from parent node {$b$}}
			edge from parent node [swap] {$c$}};
	\end{tikzpicture}
\]
\end{example}

Given a set $S$ of size $n$ we write
$\textsf{Ord}(S) = \mathsf{Iso}(S,\{1,\cdots,n\})$. 
We will also abuse notation by regarding its objects as pairs $(S,\leq)$ where $\leq$ is a total order on $S$.

\begin{proposition}\label{PLANARIZATIONCHAR PROP}
	Let $T \in \Omega$ be a tree, 
	with $V(T)$ its set of vertices.
        \index{structure on trees!vertices!vertices@$V(T)$}
	There is a bijection
\[
	\begin{tikzcd}[row sep = 0em]
		\{\text{planar structures }(T,\leq_p)\} \ar{r}{\simeq} &
		\prod_{(a^{\uparrow} \leq a) \in V(T)} \mathsf{Ord}(a^{\uparrow}) \\
		\leq_p \ar[mapsto]{r} & (\leq_p|_{a^{\uparrow}})
	\end{tikzcd}	
\]
\end{proposition}

\begin{proof}
We will keep the notation in Proposition \ref{INPUTPREDECESSORPROP PROP} throughout,
i.e. $e, f$ are $\leq_d$-incomparable edges and we write $b = e \vee f$. 

	We first show injectivity,
	i.e. that the restrictions $\leq_p|_{a^{\uparrow}}$ determine if 
	$e <_p f$ holds or not.
If $b^{\uparrow}_e <_p b^{\uparrow}_f$, the relations
$e \leq_d b^{\uparrow}_e <_p b^{\uparrow}_f \geq_d f$
and Definition \ref{PLANARIZE DEF} imply it must be $e <_p f$.
Dually, if $b^{\uparrow}_f <_p b^{\uparrow}_e$ then 
$f <_p e$. Thus 
$b^{\uparrow}_e <_p b^{\uparrow}_f \Leftrightarrow e <_p f$ and injectivity follows.

To check surjectivity, 
 it suffices (recall that $e,f$ are assumed $\leq_d$-incomparable) to check that
defining $e \leq_p f$ to hold iff $b^{\uparrow}_e < b^{\uparrow}_f$ holds in $b^{\uparrow}$ yields a planar structure.

Antisymmetry and the total order conditions are immediate, and it thus remains to check the transitivity and planar conditions.
Transitivity of $\leq_p$ in the case $e' \leq_d e <_p f$ and the planar condition, which is the case $e <_p f \geq_d f'$, follow from Proposition \ref{INPUTPREDECESSORPROP PROP}(c). Transitivity of $\leq_p$ in the case $e <_p f \leq_d f'$
follows since either $e \leq_d f'$ or else $e,f'$ are $\leq_d$-incomparable, in which case one can apply Proposition \ref{INPUTPREDECESSORPROP PROP}(c) with the roles of $f,f'$ reversed.

It remains to check transitivity in the hardest case, that of 
$e <_p f <_p g$ with $\leq_d$-incomparable $f,g$.
We write $c = e \vee f \vee g$.
By the ``therefore'' part of Proposition \ref{TERNARYJOIN PROP}, either:
\begin{enumerate*}
	\item[(i)] $e \vee f <_d c$, in which case 
	Proposition \ref{TERNARYJOIN PROP}
	implies 
	$c=e \vee g$,
	$c^{\uparrow}_e = c^{\uparrow}_f$ and transitivity follows;
	\item[(ii)] $f \vee g <_d c$, which follows just as (i);
	\item[(iii)]  
$e \vee f = f \vee g =c$, in which case 
$c^{\uparrow}_e <
c^{\uparrow}_f <
c^{\uparrow}_g $ in $c^{\uparrow}$
so that $c^{\uparrow}_e \neq c^{\uparrow}_g$ and by Proposition \ref{TERNARYJOIN PROP} it is also 
$c = e \vee g$ and transitivity follows.
\end{enumerate*}
\end{proof}

\begin{remark}\label{CLOSURE REM}
Proposition \ref{PLANARIZATIONCHAR PROP} states, in particular,
that $\leq_p$ is the closure of the $\leq_d$ relations  
and the $\leq_p$ relations within each $a^{\uparrow}$
under the planar condition in 
Definition \ref{PLANARIZE DEF}.
\end{remark}

The discussion of the substitution procedure in \S \ref{SUBS SEC} 
will be simplified by working with 
a model for the category $\Omega$
with exactly one representative
of each possible planar structure on each tree or, more precisely, a model where the only isomorphisms preserving the planar structure are the identities.
On the other hand, exclusively using such a model for $\Omega$ throughout would, among other issues, make the discussion of faces in \S \ref{OUTTALL SEC} rather awkward.
We now describe our conventions to address such issues.

Let $\Omega^p$ denote the category of \textit{planarized trees},
with objects pairs $T_{\leq_p}=(T,\leq_p)$ of trees together with a planar structure,
and morphisms \textit{underlying} maps of trees (i.e. ignoring the planar structures).
There is a full subcategory $\Omega^s \hookrightarrow \Omega^p$, whose objects we call \textit{standard models}, of those $T_{\leq_p}$ whose underlying set is one of the sets $\underline{n} = \{1,2,\cdots,n\}$ and for which $\leq_p$ coincides with the canonical order.

\begin{example}\label{STANDMODEL EX}
	Some examples of standard models, i.e. objects of $\Omega^s$, follow (further, Example \ref{PLANAREX EX} can also be interpreted as such an example).
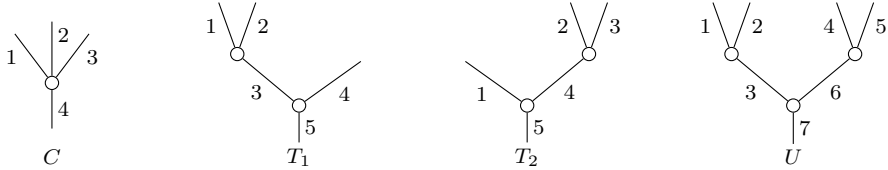
\[
	\begin{tikzpicture}[grow=up,auto,level distance=2.1em,
	every node/.style = {font=\footnotesize,inner sep=2pt},
	dummy/.style={circle,draw,inner sep=0pt,minimum size=1.75mm}]
		\node at (-0.25,0) {$C$};
		\node at (-0.25,0.3) {}
			child{node [dummy] {}
				child[sibling distance = 1.5em,level distance= 2em]{
				edge from parent node [swap, near end] {$3$}}
				child[sibling distance = 1.5em,level distance= 2.5em]{
				edge from parent node [swap, near end] {$2$}}
				child[sibling distance = 1.5em,level distance= 2em]{
				edge from parent node [near end] {$1$}}
			edge from parent node [swap] {$4$}};
		\node at (3,0) {$T_1$}
			child{node [dummy] {}
				child[sibling distance = 5em, level distance=1.8em]{
				edge from parent node [swap] {$4$}}
				child[sibling distance = 5em]{node [dummy] {}
					child[sibling distance = 1.5em]{
					edge from parent node [swap,near end] {$2$}}
					child[sibling distance = 1.5em]{
					edge from parent node [near end] {$1$}}
				edge from parent node {$3$}}
			edge from parent node [swap] {$5$}};
		\node at (6,0) {$T_2$}
			child{node [dummy] {}
				child[sibling distance = 5em]{node [dummy] {}
					child[sibling distance = 1.5em]{
					edge from parent node [swap,near end] {$3$}}
					child[sibling distance = 1.5em]{
					edge from parent node [near end] {$2$}}
				edge from parent node [swap] {$4$}}
				child[sibling distance = 5em, level distance=1.8em]{
				edge from parent node {$1$}}
			edge from parent node [swap] {$5$}};
		\node at  (9.5,0) {$U$}
			child{node [dummy] {}
				child[sibling distance = 5em]{node [dummy] {}
					child[sibling distance = 1.5em]{
					edge from parent node [swap,near end] {$5$}}
					child[sibling distance = 1.5em]{
					edge from parent node [near end] {$4$}}
				edge from parent node [swap] {$6$}}
				child[sibling distance = 5em]{node [dummy] {}
					child[sibling distance = 1.5em]{
					edge from parent node [swap,near end] {$2$}}
					child[sibling distance = 1.5em]{
					edge from parent node [near end] {$1$}}
				edge from parent node {$3$}}
			edge from parent node [swap] {$7$}};
	\end{tikzpicture}
\]
Here $T_1$ and $T_2$ are isomorphic to each other but not isomorphic to any other standard model in $\Omega^s$ while both $C$ and $U$ are the unique objects in their isomorphism classes. 
\end{example}

Given $T_{\leq_p} \in \Omega^p$ there is an obvious standard model $T_{\leq_p}^s \in \Omega^s$ given by replacing each edge by its order following $\leq_p$. Indeed, this defines a retraction 
$(\minus)^s \colon \Omega^p \to \Omega^s$
and a natural transformation 
$\sigma \colon id \Rightarrow (\minus)^s$
given by isomorphisms preserving the planar structure
(in fact, the pair $\left((\minus)^s, \sigma \right)$ is  uniquely characterized by this property).

\begin{remark}\label{FORESTPLAN REM}
	Definition \ref{PLANARIZE DEF} readily extends to 
	the broad poset definition of forests $F \in \Phi$ 
	in \cite[Def. 5.27]{Pe17}, with the analogue of
	Proposition \ref{PLANARIZATIONCHAR PROP}
	then stating that a planar structure is 
equivalent to total orderings of the nodes of $F$ together with a total ordering of its set of roots.
There are thus two competing notions of standard forests: the \cite[Def. 5.27]{Pe17} model $\Phi^s$ whose objects are planar forest structures on one of the standard sets $\{1,\cdots,n\}$ and (following the discussion at the start of \S \ref{PLANAR_SECTION})
the model $\Fin \wr \Omega^s$, whose objects are tuples, indexed by a standard set, of planar tree structures on standard sets.
An illustration follows.
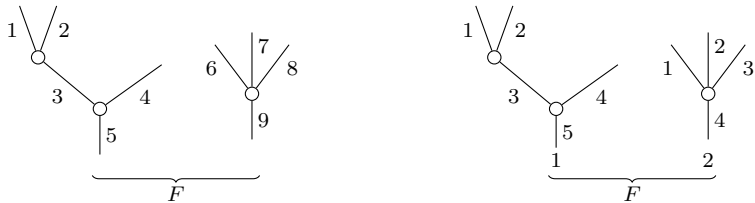
\[
	\begin{tikzpicture}[grow=up,auto,level distance=2.1em,
	every node/.style = {font=\footnotesize,inner sep=2pt},
	dummy/.style={circle,draw,inner sep=0pt,minimum size=1.75mm}]
		\node at (2,0.2) {}
			child{node [dummy] {}
				child[sibling distance = 1.5em,level distance= 2em]{
				edge from parent node [swap, near end] {$8$}}
				child[sibling distance = 1.5em,level distance= 2.5em]{
				edge from parent node [swap, near end] {$7$}}
				child[sibling distance = 1.5em,level distance= 2em]{
				edge from parent node [near end] {$6$}}
			edge from parent node [swap] {$9$}};
		\node at (0,0) {}
			child{node [dummy] {}
				child[sibling distance = 5em, level distance=1.8em]{
				edge from parent node [swap] {$4$}}
				child[sibling distance = 5em]{node [dummy] {}
					child[sibling distance = 1.5em]{
					edge from parent node [swap,near end] {$2$}}
					child[sibling distance = 1.5em]{
					edge from parent node [near end] {$1$}}
				edge from parent node {$3$}}
			edge from parent node [swap] {$5$}};
		\node at (8,0) {$2$};
		\node at (8,0.2) {}
			child{node [dummy] {}
				child[sibling distance = 1.5em,level distance= 2em]{
				edge from parent node [swap, near end] {$3$}}
				child[sibling distance = 1.5em,level distance= 2.5em]{
				edge from parent node [swap, near end] {$2$}}
				child[sibling distance = 1.5em,level distance= 2em]{
				edge from parent node [near end] {$1$}}
			edge from parent node [swap] {$4$}};
		\node at (6,0) {$1$}
			child{node [dummy] {}
				child[sibling distance = 5em, level distance=1.8em]{
				edge from parent node [swap] {$4$}}
				child[sibling distance = 5em]{node [dummy] {}
					child[sibling distance = 1.5em]{
					edge from parent node [swap,near end] {$2$}}
					child[sibling distance = 1.5em]{
					edge from parent node [near end] {$1$}}
				edge from parent node {$3$}}
			edge from parent node [swap] {$5$}};
		\draw[decorate,decoration={brace,amplitude=2.5pt}] (2.1,-0.2) -- (-0.1,-0.2) node[midway,inner sep=4pt]{$F$}; %
		\draw[decorate,decoration={brace,amplitude=2.5pt}] (8.1,-0.2) -- (5.9,-0.2) node[midway,inner sep=4pt]{$F$}; %
	\end{tikzpicture}
\]
However, there is a 
\textit{canonical} isomorphism $\Phi^s \simeq \Fin \wr \Omega^s$ 
(with both sides of the diagram above then
depicting the same planar forest). 
Moreover, while the similarly defined categories $\Phi^p$
and $\Fin \wr \Omega^p$ are only equivalent (rather than isomorphic), their retractions onto $\Phi^s \simeq \Fin \wr \Omega^s$ are compatible, and we will thus henceforth not distinguish between 
$\Phi^s$ and $\Fin \wr \Omega^s$.
\end{remark}

\begin{convention}\label{PLANARCONV CON}
      From now on we write simply $\Omega$, $\Omega_G$ to denote the categories $\Omega^s$, $\Omega_G^s$ of standard models (where planar structures are defined in the underlying forest as in Remark \ref{FORESTPLAN REM}).
      \index{categories!of trees!Gtrees1@$\Omega_G$}
      Therefore, whenever a construction produces an object or diagram in $\Omega^p$ or $\Omega^p_G$,
      we always implicitly reinterpret it by using the standardization functor $(\minus)^s$.
      
      Similarly, any finite set (resp. orbital finite $G$-set) together with a total order is implicitly reinterpreted as an object of
      $\Fin$ (resp. $\mathsf{O}_G$).
\end{convention}

\begin{example}
To illustrate our convention, consider the trees in Example \ref{STANDMODEL EX}. 

There are subtrees
$F_1 \hookrightarrow F_2 \hookrightarrow U$,
where $F_1$ is the subtree with edge set $\{1,2,6,7\}$,
and $F_2$ is the subtree with edge set $\{1,2,3,6,7\}$, both with inherited tree and planar structures. 
Applying $(\minus)^s$ to the inclusion diagram on the left below then yields a diagram as on the right.
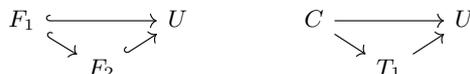
\[
\begin{tikzcd}[row sep = 0.5em,column sep =1.3em]
	F_1 \ar[hookrightarrow]{rr} \ar[hookrightarrow]{rd} & & U & &&
	C \ar{rr} \ar{rd} & & U
\\
	& F_2 \ar[hookrightarrow]{ru} & & &&
	& T_1 \ar{ru}
\end{tikzcd}
\]
Similarly, let $\leq_{(12)}$ and $\leq_{(45)}$ denote alternate planar structures for $U$ exchanging the orders of the pairs $1,2$ and $4,5$, so that one has objects 
$U_{\leq_{(12)}}$, $U_{\leq_{(45)}}$ in $\Omega^p$. 
Applying $(\minus)^s$ to the diagram of underlying identities on the left yields the permutation diagram on the right.
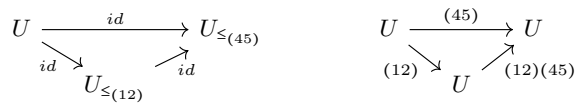
\[
\begin{tikzcd}[row sep = 0.5em,column sep =1.3em]
	U \ar{rr}{id} \ar{rd}[swap]{id} & & U_{\leq_{(45)}} & & &
	U \ar{rr}{(45)} \ar{rd}[swap]{(12)} & & U
\\
	& U_{\leq_{(12)}} \ar{ru}[swap]{id} & & & &
	& U \ar{ru}[swap]{(12)(45)}
\end{tikzcd}
\]
\end{example}

\begin{example}
An additional reason to leave the use of $(\minus)^s$ implicit
as described in Convention \ref{PLANARCONV CON} is that when depicting $G$-trees it is preferable to choose edge labels that describe the $G$-action rather than the planarization (which is already implicit anyway).

For example, for the two groups 
$G = \mathbb{Z}_{/4}$ and 
$\bar{G} = \mathbb{Z}_{/3}$, in both diagrams below the orbital representation on the left represents the isomorphism class consisting only of the two trees 
$T_1,T_2 \in \Omega_G$ and
$\bar{T}_1,\bar{T}_2 \in \Omega_{\bar{G}}$
on the right.
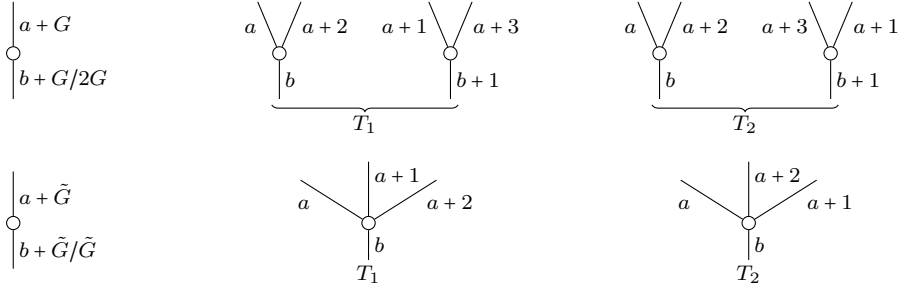
\[
	\begin{tikzpicture}[grow=up,auto,level distance=2.1em,
	every node/.style = {font=\footnotesize,inner sep=2pt},
	dummy/.style={circle,draw,inner sep=0pt,minimum size=1.75mm}]
%
		\node at (-1,-2) {}
			child{node [dummy] {}
				child[sibling distance=1.75em]{
				edge from parent node [swap]  {$a+G$}}
			edge from parent node [swap] {$b+G/2G$}};
		\node at (2.5,-2) {}
			child{node [dummy] {}
				child[sibling distance=1.75em]{
				edge from parent node [swap,near end] {$a+2$}}
				child[sibling distance=1.75em]{
				edge from parent node [near end]  {$\phantom{1+}a$}}
			edge from parent node [swap] {$b$}};
		\node at (4.75,-2) {}
			child{node [dummy] {}
				child[sibling distance=1.75em]{
				edge from parent node [swap,near end] {$a+3$}}
				child[sibling distance=1.75em]{
				edge from parent node [near end]  {$a+1$}}
			edge from parent node [swap] {$b+1$}};
		\draw[decorate,decoration={brace,amplitude=2.5pt}] (4.85,-2) -- (2.4,-2) node[midway,inner sep=4pt]{$T_1$};
		\node at (7.5,-2) {}
			child{node [dummy] {}
				child[sibling distance=1.75em]{
				edge from parent node [swap,near end] {$a+2$}}
				child[sibling distance=1.75em]{
				edge from parent node [near end]  {$\phantom{1+}a$}}
			edge from parent node [swap] {$b$}};
		\node at (9.75,-2) {}
			child{node [dummy] {}
				child[sibling distance=1.75em]{
				edge from parent node [swap,near end] {$a+1$}}
				child[sibling distance=1.75em]{
				edge from parent node [near end]  {$a+3$}}
			edge from parent node [swap] {$b+1$}};
		\draw[decorate,decoration={brace,amplitude=2.5pt}] (9.85,-2) -- (7.4,-2) node[midway,inner sep=4pt]{$T_2$};
		\node at (-1,-4.25) {}
			child{node [dummy] {}
				child[sibling distance=1.75em]{
				edge from parent node [swap]  {$a+\bar{G}$}}
			edge from parent node [swap] {$b+\bar{G}/\bar{G}$}};
		\node at (3.6725,-4.25) {$\bar{T}_1$}
			child{node [dummy] {}
				child[sibling distance = 2.75em,level distance= 1.75em]{
				edge from parent node [swap, near end] {$a+2$}}
				child[sibling distance = 2.75em,level distance= 2.5em]{
				edge from parent node [swap, near end] {$a+1$}}
				child[sibling distance = 2.75em,level distance= 1.75em]{
				edge from parent node [near end] {$\phantom{1+}a$}}
			edge from parent node [swap] {$b$}};
		\node at (8.6725,-4.25) {$\bar{T}_2$}
			child{node [dummy] {}
				child[sibling distance = 2.75em,level distance= 1.75em]{
				edge from parent node [swap, near end] {$a+1$}}
				child[sibling distance = 2.75em,level distance= 2.5em]{
				edge from parent node [swap, near end] {$a+2$}}
				child[sibling distance = 2.75em,level distance= 1.75em]{
				edge from parent node [near end] {$\phantom{1+}a$}}
			edge from parent node [swap] {$b$}};
	\end{tikzpicture}
\]
In general, isomorphism classes are of course far bigger.
The interested reader may show that there are 
$3 \cdot 3! \cdot 2 \cdot 3! \cdot 3!$
trees in the isomorphism class of the tree depicted in 
\eqref{D6SMALLER EQ}.
\end{example}

We now turn to the notion of \emph{planar map}.
In order to cover the case of forests, 
we need to recall
the notion of \emph{independent map} of forests
introduced in \cite[Def. 5.28]{Pe17}.
However, rather than work with the definition in 
\cite{Pe17}, we prefer a different characterization, as follows.

\begin{proposition}\label{INDMAPCHAR PROP}
	Let $F \xrightarrow{\varphi} F'$ be a map of forests.
	The following are equivalent:
\begin{enumerate}
	\item[(i)] $\varphi$ is an independent map in the sense of
	\cite[Def. 5.28]{Pe17};
	\item[(ii)] for any edges $e,\bar{e}$ of $F$,
	the edges
	$\varphi(e),\varphi(\bar{e})$ of $F'$
	are $\leq_d$-incomparable iff 
	$e,\bar{e}$ are;
	\item[(iii)] for distinct roots $r,\bar{r}$ of $F$,
	the edges
	$\varphi(r),\varphi(\bar{r})$ of $F'$
	are $\leq_d$-incomparable.
\end{enumerate}
\end{proposition}

\begin{proof}
	$(i) \Rightarrow (ii)$
	is the content of \cite[Lemma 5.32]{Pe17}.
	$(ii) \Rightarrow (iii)$ is clear.
	Lastly, 
	$(iii) \Rightarrow (i)$ follows by applying 
	\cite[Lemma 5.24]{Pe17} to each of the tree components of $F'$.
\end{proof}

\begin{remark}
	By (iii) above,
	a map $F \xrightarrow{\varphi} F'$ is independent whenever $F$ is a tree. 
	More generally, (ii) can hence only fail
	if $e,\bar{e}$ are in distinct tree components of $F$.
	Thus, independent maps admit the following informal description:
	$\varphi$ is independent if, for any two tree components
	$T,\bar{T}$ of $F$,
	the images of $T,\bar{T}$ are ``in separate branches of $F'$'',
	in the sense that the image of $T$ contains no edges above (or on) the image of $\bar{T}$, and vice versa.
\end{remark}

\begin{definition}\label{PLANARMAP_DEF}
	A map $S \xrightarrow{\varphi} T$ in the category $\Omega$ of forests preserving the planar structure $\leq_p$
	is called a \textit{planar map}.
	
	More generally, a map $F \xrightarrow{\varphi} F'$ in one of the categories $\Phi$, $\Phi^G$, $\Omega_G$ of forests, $G$-forests, $G$-trees is called a \textit{planar map} if it is an independent map that preserves the planar structures $\leq_p$.
\end{definition}

\begin{remark}
The need for independence is justified by
condition (iii) in Proposition \ref{INDMAPCHAR PROP}.
\end{remark}

\begin{remark}\label{INDOMGALT REM}
In the case of $\Omega_G$,
independence admits simpler characterizations:
$\varphi$ is independent iff $\varphi$ is injective on each edge orbit iff $\varphi$ is injective on the root orbit.

To see this, note first that distinct edges $e, g e$ in the same orbit must be $\leq_d$-incomparable.
Indeed, if it were 
$e \leq_d g e$ (the $g e \leq_d e$ case is similar)
it would be
$e \leq_d g e \leq_d g^2 e \leq_d \cdots
\leq_d g^n e = e$ (here $n$ is the order of $g$),
requiring $e=ge$.
The given characterizations now follow from
Proposition \ref{INDMAPCHAR PROP}(ii)(iii)
and the fact that for $F \in \Omega_G$ the roots form a single orbit.
\end{remark}

\begin{proposition}
\label{PLANARPULL PROP}
	Let $F \xrightarrow{\varphi} F'$ be an independent map in $\Phi$ (or $\Omega$, $\Omega_G$, $\Phi^G$). 
	One has a unique factorization 
	\[F \xrightarrow{\simeq} \bar{F} \to F'\]
	such that $F \xrightarrow{\simeq} \bar{F}$ is an isomorphism and $\bar{F} \to F'$ is planar.
\end{proposition}

\begin{proof}
We need to show that there is a unique planar structure 
$\leq_p^{\bar{F}}$ on the underlying forest of $F$ making the underlying map a planar map.
Simplicity \cite[Def. 5.3]{Pe17}
of the broad poset $F'$ ensures that for any vertex $e^{\uparrow} \leq e$ of $F$ the edges in $\varphi(e^{\uparrow})$ are all distinct while independence of $\varphi$ likewise ensures that the edges in $\varphi(\underline{r}_F)$ are distinct.
By (the forest version of) Proposition
\ref{PLANARIZATIONCHAR PROP},
the only possible planar structure $\leq_p^{\bar{F}}$
is the one which orders each set $e^{\uparrow}$ and the root tuple $\underline{r}_F$ according to their images.
The claim that $\varphi$ is then planar follows from 
Remark \ref{CLOSURE REM}
together with the
fact that $\varphi$ reflects $\leq_d$-comparability,
cf. Proposition \ref{INDMAPCHAR PROP}(ii).
\end{proof}

\begin{remark}\label{PULLPLANAR REM}
Proposition \ref{PLANARPULL PROP} says that planar structures can be pulled back along independent maps. However, they can not always be pushed forward. As a counter-example, in the setting of Example \ref{STANDMODEL EX}, consider the map $C \to T_1$ given by $1 \mapsto 1$, $2 \mapsto 4$, $3 \mapsto 2$, $4 \mapsto 5$.
\end{remark}

We end this section by discussing a different type of pullback.
The reader may have noticed that it follows from 
Proposition \ref{GROTHSTAB PROP}
that both vertical maps in \eqref{OGDEF EQ}
are split Grothendieck fibrations. We now introduce some terminology.

\begin{definition}\label{ROOTPULL DEF}
The map $\mathsf{r} \colon \Omega_G \to \mathsf{O}_G$
in \eqref{OGDEF EQ} is called the \textit{root functor}.
\index{key functors!Gtreesleafroot0@$\mathsf{r}$}

Further, fiber maps (those maps inducing identities, i.e. ordered bijections, on $\mathsf{r}(\minus)$) are called \textit{rooted maps}, and pullbacks with respect to $\mathsf{r}$ are called
\textit{root pullbacks}.
\end{definition}

To motivate the terminology, 
note first that unpacking definitions shows that 
$\mathsf{r}(T)$ is the ordered set of tree components of  
$T\in \Omega_G$,
which coincides with the ordered set of roots.
The exact name choice is meant to accentuate the connection with another key functor
described in \S \ref{LRVERT SEC},
which we call the \textit{leaf-root functor}.

Further, unpacking the construction in \eqref{OGDEF EQ}, one sees that,
for a $G$-tree $T = (T_x)_{x \in X}$
with structure maps $T_x \to T_{g x}$,
the pullback of $T$
along the map 
$\varphi \colon Y \to X$ in $\mathsf{O}_G$
is simply the $G$-tree
$\varphi^{\**}T=(T_{\varphi(y)})_{y \in Y}$
with structure maps 
$T_{\varphi(y)} \to T_{g \varphi(y)} = T_{\varphi(g y)}$.

\begin{example}\label{ROOTPULL EX}
Let $G=\{\pm 1, \pm i, \pm j, \pm k\}$ be the group of quaternionic units, 
$H = \langle j \rangle$ and $K = \langle -1 \rangle$.
Figure \ref{FIGURE} illustrates the pullbacks of two $G$-trees, 
$T$ and $S$,
along the 
twist map $\tau \colon G/H \to G/H$
and the unique map $\pi \colon G/H \to G/G$, respectively
(or, more precisely, noting that in our model the underlying set 
of $G/H$ is actually $\{1,2\}$,
$\tau$ is the permutation $(12)$).
We note that the stabilizers of $a,b,c$ are $\{1\},K,H$ for $T$
and $K,H,G$ for $S$.

The pullback $\tau^{\**} T$ along the map $\tau$
is obtained by interchanging the two tree components of $T$,
as in the top depiction of $\tau^{\**} T$. 
However, one drawback of this top depiction 
is that the edge orbit generators $a,b,c$ now
appear in the middle of the forest.
By choosing the leftmost 
edge orbit generators
$d = i a$, $e = i b$, $f = i c$,
one obtains the bottom depiction of $\tau^{\**} T$. 

For the pullback $\pi^{\**} S$,
since $\pi$ folds two points into one,
the underlying forest of $\pi^{\**} S$
consists of two copies of the underlying tree of $S$,
with $\pi^{\**}S \to S$
folding those copies while respecting the planarizations.
The top depiction of $\pi^{\**} S$ then
chooses edge orbit generators 
$a,b,c,\bar{a},\bar{b}$
that are as left as possible while also 
lifting the generators $a,b,c$ of $S$.
The bottom depiction of $\pi^{\**} S$,
which sets $d = i \bar{a}$, $e = i \bar{b}$, 
chooses the leftmost possible generators.
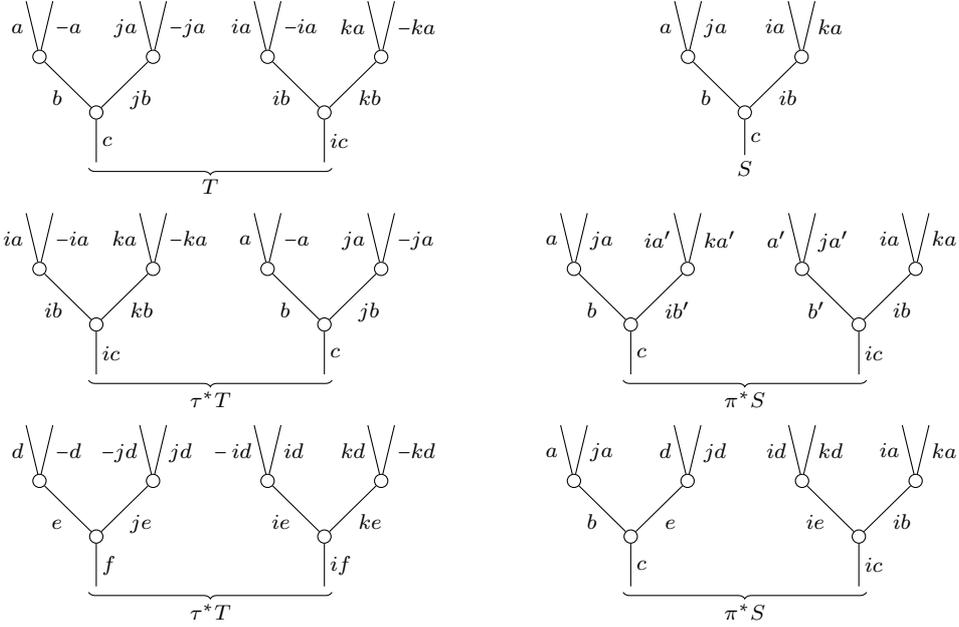
\begin{figure}[ht]
\[
	\begin{tikzpicture}[grow=up,auto,level distance=2.1em,
	every node/.style = {font=\footnotesize,inner sep =2pt},
	dummy/.style={circle,draw,inner sep=0pt,minimum size=1.75mm}]
	\begin{scope}[yshift=12em]
		\node at  (0,0) {}
			child{node [dummy] {}
				child[sibling distance = 4.25em]{node [dummy] {}
					child[sibling distance = 1em]{
					edge from parent node [swap,near end] {$-j a$}}
					child[sibling distance = 1em]{
					edge from parent node [near end] {$j a$}}
				edge from parent node [swap] {$j b$}}
				child[sibling distance = 4.25em]{node [dummy] {}
					child[sibling distance = 1em]{
					edge from parent node [swap,near end] {$-a\phantom{j}$}}
					child[sibling distance = 1em]{
					edge from parent node [near end] {$\phantom{j}a$}}
				edge from parent node {$b$}}
			edge from parent node [swap] {$c$}};
		\node at  (3,0) {}
			child{node [dummy] {}
				child[sibling distance = 4.25em]{node [dummy] {}
					child[sibling distance = 1em]{
					edge from parent node [swap,near end] {$- k a$}}
					child[sibling distance = 1em]{
					edge from parent node [near end] {$k a$}}
				edge from parent node [swap] {$k b$}}
				child[sibling distance = 4.25em]{node [dummy] {}
					child[sibling distance = 1em]{
					edge from parent node [swap,near end] {$- i a$}}
					child[sibling distance = 1em]{
					edge from parent node [near end] {$i a$}}
				edge from parent node {$i b$}}
			edge from parent node [swap] {$i c$}};
		\draw[decorate,decoration={brace,amplitude=2.5pt}] (3.1,0) -- (-0.1,0) node[midway,inner sep=4pt,font=\normalsize]{$T$};
	\end{scope}
	\begin{scope}[yshift=4em]
		\node at  (3,0) {}
			child{node [dummy] {}
				child[sibling distance = 4.25em]{node [dummy] {}
					child[sibling distance = 1em]{
					edge from parent node [swap,near end] {$-j a$}}
					child[sibling distance = 1em]{
					edge from parent node [near end] {$j a$}}
				edge from parent node [swap] {$j b$}}
				child[sibling distance = 4.25em]{node [dummy] {}
					child[sibling distance = 1em]{
					edge from parent node [swap,near end] {$-a\phantom{j}$}}
					child[sibling distance = 1em]{
					edge from parent node [near end] {$\phantom{+1}a$}}
				edge from parent node {$b$}}
			edge from parent node [swap] {$c$}};
		\node at  (0,0) {}
			child{node [dummy] {}
				child[sibling distance = 4.25em]{node [dummy] {}
					child[sibling distance = 1em]{
					edge from parent node [swap,near end] {$- k a$}}
					child[sibling distance = 1em]{
					edge from parent node [near end] {$k a$}}
				edge from parent node [swap] {$k b$}}
				child[sibling distance = 4.25em]{node [dummy] {}
					child[sibling distance = 1em]{
					edge from parent node [swap,near end] {$- i a$}}
					child[sibling distance = 1em]{
					edge from parent node [near end] {$i a$}}
				edge from parent node {$i b$}}
			edge from parent node [swap] {$i c$}};
		\draw[decorate,decoration={brace,amplitude=2.5pt}] (3.1,0) -- (-0.1,0) node[midway,inner sep=4pt,font=\normalsize]{$\tau^{\**}T$};
	\end{scope}
	\begin{scope}[yshift=-4em]
		\node at  (0,0) {}
			child{node [dummy] {}
				child[sibling distance = 4.25em]{node [dummy] {}
					child[sibling distance = 1em]{
					edge from parent node [swap,near end] {$j d$}}
					child[sibling distance = 1em]{
					edge from parent node [near end] {$-j d$}}
				edge from parent node [swap] {$j e$}}
				child[sibling distance = 4.25em]{node [dummy] {}
					child[sibling distance = 1em]{
					edge from parent node [swap,near end] {$-d\phantom{j}$}}
					child[sibling distance = 1em]{
					edge from parent node [near end] {$\phantom{+1}d$}}
				edge from parent node {$\phantom{j}e$}}
			edge from parent node [swap] {$f$}};
		\node at  (3,0) {}
			child{node [dummy] {}
				child[sibling distance = 4.25em]{node [dummy] {}
					child[sibling distance = 1em]{
					edge from parent node [swap,near end] {$-k d$}}
					child[sibling distance = 1em]{
					edge from parent node [near end] {$k d$}}
				edge from parent node [swap] {$k e$}}
				child[sibling distance = 4.25em]{node [dummy] {}
					child[sibling distance = 1em]{
					edge from parent node [swap,near end] {$i d\phantom{j}$}}
					child[sibling distance = 1em]{
					edge from parent node [near end] {$\phantom{+1}-i d$}}
				edge from parent node {$i e$}}
			edge from parent node [swap] {$i f$}};
		\draw[decorate,decoration={brace,amplitude=2.5pt}] (3.1,0) -- (-0.1,0) node[midway,inner sep=4pt,font=\normalsize]{$\tau^{\**}T$};
	\end{scope}
	\begin{scope}[yshift=12em,xshift=20em]
		\node at  (1.5,0) [font=\normalsize] {$S$}
			child{node [dummy] {}
				child[sibling distance = 4.25em]{node [dummy] {}
					child[sibling distance = 1em]{
					edge from parent node [swap,near end] {$k a$}}
					child[sibling distance = 1em]{
					edge from parent node [near end] {$i a$}}
				edge from parent node [swap] {$i b$}}
				child[sibling distance = 4.25em]{node [dummy] {}
					child[sibling distance = 1em]{
					edge from parent node [swap,near end] {$j a\phantom{j}$}}
					child[sibling distance = 1em]{
					edge from parent node [near end] {$\phantom{j}a$}}
				edge from parent node {$b$}}
			edge from parent node [swap] {$c$}};
	\end{scope}
	\begin{scope}[yshift=4em,xshift=20em]
		\node at  (0,0) {}
			child{node [dummy] {}
				child[sibling distance = 4.25em]{node [dummy] {}
					child[sibling distance = 1em]{
					edge from parent node [swap,near end] {$k \bar{a}$}}
					child[sibling distance = 1em]{
					edge from parent node [near end] {$i \bar{a}$}}
				edge from parent node [swap] {$i \bar{b}$}}
				child[sibling distance = 4.25em]{node [dummy] {}
					child[sibling distance = 1em]{
					edge from parent node [swap,near end] {$j a\phantom{j}$}}
					child[sibling distance = 1em]{
					edge from parent node [near end] {$\phantom{+1}a$}}
				edge from parent node {$\phantom{\bar{b}j}b$}}
			edge from parent node [swap] {$c$}};
		\node at  (3,0) {}
			child{node [dummy] {}
				child[sibling distance = 4.25em]{node [dummy] {}
					child[sibling distance = 1em]{
					edge from parent node [swap,near end] {$k a$}}
					child[sibling distance = 1em]{
					edge from parent node [near end] {$i a$}}
				edge from parent node [swap] {$i b\phantom{\bar{b}}$}}
				child[sibling distance = 4.25em]{node [dummy] {}
					child[sibling distance = 1em]{
					edge from parent node [swap,near end] {$j \bar{a}\phantom{j}$}}
					child[sibling distance = 1em]{
					edge from parent node [near end] {$\phantom{j}\bar{a}$}}
				edge from parent node {$\bar{b}$}}
			edge from parent node [swap] {$i c$}};
		\draw[decorate,decoration={brace,amplitude=2.5pt}] (3.1,0) -- (-0.1,0) node[midway,inner sep=4pt,font=\normalsize]{$\pi^{\**}S$};
	\end{scope}
	\begin{scope}[yshift=-4em,xshift=20em]
		\node at  (0,0) {}
			child{node [dummy] {}
				child[sibling distance = 4.25em]{node [dummy] {}
					child[sibling distance = 1em]{
					edge from parent node [swap,near end] {$j d$}}
					child[sibling distance = 1em]{
					edge from parent node [near end] {$\phantom{j}d$}}
				edge from parent node [swap] {$e{\phantom{b}}$}}
				child[sibling distance = 4.25em]{node [dummy] {}
					child[sibling distance = 1em]{
					edge from parent node [swap,near end] {$j a\phantom{j}$}}
					child[sibling distance = 1em]{
					edge from parent node [near end] {$\phantom{j}a$}}
				edge from parent node {$\phantom{j}b$}}
			edge from parent node [swap] {$c$}};
		\node at  (3,0) {}
			child{node [dummy] {}
				child[sibling distance = 4.25em]{node [dummy] {}
					child[sibling distance = 1em]{
					edge from parent node [swap,near end] {$k a$}}
					child[sibling distance = 1em]{
					edge from parent node [near end] {$i a$}}
				edge from parent node [swap] {$i b$}}
				child[sibling distance = 4.25em]{node [dummy] {}
					child[sibling distance = 1em]{
					edge from parent node [swap,near end] {$k d\phantom{j}$}}
					child[sibling distance = 1em]{
					edge from parent node [near end] {$\phantom{+1}i d$}}
				edge from parent node {$i e$}}
			edge from parent node [swap] {$i c$}};
		\draw[decorate,decoration={brace,amplitude=2.5pt}] (3.1,0) -- (-0.1,0) node[midway,inner sep=4pt,font=\normalsize]{$\pi^{\**}S$};
	\end{scope}
	\end{tikzpicture}
\]
\caption{Root pullbacks}
\label{FIGURE}
\end{figure}
\end{example}

\subsection{Outer faces, tall maps, and substitution}\label{OUTTALL SEC}
\label{SUBS SEC}

One of the key ideas needed 
to describe the free operad monad is
the notion of \textit{substitution} of tree nodes,
a process that we will prefer to repackage in terms of maps of trees.

In preparation for that discussion,
we first recall some basic definitions and results concerning outer subtrees and tree grafting, as in \cite[\S 5]{Pe17}.

\begin{definition}\label{OUTFACE DEF}
	Let $T \in \Omega$ be a tree and 
	$e_1 \cdots e_n =\underline{e} \leq e$ a broad relation in $T$.
	
	We define the \textit{planar outer face $T_{\underline{e} \leq e}$}
	to be the subtree with underlying set those edges $f \in T$ such that
\begin{equation}\label{OUTERFACE EQ}
	f \leq_d e,\quad \forall_i f \nless_d e_i,
\end{equation}
\index{structure on trees!outer faces!planarouterface@$T_{\underline{e} \leq e}$}
with generating broad relations the relations $f^{\uparrow} \leq f$ for those $f \in T$ satisfying
$\forall_i f\neq e_i$
in addition to \eqref{OUTERFACE EQ},
and planar structure pulled back from $T$ (in the sense of Remark \ref{PULLPLANAR REM}).

Moreover, inclusions of the form 
$T_{\underline{e} \leq e}
\hookrightarrow T$
are called \emph{planar outer face maps}.
\end{definition}

\begin{remark}
If one forgoes the requirement that $T_{\underline{e} \leq e}$ be equipped with the pulled back planar structure, the inclusion $T_{\underline{e} \leq e} \hookrightarrow T$ is usually called simply an \textit{outer face map}.
\end{remark}

\begin{example}\label{OUTERTREE EX}
	The following illustrates some planar outer faces of the tree $T$ in Example \ref{PLANAREX EX}.
	Pictorially, Definition \ref{OUTFACE DEF}
	says that $T_{\underline{e} \leq e}$
	is obtained from $T$ by removing those edges and vertices that are either
	not above $e$ or
	above one of the $e_i$ in $\underline{e}$.
	\[
	\begin{tikzpicture}[grow=up,auto,level distance=2.3em,
	every node/.style = {font=\footnotesize,inner sep=2pt},
	dummy/.style={circle,draw,inner sep=0pt,minimum size=1.75mm}]
	\node[font=\normalsize] at (0,0) {$T_{34 \leq 7}$}
		child[sibling distance = 8em]{node [dummy] {}
			child[sibling distance =2.25em, level distance = 1.25 em]{node [dummy] {}
			edge from parent node [swap] {$6$}}
			child[sibling distance = 1.5em]{node [dummy] {}
				child[sibling distance =1em]{
				edge from parent node [swap,near end] {$4$}}
				child[sibling distance =1em]{
				edge from parent node [near end] {$3$}}
			edge from parent node [very near end,swap] {$5$}}
			child[sibling distance =1.5em]{node [dummy] {}
			edge from parent node [very near end] {$2$}}
			child[sibling distance =2.25em,level distance =1.25em]{node [dummy] {}
			edge from parent node {$1$}}
		edge from parent node [swap] {$7$}};
\begin{scope}[xshift=8.5em]
\node[font=\normalsize] at (0,0) {$T_{2346 \leq 7}$}
child[sibling distance = 8em]{node [dummy] {}
	child[sibling distance =2.25em, level distance = 1.25 em]{
	edge from parent node [swap] {$6$}}
	child[sibling distance = 1.5em]{node [dummy] {}
		child[sibling distance =1em]{
		edge from parent node [swap,near end] {$4$}}
		child[sibling distance =1em]{
		edge from parent node [near end] {$3$}}
	edge from parent node [very near end,swap] {$5$}}
	child[sibling distance =1.5em]{node {}
	edge from parent node [very near end] {$2$}}
	child[sibling distance =2.25em,level distance =1.25em]{node [dummy] {}
	edge from parent node {$1$}}
edge from parent node [swap] {$7$}};
\end{scope}
\begin{scope}[xshift=17em]
\node[font=\normalsize] at (0,0) {$T_{1256 \leq 7}$}
child[sibling distance = 8em]{node [dummy] {}
	child[sibling distance =2.25em, level distance = 1.25 em]{node {}
	edge from parent node [swap] {$6$}}
	child[sibling distance = 1.5em]{node {}
	edge from parent node [very near end,swap] {$5$}}
	child[sibling distance =1.5em]{node {}
	edge from parent node [very near end] {$2$}}
	child[sibling distance =2.25em,level distance =1.25em]{node {}
	edge from parent node {$1$}}
edge from parent node [swap] {$7$}};
\end{scope}
\begin{scope}[xshift=22em]
\node[font=\normalsize] at (0,0) {$T_{\epsilon \leq 2}$}
child[sibling distance = 8em]{node [dummy] {}
edge from parent node [swap] {$2$}};
\end{scope}
\begin{scope}[xshift=25em]
\node[font=\normalsize] at (0,0) {$T_{2 \leq 2}$}
child[sibling distance = 8em]{node {}
	edge from parent node [swap] {$2$}};
\end{scope}
\begin{scope}[xshift=30em]
\node[font=\normalsize] at (0,0) {$T_{789ab \leq d}$}
child{node [dummy] {}
	child[sibling distance = 3em]{node [dummy] {}
		child[sibling distance = 1.75em]{
		edge from parent node [near end,swap] {$b$}}
		child[level distance=2.5em]{
		edge from parent node [very near end,swap] {$a$}}				
		child[sibling distance = 1.75em]{
		edge from parent node [near end] {$9$}}
	edge from parent node [swap] {c}}
	child[level distance =2.5em]{
	edge from parent node [swap] {$8$}}
	child[sibling distance = 3em]{
	edge from parent node {$7$}}
edge from parent node [swap] {$d$}};
\end{scope}
	\end{tikzpicture}
	\]
\end{example}

We now recap some basic results.

\begin{notation}\label{STICKTRE NOT}
	We write $\eta \in \Omega$
	for the \emph{stick tree} 
	consisting of a single edge and no vertices.
\end{notation}

\begin{proposition}
Let $T \in \Omega$ be a tree.
\begin{itemize}
\item[(a)] $T_{\underline{e} \leq e}$ is a tree with root $e$
and leaf tuple $\underline{e}$;
\item[(b)] there is a bijection
\[
	\{\text{planar outer faces of $T$} \} 
\leftrightarrow 
	\{\text{broad relations of $T$}\};
\]
\item[(c)] if $R \to S$ and $S \to T$ are (planar) outer face maps then so is $R \to T$;
\item[(d)] any pair of broad relations $\underline{g} \leq v$, $\underline{f}v \leq e$ induces a grafting pushout diagram
\begin{equation}\label{GRAPTPUSH EQ}
\begin{tikzcd}
	\eta \ar{r}{v} \ar{d}[swap]{v} & T_{\underline{g} \leq v} \ar{d}
\\
	T_{\underline{f}v \leq e} \ar{r} & T_{\underline{f}\underline{g} \leq e}.
\end{tikzcd}
\end{equation}
Further, $T_{\underline{f} \underline{g} \leq e}$ is the
unique choice of pushout that makes the maps in \eqref{GRAPTPUSH EQ} planar.
\end{itemize}
\end{proposition}

\begin{proof}
We first show (a). That $T_{\underline{e} \leq e}$ is indeed a tree is the content of \cite[Prop. 5.20]{Pe17}. More precisely, 
$T_{\underline{e} \leq e} = (T^{\leq e})_{\less \underline{e}}$
in the notation therein. That the root of $T_{\underline{e} \leq e}$ is $e$ is clear and that the leaf tuple is $\underline{e}$ follows from \cite[Remark 5.23]{Pe17}.

 (b) follows from (a), which shows that $\underline{e} \leq e$ can be recovered from
$T_{\underline{e} \leq e}$.

 (c) follows from the definition of outer face together with \cite[Lemma 5.33]{Pe17}, which states that the $\leq_d$ relations on $S,T$ coincide.
 
  Since by (b) and (c) both $T_{\underline{g} \leq v}$ and $T_{\underline{f}v \leq e}$ are outer faces of $T_{\underline{f} \underline{g} \leq e}$, 
the first part of (d) is a restatement of \cite[Prop. 5.15]{Pe17}, while the additional planarity claim 
follows by Proposition \ref{PLANARIZATIONCHAR PROP}
together with the vertex identification
$V(T_{\underline{f} \underline{g} \leq e})=
V(T_{\underline{f} v \leq e}) \amalg V(T_{\underline{g} \leq v})$.
\end{proof}

\begin{definition}
	A map $S \xrightarrow{\varphi} T$ in $\Omega$ is called a \textit{tall map} if 
	\[\varphi(\underline{l}_S) = \underline{l}_T, 
		\qquad
	\varphi(r_S)= r_T,\]
where $l_{(\minus)}$ denotes the (unordered) leaf tuple and $r_{(\minus)}$ the root.
\end{definition}

The following is a restatement of \cite[Cor. 5.24]{Pe17}

\begin{proposition}\label{TALLOUTERDEC PROP}
	Any map $S \xrightarrow{\varphi} T$ in $\Omega$ has a factorization, unique up to unique isomorphism,
	\[
		S \xrightarrow{\varphi^t} U \xrightarrow{\varphi^u} T
	\]
	as a tall map followed by an outer face (in fact, 
	$U= T_{\varphi(\underline{l}_S) \leq \varphi(r_S)}$).
\end{proposition}

Recall that a map 
$S \to T$ in $\Omega$
is called a face (resp. degeneracy)
if it is injective on edges (surjective on edges and preserves leaves).
Moreover, a face $F \to T$ is called \textit{inner} if it is obtained by iteratively removing inner edges (edges other than the root or the leaves). In particular, a face is inner if and only if it is tall. 
The usual degeneracy-face decomposition
\cite[Lemma 3.1]{MW07},\cite[Prop. 5.37]{Pe17}
thus combines with Proposition \ref{TALLOUTERDEC PROP} to give the following.

\begin{corollary}
	Any map $S \xrightarrow{\varphi} T$ in $\Omega$ has a factorization, unique up to unique isomorphisms,
\[
	S \xrightarrow{\varphi^-} U
	\xrightarrow{\varphi^i} V
	\xrightarrow{\varphi^u} T
\]
	as a degeneracy followed by an inner face followed by an outer face.
\end{corollary}
	

We will find it convenient  throughout to regard the 
groupoid $\Sigma$ of finite sets 
as the subcategory 
$\Sigma \hookrightarrow \Omega$
consisting of \textit{corollas}
(i.e. trees with a single vertex)
and isomorphisms.

\begin{notation}\label{UNIQCOR NOT}
	Given a tree $T \in \Omega$, there is a unique corolla $\mathsf{lr}(T) \in \Sigma$ and planar tall map 
	$\mathsf{lr}(T) \to T$, which we call the 
	\textit{leaf-root} of $T$ (this name is motivated by the equivariant analogue, discussed in \S \ref{LRVERT SEC}).
        \index{key functors!Gtreesleafroot1@$\mathsf{lr}$}
	Explicitly, the number of leaves of $\mathsf{lr}(T)$ matches that of $T$, together with the inherited order. 
\end{notation}

We now turn to discussing the substitution operation. We start with an example focused on the closely related notion of 
 iterated graftings of trees (as described in \eqref{GRAPTPUSH EQ}).

\begin{example}\label{GRAFTSUB EX}
The trees $U_1, U_2,\cdots, U_6$ on the left below can be grafted to obtain the tree $U$ in the middle.
More precisely (among other possible grafting orders), one has
\begin{equation}\label{UFORMULA EQ}
U = \left(
		\left(
			\left(
				\left(
					\left(U_6 \amalg_a U_2 \right)
				\right) \amalg_a U_1
			\right) \amalg_b U_3
		\right) \amalg_d U_5
	\right) \amalg_c U_4
\end{equation}
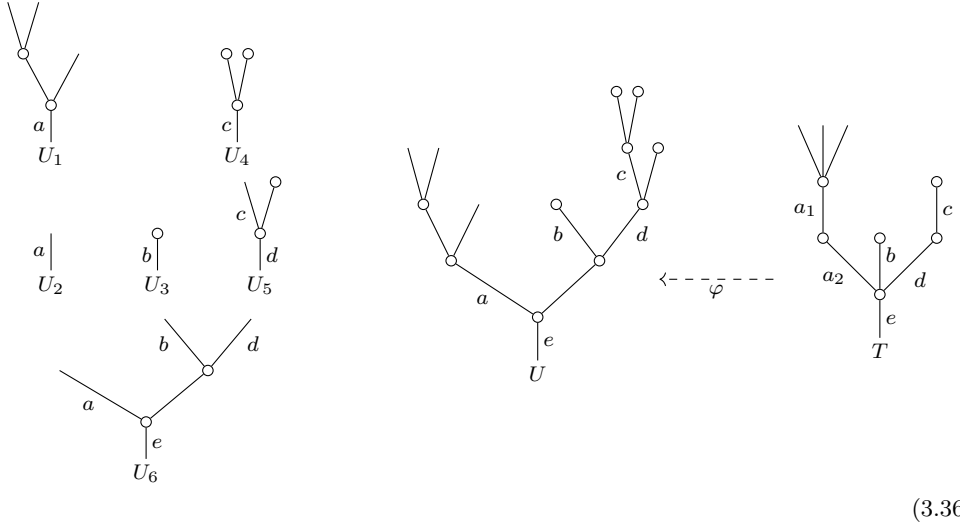
\begin{equation}\label{SUBSDATUMTREES EQ}
	\begin{tikzpicture}[grow=up,auto,level distance=2.1em,
	every node/.style = {font=\footnotesize,inner sep=2pt},
	dummy/.style={circle,draw,inner sep=0pt,minimum size=1.375mm}]
\begin{scope}[xshift=-2em]
	\begin{scope}
	\tikzstyle{level 2}=[sibling distance=2.25em]%
	\tikzstyle{level 3}=[sibling distance=1.25em]%
		\node at (-0.25,3.2) {$U_1$}
			child{node [dummy] {}
				child
				child{node [dummy] {}
					child
					child
				}
			edge from parent node {$a$}};
	\end{scope}
		\node at (-0.25,1.5) {$U_2$}
			child{
		edge from parent node {$a$}};
		\node at (1.15,1.5) {$U_3$}
			child{node [dummy] {}
		edge from parent node {$b$}};
	\begin{scope}
	\tikzstyle{level 2}=[sibling distance=0.875em]%
		\node at (2.2,3.2) {$U_4$}
			child{node [dummy] {}
				child{node [dummy] {}}
				child{node [dummy] {}}
			edge from parent node {$c$}};
	\end{scope}
	\begin{scope}
		\tikzstyle{level 2}=[sibling distance=1.25em]%
		\node at (2.5,1.5) {$U_5$}
			child{node [dummy] {}
				child{node[dummy] {}}
				child{
				edge from parent node {$c$}}
			edge from parent node [swap] {$d$}};
	\end{scope}
	\begin{scope}
	\tikzstyle{level 2}=[sibling distance=3.5em]%
	\tikzstyle{level 3}=[sibling distance=2.25em]%
		\node at (1,-1) {$U_6$}
			child{node [dummy] {}
				child[sibling distance = 5em]{node [dummy] {}
					child[sibling distance = 3.5em]{edge from parent node [swap,near end] {$d$} }
					child[sibling distance = 3.5em]{edge from parent node [near end] {$b$} }
				}
				child[sibling distance = 7em]{ edge from parent node {$a$} }
			edge from parent node [swap] {$e$}};
	\end{scope}
\end{scope}
\begin{scope}[yshift=1em]
	\begin{scope}[level distance=2.3em]
	\tikzstyle{level 2}=[sibling distance=3.5em]%
	\tikzstyle{level 3}=[sibling distance=2.25em]%
	\tikzstyle{level 4}=[sibling distance=1.25em]%
	\tikzstyle{level 5}=[sibling distance=0.875em]%
		\node at (5.5,0) {$U$}
			child{node [dummy] {}
				child[sibling distance =5em]{node [dummy] {}
					child[sibling distance =3.5em]{node [dummy] {}
						child{node [dummy] {}
						}
						child{node [dummy] {}
							child{node [dummy] {}}
							child{node [dummy] {}}
						edge from parent node [near end] {$c$}}
					edge from parent node [swap, near end] {$d$}}
					child[sibling distance =3.5em]{node [dummy] {}
					edge from parent node [near end] {$b$}}
				}
				child[sibling distance =7em]{node [dummy] {}
					child
					child{node [dummy] {}
						child
						child
					}
				edge from parent node {$a$}}
			edge from parent node [swap] {$e$}};
	\end{scope}
	\begin{scope}[level distance=2.3em]
	\tikzstyle{level 2}=[sibling distance=2.3em]%
	\tikzstyle{level 4}=[sibling distance=1em]%
		\node at (10,0.3) {$T$}
			child{node [dummy] {}
				child{node [dummy] {}
					child{node [dummy] {}
					edge from parent node [swap] {$c$}}	
				edge from parent node [swap] {$d$}}
				child{node [dummy] {}
				edge from parent node [near end,swap] {$b$}}
				child{node [dummy] {}
					child{node [dummy] {}
						child
						child
						child
					edge from parent node {$a_1$}}
				edge from parent node {$a_2$}}
			edge from parent node [swap] {$e$}};
	\end{scope}
	\draw [->,dashed] (8.6,1.25) -- node {$\varphi$} (7.1,1.25);
\end{scope}
	\end{tikzpicture}
\end{equation}
We now consider the tree $T$, which is built by converting each $U_i$ into the corolla $\mathsf{lr}(U_i)$, and then performing the same grafting operations as in \eqref{UFORMULA EQ}. $T$ can then be regarded as encoding the combinatorics of the iterated grafting in \eqref{UFORMULA EQ}, with alternative ways to reparenthesize
operations in \eqref{UFORMULA EQ} in bijection with ways to assemble $T$ out of its nodes.

One can now therefore think of the iterated grafting \eqref{UFORMULA EQ} as being instead encoded by the tree $T$ together with the (unique) planar tall maps $\varphi_i$ below.
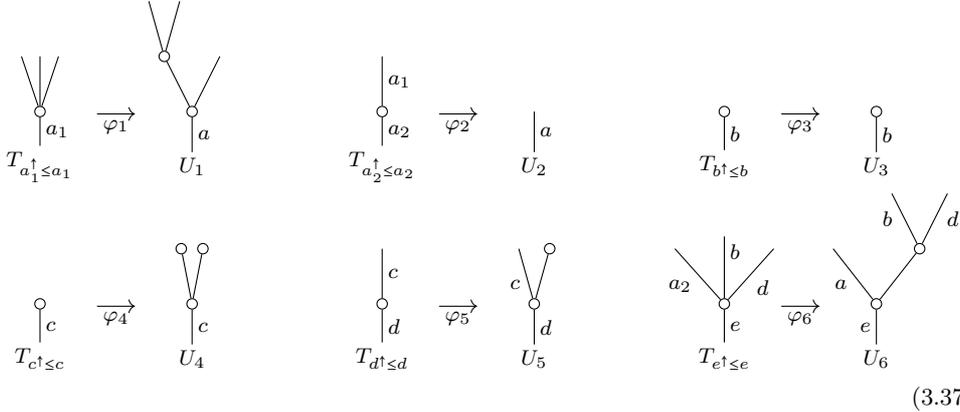
\begin{equation}\label{SUBSDATUMTREES2 EQ}
	\begin{tikzpicture}[grow=up,auto,level distance=2.25em,
	every node/.style = {font=\footnotesize, inner sep=2pt},
	dummy/.style={circle,draw,inner sep=0pt,minimum size=1.375mm}]
	\begin{scope}
	\tikzstyle{level 2}=[sibling distance=0.75em]%
		\node at (0,0) {$T_{a_1^{\uparrow}\leq a_1}$}
			child{node [dummy] {}
				child
				child
				child
			edge from parent node [swap] {$a_1$}};
	\end{scope}	
	\begin{scope}
	\tikzstyle{level 2}=[sibling distance=2.25em]%
	\tikzstyle{level 3}=[sibling distance=1.25em]%
		\node at (2,0) {$U_1$}
			child{node [dummy] {}
				child
				child{node [dummy] {}
					child
					child
				}
			edge from parent node [swap] {$a$}};
	\end{scope}
		\draw [->] (0.75,0.7) -- node [swap]{$\varphi_1$} (1.25,0.7);
		\node at (4.5,0) {$T_{a_2^{\uparrow}\leq a_2}$}
			child{node [dummy] {}
				child{
				edge from parent node [swap] {$a_1$}}
			edge from parent node [swap] {$a_2$}};
		\node at (6.5,0) {$U_2$}
			child{
			edge from parent node [swap] {$a$}};
		\draw [->] (5.25,0.7) -- node [swap]{$\varphi_2$} (5.75,0.7);
		\node at (9,0) {$T_{b^{\uparrow}\leq b}$}
			child{node [dummy] {}
			edge from parent node [swap] {$b$}};
		\node at (11,0) {$U_3$}
			child{node [dummy] {}
			edge from parent node [swap] {$b$}};
		\draw [->] (9.75,0.7) -- node [swap]{$\varphi_3$} (10.25,0.7);
	\begin{scope}[yshift=-2.55cm]
		\node at (0,0) {$T_{c^{\uparrow}\leq c}$}
			child{node [dummy] {}
			edge from parent node [swap] {$c$}};
	\begin{scope}
	\tikzstyle{level 2}=[sibling distance=0.875em]%
		\node at (2,0) {$U_4$}
			child{node [dummy] {}
				child{node [dummy] {}}
				child{node [dummy] {}}
			edge from parent node  [swap]{$c$}};
	\end{scope}
	\draw [->] (0.75,0.7) -- node [swap]{$\varphi_4$} (1.25,0.7);
		\node at (4.5,0) {$T_{d^{\uparrow}\leq d}$}
			child{node [dummy] {}
				child{
				edge from parent node [swap] {$c$}}
			edge from parent node [swap] {$d$}};
	\begin{scope}
	\tikzstyle{level 2}=[sibling distance=1.25em]%
		\node at (6.5,0) {$U_5$}
			child{node [dummy] {}
				child{node[dummy] {}}
				child{
				edge from parent node {$c$}}
			edge from parent node [swap] {$d$}};
	\end{scope}
	\draw [->] (5.25,0.7) -- node [swap]{$\varphi_5$} (5.75,0.7);
	\begin{scope}
	\tikzstyle{level 2}=[sibling distance=2em]%
		\node at (9,0) {$T_{e^{\uparrow}\leq e}$}
			child{node [dummy] {}
				child{ edge from parent node [swap] {$d$} }
				child[level distance=2.75em]{ edge from parent node [near end,swap] {$b$} }
				child{ edge from parent node {$a_2$} }
			edge from parent node [swap] {$e$}};
	\end{scope}
	\begin{scope}
	\tikzstyle{level 2}=[sibling distance=3.5em]%
	\tikzstyle{level 3}=[sibling distance=2.25em]%
		\node at (11,0) {$U_6$}
			child{node [dummy] {}
				child{node [dummy] {}
					child{ edge from parent node [swap,near end] {$d$} }
					child{ edge from parent node [near end]{$b$} }
				}
				child{ edge from parent node {$a$} }
			edge from parent node {$e$}};
	\end{scope}
	\draw [->] (9.75,0.7) -- node [swap]{$\varphi_6$} (10.25,0.7);
	\end{scope}
	\end{tikzpicture}
\end{equation}
From this perspective, $U$ can now be thought of as obtained from $T$ by \textit{substituting} each of its nodes with the corresponding $U_i$. Moreover, the $\varphi_i$ assemble to a planar tall map 
$\varphi \colon T \to U$ (such that $a_i \mapsto a,b \mapsto b,\cdots,e \mapsto e$), which likewise encodes the same information.

\end{example}

One of the fundamental ideas shaping our perspective on operads
is then that substitution data as in \eqref{SUBSDATUMTREES2 EQ}
can equivalently be repackaged using planar tall maps. 

\begin{definition}\label{SUBSTITUTIONDATUM}
	Let $T \in \Omega$ be a tree.
	
	A \textit{$T$-substitution datum} is a tuple 
	$\left(U_{e^{\uparrow} \leq e}\right)_{(e^{\uparrow} \leq e)\in V(T)}$ together with tall maps
	$T_{e^{\uparrow}\leq e} \to U_{e^{\uparrow}\leq e}$.
	Further, a map of $T$-substitution data 
	$\left(U_{e^{\uparrow} \leq e}\right) \to \left(V_{e^{\uparrow} \leq e}\right)$ is a tuple of tall maps $\left(U_{e^{\uparrow} \leq e}\to V_{e^{\uparrow} \leq e}\right)$ compatible with the substitution maps.
	
	Lastly, a substitution datum is called \textit{planar}
        if the chosen maps are planar (so that 
	$\mathsf{lr}(U_{e^{\uparrow} \leq e}) = T_{e^{\uparrow} \leq e}$),
        and a morphism between planar data is called a \textit{planar morphism} if it consists of a tuple of planar maps.
	
	We denote the category of (resp. planar) $T$-substitution data 
	by $\mathsf{Sub}(T)$ (resp. $\mathsf{Sub}_{\mathsf{p}}(T)$).
\end{definition}

\begin{definition}
	Let $T \in \Omega$ be a tree. 
	The \textit{Segal core poset $\mathsf{Sc}(T)$} is the poset with objects the single edge subtrees $\eta_e$ and vertex subtrees $T_{e^{\uparrow} \leq e}$, ordered by inclusion.
\end{definition}

\begin{remark}\label{SCTARR REM}
Note that the only non-identity arrows in $\mathsf{Sc}(T)$ are inclusions of the form $\eta_a \hookrightarrow T_{e^{\uparrow}\leq e}$.
In particular, one can not compose non-identity arrows in $\mathsf{Sc}(T)$.
\end{remark}

Given a $T$-substitution datum 
$\left(U_{e^{\uparrow} \leq e}\right)_{(e^{\uparrow} \leq e)\in V(T)}$
we abuse notation by writing
\[U_{(\minus)} \colon \mathsf{Sc}(T) \to \Omega\]
for the functor $\eta_a \mapsto \eta$, $T_{e^{\uparrow} \leq e} \mapsto U_{e^{\uparrow} \leq e}$  
and sending the inclusions $\eta_a \subset T_{e^{\uparrow} \leq e}$
to the composites
\[
\eta \xrightarrow{a} T_{e^{\uparrow} \leq e}  \to 
U_{e^{\uparrow} \leq e}.\]

\begin{proposition}\label{SUBDATAUNDERPLAN PROP}
Let $T \in \Omega$ be a tree. There is an isomorphism of categories
\[
\begin{tikzcd}[row sep =0pt]
	\mathsf{Sub}_{\mathsf{p}}(T) \ar[r,shift left=2pt] &
	T \downarrow \Omega^{\mathsf{pt}} \ar[l,shift left=2pt]
\\
	\left(U_{e^{\uparrow} \leq e}\right) \ar[r,mapsto] & 
	\left(T \to \colim_{\mathsf{Sc}(T)} U_{(\minus)}\right)
\\
	\left(U_{\varphi(e^{\uparrow}) \leq \varphi(e)}\right) &
	(T \xrightarrow{\varphi} U) \ar[l,mapsto]
\end{tikzcd}
\]
where $T \downarrow \Omega^{\mathsf{pt}}$ denotes
the category of planar tall maps under $T$
and $\colim_{\mathsf{Sc}(T)} U_{(\minus)}$ is chosen in the unique way that makes the inclusions of the $U_{e^{\uparrow} \leq e}$ planar. 
\end{proposition}

\begin{proof}
We first show in parallel that:
\begin{enumerate*}
\item[(i)] $\colim_{\mathsf{Sc}(T)} U_{(\minus)}$, which we denote $U_T$, exists;
\item[(ii)] for the datum $\left(T_{e^{\uparrow}\leq e}\right)$, it is $T = \colim_{\mathsf{Sc}(T)} T_{(\minus)}$;
\item[(iii)] $V(U_T) = \coprod_{V(T)} V(U_{e^{\uparrow} \leq e})$;
\item[(iv)] the induced map
$T \to U_T$ is planar tall.
\end{enumerate*}
 
The argument is by induction on the number of vertices of $T$, with the base cases of $T$ with $0$ or $1$ vertices being immediate, since then $T$ is the terminal object of $\mathsf{Sc}(T)$.
Otherwise, one can choose a non trivial grafting decomposition so as to write $T = R \amalg_e S$, resulting 
in identifications 
$\mathsf{Sc}(R) \subset \mathsf{Sc}(T)$, 
$\mathsf{Sc}(S) \subset \mathsf{Sc}(T)$
with  
$\mathsf{Sc}(R) \cup \mathsf{Sc}(S) = \mathsf{Sc}(T)$
and 
$\mathsf{Sc}(R) \cap \mathsf{Sc}(S) = \{\eta_e \}$.
The existence of $U_T = \colim_{\mathsf{Sc}(T)}U_{(\minus)}$
is thus equivalent to the existence of the pushout below
(where the rightmost diagram merely simplifies notation).
\begin{equation}\label{ASSEMBLYGRAFT EQ}
\begin{tikzcd}
	\eta \ar{r}{e} \ar{d}[swap]{e} & \colim_{\mathsf{Sc}(R)}U_{(\minus)} \ar[dashed,d] &
	\eta \ar{r}{e} \ar{d}[swap]{e} & U_R \ar[dashed,d]	
\\
	\colim_{\mathsf{Sc}(S)}U_{(\minus)} \ar[dashed,r] &
	\colim_{\mathsf{Sc}(T)}U_{(\minus)} &
	U_S \ar[dashed,r] &
	U_T
\end{tikzcd}
\end{equation}
By induction, $U_R$ and $U_S$ exist for any $U_{(\minus)}$, 
equal $R$ and $S$ in the case $U_{(\minus)} = T_{(\minus)}$,
$V(U_R) = \coprod_{V(R)} V(U_{e^{\uparrow} \leq e})$
and likewise for $S$ (so that there are unique choices of $U_R$, $U_S$ making the inclusions of $U_{e^{\uparrow} \leq e}$ planar),
and the maps 
$R \to \colim_{\mathsf{Sc}(R)}U_{(\minus)}$,
$S \to \colim_{\mathsf{Sc}(S)}U_{(\minus)}$
are planar tall.
But it now follows that \eqref{ASSEMBLYGRAFT EQ} is a grafting pushout diagram (cf. \eqref{GRAPTPUSH EQ}), so that the pushout indeed exists. The conditions
$T = \colim_{\mathsf{Sc}(T)}T_{(\minus)}$,
$V(U_T) = \coprod_{V(T)} V(U_{e^{\uparrow} \leq e})$, 
and that
$T \to \colim_{\mathsf{Sc}(T)}U_{(\minus)}$
is planar tall follow.

The fact that the two functors in the statement
are inverse to each other is clear from the same inductive argument.
\end{proof}

\begin{corollary}\label{SUBDATAUNDERPLAN COR}
Let $T \in \Omega$ be a tree. The formulas in
Proposition \ref{SUBDATAUNDERPLAN PROP}
give an isomorphism of categories
\[
\begin{tikzcd}[row sep =0pt]
	\mathsf{Sub}(T) \ar[r,shift left=2pt] &
	T \downarrow \Omega^{\mathsf{t}} \ar[l,shift left=2pt]
\end{tikzcd}
\]
where $T \downarrow \Omega^{\mathsf{t}}$ denotes
the category of tall maps under $T$.
\end{corollary}

\begin{proof}
	This is a consequence of Proposition \ref{PLANARPULL PROP} together with the previous result.
	Indeed, Proposition \ref{PLANARIZATIONCHAR PROP} can be restated as saying that isomorphisms $T \to T'$ are in bijection with substitution data consisting of isomorphisms, and thus  bijectiveness of $\mathsf{Sub}(T) \to T \downarrow \Omega^{\mathsf{t}}$ reduces to that in the previous result.
\end{proof}

\begin{notation}\label{UEUPE NOT}
	Following the previous results, 
	given a map of trees 
	$\varphi \colon T \to U$
	and vertex
	$(e^{\uparrow} \leq e) \in V(T)$
	we will abbreviate
	$U_{\varphi(e^{\uparrow}) \leq \varphi(e)}$
	as simply
	$U_{e^{\uparrow} \leq e}$\index{structure on trees!outer faces!UEUPE@$U_{e^{\uparrow} \leq e}$}.
\end{notation}

\begin{remark}\label{VERTEXDECOMP REM}
As noted in the proof of Proposition \ref{SUBDATAUNDERPLAN PROP}, writing $U = \colim_{\mathsf{Sc}(T)}U_{(\minus)}$,
	one has 
\begin{equation}\label{VERTEXDECOMP EQ}
	V(U) = \coprod_{(e^{\uparrow} \leq e) \in V(T)}
	V(U_{e^{\uparrow} \leq e}).
  \end{equation}
    Alternatively, \eqref{VERTEXDECOMP EQ} can be regarded as a map 
    $\varphi^{\**} \colon V(U) \to V(T)$ induced by the planar tall map 
    $\varphi \colon T \to U$.
    Explicitly, $\varphi^{\**}(U_{u^{\uparrow} \leq u})$ 
    is the unique $T_{t^{\uparrow}\leq t}$ such that
    there is an inclusion of outer faces $U_{u^{\uparrow} \leq u} \hookrightarrow U_{t^{\uparrow} \leq t}$,
    so that $\varphi^{\**}$ indeed depends contravariantly on the tall map $\varphi$.
\end{remark}

\begin{remark}\label{INPPATH REM}
Suppose that $e \in T$ has input path
$I_T(e) = (e=e_n < e_{n-1} < \cdots < e_0)$.
It is intuitively clear that for a tall map 
$\varphi \colon T \to U$ the input path of $\varphi(e)$ is built by gluing input paths in the $U_{t^{\uparrow} \leq t}$. More precisely (and omitting $\varphi$ for readability), one has
\[
	I_U\left(e_n \right) \simeq 
	I_{n-1}(e_n) \amalg_{e_{n-1}} I_{n-2}(e_{n-1})
	\amalg_{e_{n-2}} \cdots
	\amalg_{e_1} I_1(e_0).
\]
where $I_k(\minus)$ denotes the input path in $U_{e_k^{\uparrow} \leq e_k}$.
More formally, this follows from the characterization of 
predecessors in Proposition \ref{INPUTPREDECESSORPROP PROP}(b).
\end{remark}

We end this section with a pair of results that will allow us to reverse the substitution procedure of 
Proposition \ref{SUBDATAUNDERPLAN PROP}
and will be needed in \S \ref{EXTTREE SEC}.
Recall that the single edge tree $\eta \in \Omega$ is called the stick tree,
cf. Notation \ref{STICKTRE NOT}.

\begin{proposition}\label{BUILDABLE PROP}
	Let $U \in \Omega$ be a tree. Then:
\begin{itemize}
	\item[(i)] given non-stick outer subtrees $U_i$ such that 
	$V(U) = \coprod_i V(U_i)$ there is a unique tree $T$ and planar inner face $T \to U$ such that the sets $\{U_i\}$, $\{U_{e^{\uparrow}\leq e}\}$ coincide;
	\item[(ii)] given multiplicities $m_e \geq 1$ for each edge $e \in U$, there is a unique planar degeneracy $\rho \colon T \to U$ such that $\rho^{-1}(e)$ has $m_e$ elements;
	\item[(iii)] planar tall maps $T \to U$ are in bijection with collections $\{U_i\}$ of outer subtrees such that $V(U) = \coprod_i V(U_i)$ and $U_j$ is not an inner edge of any $U_i$ whenever $U_j \simeq \eta$ is a stick.
\end{itemize}
\end{proposition}

\begin{proof}
	We first show (i) by induction on the number of subtrees $U_i$. The base case $\{U_i\}=\{U\}$ is immediate, setting 
	$T= \mathsf{lr}(U)$. Otherwise, $U$ must not be a corolla.
	Letting $e$ be an edge that is both an inner edge of $U$ and a root of some $U_i$, one can form a grafting pushout diagram
\begin{equation} \label{DECOMPPROOF EQ}
\begin{tikzcd}
	\eta \ar{r}{e} \ar{d}[swap]{e} & U^{\leq e} \ar{d}
\\
	U_{\nless e} \ar{r} & U
\end{tikzcd}
\end{equation}
where $U^{\leq e}$ (resp. $U_{\nless e}$) is the outer face consisting of the edges $u \leq_d e$ (resp. $u \nless_d e$).
Since there is a unique $U_i$ containing the vertex $e^{\uparrow} \leq e$, 
it follows from the definition of outer face that there is a
non trivial partition 
$\{U_i\} = \{U_i|U_i \hookrightarrow U^{\leq e}\} 
\amalg \{U_i|U_i \hookrightarrow U_{\nless e}\}$. 
The existence of $T \to U$ now follows from the induction hypothesis. For uniqueness, the condition that no $U_i$ is a stick guarantees that $T$ possesses a single inner edge mapping to $e$, and thus admits a compatible decomposition as in \eqref{DECOMPPROOF EQ}, so that uniqueness too follows from the induction hypothesis.

For (ii), we argue existence by nested induction on the number of vertices $|V(U)|$ and the sum of the multiplicities $m_e$. The base case $|V(U)|=0$, i.e., $U = \eta$ is immediate. Otherwise, writing $m_e = m'_e +1$, one can form a decomposition \eqref{DECOMPPROOF EQ} where either $|V(U^{\leq e})|,|V(U_{\nless e})|<|V(U)|$ or one of $U^{\leq e},U_{\nless e}$ is $\eta$, so that $T \to U$ can be built via the induction hypothesis. For uniqueness, note first that 
by \cite[Lemma 5.33]{Pe17} each pre-image $\rho^{-1}(e)$ is linearly ordered and by the ``further'' claim in 
\cite[Cor. 5.39]{Pe17} the remaining broad relations are precisely the pre-image of the non-identity relations in $U$, showing that the underlying broad poset of the tree $T$ is unique  up to isomorphism. Strict uniqueness is then 
Proposition \ref{PLANARPULL PROP}.

(iii) follows by combining (i) and (ii). Indeed, any planar tall map $T \to U$ has a unique factorization 
$T \twoheadrightarrow \bar{T} \hookrightarrow U$
as a planar degeneracy followed by a planar inner face, and each  of these maps is classified by the data in (b) and (a).
\end{proof}

\begin{lemma}\label{OUTERFACEUNION LEM}
	Suppose $T_1,T_2 \hookrightarrow T$ are two outer faces with at least one common edge $e$. Then there exists an unique outer face $T_1 \cup T_2$ such that 
	$V(T_1 \cup T_2) = V(T_1) \cup V(T_2)$.
\end{lemma}

\begin{proof}
	The result is obvious if either
	$T$ is a corolla or if one of $T_1,T_2$
	is a stick subtree for a leaf or root.
	Otherwise, one can necessarily choose $e$ to be an inner edge of $T$, in which case all three of $T_1,T_2,T$ admit compatible decompositions as in \eqref{DECOMPPROOF EQ} and the result follows by induction on $|V(T)|$.
\end{proof}

\subsection{Equivariant leaf-root and vertex functors}\label{LRVERT SEC}

This section introduces two functors that are central to our definition of the category $\mathsf{Op}_G$ of
genuine equivariant operads: the leaf-root and vertex functors.

We start by recalling a key class of maps of $G$-trees.

\begin{definition}\label{QUOT DEF}
	Let $S = (S_y)_{y \in Y}$ and $T = (T_x)_{x \in X}$
	be $G$-trees.
	A map of $G$-trees 
	\[
	\varphi = (\phi, (\varphi_y))\colon S \to T
	\]
	is called a \textit{quotient} if each of the constituent tree maps
	\[
	\varphi_y \colon S_y \to T_{\phi(y)}
	\]
is an isomorphism of trees.	

The category of $G$-trees and quotients is denoted $\Omega_{G}^0$ (this notation is justified in \S \ref{PLANARSTRING SEC}).
\index{categories!of trees!Gtrees0stringsn@$\Omega_G^n$}
\end{definition}

\begin{remark}
	Quotients can alternatively be described as the cartesian arrows for the Grothendieck fibration
	$\Omega_G \xrightarrow{\mathsf{r}} \mathsf{O}_G$.
	We note that this is more general than the notion of root pullbacks (Figure \ref{FIGURE}), which are the \textit{chosen} cartesian arrows.
	More explicitly,
	root pullbacks are those quotientes for which
	$\varphi_y \colon S_y \to T_{\phi(y)}$
	is a planar isomorphism, i.e., an identity.
\end{remark}

\begin{definition}\label{GSYMCAT DEF}
	\index{categories!of operads/symmetric sequences!SymG@$\Sym_G(\mathcal V) = \V^{\Sigma_G^{op}}$}
	The \textit{$G$-symmetric category},
	whose objects we call \textit{$G$-corollas}, is the full subcategory 
	$\Sigma_G \hookrightarrow \Omega_{G}^0$ of those $G$-trees
	$C = (C_x)_{x \in X}$ such that some (and thus all) $C_x$ is a corolla $C_x \in \Sigma \hookrightarrow \Omega$
	(cf. Notation \ref{UNIQCOR NOT}).
        \index{categories!of trees!SigmaG@$\Sigma_G$}
\end{definition}

\begin{definition}
      The \textit{leaf-root functor} is the functor $\Omega_{G}^0 \xrightarrow{\mathsf{lr}} \Sigma_G$ defined by 
      \[
            \mathsf{lr}\left((T_x)_{x \in X}\right)=
            \left(\mathsf{lr}(T_x)\right)_{x \in X}.
      \]
      \index{key functors!Gtreesleafroot1@$\mathsf{lr}$}
\end{definition}

\begin{remark}
	The leaf-root functor extends 
	to a functor $\mathsf{lr} \colon \Omega^{\mathsf{t}}_G \to \Sigma_G$, 
	where $\Omega^{\mathsf{t}}_G$ is the category of tall maps, defined exactly as in Definition \ref{QUOT DEF}, but not to a functor defined on all arrows of $\Omega_G$.
        \index{categories!of trees!Gtrees0tall@$\Omega_G^{\mathsf t}$}
	Nonetheless, we will be primarily interested in the 
 restriction  
	$\Omega_{G}^0 \xrightarrow{\mathsf{lr}} \Sigma_G$.
\end{remark}

\begin{remark}\label{LEAFROOTEXAMP REM}
	Generalizing the remark in Notation \ref{UNIQCOR NOT},
	$\mathsf{lr}(T)$ can alternatively be characterized as being the \textit{unique} $G$-corolla which admits a (likewise unique) planar tall map $\mathsf{lr}(T) \to T$. Moreover, $\mathsf{lr}(T)$ can usually be regarded as the ``smallest inner face'' of $T$, obtained by removing all the inner edges, although this characterization fails when 
	$T=(\eta_x)_{x \in X}$ is a stick $G$-tree. Some examples with $G=\mathbb{Z}_{/4}$ follow.
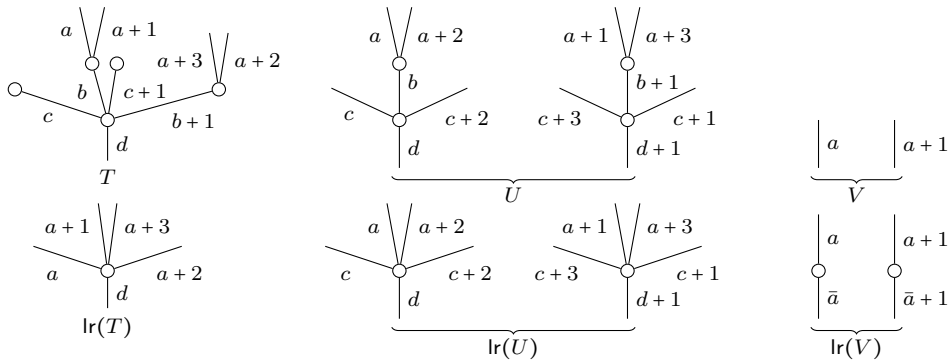
\[
	\begin{tikzpicture}[grow=up,auto,level distance=2.3em,
	every node/.style = {font=\footnotesize,inner sep=3pt},
	dummy/.style={circle,draw,inner sep=0pt,minimum size=1.75mm}]
		\node at (0,0) {$T$}
			child{node [dummy] {}
				child[level distance=1.25em,sibling distance = 3em]{node [dummy] {}
					child[level distance=2.3em,sibling distance =0.75em]{
					edge from parent node [swap,near end] {$a+2$}}
					child[level distance=2.3em,sibling distance =0.75em]{
					edge from parent node [near end] {$a+3$}}
				edge from parent node [swap] {$b+1$}}
				child[sibling distance = 0.75em]{node [dummy] {}
				edge from parent node [swap, very near end] {$c+1$}}
				child[sibling distance = 1.25em]{node [dummy] {}
					child[sibling distance = 1em]{
					edge from parent node [swap,very near end] {$a+1$}}
					child[sibling distance = 1em]{
					edge from parent node [very near end] {$\phantom{1+}a$}}
				edge from parent node [very near end] {$b$}}
				child[level distance=1.25em,sibling distance = 2.5em]{node [dummy] {}
				edge from parent node {$c$}}
			edge from parent node [swap] {$d$}};
		\node at (0,-2) {$\mathsf{lr}(T)$}
			child{node [dummy] {}
				child[sibling distance=2em, level distance=1em]{
				edge from parent node [swap] {$a+2$}}
				child[sibling distance=0.75em,level distance=2.75em]{
				edge from parent node [very near end,swap] {$a+3$}}
				child[sibling distance=0.75em,level distance=2.75em]{
				edge from parent node [very near end] {$a+1$}}
				child[sibling distance=2em, level distance=1em]{
				edge from parent node {$\phantom{1+}a$}}
			edge from parent node [swap] {$d$}};
	\begin{scope}[xshift=-0.5em]
		\node at (4,0) {}
			child{node [dummy] {}
				child[level distance=1.3em,sibling distance=2.75em]{
				edge from parent node [swap] {$c+2$}}
				child{node [dummy] {}
					child[sibling distance=1em]{
					edge from parent node [near end,swap] {$a+2$}}
					child[sibling distance=1em]{
					edge from parent node [near end] {$\phantom{1}a$}}
				edge from parent node [swap, near end] {$b$}}
				child[level distance=1.3em,sibling distance=2.75em]{
				edge from parent node {$c$}}
			edge from parent node [swap] {$d$}};
		\node at (7,0) {}
			child{node [dummy] {}
				child[level distance=1.3em,sibling distance=2.75em]{
				edge from parent node [swap] {$c+1$}}
				child{node [dummy] {}
					child[sibling distance=1em]{
					edge from parent node [near end,swap] {$a+3$}}
					child[sibling distance=1em]{
					edge from parent node [near end] {$a+1$}}
				edge from parent node [swap, near end] {$b+1$}}
				child[level distance=1.3em,sibling distance=2.75em]{
				edge from parent node {$c+3$}}
			edge from parent node [swap] {$d+1$}};
			\draw[decorate,decoration={brace,amplitude=2.5pt}] (7.1,0) -- (3.9,0) node[midway,inner sep=4pt]{$U$};
		\node at (4,-2) {}
			child{node [dummy] {}
				child[sibling distance=2em, level distance=1em]{
				edge from parent node [swap] {$c+2$}}
				child[sibling distance=1em,level distance=2.75em]{
				edge from parent node [very near end,swap] {$a+2$}}
				child[sibling distance=1em,level distance=2.75em]{
				edge from parent node [very near end] {$\phantom{1+}a$}}
				child[sibling distance=2em, level distance=1em]{
				edge from parent node {$\phantom{1+}c$}}
			edge from parent node [swap] {$d$}};
		\node at (7,-2) {}
			child{node [dummy] {}
				child[sibling distance=2em, level distance=1em]{
				edge from parent node [swap] {$c+1$}}
				child[sibling distance=1em,level distance=2.75em]{
				edge from parent node [very near end,swap] {$a+3$}}
				child[sibling distance=1em,level distance=2.75em]{
				edge from parent node [very near end] {$a+1$}}
				child[sibling distance=2em, level distance=1em]{
				edge from parent node {$\phantom{1+}c+3$}}
			edge from parent node [swap] {$d+1$}};
\draw[decorate,decoration={brace,amplitude=2.5pt}] (7.1,-2) -- (3.9,-2) node[midway,inner sep=4pt]{$\mathsf{lr}(U)$};
	\end{scope}
	\begin{scope}[xshift=-2em]
			\node at (10,0) {}
				child{
				edge from parent node [swap] {$a$}};
			\node at (11,0) {}
				child{
				edge from parent node [swap] {$a+1$}};
\draw[decorate,decoration={brace,amplitude=2.5pt}] (11.1,0) -- (9.9,0) node[midway,inner sep=4pt]{$V$};
			\node at (10,-2) {}
				child{node[dummy]{}
					child{
					edge from parent node [swap] {$a$}}
				edge from parent node [swap] {$\bar{a}$}};
			\node at (11,-2) {}
				child{node[dummy]{}
					child{
					edge from parent node [swap] {$a+1$}}
				edge from parent node [swap] {$\bar{a}+1$}};
\draw[decorate,decoration={brace,amplitude=2.5pt}] (11.1,-2) -- (9.9,-2) node[midway,inner sep=4pt]{$\mathsf{lr}(V)$};
	\end{scope}
	\end{tikzpicture}
\]	
\end{remark}

\begin{remark}\label{LRROOTMAP REM}
	Since planarizations can not be pushed forward along tree maps (cf. Remark \ref{PULLPLANAR REM}),
	the leaf-root functor $\mathsf{lr} \colon \Omega_{G}^0 \to \Sigma_G$ is not a Grothendieck fibration,
	but instead only a map of Grothendieck fibrations over $\mathsf{O}_G$ 
	(for the obvious root functor $\mathsf{r} \colon \Sigma_G \to \mathsf{O}_G$).
\end{remark}

\begin{definition}\label{VG DEF}
Given $T = (T_x)_{x \in X} \in \Omega_G$ we define its set of \textit{vertices} to be $V(T) = \coprod_{x \in X} V(T_x)$
and its set of
\textit{$G$-vertices} $V_G(T)$ to be the orbit set $V(T)/G$.
\index{structure on trees!vertices!vertices@$V(T)$}
\index{structure on trees!vertices!verticesG@$V_G(T) = V(T)/G$}

Furthermore, we will regard 
$V(T)$ as an object of $\Fin$ by using the induced planar order
(with $e^{\uparrow}\leq e$ ordered according to $e$)
and likewise $V_G(T)$ will be regarded as an object of $\Fin$ by using the lexicographic order: i.e. vertex equivalence classes 
$[e^{\uparrow} \leq e]$ are ordered according to the planar order $\leq_p$ of the smallest representative $ge$, $g \in G$.
\end{definition}

\begin{remark}\label{VERTEXDECOMPG REM}
	Following Remark \ref{VERTEXDECOMP REM},
	a tall map $\varphi \colon T \to U$ of $G$-trees
	induces a $G$-equivariant map
	$\varphi^{\**} \colon V(U) \to V(T)$
	and thus also a map of orbits
	$\varphi^{\**} \colon V_G(U) \to V_G(T)$.
	We note, however, that $\varphi^{\**}$ is not in general compatible with the order on $V_G(\minus)$ even if $\varphi$ is planar, as is indeed the case even in the non-equivariant setting.

A minimal example follows.
		
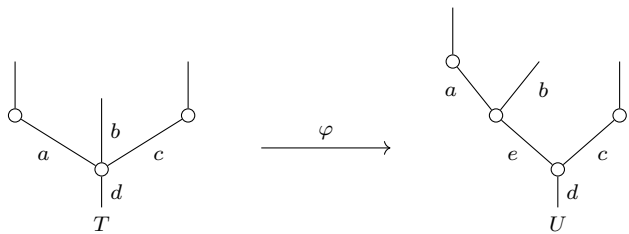
\[
		\begin{tikzpicture}[grow=up,auto,level distance=2.2em,
		every node/.style = {font=\footnotesize},
		dummy/.style={circle,draw,inner sep=0pt,minimum size=1.75mm}]
		\node at (0,0) {$T$}
			child{node [dummy] {}
				child[sibling distance = 3.5em]{node [dummy] {}
					child
				edge from parent node[swap] {$c$}}
				child[level distance=2.9em]{
				edge from parent node [swap] {$b$}}
				child[sibling distance = 3.5em]{node [dummy] {}
					child
				edge from parent node {$a$}}		
			edge from parent node [swap] {$d$}};
		\node at (6,0) {$U$}
			child{node [dummy] {}
				child[sibling distance = 5em]{node [dummy] {}
					child
				edge from parent node [swap] {$c$}}
				child[sibling distance = 5em]{node [dummy] {}
					child[sibling distance = 3.5em]{
					edge from parent node [swap,near end] {$b$}}
					child[sibling distance = 3.5em]{node [dummy] {}
						child
					edge from parent node [near end]{$a$}}
				edge from parent node {$e$}}
			edge from parent node [swap] {$d$}};
		\draw[->] (2.1,1) -- node [above] {$\varphi$} (3.8,1) ;
		\end{tikzpicture}
		\]
In $V(T)$ the vertices are ordered as $a<c<d$ while in $V(U)$ they are ordered as $a<e<c<d$ but the map 
$\varphi^{\**} \colon V(U) \to V(T)$ is given by 
$a \mapsto a, c \mapsto c, d \mapsto d, e \mapsto d$.
\end{remark}

\begin{notation}\label{GVERT NOT}
Given $T=(T_x)_{x \in X} \in \Omega_G$
and $(e^{\uparrow} \leq e) \in V(T)$ 
we write $T_{e^{\uparrow}\leq e}$
as a shorthand for $T_{x,e^{\uparrow}\leq e}$, where $e \in T_x$.

Further, each element of $V_G(T)$ corresponds to an unique edge orbit $Ge$ for $e$ not a leaf.
We will prefer to write $G$-vertices as $v_{Ge}$\index{structure on trees!vertices!Gvertex@$v_{Ge}$}, 
and write
\begin{equation}\label{TVGE DEF}
	T_{v_{Ge}}\index{structure on trees!vertices!Gvertexsubtree@$T_{v_{Ge}}$} = (T_{f^{\uparrow} \leq f})_{f \in Ge}
\end{equation}
where $Ge$ inherits the planar order.
\end{notation}

Note that $T_{v_{Ge}}$ is always a $G$-corolla, leading to the following definition.

\begin{definition}\label{GVERTFUN DEF}
The \textit{$G$-vertex functor} is the functor
\begin{equation}
        \label{GVERT_EQ}
  	  \begin{tikzcd}[row sep=0]
	  \Omega_{G}^0 \ar{r}{\boldsymbol{V}_G} & \Fin_s \wr \Sigma_G \\
	  T \ar[mapsto]{r} & (T_{v_{Ge}})_{v_{Ge} \in V_G(T)},
	  \end{tikzcd}	
\end{equation}
\index{key functors!vertices@$\boldsymbol{V}_G$}
where $\Fin_s$ is the category of finite sets and surjections of
Remark \ref{FINSURJ REM}.
\end{definition}

\begin{remark}
	Note that, though the composite
	$\Omega_G^0 \to \Fin_s \wr \Sigma_G \to \Fin_s$
	coincides on objects with the functor described in Remark \ref{VERTEXDECOMPG REM},
	the variance is now reversed. 
\end{remark}

\begin{remark}\label{NEED_WREATH_REMARK}
	In the non-equivariant case the vertex functor can be defined to land instead in $\Sigma \wr \Sigma$.
	The need to introduce the $\Fin_s \wr (\minus)$ construction comes from the fact that
	in general quotient maps do not preserve the number of $G$-vertices.
	To illustrate,
	and keeping the set-up of Example \ref{ROOTPULL EX},
	let $G=\{\pm 1, \pm i, \pm j, \pm k\}$ and consider the pullback map 
	$\varphi \colon \pi^{\**} S \to S$
	given by 
	$a \mapsto a$, $b \mapsto b$, $c \mapsto c$, $d \mapsto i a$, $e \mapsto i b$,	
	which we present below in orbital notation. 
		\[
		\begin{tikzpicture}[grow=up,auto,level distance=2.4em,every node/.style = {font=\footnotesize},dummy/.style={circle,draw,inner sep=0pt,minimum size=1.75mm}]
		\node at (0,0) {$\pi^{\**}S$}
			child{node [dummy] {}
				child{node [dummy] {}
					child{
					edge from parent node [swap] {$G/\langle j \rangle  \cdot d$}}
				edge from parent node[near end,swap] {$G/\langle -1 \rangle  \cdot e$}}
				child{node [dummy] {}
					child{
					edge from parent node {$G/\langle j \rangle  \cdot a$}}
				edge from parent node [near end] {$G/\langle -1 \rangle  \cdot b$}}		
			edge from parent node[swap] {$G/\langle -1 \rangle  \cdot c$}};
		\node at (4.53,0) {$S$}
			child{node [dummy] {}
				child{node [dummy] {}
					child{
					edge from parent node [swap] {$G/\langle j \rangle \cdot a$}}
				edge from parent node [swap] {$G / \langle -1 \rangle \cdot b$}}
			edge from parent node [swap] {$G/G \cdot c$}};
		\draw[->] (2.6,1) -- node [above] {$\varphi$} (3.6,1) ;
		\end{tikzpicture}
		\]
Note that $T = \pi^{\**} S$ has three $G$-vertices $v_{G c}$, $v_{G b}$, $v_{G e}$ while $S$ has only two $G$-vertices $v_{G c}$ and $v_{G b}$. $\boldsymbol{V}_G(\varphi)$ then maps the two $G$-corollas 
$T_{v_{G b}}$ and $T_{v_{G e}}$
isomorphically onto $S_{v_{G b}}$
and the $G$-corolla $T_{v_{Gc}}$ by a non-isomorphism quotient onto $S_{v_{G c}}$.
\end{remark}

The following elementary statement will play an important auxiliary role.

\begin{lemma}\label{VGPULL LEM}
The $G$-vertex functor 
$
 	\boldsymbol{V}_G \colon \Omega_{G}^0 \to \Fin_s \wr \Sigma_G
$
sends pullbacks over $\mathsf{O}_G$ (i.e. root pullbacks)
to pullbacks over $\Fin_s \wr \mathsf{O}_G$
(cf. Lemma \ref{FWRGROTH LEM}).
\end{lemma}

\begin{proof}
Note first that an arrow 
$(\phi,(\varphi_i))\colon (C_i)_{i \in I} \to (C'_j)_{j\in J}$
is a pullback for the split fibration 
$\Fin_s \wr \Sigma_G \to \Fin_s \wr \mathsf{O}_G$
iff each of the constituent arrows
$\varphi_i \colon C_i \to C'_{\phi(i)}$
are pullbacks for the split fibration $\Sigma_G \to \mathsf{O}_G$.

The pullback
$\psi^{\**} T \xrightarrow{\bar{\psi}} T$
of $T = (T_x)_{x \in X} \in \Omega_{G}^{0}$
over $\psi \colon Y \to X$
has the form 
$(T_{\psi(y)})_{y \in Y} \to (T_x)_{x \in X}$
and it now suffices to check that each of the vertex maps
$
	(\psi^{\**} T)_{v_{G e}} \to T_{v_{G \bar{\psi}(e)}}
$
is itself a pullback.
By (\ref{TVGE DEF}), this is the statement that for 
$f \in G e$ the induced map
\begin{equation}\label{VGPULL EQ}
	(\psi^{\**}T)_{f^{\uparrow} \leq f} \to 
	T_{\bar{\psi}(f^{\uparrow}) \leq \bar{\psi} (f)}
\end{equation}
is an identity (i.e. planar isomorphism),
and letting $y$ be such that $f \in T_{\psi(y)}$
one sees that \eqref{VGPULL EQ}
is the identity
$T_{\psi(y),f^{\uparrow} \leq f} = 
T_{x,\bar{\psi}(f)^{\uparrow} \leq \bar{\psi}(f)}$, where $x=\psi(y)$, finishing the proof.
\end{proof}

\begin{example}
The following depicts one of the maps 
\eqref{VGPULL EQ}
for the pullback $\tau^{\**} T \to T$
in Example \ref{ROOTPULL EX}.
\[
	\begin{tikzpicture}[grow=up,auto,level distance=2.1em,
	every node/.style = {font=\footnotesize,inner sep=2pt},
	dummy/.style={circle,draw,inner sep=0pt,minimum size=1.75mm}]
	\begin{scope}[xshift=7cm]
		\node at (0,0) {}
			child{node [dummy] {}
				child[sibling distance=1.25em]{
				edge from parent node [swap,near end] {$-a\phantom{j}$}}
				child[sibling distance=1.25em]{
				edge from parent node [near end]  {$\phantom{j}a$}}
			edge from parent node [swap] {$b$}};
		\node at (1.5,0) {}
			child{node [dummy] {}
				child[sibling distance=1.25em]{
				edge from parent node [swap,near end] {$-j a$}}
				child[sibling distance=1.25em]{
				edge from parent node [near end]  {$j a$}}
			edge from parent node [swap] {$j b$}};
		\node at (3,0) {}
			child{node [dummy] {}
				child[sibling distance=1.25em]{
				edge from parent node [swap,near end] {$-i a$}}
				child[sibling distance=1.25em]{
				edge from parent node [near end]  {$i a$}}
			edge from parent node [swap] {$i b$}};
		\node at (4.5,0) {}
			child{node [dummy] {}
				child[sibling distance=1.25em]{
				edge from parent node [swap,near end] {$-k a$}}
				child[sibling distance=1.25em]{
				edge from parent node [near end]  {$k a$}}
			edge from parent node [swap] {$k b$}};
		\draw[decorate,decoration={brace,amplitude=2.5pt}] (4.6,0) -- (-0.1,0) node[midway]{$T_{v_{G b}}$};
	\end{scope}
		\node at (0,0) {}
			child{node [dummy] {}
				child[sibling distance=1.25em]{
				edge from parent node [swap,near end] {$-d\phantom{j}$}}
				child[sibling distance=1.25em]{
				edge from parent node [near end]  {$\phantom{j}d$}}
			edge from parent node [swap] {$e$}};
		\node at (1.5,0) {}
			child{node [dummy] {}
				child[sibling distance=1.25em]{
				edge from parent node [swap,near end] {$j d$}}
				child[sibling distance=1.25em]{
				edge from parent node [near end]  {$-j d$}}
			edge from parent node [swap] {$j e$}};
		\node at (3,0) {}
			child{node [dummy] {}
				child[sibling distance=1.25em]{
				edge from parent node [swap,near end] {$i d$}}
				child[sibling distance=1.25em]{
				edge from parent node [near end]  {$-i d$}}
			edge from parent node [swap] {$i e$}};
		\node at (4.5,0) {}
			child{node [dummy] {}
				child[sibling distance=1.25em]{
				edge from parent node [swap,near end] {$-k d$}}
				child[sibling distance=1.25em]{
				edge from parent node [near end]  {$k d$}}
			edge from parent node [swap] {$k e$}};
		\draw[decorate,decoration={brace,amplitude=2.5pt}] (4.6,0) -- (-0.1,0) node[midway]{$(\tau^{\**}T)_{v_{G e}}$};
	\draw[->] (5.3,0.6) -- node [below] {$\substack{d \mapsto i a \\ e \mapsto i b}{}$} (6.2,0.6) ;
	\end{tikzpicture}
\]
Note that 
$(\tau^{\**} T)_{v_{G e}} = \rho^{\**} T_{v_{G b}}$
for $\rho$ the map
$\{e,j e, i e, k e\} \to \{b, j b, i b, k b\}$ 
defined by $e \mapsto i b$ so that,
accounting for orders,
$\rho$ is the block permutation $\rho = (13)(24)$.
\end{example}

We are now in a position to generalize 
Definition \ref{SUBSTITUTIONDATUM}.

\begin{definition}\label{SUBSTITUTIONDATUMG DEF}
	Let $T \in \Omega_G$ be a $G$-tree.
	
	A \textit{(resp. planar) $T$-substitution datum} is a tuple 
	$\left(S_{f^{\uparrow} \leq f} \right)_{V(T)}$ of trees together with
\begin{itemize}	
\item[(i)] associative and unital $G$-action maps
	$S_{f^{\uparrow} \leq f} \to S_{g f^{\uparrow} \leq g f}$; 
\item[(ii)]	(planar) tall maps 
	$T_{f^{\uparrow} \leq f} \to S_{f^{\uparrow} \leq f}$ compatible with the $G$-action maps.
\end{itemize}	
	Further, a map of (planar) $T$-substitution data 
	$\left(S_{f^{\uparrow} \leq f}\right) \to
	\left(R_{f^{\uparrow} \leq f}\right)$ is a compatible tuple of (planar) tall maps 
	$\left(S_{f^{\uparrow} \leq f} \to R_{f^{\uparrow} \leq f} \right)$.
	
	We denote the category of (planar) $T$-substitution data 
	by $\mathsf{Sub}(T)$
	(resp. $\mathsf{Sub}_{\mathsf{p}}(T)$).
\end{definition}

Recall that a map of $G$-trees is called 
\textit{rooted} if it induces an ordered isomorphism on the root orbit (cf. Definition \ref{ROOTPULL DEF}),
and we note that, by Definition \ref{PLANARMAP_DEF}, planar tall maps of $G$-trees are always rooted.

\begin{remark}\label{SUBSGREF DEF}
Writing $S_{v_{G e}} = (S_{f^{\uparrow} \leq f})_{f \in Ge}$,
a $T$-substitution datum can equivalently be encoded by the tuple
$\left(S_{v_{G e}}\right)_{V_G(T)}$ together with \textit{rooted} tall maps 
$T_{v_{Ge}} \to S_{v_{G e}}$.

Further, the $T$-substitution datum is planar iff the maps $T_{v_{Ge}} \to S_{v_{G e}}$ are as well.

We caution that, in the non-planar case, the
$S_{v_{G e}}$ notation requires some care,
as discussed in Remark \ref{WHYALT REM}. 
\end{remark}

\begin{remark}\label{SUBSDATUMCONV REM}
	Writing $T = (T_x)_{x \in X}$ as usual,
	one obtains (non-equivariant) $T_x$-substitution data 
	$S_{x,(\minus)}$ for each $T_x$.
	As in the discussion after Remark \ref{SCTARR REM},
	we again write
	$S_{x,(\minus)} \colon \mathsf{Sc}(T_x) \to \Omega$
	and note that these are compatible with the $G$-action, in the sense that the obvious diagram
\[
\begin{tikzcd}[row sep =3pt]
	\mathsf{Sc}(T_x) \ar{rr}{S_{x,(\minus)}} \ar{rd}[swap]{g} &&
	\Omega
\\
	& \mathsf{Sc}(T_{gx}) \ar{ru}[swap]{S_{gx,(\minus)}}
\end{tikzcd}
\]
commutes.
Writing $\mathsf{Sc}(T) = \coprod_x \mathsf{Sc}(T_x)$,
these diagrams assemble into a functor
$G \ltimes \mathsf{Sc}(T) \to \Omega$,
where $G \ltimes \mathsf{Sc}(T)$ is the Grothendieck construction for the $G$-action
(which, explicitly, adds arrows 
$\eta_a \to \eta_{ga}$, 
$T_{e^{\uparrow} \leq e} \to T_{ge^{\uparrow} \leq ge}$
to $\mathsf{Sc}(T)$ that satisfy obvious compatibilities).
\end{remark}

In the following, we write
$\colim_{\mathsf{Sc}(T)}S_{(\minus)}$
to mean
$(\colim_{\mathsf{Sc}(T_x)}S_{x,(\minus)})_{x \in X}$ or, in other words, we take the colimit 
in $\Phi = \Fin \wr \Omega$ rather than in $\Omega$
(as is needed since $\Omega$ lacks coproducts).

\begin{corollary}\label{SUBDATAUNDERPLANG COR}
Let $T \in \Omega_G$ be a $G$-tree. There are isomorphisms of categories
\[
\begin{tikzcd}[row sep =0pt]
	\mathsf{Sub}_{\mathsf{p}}(T) \ar[r,shift left=2pt] &
	T \downarrow \Omega_{G}^{\mathsf{pt}} \ar[l,shift left=2pt] &
	\mathsf{Sub}(T) \ar[r,shift left=2pt] &
	T \downarrow \Omega_{G}^{\mathsf{rt}} \ar[l,shift left=2pt]
\\
	\left(S_{f^{\uparrow} \leq f}\right)_{V(T)} \ar[r,mapsto] & 
	\left(T \to \colim_{\mathsf{Sc}(T)} S_{(\minus)}\right) &
	\left(S_{f^{\uparrow} \leq f}\right)_{V(T)} \ar[r,mapsto] & 
	\left(T \to \colim_{\mathsf{Sc}(T)} S_{(\minus)}\right)
\end{tikzcd}
\]
where $T \downarrow \Omega_G^{\mathsf{pt}}$ 
(resp. $T \downarrow \Omega_G^{\mathsf{rt}}$)
is the category of planar tall (resp. rooted tall) maps under $T$.
\end{corollary}

\begin{proof}
	This is a direct consequence of the 
	non-equivariant analogues Proposition \ref{SUBDATAUNDERPLAN PROP} and Corollary \ref{SUBDATAUNDERPLAN COR}
	applied to each individual $T_x$ together with the equivariance analysis in
	Remark \ref{SUBSDATUMCONV REM}.
\end{proof}

In the following, 
note that tall rooted maps are independent maps,
cf. Proposition \ref{INDMAPCHAR PROP}.

\begin{notation}\label{UEUPEG NOT}
	Combining Notations \ref{UEUPE NOT} and \ref{GVERT NOT}
	(and in accordance with Remark \ref{SUBSGREF DEF}),
	given an independent map
	$\varphi \colon T \to S$ of $G$-trees
	and $G$-vertex $v_{Ge} \in V_G(T)$
	we write
	\begin{equation}\label{UEUPEG EQ}
	S_{v_{Ge}} = 
	\left(S_{\varphi(f^{\uparrow}) \leq \varphi(f)}\right)_{f \in Ge}.\index{structure on trees!outer faces!SVGE@$S_{v_{Ge}}$}
	\end{equation}
\end{notation}

\begin{remark}\label{WHYALT REM}
	In Notation \ref{UEUPE NOT} one has,
	by definition,
	that the maps $U_{e^{\uparrow} \leq e} \to U$
	are planar maps, regardless of whether the map
	$\varphi \colon T \to U$ therein is planar.

	However, Notation \ref{UEUPEG NOT} is subtler.
	When $\varphi \colon T \to S$ is not a planar map,
	the maps $S_{v_{Ge}} \to S$ need to be planar on each tree component of $S_{v_{Ge}}$, but need not respect the order of the roots of $S_{v_{Ge}}$.
	This is because the tuple \eqref{UEUPEG EQ} 
	is ordered according to the edge orbit $Ge$ of $T$
	(this is needed for Remark \ref{SUBSGREF DEF}),
	rather than the edge orbit $G\varphi(e)$ of $S$.
	
	To make this more precise, write 
	$\varphi_{Ge} \colon Ge \to G \varphi(e)$
	for the induced map, which
	is an isomorphism by
	Proposition \ref{INDMAPCHAR PROP},
	and write $S_{v_{Ge}}'$ for the tuple \eqref{UEUPEG EQ},
	except reordered according to $G \varphi(e)$.
	One then has that the map
	$S_{v_{Ge}}' \to S$ is planar
	and $S_{v_{Ge}} = \varphi_{Ge}^{\**} S_{v_{Ge}}'$.
\end{remark}

\begin{remark}\label{PULLCOMP REM}
	The isomorphisms in Corollary \ref{SUBDATAUNDERPLANG COR}
	are compatible with root pullbacks of trees. 
	More concretely, as in the proof of Lemma \ref{VGPULL LEM},
	each pullback 
	$\bar{\psi} \colon \psi^{\**} T \to T$
	determines pullback maps
	$\bar{\psi}_{G e} \colon
	(\psi^{\**} T)_{v_{Ge}} \to T_{v_{G \bar{\psi}(e)}}$,
	which we note are pullbacks over the maps
	$\bar{\psi}_{G e} \colon Ge \to G \bar{\psi}(e)$
	in $\mathsf{O}_G$. The definition of pullback then allows us to uniquely fill any diagram (where we reformulate substitution data as in Remark \ref{SUBSGREF DEF})
\[
\begin{tikzcd}[row sep =10pt]
	(\psi^{\**} T)_{v_{Ge}} \ar{d} \ar[dashed]{r} &
	\bar{\psi}_{G e}^{\**}S_{v_{G \bar{\psi}(e)}} \ar{d}
\\
	T_{v_{G \bar{\psi}(e)}} \ar{r} &
	S_{v_{G \bar{\psi}(e)}}
\end{tikzcd}	
\]
defining the left vertical functors (with the right functors defined analogously) in each of the commutative diagrams below.
\begin{equation}\label{SUBDATAUNDERPLANG2 EQ}
\begin{tikzcd}[row sep =12pt,column sep= 14pt]
	\mathsf{Sub}_{\mathsf{p}}(\psi^{\**} T) \ar[r,shift left=2pt] &
	\psi^{\**} T \downarrow \Omega_{G}^{\mathsf{pt}} \ar[l,shift left=2pt] & &
	\mathsf{Sub}(\psi^{\**} T) \ar[r,shift left=2pt] &
	\psi^{\**} T \downarrow \Omega_{G}^{\mathsf{rt}} \ar[l,shift left=2pt]
\\
	\mathsf{Sub}_{\mathsf{p}}(T) \ar[r,shift left=2pt] \ar{u}{(\bar{\psi}_{Ge}^{\**})} &
	T \downarrow \Omega_{G}^{\mathsf{pt}} \ar[l,shift left=2pt] \ar{u}[swap]{\psi^{\**}} & &
	\mathsf{Sub}(T) \ar[r,shift left=2pt] \ar{u}{(\bar{\psi}_{Ge}^{\**})} &
	T \downarrow \Omega_{G}^{\mathsf{rt}} \ar[l,shift left=2pt] \ar{u}[swap]{\psi^{\**}}
\end{tikzcd}
\end{equation}
\end{remark}

\subsection{Planar strings}\label{PLANARSTRING SEC}

We now use the leaf-root and vertex functors in \S \ref{LRVERT SEC}
to repackage our substitution results in a format 
that will be more convenient for our definition of genuine equivariant operads in \S \ref{GENUINE_OP_MONAD_SECTION}.

\begin{definition}\label{PLANSTR DEF}
	The category $\Omega_{G}^n$ of 
	\textit{planar $n$-strings} is the category whose objects are strings
        \index{categories!of trees!Gtrees0stringsn@$\Omega_G^n$}
        \begin{equation}\label{STRINGOBJ EQ}
	\begin{tikzcd}
	T_0 \ar{r}{\varphi_1} & T_1 \ar{r}{\varphi_2} & \cdots \ar{r}{\varphi_n} & T_n
	\end{tikzcd}	
\end{equation}
	where $T_i \in \Omega_G$ and the $\varphi_i$ are planar \textit{tall} maps, while arrows are commutative diagrams 
	\begin{equation} \label{PTNARROW EQ}
	\begin{tikzcd}
	T_0 \ar{r}{\varphi_1} \ar{d}[swap]{\rho_0} & T_1 \ar{r}{\varphi_2} \ar{d}[swap]{\rho_1} & \cdots \ar{r}{\varphi_n} & T_n \ar{d}[swap]{\rho_n}
\\
	T'_0 \ar{r}[swap]{\varphi'_1} & T'_1 \ar{r}[swap]{\varphi'_2} & \cdots \ar{r}[swap]{\varphi'_n} & T'_n
	\end{tikzcd}	
	\end{equation}
where each $\rho_i$ is a quotient map.
\end{definition}


\begin{notation}\label{SIMPOPERATORS NOT}
	Since compositions of planar tall arrows are planar tall
	and identity arrows are planar tall,	
	it follows that $\Omega_{G}^{\bullet}$
	forms a simplicial object in $\mathsf{Cat}$, 
	with faces given by composition and degeneracies by inserting identities.
        \index{key functors!Gtreessimpdegen@$s_i$}
        \index{key functors!Gtreessimpface@$d_i$}

	Further setting 
	$\Omega_{G}^{-1} = \Sigma_G$,
        \index{categories!of trees!Gtrees0stringsn@$\Omega_G^n$}
        the leaf-root functor $\Omega_{G}^{0} \xrightarrow{\mathsf{lr}} \Sigma_G$ makes 
	$\Omega_{G}^{\bullet}$ into an augmented simplicial object and, furthermore, the maps 
	$s_{-1} \colon \Omega_{G}^{n} \to \Omega_{G}^{n+1}$
        sending $T_0 \to T_1 \to \cdots \to T_n$ to
        $\mathsf{lr}(T_0) \to T_0 \to T_1 \to \cdots \to T_n$ equip it with extra degeneracies.
  \end{notation}

\begin{remark}
      The identification $\Omega_{G}^{-1} = \Sigma_G$ can be understood by noting that a string as in \eqref{STRINGOBJ EQ} is equivalent to a string
      \begin{equation}\label{STRINGOBJALT EQ}
            \begin{tikzcd}
                  T_{-1} \ar{r}{\varphi_0} & T_0 \ar{r}{\varphi_1} & T_1 \ar{r}{\varphi_2} & \cdots \ar{r}{\varphi_n} & T_n
            \end{tikzcd}	
      \end{equation}
      where $T_{-1} = \mathsf{lr}(T_0) = \cdots = \mathsf{lr}(T_n)$.
\end{remark}

\begin{remark}\label{ALLSPLITMAPS REM}
Since for any planar $n$-string we have 
$\mathsf{r}(T_i) = \mathsf{r}(T_j)$
for any $1 \leq i,j \leq n$, 
there is a well-defined \emph{root functor}
$\mathsf{r} \colon \Omega_{G}^{n} \to \mathsf{O}_G$,
\index{key functors!Gtreesleafroot0@$\mathsf{r}$}
which is readily seen to be a split Grothendieck fibration.
Furthermore, generalizing Remark \ref{LRROOTMAP REM},
all operators $d_i$, $s_j$ 
are maps of split Grothendieck fibrations.
\end{remark}

\begin{notation}\label{VGDEF NOT}
We extend the vertex functor 
in Definition \ref{GVERTFUN DEF} to a functor 
$\boldsymbol{V}_G \colon \Omega_{G}^{n} \to \Fin_s \wr \Omega_{G}^{n-1}$
by
\begin{equation}\label{VGDEF EQ}
	\boldsymbol{V}_G(T_0 \to T_1 \to \cdots \to T_n) = 
	(T_{1,v_{Ge}} \to \cdots \to
	T_{n,v_{Ge}})_{v_{Ge} \in V_G(T_0)}	
\end{equation}
\index{key functors!vertices@$\boldsymbol{V}_G$}
where $T_{i,v_{Ge}}$ is as in
Notation \ref{UEUPEG NOT}
for the map $T_0 \to T_i$ and $v_{Ge} \in V_G(T_0)$.

Alternatively, regarding $T_0 \to \cdots \to T_n$ as a string of $n$ arrows in $T_0 \downarrow \Omega_G^{\mathsf{pt}}$, 
the object $\boldsymbol{V}_G(T_0 \to \cdots \to T_n)$
can be thought of as the image of the inverse functor in
Corollary \ref{SUBDATAUNDERPLANG COR} (in the planar case),
written according to the reformulation in 
Remark \ref{SUBSGREF DEF}.
Note, however, that from this perspective
functoriality needs to be addressed separately.
\end{notation}

\begin{notation}\label{DDDDD NOT}
	For $I \subseteq \{0,1,\cdots,n\}$
	we write
	$d_I \colon \Omega^n_G \to \Omega^{n-|I|}_G$
	for the functor which sends 
	$T_0 \to T_1 \to \cdots \to T_n$
	to the string with $T_i, i\in I$ omitted.
	
	Note that, in light of \eqref{STRINGOBJALT EQ},
	this makes sense even when
	$I = \{0,1,\cdots,n\}$,
	resulting in a functor
	$\mathsf{lr} = d_{0,1,\cdots,n}
	\colon \Omega_G^n \to \Sigma_G$.
\end{notation}

\begin{notation}\label{ROOTUNDER NOT}
	For the map 
	$\mathsf{lr} \colon \Omega_G^n \to \Sigma_G$
	of Grothendieck fibrations over $\mathsf{O}_G$
	we abbreviate the partial undercategory
	$C \downarrow_{\mathsf{O}_G} \Omega^n_G$
	for $C \in \Sigma_G$
	defined before Proposition \ref{FIBERKANMAP PROP}
	as $C \downarrow_{\mathsf{r}} \Omega^n_G$,
        \index{categories!of trees!undercat@$C \downarrow_{\mathsf r} \Omega^n_G$}
	and refer to this as the \emph{rooted undercategory}.
\end{notation}

We now obtain a key reinterpretation (and slight strengthening) of Corollary \ref{SUBDATAUNDERPLANG COR}.

\begin{proposition} \label{SUBSASPULL PROP}
For any $n\geq 0$ the commutative diagram
	\begin{equation}\label{PTPULL EQ}
	\begin{tikzcd}
		\Omega_{G}^{n} \ar{r}{\boldsymbol{V}_G} 
		\ar{d}[swap]{d_{1,\cdots,n}} & \Fin_s \wr \Omega_{G}^{n-1} 
		\ar{d}{\Fin \wr d_{0,\cdots,n-1}}
	\\
		\Omega_{G}^{0} \ar{r}[swap]{\boldsymbol{V}_G} & \Fin_s \wr \Sigma_G
	\end{tikzcd}
	\end{equation}
is a pullback diagram in $\mathsf{Cat}$.
\end{proposition}

\begin{proof}
Let us write 
$P = \Omega_{G}^{0} \times_{\Fin_s \wr \Sigma_G} \Fin_s \wr \Omega_{G}^{n-1}$ for the pullback,
so that our goal is to show that the canonical map
$\Omega_{G}^{n} \to P$ is an isomorphism. 

That $\Omega_{G}^{n} \to P$ is an isomorphism on objects 
follows by combining the alternative description of $\boldsymbol{V}_G$
in Notation \ref{VGDEF NOT} with the planar half of
Corollary \ref{SUBDATAUNDERPLANG COR}
(in fact, this yields isomorphisms of the fibers over 
$\Omega_{G}^{0}$, but we will not directly use this fact).
We will hence write $T_0 \to \cdots \to T_n$
to denote an object of $P$ as well.

An arrow in $P$ from 
$T_0 \to \cdots \to T_n$ to 
$T'_0 \to \cdots \to T'_n$
then consists of a quotient 
$\rho_0 \colon T_0 \to T'_0$
together with a $V_G(T_0)$ indexed tuple of quotients of strings (where we write $e'=\rho_0(e)$)
\begin{equation} \label{PTNARROWLOC EQ}
\begin{tikzcd}[column sep=1.2em]
	T_{0,v_{G e}} \ar{r}{} \ar{d}[swap]{\rho_{0,e}} & 
	T_{1,v_{G e}} \ar{r}{} \ar{d}[swap]{\rho_{1,e}} &
	\cdots \ar{r}{} &
	T_{n,v_{G e}} \ar{d}{\rho_{n,e}}
\\
	T'_{0,v_{G e'}} \ar{r}{} &
	T'_{1,v_{G e'}} \ar{r}{} &
	\cdots \ar{r}{} &
	T'_{n,v_{G e'}}.
\end{tikzcd}	
\end{equation}
That $\Omega_{G}^{n} \to P$ is injective on arrows is then clear.

For surjectivity, note first that, by Lemma \ref{VGPULL LEM}, the composite $P \to \Omega_{G}^{0} \to \mathsf{O}_G$
is a split Grothendieck fibration and 
$P \to \Omega_{G}^{0}$ is a map of split Grothendieck fibrations. 
Indeed, pullbacks in $P$ can be built explicitly as those arrows such that $\rho_0$ and all $\rho_{i,e}$ in 
\eqref{PTNARROWLOC EQ}
are pullbacks (alternatively, an abstract argument also works).
The alternative description of $\boldsymbol{V}_G$ in 
Notation \ref{VGDEF NOT} combined with
\eqref{SUBDATAUNDERPLANG2 EQ} 
then show that 
$\Omega_{G}^{n} \to P$ preserves pullback arrows,
so that surjectivity needs only be checked for maps in the fibers over $\mathsf{O}_G$, i.e. on rooted maps.
Tautologically, a map in $P$ is rooted iff $\rho_0 \colon T_0 \to T'_0$ is.
But, since a quotient is an isomorphism iff it is so on roots,
we further have that 
a map in $P$ is rooted iff $\rho_0 \colon T_0 \to T'_0$
is a rooted isomorphism and
each $\rho_{i,e}$ in \eqref{PTNARROWLOC EQ} is an isomorphism.
But now reinterpreting \eqref{PTNARROWLOC EQ} as a tuple of diagrams indexed by $f \in G e$,
one obtains a diagram in $\mathsf{Sub}(T_0)$ of the same shape  which, once converted to a diagram in 
$T_0 \downarrow \Omega_G^{\mathsf{rt}}$
using the rooted half of Corollary \ref{SUBDATAUNDERPLANG COR},
yields the desired rooted map \eqref{PTNARROW EQ}
in $\Omega_{G}^{n}$ lifting the rooted map in $P$.
\end{proof}


\begin{notation}\label{INDVNG NOT}
	For $0 \leq k \leq n$ we let 
\[
	\boldsymbol{V}_{G}^{k} \colon \Omega_{G}^{n} \to \Fin_s \wr \Omega_{G}^{n-k-1}
\]
\index{key functors!vertices@$\boldsymbol{V}_G^k$}  
be inductively defined by setting $\boldsymbol{V}_{G}^{0} = \boldsymbol{V}_G$ and
letting
$\boldsymbol{V}_{G}^{k+1}$ be the composite
\[
\begin{tikzcd}[column sep =1.7em]
\Omega_{G}^{n} \ar{r}{\boldsymbol{V}_G} &
\Fin_s \wr \Omega_{G}^{n-1} \ar{r}{V^k_G} &
\Fin_s \wr \Fin_s \wr \Omega_{G}^{n-k-2} \ar{r}{\sigma^0} &
\Fin_s \wr \Omega_{G}^{n-k-2}.
\end{tikzcd}
\]
\end{notation}

\begin{remark}\label{VGN REM}
When $n = 2$, $\boldsymbol{V}_{G}^{2}$ is thus the composite
\[
\begin{tikzcd}[column sep =1.7em]
	\Omega_{G}^{2} \ar{r}{\boldsymbol{V}_G} &
	\Fin_s \wr \Omega_{G}^{1} \ar{r}{\boldsymbol{V}_G} &
	\Fin_s \wr \Fin_s \wr \Omega_{G}^{0} \ar{r}{\boldsymbol{V}_G} &
	\Fin_s \wr \Fin_s \wr \Fin_s \wr \Sigma_{G} \ar{r}{\sigma^0} &
	\Fin_s \wr \Fin_s \wr \Sigma_{G} \ar{r}{\sigma^0} &
	\Fin_s \wr \Sigma_{G}
\end{tikzcd}
\]
while, for $n=4$,  $\boldsymbol{V}_{G}^{1}$ is the composite
\[
\begin{tikzcd}[column sep =1.7em]
	\Omega_{G}^{4} \ar{r}{\boldsymbol{V}_G} &
	\Fin_s \wr \Omega_{G}^{3} \ar{r}{\boldsymbol{V}_G} &
	\Fin_s \wr \Fin_s \wr \Omega_{G}^{2} \ar{r}{\sigma^0} &
	\Fin_s \wr \Omega_{G}^{2}.
\end{tikzcd}
\]
In light of Remarks \ref{VERTEXDECOMP REM} and \ref{VERTEXDECOMPG REM}, 
$\boldsymbol{V}_{G}^{n}(T_0 \to \cdots \to T_n)$ is identified with the tuple 
\begin{equation}\label{VGNISO EQ}
	(T_{k,v_{G e}}\to \cdots \to T_{n,v_{G e}})_{v_{G e} \in V_G(T_k)},
\end{equation}
where we note that strings are written in prepended notation as in \eqref{STRINGOBJALT EQ}, so that $T_{k,v_{G e}}$ is superfluous unless $k=n$.
Further, note that this requires changing the order of $V_G(T_k)$.
Rather than using the order induced by $T_k$, one instead equips 
$V_G(T_k)$ with the order induced lexicographically
from the maps 
$V_G(T_k) \to V_G(T_{k-1}) \to \cdots \to V_G(T_0)$ 
of Remark \ref{VERTEXDECOMP REM}. I.e., for 
$v,w \in V_G(T_k)$ the condition $v<w$ is determined by the lowest $l$ such that the images of $v,w$ in $V_G(T_l)$ are distinct.

Therefore, for each $d_i$ with $i < k$,
there are natural isomorphisms as on 
the left below which interchange the
lexicographical order on the indexing set $V_G(T_k)$
induced by the string
$V_G(T_k) \to V_G(T_{k-1}) \to \cdots \to V_G(T_0)$ 
with the one induced by the string
$V_G(T_k) \to V_G(T_{k-1}) \to \cdots
\widehat{V_G(T_i)}
\cdots \to V_G(T_0)$ 
that omits $V_G(T_i)$.
For $d_i$ with $i>k$ one has commutative diagrams as on the right below.
Note that no such diagram is defined for $d_k$.
\begin{equation}\label{PIIDEFDI EQ}
\begin{tikzcd}[row sep=1.7em,column sep = 3em]
	\Omega_{G}^{n} \ar{d}[swap]{d_{i}} \ar{r}{\boldsymbol{V}_G^k} &
	|[alias=F1]|
	\Fin_s \wr \Omega_{G}^{n-k-1}
	\ar[equal]{d} 
&
	\Omega_{G}^{n} \ar{d}[swap]{d_{i}} \ar{r}{\boldsymbol{V}_G^k} &
	\Fin_s \wr \Omega_{G}^{n-k-1}
	\ar{d}{d_{i-k-1}} 
\\
	|[alias=G2]|
	\Omega_{G}^{n-1} \ar{r}[swap]{\boldsymbol{V}_G^{k-1}}&
	\Fin_s \wr \Omega_{G}^{n-k-1}  
&
	\Omega_{G}^{n-1} \ar{r}[swap]{\boldsymbol{V}_G^{k}}&
	\Fin_s \wr \Omega_{G}^{n-k-2}  
\arrow[Leftrightarrow, from=F1, to=G2,shorten >=0.15cm,shorten <=0.15cm,"\pi_{i}"]
\end{tikzcd}
\end{equation}
\index{key functors!Gtreessimppi@$\pi_i$}
Similarly, for $s_j$ with $j<k$ (resp. $j \geq k$) one
has commutative diagrams as on the left (resp. right) below. Note that for $s_k$ one uses the extra degeneracy 
$s_{k-k-1}=s_{-1}$.

\begin{equation}\label{PIIDEFDI2 EQ}
\begin{tikzcd}[row sep=1.7em,column sep = 3em]
	\Omega_{G}^{n} \ar{d}[swap]{s_{j}} \ar{r}{\boldsymbol{V}_G^k} &
	|[alias=F1]|
	\Fin_s \wr \Omega_{G}^{n-k-1}
	\ar[equal]{d} 
&
	\Omega_{G}^{n} \ar{d}[swap]{s_{j}} \ar{r}{\boldsymbol{V}_G^k} &
	\Fin_s \wr \Omega_{G}^{n-k-1}
	\ar{d}{s_{j-k-1}} 
\\
	|[alias=G2]|
	\Omega_{G}^{n+1} \ar{r}[swap]{\boldsymbol{V}_G^{k+1}}&
	\Fin_s \wr \Omega_{G}^{n-k-1}  
&
	\Omega_{G}^{n+1} \ar{r}[swap]{\boldsymbol{V}_G^{k}}&
	\Fin_s \wr \Omega_{G}^{n-k}  
\end{tikzcd}
\end{equation}
\end{remark}

The functors $V^k_G$ and isomorphisms $\pi_i$ satisfy a number of compatibilities that we now catalog.

\begin{proposition}\label{PIIPROP PROP}
\begin{enumerate}[label=(\alph*)]
\item The composite
\[
\begin{tikzcd}
	\Omega_G^n \ar{r}{\boldsymbol{V}_G^k} &
	\Fin_s \wr \Omega_G^{n-k-1} \ar{r}{\boldsymbol{V}_G^l} &
	\Fin_s^{\wr 2} \wr \Omega_G^{n-k-l-2} \ar{r}{\sigma^0} &
	\Fin_s \wr \Omega_G^{n-k-l-2}
\end{tikzcd}
\]
equals the functor $\boldsymbol{V}_{G}^{k+l+1}$.

\item The functors $\boldsymbol{V}_G^k$ send pullback arrows for the split Grothendieck fibration $\Omega_G^k \to \mathsf{O}_G$
to pullback arrows for $\Fin_s \wr \Omega_G^{n-k-1} \to \Fin_s$.

\item The isomorphisms $\pi_i(T_0 \to \cdots \to T_n)$
are pullback arrows for the split Grothendieck fibration 
$\Fin_s \wr \Omega_G^{n-k-1} \to \Fin_s$. Moreover, the projection of $\pi_i(T_0 \to \cdots \to T_n)$ onto $\Fin_s$
is the permutation interchanging 
the lexicographical order on the set
$V_G(T_k)$ determined by
$V_G(T_k) \to \cdots \to V_G(T_0)$
with that determined by
$V_G(T_k) \to \cdots
\widehat{V_G(T_{i})}
\cdots \to V_G(T_0)$.

\item The rightmost diagrams in both \eqref{PIIDEFDI EQ}
and \eqref{PIIDEFDI2 EQ} are pullback diagrams in $\mathsf{Cat}$.

\item
For $i < k \leq n$ the composite natural transformation in the diagram below is $\pi_i$.
\begin{equation}\label{INDPI1 EQ}
\begin{tikzcd}[row sep=1.7em,column sep = 3.5em]
	\Omega_{G}^{n} \ar{d}[swap]{d_{i}} \ar{r}{\boldsymbol{V}_G^k} &
	|[alias=F1]|
	\Fin_s \wr \Omega_{G}^{n-k-1} \ar{r}{\Fin_s \wr \boldsymbol{V}_{G}^l} 
	\ar[equal]{d} &
	\Fin_s^{\wr 2} \wr \Omega_G^{n-k-l-2} \ar[equal]{d} \ar{r}{\sigma^0} &
	\Fin_s \wr \Omega_G^{n-k-l-2} \ar[equal]{d}
\\
	|[alias=G2]|
	\Omega_{G}^{n-1} \ar{r}[swap]{\boldsymbol{V}_G^{k-1}}&
	\Fin_s \wr \Omega_{G}^{n-k-1} \ar{r}[swap]{\Fin_s \wr \boldsymbol{V}_{G}^{l}} &
	\Fin_s^{\wr 2} \wr  \Omega_G^{n-k-l-2} \ar{r}[swap]{\sigma^0} &
	\Fin_s \wr  \Omega_G^{n-k-l-2}
\arrow[Leftrightarrow, from=F1, to=G2,shorten >=0.15cm,shorten <=0.15cm,"\pi_{i}"]
\end{tikzcd}
\end{equation}
For $k< i < k+l+1 \leq n$ the composite natural transformation in the diagram below is $\pi_{i}$.
\begin{equation}\label{INDPI2 EQ}
\begin{tikzcd}[row sep=1.7em,column sep = 3.5em]
	\Omega_{G}^{n} \ar{d}[swap]{d_{i}} \ar{r}{\boldsymbol{V}_G^k} &
	\Fin_s \wr \Omega_{G}^{n-k-1} \ar{r}{\Fin_s \wr \boldsymbol{V}_{G}^l} 
	\ar{d}[swap]{\Fin_s \wr d_{i-k-1}} &
	|[alias=F1]|
	\Fin_s^{\wr 2} \wr \Omega_G^{n-k-l-2} \ar[equal]{d} \ar{r}{\sigma^0} &
	\Fin_s \wr \Omega_G^{n-k-l-2} \ar[equal]{d}
\\
	\Omega_{G}^{n-1} \ar{r}[swap]{\boldsymbol{V}_G^k}&
	|[alias=G2]|
	\Fin_s \wr \Omega_{G}^{n-k-2} \ar{r}[swap]{\Fin_s \wr \boldsymbol{V}_{G}^{l-1}} &
	\Fin_s^{\wr 2} \wr  \Omega_G^{n-k-l-2} \ar{r}[swap]{\sigma^0} &
	\Fin_s \wr  \Omega_G^{n-k-l-2}
\arrow[Leftrightarrow, from=F1, to=G2,shorten >=0.15cm,shorten <=0.15cm,"\Fin_s \wr \pi_{i-k-1}"]
\end{tikzcd}
\end{equation}

\item Restricting to the case $k=n$, the pairs $(d_i,\pi_i)$ and
$(s_j,id_{\boldsymbol{V}_{G}^{n}})$ satisfy all possible simplicial identities (i.e. those with $i \neq n$).
Explicitly, for $0 \leq i' < i < n$
the composite natural transformations
in the diagrams
\begin{equation}\label{SIMPPI EQ}
\begin{tikzcd}[row sep=1.7em,column sep = 3em]
	\Omega_{G}^{n} \ar{r}{} \ar{d}[swap]{d_{i}} &
	|[alias=F1]|
	\Fin_s \wr \Sigma_G \ar[equal]{d}
&&
	\Omega_{G}^{n} \ar{r}{} \ar{d}[swap]{d_{i'}} &
	|[alias=F12]|
	\Fin_s \wr \Sigma_G \ar[equal]{d}
\\
	|[alias=G2]|
	\Omega_{G}^{n-1} \ar{r}[swap]{}  \ar{d}[swap]{d_{i'}} &
	|[alias=F2]|
	\Fin_s \wr \Sigma_G \ar[equal]{d}
&&
	|[alias=G22]|
	\Omega_{G}^{n-1} \ar{r}[swap]{}  \ar{d}[swap]{d_{i-1}} &
	|[alias=F22]|
	\Fin_s \wr \Sigma_G \ar[equal]{d}
\\
	|[alias=G3]|
	\Omega_{G}^{n-2} \ar{r}{} &
	\Fin_s \wr \Sigma_G
&&
	|[alias=G32]|
	\Omega_{G}^{n-2} \ar{r}{} &
	\Fin_s \wr \Sigma_G
\arrow[Leftrightarrow, from=F1, to=G2,shorten >=0.15cm,shorten <=0.15cm,"\pi_{i}"]
\arrow[Leftrightarrow, from=F2, to=G3,shorten >=0.15cm,shorten <=0.15cm,"\pi_{i'}"]
\arrow[Leftrightarrow, from=F12, to=G22,shorten >=0.15cm,shorten <=0.15cm,"\pi_{i'}"]
\arrow[Leftrightarrow, from=F22, to=G32,shorten >=0.15cm,shorten <=0.15cm,"\pi_{i-1}"]
\end{tikzcd}
\end{equation}
coincide, and similarly for the face-degeneracy relations.
\end{enumerate}
\end{proposition}

\begin{proof}
(a) follows by induction on $k$, with $k=0$ being the definition. More generally (and writing $\Fin$ for $\Fin_s$)
one has
\begin{align*}
	\sigma^0(\Fin \wr \boldsymbol{V}_G^l)\boldsymbol{V}_G^{k+1}= &
	\sigma^0(\Fin \wr \boldsymbol{V}_G^l)\sigma^0 (\Fin \wr \boldsymbol{V}_G^k) \boldsymbol{V}_G =
	\sigma^0 \sigma^0 (\Fin^{\wr 2} \wr \boldsymbol{V}_G^l)(\Fin \wr \boldsymbol{V}_G^k) \boldsymbol{V}_G
\\
	= & \sigma^0 \sigma^1 (\Fin^{\wr 2} \wr \boldsymbol{V}_G^l)(\Fin \wr \boldsymbol{V}_G^k) \boldsymbol{V}_G =
	\sigma^0 (\Fin \wr \sigma^0)  (\Fin^{\wr 2} \wr \boldsymbol{V}_G^l)(\Fin \wr \boldsymbol{V}_G^k) \boldsymbol{V}_G 
\\
	= & \sigma^0 \left(\Fin \wr \left( \sigma^0 (\Fin \wr \boldsymbol{V}_G^l) \boldsymbol{V}_G^k \right)\right) \boldsymbol{V}_G = \sigma^0 \left(\Fin \wr \boldsymbol{V}_G^{k+l+1}\right) \boldsymbol{V}_G =
\boldsymbol{V}_G^{k+l+2}.
\end{align*}

(b) generalizes Lemma \ref{VGPULL LEM}, and follows by induction using that result, Lemma \ref{FWRGROTH LEM},
and the obvious claim that $\Fin \wr \Fin \wr A \xrightarrow{\sigma^0} \Fin \wr A$ sends pullbacks over $\Fin \wr \Fin$ to pullbacks over $\Fin$.

(c) is clear from the definition of $\pi_i$. Also, (e) and (f) are easy consequences of (b) and (c): since all natural transformations involved consist of pullback arrows, one needs only check each claim after forgetting to the $\Fin_s$ coordinate, which is straightforward.

Lastly, we argue (d) by induction on $k$ and $n$. The case $k=0$ for the rightmost diagram in \eqref{PIIDEFDI EQ} follows by the diagram on the left below, combined with
Proposition \ref{SUBSASPULL PROP} applied to the bottom and total squares. The general case then follows from the right diagram, 
where the left square is in the case $k=0$,
the middle square is a pullback by induction 
(and since $\Fin \wr (\minus)$ preserves pullback squares),
and the rightmost square is clearly a pullback.
\begin{equation}\label{PROOFD EQ}
\begin{tikzcd}[row sep=1.7em,column sep = 2.5em]
	\Omega_{G}^{n} \ar{d}[swap]{d_{i}} \ar{r}{\boldsymbol{V}_G} &
	\Fin_s \wr \Omega_{G}^{n-1}
	\ar{d}{d_{i-1}} &
	\Omega_G^n \ar{r}{\boldsymbol{V}_G} \ar{d}[swap]{d_i} &
	\Fin_s \wr \Omega_G^{n-1} \ar{r}{\boldsymbol{V}_G^k} \ar{d}[swap]{\Fin_s \wr d_{i-1}} &
	\Fin_s^{\wr 2} \wr \Omega_G^{n-k-2} \ar{r}{\sigma^0} \ar{d}[swap]{\Fin_s^{\wr 2} \wr d_{i-1}}  &
	\Fin_s \wr \Omega_G^{n-k-2} \ar{d}[swap]{\Fin_s \wr d_{i-1}}
\\
	\Omega_{G}^{n-1} \ar{d}[swap]{d_{1,\cdots,n}} \ar{r}{\boldsymbol{V}_G} &
	\Fin_s \wr \Omega_{G}^{n-2}
	\ar{d}{d_{0,\cdots,n-1}}  &
	\Omega_G^{n-1} \ar{r}[swap]{\boldsymbol{V}_G} &
	\Fin_s \wr \Omega_G^{n-3} \ar{r}[swap]{\boldsymbol{V}_G^k} &
	\Fin_s^{\wr 2} \wr \Omega_G^{n-k-3} \ar{r}[swap]{\sigma^0} &
	\Fin_s \wr \Omega_G^{n-k-3}
\\
	\Omega_{G}^{0} \ar{r}{\boldsymbol{V}_G}&
	\Fin_s \wr \Sigma_G 
\end{tikzcd}
\end{equation}
The claim for the rightmost square in \eqref{PIIDEFDI2 EQ} follows by the analogous diagrams with the $d_i$ (but not $d_{1,\cdots,n}$, 
$d_{0,\cdots,n-1}$) replaced with $s_j$.
\end{proof}

\section{Genuine equivariant operads}\label{GENUINE_OP_MONAD_SECTION}

In this section we now build the category 
$\mathsf{Op}_G (\mathcal{V})$
of genuine equivariant operads.
We do so by building a monad $\mathbb{F}_G$
on the category
$\mathsf{Sym}_G(\mathcal{V}) = 
\mathsf{Fun}(\Sigma_G^{op},\mathcal{V})$
of $G$-symmetric sequences on $\mathcal{V}$, for $\V$ a symmetric monoidal category with diagonals 
(cf. Remark \ref{FINSURJ REM}).
The underlying endofunctor of $\mathbb{F}_G$ is easy to describe:
given $X \in \mathsf{Sym}_G(\mathcal{V})$, $\mathbb{F}_G X$ is given by the left Kan extension diagram
\begin{equation}\label{FGXDEF EQ}
\begin{tikzcd}[row sep=2em,column sep = 3.3em]
	(\Omega_{G}^{0})^{op} \ar{r}{\boldsymbol{V}_{G}^{op}} \ar{d}[swap]{\mathsf{lr}} &
	|[alias=F1]|
(\Fin_s \wr \Sigma_G)^{op} \ar{r}[swap,name=F2]{}{(\Fin_s \wr X^{op})^{op}}& (\Fin_s \wr \mathcal{V}^{op})^{op} \ar{r}{\otimes} & \mathcal{V}
\\
	|[alias=G2]|
	\Sigma_G^{op}  \ar{urrr}[swap]{\mathbb{F}_G X} & &
\arrow[Rightarrow, from=F1, to=G2,shorten >=0.25cm,shorten <=0.35cm]
\end{tikzcd}
\end{equation}
Explicitly, using Proposition \ref{FIBERKANMAP PROP}
	and the fact that 
	the rooted undercategories
	$C \downarrow_{\mathsf{r}} \Omega_G^0$ 
	(cf. Notation \ref{ROOTUNDER NOT})
	only depend on the isomorphisms
	in $\Omega^0_G,\Sigma_G$,
	the left Kan extension can be computed by
	replacing both of $\Omega^0_G,\Sigma_G$ 
	with their groupoids of isomorphims,
	yielding the formula
\begin{equation}\label{FGXDEFEXP EQ}
\mathbb{F}_G X (C) \simeq
\coprod_{T \in 
\mathsf{Iso}(C \downarrow_{\mathsf{r}} \Omega_G^0)}
\left(
\bigotimes_{v \in V_G(T)}
 X(T_v)
\right) 
\cdot_{\mathsf{Aut}(T)} \mathsf{Aut}(C),
\end{equation}
though we will prefer to work with \eqref{FGXDEF EQ} throughout.

To intuitively motivate the monad structure of $\mathbb{F}_G X$, note that 
\eqref{FGXDEFEXP EQ} roughly states that 
$\mathbb{F}_G X$ consists of ``$G$-trees $T$ with $G$-nodes suitably labeled by $X$'', 
and thus that $\mathbb{F}_G \mathbb{F}_G X$
consists of ``$G$-trees $T_0$ with $G$-nodes labeled 
by $G$-trees $T_{1,i}$ with $G$-nodes labeled by $X$''.
The substitution discussion in 
\S \ref{OUTTALL SEC}, 
\S \ref{PLANARSTRING SEC}
then says that $\mathbb{F}_G \mathbb{F}_G X$ roughly consists of ``planar tall maps of $G$-trees $T_0 \to T_1$ with $G$-nodes of $T_1$ labeled by $X$'' (for a precise statement, see Remark \ref{REPACKAGERES REM}),
so that the multiplication
$\mathbb{F}_G \mathbb{F}_G \to \mathbb{F}_G$
is obtained by ``forgetting $T_0$''.

To rigorously describe the monad structure on $\mathbb{F}_G$, however, we will find it preferable to separate the left Kan extension step in \eqref{FGXDEF EQ}
from the remaining construction.
As such, we will 
build a monad $N$ 
on a larger category
$\mathsf{WSpan}^l(\Sigma_G^{op},\mathcal{V})$
in \S \ref{MONSPAN SEC}
(see Proposition \ref{MONSPAN PROP}),
which we then transfer
to $\mathsf{Sym}_G(\mathcal{V})$
in \S \ref{FGMON SEC} 
by using the $(\mathsf{Lan},\upsilon)$ adjunction in Remark \ref{RANLANADJ REM}.
\S \ref{COMPARISON_REGULAR_SECTION} then compares genuine equivariant operads with regular equivariant operads, obtaining   the pair of adjunctions in Corollary \ref{TWOADJOINTSOP_COR}, which are required when formulating and proving our main results.
Lastly, \S \ref{INDEXING_SECTION} shows that the indexing systems of Blumberg-Hill (or, more precisely, a slight generalization called ``weak indexing systems''; see Remark \ref{WINDEX_GAMMA_REM}) naturally give rise to notions of ``partial genuine operads''.


\subsection{A monad on spans}\label{MONSPAN SEC}

\begin{definition}\label{WSPAN DEF}
For categories $\C,\D$ we write 
$\mathsf{WSpan}^l(\C,\D)$
(resp.
$\mathsf{WSpan}^r(\C,\D)$),
which we call the category of  \textit{left weak spans} (resp. \textit{right weak spans}),
to denote the category with objects the spans
\[
\begin{tikzcd}
\C & A \ar{l}[swap]{k} \ar{r}{X} & \D,
\end{tikzcd}
\]
arrows the diagrams as on the left (resp. right) below 
\[
	\begin{tikzcd}[row sep=small]
	&
	A_1 \ar{dl}[swap,name=k1]{k_1} \ar{dr}[name=F11]{X_1} \ar{dd}[swap]{i} & &
	&
	A_1 \ar{dl}[swap,name=k1]{k_1} \ar{dr}[name=F1]{X_1} \ar{dd}[swap]{i} 
\\
	\C & & \D &
	\C & & \D 
\\
		& |[alias=G21]| A_2  \ar{ul}{k_2} \ar{ur}[swap]{X_2} & &
		& |[alias=G2]| A_2  \ar{ul}{k_2} \ar{ur}[swap]{X_2} &
		\arrow[Leftarrow, from=F1, to=G2,shorten >=0.25cm,shorten <=0.25cm,"\varphi"]
		\arrow[Rightarrow, from=F11, to=G21,shorten >=0.25cm,shorten <=0.25cm,"\varphi"]
	\end{tikzcd}
\]
which we write as $(i,\varphi) \colon (k_1,X_1) \to (k_2,X_2)$, and composition given in the obvious way.
\end{definition}

\begin{remark}
There are canonical natural isomorphisms
\[
	\mathsf{WSpan}^r(\C,\D) \simeq 
	\mathsf{WSpan}^l(\C^{op},\D^{op}).
\]
\end{remark}

\begin{remark}\label{RANLANADJ REM}
The terms \textit{left/right} are motivated by the existence of adjunctions (which are seen to be equivalent by 
the previous remark)
\[
	\mathsf{Lan} \colon
	\mathsf{WSpan}^l(\C, \D)
		\rightleftarrows
	\mathsf{Fun}(\C, \D)
	\colon \upsilon
\]
\[
	\upsilon \colon 
	\mathsf{Fun}(\C, \D)
		\rightleftarrows
	\mathsf{WSpan}^r(\C, \D)^{op}
	\colon \mathsf{Ran}
\]
where the functors $\upsilon$ denote the obvious inclusions 
(note the need for the $(\minus)^{op}$ in the second adjunction) 
and $\mathsf{Lan}$/$\mathsf{Ran}$ denote the left/right Kan extension functors.
\end{remark}

We will be mainly interested in the span categories 
$\mathsf{WSpan}^l(\Sigma_G^{op},\mathcal{V})\simeq 
\mathsf{WSpan}^r(\Sigma_G,\mathcal{V}^{op})$.

\begin{notation}\label{OMEGAGNA NOT}
	Given a functor $\rho \colon A \to \Sigma_G$, $n \geq 0$, we let $\Omega_G^n \wr A$ denote the pullback in $\mathsf{Cat}$
\begin{equation}\label{OMGGNA}
	\begin{tikzcd}
	\Omega_{G}^n \wr A \ar{r}{\boldsymbol{V}_{G}^{n}} \ar{d}& 
	\Fin_s \wr A \ar{d}
\\
	\Omega_{G}^{n} \ar{r}[swap]{\boldsymbol{V}_{G}^{n}} &
	\Fin_s \wr \Sigma_G
	\end{tikzcd}
\end{equation}
We will write the top $V^n_G$ functor as $\boldsymbol{V}_G^n \wr A$ whenever we need to distinguish such functors.

Explicitly, by Remark \ref{VGN REM}
the objects of $\Omega_{G}^{n} \wr A$ are pairs 
\begin{equation}\label{OMEGAGNA EQ}
(T_0 \to \cdots \to T_n,
(a_{v_{G e}})_{v_{G e} \in V_G(T_n)})
\end{equation}
such that $\rho(a_{v_{G e}}) = T_{n,v_{G e}}$, and
where $V_G(T_n)$ is ordered lexicographically
(cf. Remark \ref{VGN REM})
according to the string $T_0 \to \cdots \to T_n$.
\end{notation}

\begin{remark}
	Generalizing the notation $\Omega_{G}^{-1} = \Sigma_G$, we will also write $\Omega_G^{-1} \wr A  = A$, in which case
	$\boldsymbol{V}_{G}^{-1} \wr A \colon \Omega_G^{-1} \wr A \to \Fin_s \wr A$
	is the obvious ``singleton map'' $\delta^0 \colon A \to \Fin_s \wr A$.
\end{remark}

\begin{remark}
An alternative, and arguably more suggestive, notation for 
$\Omega_{G}^n \wr A$ would be $\Omega_{G}^n \wr_{\Sigma_G} A$,
since we are really defining a ``relative'' analogue of the wreath product 
(so that in particular $\Omega_{G}^n \wr_{\Sigma_G} \Sigma_G \simeq \Omega_G^n)$.
However, we will prefer $\Omega_{G}^n \wr A$ due to space concerns.
\end{remark}

Our primary interest here will be in the 
$\Omega_{G}^{0}\wr (\minus)$ construction,
which can be iterated thanks to the existence of the composite maps
$\Omega_{G}^{0} \wr A \to \Omega_{G}^{0} \to \Sigma_G$.
The role of the higher strings $\Omega_{G}^{n} \wr A $ will then be to provide more convenient models for iterated 
$\Omega_{G}^{0}\wr (\minus)$ constructions.
Indeed, Proposition \ref{SUBSASPULL PROP} can be reinterpreted as providing a canonical identification
$\Omega_{G}^{0} \wr \Omega_{G}^{n} \simeq \Omega_{G}^{n+1}$,
with the functor $\boldsymbol{V}_G^0 \wr \Omega_G^n$ identified with the functor $\boldsymbol{V}_G$ as defined in Notation \ref{VGDEF NOT}.
Moreover, arguing by induction on $n$, the fact that the rightmost squares in \eqref{PIIDEFDI EQ} are pullbacks
(Proposition \ref{PIIPROP PROP})
provides further identifications
$\Omega_{G}^{k} \wr \Omega_{G}^{n} \simeq \Omega_{G}^{n+k+1}$
with $\boldsymbol{V}_G^k \wr \Omega_G^n$ identified with $\boldsymbol{V}_G^k$ as defined by Notation \ref{INDVNG NOT}.

Our first task is now to produce analogous identifications between
$\Omega_{G}^{k} \wr \Omega_{G}^{n} \wr A =
\Omega_{G}^{k} \wr (\Omega_{G}^{n} \wr A)$
and 
$\Omega_{G}^{n+k+1} \wr A$
(note that iterated wreath expressions should always be read as bracketed on the right, i.e. we do \textit{not} define the expression
$ (\Omega_{G}^{k} \wr \Omega_{G}^{n}) \wr A$).
We start by generalizing the key functors from \S \ref{PLANARSTRING SEC}.

\begin{proposition}\label{PIIPROPA PROP}
There are functors
\[
	\begin{tikzcd}
	\Omega_G^n \wr A \ar{r}{\boldsymbol{V}_G^k} & \Fin_s \wr \Omega_G^{n-k-1}\wr A
&
	\Omega_G^n \wr A \ar{r}{d_i} & \Omega_G^{n-1}\wr A
&
	\Omega_G^n \wr A \ar{r}{s_j} & \Omega_G^{n+1}\wr A
	\end{tikzcd}
\]
where $i<n$, and natural isomorphisms 
\[
	\pi_i \colon \boldsymbol{V}_G^k \Rightarrow \boldsymbol{V}_G^{k-1} \circ d_i
\]
for $i < k$.
Further, all of these are natural in $A$
and they satisfy all the analogues of the properties listed in 
Proposition \ref{PIIPROP PROP}.
\end{proposition}

\begin{proof}
Though it is not hard to explicitly write formulas for $\boldsymbol{V}_G^k$, $d_i$, $s_j$, $\pi_i$ 
(see Remark \ref{VGDEFA REM} below),
and then verify the desired properties,
we here instead argue that the desiderata themselves can be used to uniquely, and coherently, define those functors. 

Firstly, the functors $\boldsymbol{V}_G = \boldsymbol{V}_G^0$ are defined from the following diagram
\[
\begin{tikzcd}[row sep = 1.3em,column sep = 3em]
	\Omega_{G}^{n+1} \wr A \ar[dashed]{r}{\boldsymbol{V}_G} \ar{d}& 
	\Fin_s \wr \Omega_G^n \wr A \ar{r}{\Fin_s \wr \boldsymbol{V}_{G}^n} \ar{d}&
	\Fin_s^{\wr 2} \wr A  \ar{d} \ar{r}{\sigma^0} &
	\Fin_s \wr A \ar{d}
\\
	\Omega_{G}^{n+1} \ar{r}{\boldsymbol{V}_G} &
	\Fin_s \wr \Omega_{G}^{n} \ar{r}{\Fin_s \wr \boldsymbol{V}_{G}^{n}} &
	\Fin_s^{\wr 2} \wr \Sigma_G \ar{r}{\sigma^0} &
	\Fin_s \wr \Sigma_G
\end{tikzcd}
\]
by noting that the middle and right squares are pullbacks, 
and choosing $\boldsymbol{V}_G$ to be the unique functor such that the top composite is $\boldsymbol{V}_G^{n+1}.$
The higher functors $\boldsymbol{V}_G^k$ are defined exactly as in \eqref{VGDEF EQ}, and the analogue of Proposition \ref{PIIPROP PROP}(a) follows by the same proof.

The analogue of Proposition \ref{PIIPROP PROP}(b) is tautological, as pullback arrows for 
$\Omega_G^n \wr A \to \mathsf{O}_G$
are defined as compatible pairs of pullbacks in 
$\Omega_G^n$ and $\Fin_s \wr A$.

To define $d_i$, we consider the diagram below (for some $i<k$).
\[
\begin{tikzcd}[column sep = small, row sep = small]
	\Omega_{G}^n \wr A \ar{rrrr}{\boldsymbol{V}_G^{k}} \ar[dashed]{rd}[swap,near end]{d_i} \ar{dd}
	&&
	&&
	|[alias=FFOmegan]|\Fin_s \wr \Omega_G^{n-k-1} \wr A  \ar[dd] \ar[equal]{rd}
\\
	&
	|[alias=Omeganp1]|\Omega_{G}^{n-1} \wr A \ar[crossing over]{rrrr}[swap,near start]{\boldsymbol{V}_G^{k-1}} &&&&
	|[alias=Omeganp2]|
	\Fin_s \wr \Omega_G^{n-k-1} \ar{dd}&
\\
	\Omega_{G}^{n} \ar{rrrr} \ar{rd}[swap,near end]{d_i}&&
	 &&
	|[alias=FFOmeganm1]| \Fin_s \wr \Omega_G^{n-k-1} \ar[rd,equal] 
\\
	&
	|[alias=Omegan]|\Omega_G^{n-1} \ar{rrrr}[swap,near start]{\boldsymbol{V}_G^{k-1}} &&&&
	\Fin_s \wr \Omega_G^{n-k-1} &
	\arrow[Leftrightarrow, from=Omeganp1, to=FFOmegan,shorten <=0.15cm,shorten >=0.15cm,swap,"\pi_i"]
	\arrow[Leftrightarrow, from=FFOmeganm1, to=Omegan,shorten <=0.15cm,shorten >=0.15cm]
	\arrow[from=Omeganp1, to=Omegan,crossing over]
\end{tikzcd}
\]
The desiderata that the top $\pi_i$ consist of pullback arrows lifting the lower $\pi_i$ implies that it is uniquely determined by the top $\boldsymbol{V}_G^k$ functor, and hence so is the top composite 
$\boldsymbol{V}_G^{k-1}d_i$. But since the front face is a pullback square
(by arguing via induction on $k$ as in \eqref{PROOFD EQ}), there is a unique choice for $d_i$. 
That this definition of $d_i \wr A$ is
independent of $k$ 
is a consequence of the fact that the composite natural transformation in \eqref{INDPI1 EQ} is $\pi_i$.
Similarly, that the analogues of the left diagrams in 
\eqref{PIIDEFDI2 EQ}
hold follows by an identical argument from the fact that the composites of \eqref{INDPI2 EQ} are $\pi_{i+1}$.

The definitions of the $s_j$ are similar, except easier since there are no $\pi_i$ to contend with.

The analogues of Proposition \ref{PIIPROP PROP}(c),(e),(f) are then tautological, and the analogue of 
Proposition \ref{PIIPROP PROP}(d)
follows by an identical argument.
\end{proof}

\begin{remark}\label{VGDEFA REM}
Explicitly,
$\boldsymbol{V}_G^{k} \colon \Omega_{G}^{n} \wr A
\to \Fin_s \wr \Omega_{G}^{n-k-1} \wr A $
is defined by sending \eqref{OMEGAGNA EQ} to
\[
	\left(
		\left(
		T_{k,v_{G f}} \to \cdots \to T_{n,v_{G f}},
		\left(
		a_{v_{G e}}
		\right)_{v_{G e} \in V_G\left(T_{n,v_{G f}}\right)}
		\right)
	\right)_{v_{G f} \in V_G(T_k)}
\]
where both $V_G(T_k)$ and $T_{n,v_{G f}}$ are ordered lexicographically according to the associated planar strings.

Similarly, functors 
$d_i \colon \Omega_{G}^{n} \wr A \to \Omega_{G}^{n-1} \wr A$
for $0 \leq i < n$
and 
$s_j \colon \Omega_{G}^{n} \wr A \to \Omega_{G}^{n+1} \wr A$
for $-1 \leq j \leq n$
are defined on the object in \eqref{OMEGAGNA EQ}
by performing the corresponding operation on the $T_0 \to \cdots \to T_n$ coordinate and, in the $d_i$ case,
 suitably reordering $V_G(T_n)$.
\end{remark}

\begin{remark}
One upshot of Proposition \ref{PIIPROPA PROP} is that formally applying the symbol $(\minus) \wr A$ 
to the diagrams in Proposition \ref{PIIPROP PROP} yields sensible statements. As such, we will simply refer to the corresponding part of
Proposition \ref{PIIPROP PROP} when
using one of the generalized claims.
\end{remark}

\begin{corollary}\label{IDEN COR}
One has identifications 
$\Omega_G^k \wr \Omega_G^n \wr A \simeq \Omega_{G}^{n+k+1} \wr A$ which identify $\boldsymbol{V}_G^k \wr \Omega_G^n \wr A$ with 
$\boldsymbol{V}_G^k \wr A$.
Further, these are associative in the sense that the identifications
\[
	\Omega_G^k \wr \Omega_G^l \wr \Omega_G^n \wr A \simeq 
	\Omega_G^{k+l+1} \wr \Omega_G^n \wr A \simeq 
	\Omega_G^{k+l+n+2} \wr A 
\]
\[
	\Omega_G^k \wr \Omega_G^l \wr \Omega_G^n \wr A \simeq 
	\Omega_G^{k} \wr \Omega_G^{l+n+1} \wr A \simeq 
	\Omega_G^{k+l+n+2} \wr A 
\]
coincide.
Lastly, one obtains identifications
\[
	d_i \wr \Omega_G^n \simeq d_i \quad
	\pi_i \wr \Omega_G^n \simeq \pi_i \quad
	s_j \wr \Omega_G^n \simeq s_j \quad
	\Omega_G^k \wr d_i \simeq d_{i+k+1} \quad
	\Omega_G^k \wr s_j \simeq s_{j+k+1}
\]
\end{corollary}

\begin{proof}
The identification $\Omega_G^k \wr \Omega_G^n \wr A \simeq \Omega_{G}^{n+k+1} \wr A$ follows since by 
Proposition \ref{PIIPROP PROP}(a)
both expressions 
compute the limit of the solid part of the diagram below.
\[
\begin{tikzcd}[row sep = 1.5em]
	\bullet \ar[dashed]{r} \ar[dashed]{d}&
	\bullet \ar[dashed]{r} \ar[dashed]{d}&
	\Fin_s^{\wr 2} \wr A \ar{r}{\sigma^0} \ar{d}&
	\Fin_s \wr A \ar{d}
\\
	\Omega_G^{n+k+1} \ar{r}[swap]{\boldsymbol{V}_G^k} \ar{d}&
	\Fin_s \wr \Omega^n_G \ar{r}[swap]{\Fin_s \wr \boldsymbol{V}_G^n} \ar{d}&
	\Fin_s^{\wr 2} \wr \Sigma_G \ar{r}[swap]{\sigma^0} &
	\Fin_s \wr \Sigma_G &
\\
	\Omega_G^k \ar{r}[swap]{\boldsymbol{V}_G^k} &
	\Fin_s \wr \Sigma_G
\end{tikzcd}
\]
Associativity follows similarly. 
The remaining identifications 
follow from the $(-) \wr A$ analogues of 
\eqref{INDPI1 EQ}, \eqref{INDPI2 EQ},
and the right side of \eqref{PIIDEFDI2 EQ}.
\end{proof}

We now have all the necessary ingredients to define our monad on 
spans.

\begin{definition}\index{monads!spanmonad@$N$}
  \label{WSPAN_MONAD_DEFINITION}
	Suppose $\mathcal{V}$ has finite products or, more generally, that it is a symmetric monoidal category with diagonals in the sense of Remark \ref{FINSURJ REM}.
	
	We define an endofunctor $N$ of 
	$\mathsf{Wspan}^r(\Sigma_G,\mathcal{V}^{op})$
	by letting $N(\Sigma_G \leftarrow A \to \mathcal{V}^{op})$
	be the span $\Sigma_G \leftarrow \Omega_G^0 \wr A \to \mathcal{V}^{op}$ given by composition of the diagram
\[
	\begin{tikzcd}
	\Omega_G^0 \wr A \ar{r}{\boldsymbol{V}_G} \ar{d}&
	\Fin_s \wr A \ar{r} \ar{d}&
	\Fin_s \wr \mathcal{V}^{op} \ar{r}{\otimes^{op}} &
	\mathcal{V}^{op}
\\
	\Omega_{G}^{0} \ar{r}[swap]{\boldsymbol{V}_G} \ar{d} &
	\Fin_s \wr \Sigma_G
\\
	\Sigma_G
	\end{tikzcd}
\]
and defined on maps of spans in the obvious way.

One has a multiplication $\mu \colon N \circ N \Rightarrow N$ given by the natural isomorphism
\begin{equation}\label{MULTDEFSPAN EQ}
	\begin{tikzcd}
	\Sigma_G \ar[equal]{d}&
	\Omega_{G}^1 \wr A \ar{r}{\boldsymbol{V}_G} \ar{d}[swap]{d_{0}} \ar{l}&
	\Fin_s \wr \Omega_{G}^0 \wr A \ar{r}{\Fin_s \wr \boldsymbol{V}_G} &
	|[alias=FFOmega]| \Fin_s^{\wr 2} \wr A \ar{d}{\sigma^0} \ar{r} &
	\Fin_s^{\wr 2} \wr \mathcal{V}^{op} \ar{d}{\sigma^0} \ar{r}{\otimes^{op}} &
	\Fin_s \wr \mathcal{V}^{op} \ar{r}{\otimes^{op}} &
	|[alias=dog]|
	\mathcal{V}^{op} \ar[equal]{d}
\\
	\Sigma_G &
	|[alias=Omega]|\Omega_{G}^{0} \wr A \ar{rr}[swap]{\boldsymbol{V}_G} \ar{l}&&
	\Fin_s \wr A \ar{r} &
	|[alias=cat]|
	\Fin_s \wr \mathcal{V}^{op} \ar{rr}[swap]{\otimes^{op}} &&
	\mathcal{V}^{op}
	\arrow[Leftrightarrow, from=FFOmega, to=Omega,shorten <=0.15cm,,shorten >=0.15cm,"\pi_0"]
	\arrow[Leftrightarrow, from=dog, to=cat,shorten <=0.15cm,,shorten >=0.15cm,"\alpha"]
	\end{tikzcd}
\end{equation}
where we note that the top right composite in the 
$\pi_0$ square is indeed $\boldsymbol{V}_{G}^{1}$,
thanks to the inductive description in (the $(\minus) \wr A$ analogue of) Notation \ref{INDVNG NOT}.

Lastly, there is a unit $\eta \colon id \Rightarrow N$ given by the strictly commutative diagrams
\begin{equation}\label{UNITSPAN EQ}
	\begin{tikzcd}
	\Sigma_G \ar[equal]{d} &
	A \ar{l} \ar{d}[swap]{s_{-1}} \ar[equal]{r} &
	A \ar{d}{\delta^0} \ar{r} &
	\mathcal{V}^{op} \ar{d}{\delta^0} \ar[equal]{r}&
	\mathcal{V}^{op} \ar[equal]{d}
\\
	\Sigma_G &
	\Omega_{G}^{0} \wr A \ar{l} \ar{r}[swap]{\boldsymbol{V}_G}&
	\Fin_s \wr A \ar{r} &
	\Fin_s \wr \mathcal{V}^{op} \ar{r}[swap]{\otimes^{op}} &
	\mathcal{V}^{op}.
	\end{tikzcd}
\end{equation}	
\end{definition}

\begin{proposition}\label{MONSPAN PROP}
$(N,\mu,\eta)$ is a monad on $\mathsf{Wspan}^r(\Sigma_G,\mathcal{V}^{op})$.
\end{proposition}

\begin{proof}
Throughout the proof we abbreviate $\Fin_s$ as $\Fin$.

The natural transformation component of $\mu \circ (N \mu)$ is given by the composite diagram
\begin{equation}\label{ASSOCSPAN1 EQ}
	\begin{tikzcd}[column sep=1em]
	\Omega_{G}^2 \wr A \ar{d}[swap]{d_1} \ar{r} &
	\Fin \wr \Omega_{G}^1 \wr A \ar{d}[swap]{\Fin \wr d_0} \ar{r}&
	\Fin^{\wr 2} \wr \Omega_{G}^0 \wr A \ar{r} &
	|[alias=dog2]|
	\Fin^{\wr 3} \wr A \ar{d}{\sigma^1} \ar{r} &
	\Fin^{\wr 3} \wr \mathcal{V}^{op} \ar{d}{\sigma^1} \ar{r} &
	\Fin^{\wr 2} \wr \mathcal{V}^{op} \ar{r}&
	|[alias=cat3]|
	\Fin \wr \mathcal{V}^{op} \ar[equal]{d} \ar{r} &
	\mathcal{V}^{op} \ar[equal]{d}
\\
	\Omega_{G}^1 \wr A \ar{r} \ar{d}[swap]{d_{0}}&
	|[alias=cat2]|
	\Fin \wr \Omega_{G}^0 \wr A \ar{rr} &&
	|[alias=FFOmega]|\Fin^{\wr 2} \wr A \ar{d}{\sigma^0} \ar{r} &
	|[alias=dog3]|
	\Fin^{\wr 2} \wr \mathcal{V}^{op} \ar{d}{\sigma^0} \ar{rr} &&
	\Fin \wr \mathcal{V}^{op} \ar{r} &
	|[alias=dog]|
	\mathcal{V}^{op} \ar[equal]{d}
\\
	|[alias=Omega]|\Omega_{G}^0 \wr A \ar{rrr} &&&
	\Fin \wr A \ar{r} &
	|[alias=cat]|
	\Fin \wr \mathcal{V}^{op} \ar{rrr} &&&
	\mathcal{V}^{op}
	\arrow[Leftrightarrow, from=FFOmega, to=Omega,shorten <=0.15cm,shorten >=0.15cm,"\pi_0"]
	\arrow[Leftrightarrow, from=dog, to=cat,shorten <=0.15cm,shorten >=0.15cm,"\alpha"]
	\arrow[Leftrightarrow, from=dog2, to=cat2,shorten <=0.15cm,shorten >=0.15cm,"\Fin \wr \pi_0"]
	\arrow[Leftrightarrow, from=cat3, to=dog3,shorten <=0.15cm,shorten >=0.15cm,"\Fin \wr \alpha"]
	\end{tikzcd}
\end{equation}
whereas the natural transformation component of $\mu \circ (\mu N)$ is given by
\begin{equation}\label{ASSOCSPAN2 EQ}
	\begin{tikzcd}[column sep=1.05em]
	\Omega_{G}^{2} \wr A \ar{d}[swap]{d_0} \ar{r} &
	\Fin \wr \Omega_{G}^{1} \wr A \ar{r} &
	|[alias=dog2]|
	\Fin^{\wr 2} \wr \Omega_{G}^{0} \wr A \ar{r} \ar{d}{\sigma^0}&
	\Fin^{\wr 3} \wr A \ar{r} \ar{d}{\sigma^0} &
	\Fin^{\wr 3} \wr \mathcal{V}^{op} \ar{r} \ar{d}{\sigma^0} &
	\Fin^{\wr 2} \wr \mathcal{V}^{op} \ar{r} \ar{d}&
	\Fin \wr \mathcal{V}^{op} \ar{r} &
	|[alias=dog3]|
	\mathcal{V}^{op} \ar[equal]{d}
\\
	|[alias=cat2]|
	\Omega_{G}^{1} \wr A \ar{rr} \ar{d}[swap]{d_{0}} &&
	\Fin \wr \Omega_{G}^{0} \wr A \ar{r} &
	|[alias=FFOmega]|
	\Fin^{\wr 2} \wr A \ar{d}{\sigma^0} \ar{r} &
	\Fin^{\wr 2} \wr \mathcal{V}^{op} \ar{d}{\sigma^0} \ar{r} &
	|[alias=cat3]|
	\Fin \wr \mathcal{V}^{op} \ar{rr} &&
	|[alias=dog]|
	\mathcal{V}^{op} \ar[equal]{d}
\\
	|[alias=Omega]|\Omega_{G}^{0} \wr A \ar{rrr} &&&
	\Fin \wr A \ar{r} &
	|[alias=cat]|
	\Fin \wr \mathcal{V}^{op} \ar{rrr} &&&
	\mathcal{V}^{op}
	\arrow[Leftrightarrow, from=FFOmega, to=Omega,shorten <=0.15cm,,shorten >=0.15cm,"\pi_0"]
	\arrow[Leftrightarrow, from=dog, to=cat,shorten <=0.15cm,,shorten >=0.15cm,"\alpha"]
	\arrow[Leftrightarrow, from=dog2, to=cat2,shorten <=0.15cm,,shorten >=0.15cm,"\pi_0"]
	\arrow[Leftrightarrow, from=dog3, to=cat3,shorten <=0.15cm,,shorten >=0.15cm,"\alpha"]
	\end{tikzcd}
\end{equation}
That the rightmost sides of \eqref{ASSOCSPAN1 EQ} and \eqref{ASSOCSPAN2 EQ} coincide follows from the associativity of the isomorphisms $\alpha$ in \eqref{COHER2 EQ}.
On the other hand, the leftmost sides coincide since they are instances of the ``simplicial relation'' diagrams in \eqref{SIMPPI EQ}, as is seen by using 
\eqref{INDPI1 EQ} and \eqref{INDPI2 EQ}
to reinterpret the top left sections.

As for the unit conditions, $\mu \circ (N \eta)$ is represented by
\begin{equation}\label{UNITSPAN1 EQ}
	\begin{tikzcd}[column sep=1.05em]
	\Omega_G^0 \wr A \ar{d}[swap]{s_{0}} \ar{r} &
	\Fin \wr A \ar{d}[swap]{s_{-1}} \ar[equal]{r} &
	\Fin \wr A \ar{d}{\delta^1} \ar{r} &
	\Fin \wr \mathcal{V}^{op} \ar{d}{\delta^1} \ar[equal]{r} &
	\Fin \wr \mathcal{V}^{op} \ar[equal]{d} \ar{r} &
	\mathcal{V}^{op} \ar[equal]{d}
\\
	\Omega_{G}^{1} \wr A \ar{r} \ar{d}[swap]{d_{0}}&
	\Fin \wr \Omega_{G}^{0} \wr A \ar{r} &
	|[alias=FFOmega]|\Fin^{\wr 2} \wr A \ar{d}{\sigma^0} \ar{r} &
	\Fin^{\wr 2} \wr \mathcal{V}^{op} \ar{d}{\sigma^0} \ar{r} &
	\Fin \wr \mathcal{V}^{op} \ar{r} &
	|[alias=dog]|
	\mathcal{V}^{op} \ar[equal]{d}
\\
	|[alias=Omega]|\Omega_{G}^0 \wr A \ar{rr}&&
	\Fin \wr A \ar{r} &
	|[alias=cat]|
	\Fin \wr \mathcal{V}^{op} \ar{rr} &&
	\mathcal{V}^{op}
	\arrow[Leftrightarrow, from=FFOmega, to=Omega,shorten <=0.15cm,,shorten >=0.15cm,"\pi_0"]
	\arrow[Leftrightarrow, from=dog, to=cat,shorten <=0.15cm,,shorten >=0.15cm,"\alpha"]
	\end{tikzcd}
\end{equation}
while $\mu \circ (\eta N)$ is represented by 
\begin{equation}\label{UNITSPAN2 EQ}
	\begin{tikzcd}[column sep=1.05em]
	\Omega_G^0 \wr A \ar{d}[swap]{s_{-1}} \ar[equal]{r} &
	\Omega_G^0 \wr A \ar{d}{\delta^0} \ar{r} &
	\Fin \wr A \ar{d}{\delta^0} \ar{r} &
	\Fin \wr \mathcal{V}^{op} \ar{d}{\delta^0} \ar{r} &
	\mathcal{V}^{op} \ar{d}{\delta^0} \ar[equal]{r} &
	\mathcal{V}^{op} \ar[equal]{d}
\\
	\Omega_{G}^{1} \wr A \ar{r} \ar{d}[swap]{d_{0}}&
	\Fin \wr \Omega_{G}^0 \wr A \ar{r} &
	|[alias=FFOmega]|\Fin^{\wr 2} \wr A \ar{d}{\sigma^0} \ar{r} &
	\Fin^{\wr 2} \wr \mathcal{V}^{op} \ar{d}{\sigma^0} \ar{r} &
	\Fin \wr \mathcal{V}^{op} \ar{r} &
	|[alias=dog]|
	\mathcal{V}^{op} \ar[equal]{d}
\\
	|[alias=Omega]|\Omega_{G}^0 \wr A \ar{rr} &&
	\Fin \wr A \ar{r} &
	|[alias=cat]|
	\Fin \wr \mathcal{V}^{op} \ar{rr} &&
	\mathcal{V}^{op}
	\arrow[Leftrightarrow, from=FFOmega, to=Omega,shorten <=0.15cm,,shorten >=0.15cm,"\pi_0"]
	\arrow[Leftrightarrow, from=dog, to=cat,shorten <=0.15cm,,shorten >=0.15cm,"\alpha"]
	\end{tikzcd}
\end{equation}
That \eqref{UNITSPAN1 EQ} and \eqref{UNITSPAN2 EQ} coincide follows analogously by the unital condition for $\alpha$
and the face-degeneracy relations in 
Proposition \ref{PIIPROP PROP}(f).
\end{proof}

\renewcommand{\F}{\mathbb{F}}

\subsection{The genuine equivariant operad monad} \label{FGMON SEC}

Since 
$\mathsf{Wspan}^r(\Sigma_G,\mathcal{V}^{op}) \simeq 
\mathsf{Wspan}^l(\Sigma_G^{op},\mathcal{V})$,
Proposition \ref{MONSPAN PROP} and Remark \ref{RANLANADJ REM} give an adjuntion
\[\index{key functors!upsilon@$\upsilon$}
	\mathsf{Lan} \colon
	\mathsf{WSpan}^l(\Sigma^{op}_G, \mathcal{V})
		\rightleftarrows
	\mathsf{Fun}(\Sigma^{op}_G, \mathcal{V})
	\colon \upsilon
\]
together with a monad $N$ in the leftmost category $\mathsf{WSpan}^l(\Sigma^{op}_G, \mathcal{V})$.

We will now show that,
under reasonable conditions on $\mathcal{V}$,
this monad can be transferred by using 
Proposition \ref{MONADADJ PROP},
i.e. we will show that the natural transformations 
$
\mathsf{Lan} \circ N \Rightarrow \mathsf{Lan} \circ N \circ \upsilon \circ \mathsf{Lan}
$
and
$\mathsf{Lan} \circ \upsilon \Rightarrow id$
are isomorphisms.

This will require us to introduce a slight modification of the category of spans.
For motivation, note that iterations
$N^{\circ n+1} \circ \upsilon$ produce spans of the form
$\Sigma_G \leftarrow \Omega_G^{n} \to \mathcal{V}^{op}$
(where we use the identification $\Omega_G^{n} \wr \Sigma_G \simeq \Omega_G^{n}$). 
As noted in Remark \ref{ALLSPLITMAPS REM}, the maps 
$\Omega_G^{n} \to \Sigma_G$ are maps of split fibrations over $\mathsf{O}_G$, as are all other simplicial operators $d_i$, $s_j$.

\begin{definition}
The category $\mathsf{Wspan}^l_{\mathsf{r}}(\Sigma_G^{op},\mathcal{V})$ of \textit{rooted (left) spans}
has as objects spans
$\Sigma_G^{op} \leftarrow A^{op} \to \mathcal{V}$
together with a split Grothendieck fibration 
$\mathsf{r} \colon A \to \mathsf{O}_G$
such that $A \to \Sigma_G$ is a map of split fibrations.
Similarly, arrows are maps of spans inducing maps of split fibrations.
\end{definition}

We refer to split fibrations $A \to \mathsf{O}_G$
as \textit{root fibrations}
and to maps between them as \textit{root fibration maps}.

\begin{remark}
The condition that $A \to \mathsf{O}_G$
be a root fibration requires additional \textit{choices} of root pullbacks. Therefore, the forgetful functor 
$\mathsf{Wspan}^l_{\mathsf{r}}(\Sigma_G^{op},\mathcal{V})
\to
\mathsf{Wspan}^l(\Sigma_G^{op},\mathcal{V})$
is not quite injective on objects.
\end{remark}

The relevance of rooted spans is given by the following couple of lemmas.

\begin{lemma}\label{ROOTFIBPULL LEM}
If $A \to \Sigma_G$ is a root fibration map then so is 
$\Omega_G^0 \wr A \to \Omega_G^0$, naturally in $A$.
\end{lemma}

\begin{proof}
The hypothesis that $A \to \Sigma_G$ is a root fibration map
implies that the rightmost vertical map below 
is a map of split fibrations over
$\Fin_s \wr \mathsf{O}_G$.
\[
\begin{tikzcd}
	\Omega_{G}^0 \wr A \ar{r}{\boldsymbol{V}_G} \ar{d} &
	\Fin_s \wr A \ar{d}
\\
	\Omega_{G}^0 \ar{r}[swap]{\boldsymbol{V}_G} &
	\Fin_s \wr \Sigma_G
\end{tikzcd}
\]
Since, by Lemma \ref{VGPULL LEM}, the map $\boldsymbol{V}_G$ sends pullback  arrows in $\Omega_{G}^0$
(over $\mathsf{O}_G$) to pullback arrows in $\Fin_s \wr \Sigma_G$ (over $\Fin_s \wr \mathsf{O}_G$), 
the root pullback arrows in 
$\Omega_G^0 \wr A$ can be defined as compatible pairs of pullback arrows in $\Omega^0_G$ and $\Fin_s \wr A$,
and the result follows.
\end{proof}

\begin{remark}\label{PULLEXP REM}
Explicitly, if $\psi \colon Y \to X$ is a map in $\mathsf{O}_G$,
and $\tilde{T} =(T,(A_{v_{Ge}})_{V_G(T)}) \in \Omega_G^0 \wr A$ lies over $X$,
the pullback $\psi^{\**} \tilde{T}$ is given by
\[
\left(\psi^{\**}T,(\bar{\psi}^{\**}_{Ge}
 A_{v_{Ge}})_{V_G(\psi^{\**}T)}\right)
\]
where $\bar{\psi}$ is the map 
$\bar{\psi} \colon \psi^{\**}T \to T$
and $\bar{\psi}_{G e}$ is the restriction
$\bar{\psi} \colon G e \to G \bar{\psi}(e)$, 
cf. Remark \ref{PULLCOMP REM}.
\end{remark}

\begin{lemma}\label{LANPULLCOMA LEM}
	Suppose that $\mathcal{V}$ is complete and that $\rho \colon A \to \Sigma_G$ is a root fibration map. If the rightmost triangle in 
\[
\begin{tikzcd}
	\Omega_{G}^{0} \wr A \ar{r}{\boldsymbol{V}_G} 
	\ar{d} & 
	\Fin_s \wr A  
	\ar{d}  \ar{r}[swap,name=F]{}&
	\mathcal{V}^{op}
\\
	\Omega_{G}^{0} \ar{r}[swap]{\boldsymbol{V}_G} & 
	|[alias=FEG]|\Fin_s \wr \Sigma_G \ar{ru}
\arrow[Rightarrow, from=FEG, to=F,shorten <=0.15cm]
\end{tikzcd}
\]
is a right Kan extension diagram then so is the composite diagram.
\end{lemma}

\begin{proof}
	Unpacking definitions using the pointwise formula for right  Kan extensions (cf. \cite[X.3 Thm. 1]{McL} or \eqref{FIBERKAN EQ}), 
	it suffices to check that for each $T \in \Omega_{G}^{0}$ the induced functor
\[
\begin{tikzcd}
	T \downarrow \Omega_{G}^{0} \wr A \ar{r}{\boldsymbol{V}_G} & 
	\boldsymbol{V}_G(T) \downarrow \Fin_s \wr A
\end{tikzcd}
\]
is initial.
We will slightly abuse notation by writing 
$(T \to S, (A_{v_{G f}})_{V_G(S)})$ for the objects of 
$T \downarrow \Omega_{G}^{0} \wr A$,
as well as 
$
\left(
	(T_{v_{G e}} \to S_{\phi(v_{Ge})})_{v_{G e} \in V_G(T)},
	(A_v)_{v \in V}
\right)
$
for the objects of 
$\boldsymbol{V}_G(T) \downarrow \Fin_s \wr A$,
with the map $\phi \colon V_G(T) \to V$ and the condition 
$\rho(A_v) = S_v$ left implicit.

By Proposition \ref{FIBERKANMAP PROP}, $T \downarrow \Omega_{G}^{0} \wr A$ has an initial subcategory
$T \downarrow_{\mathsf{r}} \Omega_{G}^{0} \wr A$
of those objects such that $T \to S$ is the identity on roots. 
Similarly, again by Proposition \ref{FIBERKANMAP PROP},
$\boldsymbol{V}_G(T) \downarrow \Fin_s \wr A$
has an initial subcategory
\begin{equation}\label{INITCAT EQ} 
	\left(
	(T_{v_{Ge}})_{V_G(T)} 
	\downarrow_{\Fin_s \wr \mathsf{O}_G}
	\Fin_s \wr A
	\right)
\simeq
	\left(
	\prod_{v_{Ge} \in V_G(T)} 
	T_{v_{Ge}} \downarrow_{\mathsf{r}} A
	\right)
\end{equation}
of those objects inducing an identity on $\Fin_s \wr \mathsf{O}_G$. Moreover, 
\eqref{INITCAT EQ} comes together with a right retraction $r$,
i.e. a right adjoint to the inclusion $i$ into $\boldsymbol{V}_G(T) \downarrow \Fin_s \wr A$, 
which is built using pullbacks. 
Explicitly, unpacking the proof of Proposition \ref{FIBERKANMAP PROP} one has that $r$ is given by the assignment
\begin{equation}\label{RETDES EQ}
\left(
	(T_{v_{Ge}})_{V_G(T)}
	\xrightarrow{\tau}
	(S_x)_{X},
	(A_x)_{X}
\right)
\mapsto
\left(
	\left(
	T_{v_{G e}} \to 
	(\mathsf{r} \tau_{v_{Ge}})^{\**}
	S_{\tau(v_{Ge})}
	\right),
	\left(
	(\mathsf{r} \tau_{v_{Ge}})^{\**}
	A_{\tau(v_{Ge})}
	\right)
\right)
\end{equation}
where we recall that the leftmost $\tau$
is described by a map of sets
$\tau \colon V_G(T) \to X$
and maps
$\tau_{v_{Ge}} \colon T_{v_{Ge}} \to S_{\tau(v_{Ge})}$ in $\Sigma_G$,
and that $\mathsf{r} \colon \Sigma_{G} \to \mathsf{O}_G$
is the root functor.

We now compute the following composite
(where we abbreviate expressions $T_{v_{G e}}$ as 
$T_{Ge}$ and implicitly assume that tuples with index $G e$ (resp. $G f$) run over $V_G(T)$ (resp. $V_G(S)$)).
\[
\begin{tikzcd}[column sep =1.5em , row sep =0em]
	T \downarrow_{\mathsf{r}} \Omega_{G}^{0} \wr A 
	\ar{r}{\boldsymbol{V}_G} &
	\boldsymbol{V}_G(T) \downarrow \Fin_s \wr A \ar{r}{r} &
	\underset{v_{Ge} \in V_G(T)}{\prod} 
	T_{v_{Ge}} \downarrow_{\mathsf{r}} A
\\
	(T \xrightarrow{\psi} S, (A_{{G f}})) \ar[mapsto]{r} &
	\left(
		(T_{{G e}} \to S_{{G \psi(e)}}),
		(A_{{G f}})
	\right) \ar[mapsto]{r} &
	\left(
		(T_{{G e}} \to \psi_{G e}^{\**} S_{{G \psi(e)}}),
		(\psi_{G e}^{\**} A_{{G \psi(e)}})
	\right)
\end{tikzcd}
\]
Note that we wrote the map $\mathsf{r} \psi_{v_{Ge}}$ in $\mathsf{O}_G$
from \eqref{RETDES EQ}
as 
$\psi_{Ge} \colon Ge \to G\psi(e)$,
following the notation
in Remarks \ref{PULLCOMP REM} and \ref{PULLEXP REM}.
Since rooted quotients are isomorphisms, the $\psi$
and $\psi_{Ge}$ appearing above are isomorphisms, 
and hence the natural transformation
$i \circ r \circ \boldsymbol{V}_G \Rightarrow \boldsymbol{V}_G$
is a natural isomorphism. 
Therefore,
$\boldsymbol{V}_G$ will be initial provided that so is
$i \circ r \circ \boldsymbol{V}_G$,
and since the inclusion $i$ is initial, it suffices to show that
$r \circ \boldsymbol{V}_G$ is an isomorphism.

But now note that an arbitrary choice of rooted isomorphisms
$T_{v_{G_e}} \to S_{v_{G_e}}$
uniquely determines a compatible planar structure on $T$, and thus a unique isomorphism $\psi \colon T \to S$.
Therefore, arbitrary choices of 
$\psi_{G e}^{\**} S_{{G \psi(e)}}$,
$\psi_{G e}^{\**} A_{{G \psi(e)}}$
uniquely determine $S$, $A_{G f}$, finishing the proof.
\end{proof}

Lemma \ref{ROOTFIBPULL LEM} implies that copying Definition \ref{WSPAN_MONAD_DEFINITION} yields a monad $N_{\mathsf{r}}$
on
$\mathsf{Wspan}^l_{\mathsf{r}}(\Sigma_G^{op},\mathcal{V})$
lifting the monad $N$.

\begin{corollary}\label{MONDEFCOR COR}
Suppose that finite products in $\mathcal{V}$ commute with colimits in each variable or, more generally, that 
$\mathcal{V}$ is a symmetric monoidal category with diagonals such that $\otimes$ preserves colimits in each variable.
Then the natural transformations
\[
	\mathsf{Lan} \circ N_{\mathsf{r}} \Rightarrow
	\mathsf{Lan} \circ N_{\mathsf{r}} \circ \upsilon \circ \mathsf{Lan},
\qquad
	\mathsf{Lan} \circ \upsilon \Rightarrow id
\]
are natural isomorphisms.
\end{corollary}

\begin{proof}
This follows by combining Lemma \ref{LANPULLCOMA LEM} with Lemma \ref{FINWREATPRODLIM LEM}.
\end{proof}

Recalling Proposition \ref{MONADADJ PROP} now leads to the following.
\begin{definition}\label{THEMONAD DEF}
	\index{categories!of operads/symmetric sequences!OpG@$\Op_G(\V)$}
	\index{monads!genopmonad@$\mathbb{F}_G$}
The \textit{genuine equivariant operad monad} is the monad
$\F_G$ on $\mathsf{Sym}_G(\mathcal{V})=\mathsf{Fun}(\Sigma_G^{op}, \mathcal{V})$
given by
\[
	\F_G = \mathsf{Lan} \circ N_{\mathsf{r}} \circ \upsilon
\]
and with multiplication and unit given by the composites
\[
\mathsf{Lan} \circ N_{\mathsf{r}} \circ \upsilon \circ
\mathsf{Lan} \circ N_{\mathsf{r}} \circ \upsilon
\overset{\simeq}{\Leftarrow}
\mathsf{Lan} \circ N_{\mathsf{r}} \circ  N_{\mathsf{r}} \circ \upsilon
\Rightarrow
\mathsf{Lan} \circ N_{\mathsf{r}} \circ \upsilon
\]
\[
id \overset{\simeq}{\Leftarrow} \mathsf{Lan} \circ \upsilon
\Rightarrow
\mathsf{Lan} \circ N_{\mathsf{r}} \circ \upsilon.
\]
We will write $\Op_G(\V)$ for the category 
$\mathsf{Alg}_{\mathbb{F}_G}(\mathsf{Sym}_G(\mathcal{V}))$ of \textit{genuine equivariant operads}.
\end{definition}

\begin{remark}
	The functor $\mathsf{Lan} \circ N_{\mathsf{r}} \circ \upsilon$ is isomorphic to 
	$\mathsf{Lan} \circ N \circ \upsilon$, and this isomorphism is compatible with the multiplication and unit	in Definition \ref{THEMONAD DEF}, and as such we will henceforth simply write $N$ rather than $N_{\mathsf{r}}$.
	
	From this point of view, root fibrations play an auxiliary role in verifying that $\mathsf{Lan} \circ N \circ \upsilon$ is indeed a monad, but are unnecessary to describe the monad structure itself.
\end{remark}

\begin{remark}\label{REPACKAGERES REM}
Since a map 
\[\F_G X =\mathsf{Lan} \circ N \circ \upsilon X \to X\]
is adjoint to a map
\[N \circ \upsilon X \to \upsilon X \]
one easily verifies that 
$X$ is a genuine equivariant operad, i.e. 
an $\F_G$-algebra, iff 
$\upsilon X$ is an $N$-algebra
(cf. Proposition \ref{MONADADJ PROP}(ii)).

Moreover, the bar resolution
$ \F_G^{\circ n +1} X $
is isomorphic to
$
	\mathsf{Lan} \left( N^{\circ n +1} \upsilon X \right)
$.
\end{remark}

\subsection{Comparison with (regular) equivariant operads}
\label{COMPARISON_REGULAR_SECTION}

In the case $G = \**$, genuine operads coincide with the usual notion of symmetric operads, i.e. 
$\mathsf{Sym}_{\**}(\mathcal{V})
\simeq \mathsf{Sym}(\mathcal{V})$ 
and
$\mathsf{Op}_{\**}(\mathcal{V})
\simeq \mathsf{Op}(\mathcal{V})$, 
and in what follows we will adopt the notations
$\mathsf{Sym}^G(\mathcal{V})$ and
$\mathsf{Op}^G(\mathcal{V})$ 
for the corresponding categories of $G$-objects.
Our goal in this section will be to relate these to the categories
$\mathsf{Sym}_G(\mathcal{V})$ and $\mathsf{Op}_G(\mathcal{V})$
of genuine equivariant sequences and genuine equivariant operads.

\begin{remark}\label{TOOPORNOT REM}
	Unpacking notation,
	$\mathsf{Sym}^G(\mathcal{V})$ is defined to be
	$\mathcal{V}^{G \times \Sigma^{op}}$
	rather than 
	$\mathcal{V}^{G \times \Sigma}$,
	though it is of course  
	$\mathcal{V}^{G \times \Sigma^{op}}
	\simeq \mathcal{V}^{G \times \Sigma}$
	via the inversion isomorphism
	$\Sigma^{op} \simeq \Sigma$.
	However, in light of \eqref{IOTADEF EQ} below, 
	it is in practice preferable to work with 
	$\mathcal{V}^{G \times \Sigma^{op}}$
	when dealing with $\mathsf{Sym}^G(\mathcal{V})$ and
	$\mathsf{Sym}_G(\mathcal{V})$ simultaneously.
	Note that, translating \eqref{GRAPHSUBIN EQ}, 
	graph subgroups of
	$G \times \Sigma_n^{op}$
	have the form
	$\Gamma = \left\{(h,\phi(h)^{-1})|h\in H\right\}$
	for $H\leq G$
	and some homomorphism $\phi \colon H \to \Sigma_n$. 
\end{remark}

Throughout this section we fix a total order of $G$ such that the identity $e$ is the first element, though we note that the exact order is unimportant, as any other such choice would lead to unique isomorphisms between the constructions described herein.

We now have an inclusion functor
\begin{equation}\label{IOTADEF EQ}
\begin{tikzcd}[row sep =0]
	\iota \colon G^{op} \times \Sigma \ar[hookrightarrow]{r} &
	\Sigma_G
\\
	C \ar[mapsto]{r} & G \cdot C
\end{tikzcd}
\end{equation}
where $G \cdot C$ is the constant tuple $(C)_{g \in G}$,
which we think of as $|G|$ copies of $C$, planarized according to $C$ and the order on $G$.
Moreover, letting $\Sigma_G^{\text{fr}} \hookrightarrow \Sigma_G$ denote the full subcategory of $G$-free corollas, there is an induced retraction 
$\rho \colon \Sigma_{G}^{\text{fr}} \to G^{op} \times \Sigma$
defined by 
$\rho\left( (C_i)_{1\leq i \leq |G|} \right) = G \cdot C_1$
together with isomorphisms 
$C \simeq \rho(C)$
uniquely determined by the condition that they
are the identity on the first tree component $C_1$.

We now consider the associated adjunctions.
\begin{equation}\label{TWOADJOINTS EQ}
\begin{tikzcd}[column sep =5em]
\index{key functors!iotau@$\iota^{\**}$}
\index{key functors!iotal@$\iota_{\**}$}
\index{key functors!iotar@$\iota_{"!}$}
	\mathsf{Sym}_G(\mathcal{V}) \ar{r}[swap]{\iota^{\**}} 
	&
	\mathsf{Sym}^G(\mathcal{V})
	\ar[bend right]{l}[swap,midway]{\iota_{!}}
	\ar[bend left]{l}{\iota_{\**}}
\end{tikzcd}
\end{equation}
Explicitly, we have the formulas (where we write $G$-corollas as $(C_i)_{I}$ for $I \in \mathsf{O}_G$)
\begin{equation}\label{IOTAFUNS EQ}
	\iota_!Y\left( (C_i)_I \right)=
	\begin{cases}
	Y(C_1), & (C_i)_I \in \Sigma_G^{\text{fr}} \\
	\emptyset, & (C_i)_I \nin \Sigma_G^{\text{fr}}
	\end{cases},
\quad
	\iota^{\**}X (C) = X(G \cdot C),
\quad
	\iota_{\**}Y ((C_i)_I)=
	\left(\prod_{I} Y(C_i)\right)^G,
\end{equation}
where in the formula for $\iota_{\**}$
the action of $G$ interchanges factors according to the action on the indexing set $I$.
More precisely,
the action of $g \in G$
is the product of the composites
$Y(C_{g^{-1} i }) \to Y(C_{i}) \xrightarrow{g} Y(C_i)$
where the first map is the given by functoriality of $Y$ 
on the isomorphism $C_i \to C_{g^{-1} i}$
(which is part of the structure of $C \in \Sigma_G$)
and the second map is given by the $G$-action in 
$Y$.

As a side note, the formulas for $\iota_!$ and $\iota_{\**}$ are independent of the chosen order of $G$.

\begin{remark}\label{IOTAFUNSALT REM}
	The formula for 
	$\iota_{\**}$ in \eqref{IOTAFUNS EQ}
	emphasizes functoriality on
	$C = (C_i)_{I}$,
	but in practice we will find it more convenient to use 
	alternative formulas.
	
	To obtain these formulas,
	write $1 \in I$ for the first element and 
	$H \leq G$ for its isotropy.
	Note that the $G$-action described after \eqref{IOTAFUNS EQ} defines an $H$-action on $Y(C_1)$.
	Moreover, viewing 
	$C_1 \in \Sigma$ as an integer arity $n\geq 0$,
	so that $Y(C_1) = Y(n)$
	comes with a natural
	$G \times \Sigma_n^{op}$-action,
	the $H$-action on $Y(C_1)$
	is identified with the action of the graph subgroup
	$\Gamma = \{(h,\phi(h)^{-1}) | h \in H \}$
	of $G \times \Sigma_n^{op}$
	associated to the homomorphism
	$\phi \colon H \to \Sigma_n$
	encoding the action of $H$ on $C_1$.
	We then have the formulas 
\begin{equation}\label{IOTAFUNSALT EQ}
	\iota_{\**}Y ((C_i)_I)
	=
	\left(\prod_{I} Y(C_i)\right)^G
	\simeq
	Y(C_1)^H
	\simeq
	Y(n)^{\Gamma}
\end{equation}
where the second identification follows by unpacking universal properties to show that a map
$A \to \left(\prod_{I} Y(C_i)\right)^G$
is equivalent to the induced map
$A \to Y(C_1)^H$
onto the first factor.
\end{remark}

\begin{remark}\label{REFLCOREFL REM}
	$\iota_!$ essentially identifies 
	$\mathsf{Sym}^G(\mathcal{V})$ as the coreflexive subcategory of sequences 
	$X \in \mathsf{Sym}_G(\mathcal{V})$ such that $X(C)=\emptyset$ whenever $C$ is not a free corolla.

On the other hand, $\iota_{\**}$ identifies 
$\mathsf{Sym}^G(\mathcal{V})$ with the more interesting reflexive subcategory of those sequences 
$X \in \mathsf{Sym}_G(\mathcal{V})$ 
such that $X(C)$ for each $C = (C_i)_I$ not a free corolla must satisfy a fixed point condition. 
Explicitly, letting $G \cdot C_1 \to C$
be the quotient map determined by the inclusion
$C_1 \to C$, one has
\[
	X(C) \xrightarrow{\simeq}
	X(G \cdot C_1)^{\Gamma}
\]
for $\Gamma \leq \mathsf{Aut}(G \cdot C_1) \simeq 
G^{op} \times \mathsf{Aut}(C_1)$
the subgroup preserving the quotient map
$G \cdot C_1 \to C$
under precomposition 
(note that $(G \cdot C_1) \in \Sigma_G^{\text{fr}}$).
\end{remark}

There is an obvious natural transformation $\beta \colon \iota_! \Rightarrow \iota_{\**}$ which,
for $(C_i)_I \in \Sigma_G^{\text{fr}}$,
sends $Y(C_1)$ to the ``$G$-twisted diagonal'' of 
$\prod_I Y(C_i)$.
Moreover, letting $\eta_!,\epsilon_!$ 
(resp. $\eta_{\**},\epsilon_{\**}$)
denote the unit and counit of the $(\iota_!,\iota^{\**})$ adjunction 
(resp. $(\iota^{\**},\iota_{\**})$ adjunction)
it is straightforward to check that the following diagram commutes.
\begin{equation}\label{BETADEFSQUARE EQ}
\begin{tikzcd}
		\iota_{!} \iota^{\**} \iota_{\**}
		\ar[Rightarrow]{r}{\epsilon_!}
		\ar[Rightarrow]{d}{\simeq}[swap]{\epsilon_{\**}}
	&
		\iota_{\**}
		\ar[Rightarrow]{d}{\eta_{!}}[swap]{\simeq}
\\
		\iota_!
		\ar[Rightarrow]{r}[swap]{\eta_{\**}}
		\ar[Rightarrow]{ru}[swap]{\beta}
	&
		\iota_{\**} \iota^{\**} \iota_{!}
\end{tikzcd}
\end{equation}
\begin{remark} An exercise in adjunctions shows the outer square in \eqref{BETADEFSQUARE EQ}
 commutes provided at least one of the adjunctions in \eqref{TWOADJOINTS EQ} is (co)reflexive, so that \eqref{BETADEFSQUARE EQ} can be regarded as an alternative definition of $\beta$.
\end{remark}


\begin{proposition}
        \label{MONAD_COMPARISON_PROP}
	One has the following:
\begin{itemize}
	\item[(i)]
	the map 
	$\iota^{\**} \mathbb{F}_G
		\xrightarrow{\eta_{\**}}
	\iota^{\**} \mathbb{F}_G \iota_{\**} \iota^{\**}$
	is an isomorphism, 
	and thus (cf. Prop. \ref{MONADADJ PROP})
	$\iota^{\**} \mathbb{F}_G \iota_{\**}$
	is a monad;
	\item[(ii)] the map 
	$\iota^{\**} \mathbb{F}_G \iota_{!}
	\xrightarrow{\beta}	
	\iota^{\**} \mathbb{F}_G \iota_{\**}$ is an isomorphism of monads;
	\item[(iii)] the map 
	$\iota_{!}\iota^{\**} \mathbb{F}_G \iota_{!}
	\xrightarrow{\epsilon_!}
	\mathbb{F}_G \iota_{!}$ is an isomorphism;
	\item[(iv)] there is a natural isomorphism of monads
	$\alpha \colon \mathbb{F} \to \iota^{\**} \mathbb{F}_G \iota_{!}$.
\end{itemize}
\end{proposition}

\begin{proof}
We first show (i), starting with some notation. 
In analogy with $\Sigma_{G}^{\text{fr}}$,
we write $\Omega_{G}^{0,\text{fr}}$ for the subcategory of free trees,
and note that the leaf-root and vertex functors then restrict to functors
$\mathsf{lr} \colon \Omega_{G}^{0,\text{fr}} \to \Sigma_G^{\text{fr}}$,
$\boldsymbol{V}_G \colon \Omega_{G}^{0,\text{fr}} \to \Fin_s \wr \Sigma_G^{\text{fr}}$.
Moreover, for each $C \in \Sigma_G^{\text{fr}}$ one has an equality of rooted undercategories between
$C \downarrow_{\mathsf{r}} \Omega_{G}^0$
and
$C \downarrow_{\mathsf{r}} \Omega_{G}^{0,\text{fr}}$,
and thus 
$\iota^{\**} \mathbb{F}_G X$ is computed by the Kan extension of the following diagram.
\[
\begin{tikzcd}
	\Omega_{G}^{0,\text{fr}} \ar{d} \ar{r} &
	\Fin_s \wr \Sigma_G^{\text{fr}} \ar{r}{\Fin_s \wr X} &
	\Fin_s \wr \mathcal{V}^{op} \ar{r} &
	\mathcal{V}^{op}
\\
	\Sigma_G^{\text{fr}}
\end{tikzcd}
\]
(i) now follows by noting that 
$X \to \iota_{\**} \iota^{\**} X$
is an isomorphism when restricted to $\Sigma_G^{\text{fr}}$.

For (ii), to show that 
	$\iota^{\**} \mathbb{F}_G \iota_{!}
	\to	
	\iota^{\**} \mathbb{F}_G \iota_{\**}$ 
is an isomorphism of functors one just repeats the argument in the previous paragraph by noting that $\iota_! \to \iota_{\**}$ is an isomorphism when restricted to $\Sigma_G^{\text{fr}}$.
	To check that this is a map of monads, we first recall  that the monad structure on $\iota^{\**} \mathbb{F}_G \iota_{\**}$
is given as described in Proposition \ref{MONADADJ PROP}.
Unpacking definitions, compatibility with multiplication reduces to showing that the composite 
$\iota_! \iota^{\**} \xrightarrow{\epsilon_!} 
id \xrightarrow{\eta_{\**}} \iota_{\**} \iota^{\**}$
coincides with $\beta \iota^{\**}$
while compatibility with units 
reduces to showing that the composite
$
	id \xrightarrow{\eta_!} 
	\iota^{\**} \iota_! \xrightarrow{\iota^{\**} \beta}
	\iota^{\**} \iota_{\**} \xrightarrow{\epsilon_{\**}}
	id
$
is the identity. Both of these are a consequence of 
\eqref{BETADEFSQUARE EQ}, following from the diagrams below 
(where the top composites are identities).
\[
\begin{tikzcd}[column sep =3em]
		\iota_! \iota^{\**}
		\ar[Rightarrow]{r}{\iota_! \iota^{\**} \eta_{\**}}
		\ar[Rightarrow]{d}[swap]{\epsilon_!}
	&
		\iota_{!} \iota^{\**} \iota_{\**} \iota^{\**}
		\ar[Rightarrow]{d}[swap]{ \epsilon_! \iota_{\**}\iota^{\**}}
		\ar[Rightarrow]{r}[swap]{\simeq}{\iota_{!}\epsilon_{\**}\iota^{\**}}
	&
		\iota_! \iota^{\**}
		\ar[Rightarrow]{ld}{\beta \iota^{\**}}
	&	
		\iota^{\**} \iota_{\**}
		\ar[Rightarrow]{r}{\eta_! \iota^{\**} \iota_{\**}}[swap]{\simeq}
		\ar[Rightarrow]{d}[swap]{\epsilon_{\**}}{\simeq}
	&
		\iota^{\**} \iota_{!} \iota^{\**} \iota_{\**}
		\ar[Rightarrow]{r}{\iota^{\**} \epsilon_! \iota_{\**}}
		\ar[Rightarrow]{d}{\simeq}[swap]{\iota^{\**}\iota_{!}\epsilon_{\**}}
	&
		\iota^{\**}\iota_{\**}
\\
		id
		\ar[Rightarrow]{r}[swap]{\eta_{\**}}
	&
		\iota_{\**} \iota^{\**}
	&
	&	
		id
		\ar[Rightarrow]{r}[swap]{\eta_!}{\simeq}
	&
		\iota^{\**} \iota_!
		\ar[Rightarrow]{ru}[swap]{\iota^{\**} \beta}
\end{tikzcd}
\]
Part (iii) amounts to showing that, if $X(C) =\emptyset$ whenever 
$C \nin \Sigma_G^{\text{fr}}$,
then we must also have that
$\mathbb{F}_G X(C) =\emptyset$
for $C \nin \Sigma_G^{\text{fr}}$.
Indeed, since for  
$C \nin \Sigma_G^{\text{fr}}$
the undercategory
$C \downarrow \Omega_{G}^{0}$
consists of trees with at least one non-free vertex (namely the root vertex), the composite
\[
\begin{tikzcd}
	C \downarrow \Omega_{G}^{0} \ar{r}{\boldsymbol{V}_G} &
	\Fin_s \wr \Sigma_G \ar{r}{\Fin_s \wr X} &
	\Fin_s \wr \mathcal{V}^{op} \ar{r}{\otimes}&
	\mathcal{V}^{op}
\end{tikzcd}
\]
is constant equal to $\emptyset$, and (iii) follows.

Finally, we show (iv). We will slightly abuse notation by writing 
$G^{op} \times \Sigma \hookrightarrow \Sigma_G$ for the image of $\iota$
and similarly
$G^{op} \times \Omega^0 \hookrightarrow \Omega_{G}^{0}$ for the image of the obvious analogous functor
$\iota \colon G^{op} \times \Omega^0 \to \Omega_{G}^{0}$.
The map 
$\alpha \colon \mathbb{F} \to \iota^{\**} \mathbb{F}_G \iota_{!}$
is the adjoint to the map 
$\tilde \alpha: \mathbb{F} \iota^{\**} \to \iota^{\**} \mathbb{F}_G$ encoded on spans by the following diagram.
\begin{equation}\label{MONADEQUIV DEF}
\begin{tikzcd}[row sep=5pt, column sep =8pt]
	G^{op} \times \Omega^0	\ar{dd} \ar{rd} \ar{rr} &&
	\Fin_s \wr (G^{op} \times \Sigma) \ar{rd}  \ar{rr}{\iota^{\**} X}&&
	\Fin_s \wr \mathcal{V}^{op} \ar{r} \ar[equal]{rd}&
	\mathcal{V}^{op} \ar[equal]{rd}
\\
	& \Omega_{G}^{0} \ar{dd} \ar{rr} &&
	\Fin_s \wr \Sigma_G  \ar{rr}[swap]{X} &&
	\Fin_s \wr \mathcal{V}^{op} \ar{r} &
	\mathcal{V}^{op}
\\
	G^{op} \times \Sigma \ar{rd} 
\\
	& \Sigma_G
\end{tikzcd}
\end{equation}
That $\alpha$ is a natural isomorphism
follows by the previous identifications 
$C \downarrow_{\mathsf{r}} \Omega_{G}^{0} \simeq
C \downarrow_{\mathsf{r}} \Omega_{G}^{0,\text{fr}}$
for $C \in G^{op} \times \Sigma$,
together with the fact that the retraction 
$\rho \colon \Omega_{G}^{0,\text{fr}} \to G^{op} \times \Omega^0$
(built just as the retraction
$\rho \colon \Sigma_G^{\text{fr}} \to G^{op} \times \Sigma$)
retracts 
$C \downarrow_{\mathsf{r}} \Omega_{G}^{0,\text{fr}}$
to the undercategory
$C \downarrow_{\mathsf{r}} G^{op} \times \Omega^0$, which is thus initial (as well as final).

Intuitively, the final claim that 
$\alpha$ is a map of monads 
follows from the fact that the composite 
$
\mathbb{F} \mathbb{F}
	\to 
\iota^{\**} \mathbb{F}_G \iota_{!} \iota^{\**} \mathbb{F}_G \iota_{!}
	\to
\iota^{\**} \mathbb{F}_G \mathbb{F}_G \iota_{!}
$
is encoded by the analogous natural transformation of diagrams for strings $G^{op} \times \Omega^1 \hookrightarrow \Omega_{G}^{1,\text{fr}}$.
However, since the presence of left Kan extensions in the 
definitions of $\mathbb{F}$, $\mathbb{F}_G$
can make a rigorous direct proof of this last claim fairly cumbersome, we sketch here a workaround argument.
We first consider the adjunction
$
	\iota_{!} \colon
	\mathsf{WSpan}^l(G \times \Sigma^{op},\mathcal{V})
		\rightleftarrows
	\mathsf{WSpan}^l(\Sigma_G^{op},\mathcal{V})
	\colon \iota^{\**}
$,
where $\iota_!$ is composition with $\iota$ and $\iota^{\**}$ is the pullback of spans. 
Writing $N$, $N_G$ for the monads 
on the span categories, mimicking (\ref{MONADEQUIV DEF}) yields
a map 
$\tilde{\alpha} \colon N \to \iota^{\**} N_G \iota_{!}$
encoded by the diagram (where the front and back squares are pullbacks).
\[
\begin{tikzcd}[row sep=3pt, column sep =8pt]
	(G^{op} \times \Omega^0) \wr \iota^{\**} A	\ar{dd} \ar{rd} \ar{rr} &&
	\Fin_s \wr \iota^{\**} A \ar{rd} \ar{dd} \ar{rr}&&
	\Fin_s \wr \mathcal{V}^{op} \ar{r} \ar[equal]{rd} &
	\mathcal{V}^{op} \ar[equal]{rd}
\\
	& 
	|[alias=O1]|
	\Omega_{G}^{0} \wr A \ar[crossing over]{rr} &&
	\Fin_s \wr A \ar{dd} \ar{rr} &&
	\Fin_s \wr \mathcal{V}^{op} \ar{r} &
	\mathcal{V}^{op}
\\
	G^{op} \times \Omega^0	\ar{dd} \ar{rd} \ar{rr}&&
	\Fin_s \wr (G^{op} \times \Sigma) \ar{rd}
\\
	&
	|[alias=O2]|
	\Omega_{G}^{0} \ar{dd} \ar{rr} &&
	\Fin_s \wr \Sigma_G
\\
	G^{op} \times \Sigma \ar{rd}
\\
	& \Sigma_{G}
\arrow[from=O1, to=O2,crossing over]
\end{tikzcd}
\]
The claim that $\tilde{\alpha}$ is a map of monads is then straightforward. Writing
\[
	\mathsf{Lan} \colon
	\mathsf{WSpan}^l(G \times \Sigma^{op},\mathcal{V})
		\rightleftarrows
	\mathsf{Fun}(G \times \Sigma^{op},\mathcal{V})
	\colon j
\quad
	\mathsf{Lan}_G \colon
	\mathsf{WSpan}^l(\Sigma_G^{op},\mathcal{V})
		\rightleftarrows
	\mathsf{Fun}(\Sigma_G^{op},\mathcal{V})
	\colon j_G
\]
for the span-functor adjunctions,  
$\alpha \colon \mathbb{F} \to \iota^{\**} \mathbb{F}_G \iota_{!}$ can then be written as the composite
\[
	\mathsf{Lan} N j \to 
	\mathsf{Lan} \iota^{\**} N_G \iota_! j  \to
	\iota^{\**} \mathsf{Lan}_G  N_G  j_G \iota_!
\]
where the first map is the isomorphism of monads induced by $\tilde{\alpha}$ and the second map can be shown directly to be a monad map by unpacking the monad structures in 
Propositions \ref{MONADADJ1 PROP} and \ref{MONADADJ PROP}.
\end{proof}


\begin{corollary}\label{TWOADJOINTSOP_COR}
\index{categories!of operads/symmetric sequences!OpF@$\mathsf{Op}_{\mathcal{F}}(\mathcal{V})$}
The adjunctions \eqref{TWOADJOINTS EQ} lift to adjunctions
\[
\begin{tikzcd}[column sep =5em]
	\mathsf{Op}_G(\mathcal{V}) \ar{r}[swap]{\iota^{\**}} 
	&
	\mathsf{Op}^G(\mathcal{V})
	\ar[bend right]{l}[swap,midway]{\iota_{!}}
	\ar[bend left]{l}{\iota_{\**}}.
\end{tikzcd}
\]
In particular, $\iota_{\**}$
identifies $\mathsf{Op}^G(\mathcal{V})$ as a reflexive subcategory of 
$\mathsf{Op}_G(\mathcal{V})$.
\end{corollary}

\begin{proof}
	For the top (resp. bottom)
	adjunction
	this follows from Proposition \ref{MONADADJ1 PROP}
	(resp. Proposition \ref{MONADADJ PROP}),
	the isomorphism of functors
	$\iota_! \iota^{\**} \mathbb{F}_G \iota_! \simeq
	\mathbb{F}_G \iota_!$ 
	(resp. $\iota^{\**} \mathbb{F}_G \simeq
	\iota^{\**} \mathbb{F}_G \iota_{\**} \iota^{\**}$)
	and the isomorphism of monads
	$\mathbb{F} \simeq \iota^{\**} \mathbb{F}_G \iota_!$
	(resp. 
	$\mathbb{F} \simeq 
	\iota^{\**} \mathbb{F}_G \iota_! \simeq
	\iota^{\**} \mathbb{F}_G \iota_{\**}$),
	cf. Proposition \ref{MONAD_COMPARISON_PROP}.
\end{proof}

\begin{remark}\label{MUTMUT REM}
	Remark \ref{REFLCOREFL REM} extends to operads mutatis mutandis.
\end{remark}

\begin{remark}\label{MUTIOTAUP REM}
	Parts (iv),(ii),(i) of 
	Proposition \ref{MONAD_COMPARISON_PROP}
	yield a string of isomorphisms
\[
	\mathbb{F} \iota^{\**}
\simeq
	\iota^{\**} \mathbb{F}_G \iota_! \iota^{\**}
\simeq 
	\iota^{\**} \mathbb{F}_G \iota_{\**} \iota^{\**}
\simeq
	\iota^{\**} \mathbb{F}_G.
\]
The identification $\mathbb{F} \iota^{\**} \simeq \iota^{\**} \mathbb{F}_G$
reflects the fact that,
for $X \in \mathsf{Op}_G(\mathcal{V})$,
the genuine operad structure maps restricted to the levels
$X(G \cdot C), C \in \Sigma$
correspond precisely to the structure of
an equivariant operad.
Recalling the adjunction
$\mathcal{V}^{\mathsf{O}_G^{op}} 
\rightleftarrows
\mathcal{V}^G$ (cf. \eqref{COFADJINT EQ}),
this is analogous to the fact that,
for $X \in \mathcal{V}^{\mathsf{O}_G^{op}}$,
the presheaf structure restricts to make $X(G)$ into a $G$-set.
\end{remark}

\begin{remark}\label{MUTMUT1 REM}
	The isomorphism
	$\iota_{!}\iota^{\**} \mathbb{F}_G \iota_{!}
	\xrightarrow{\epsilon_!}
	\mathbb{F}_G \iota_{!}$
	shows that $\mathbb{F}_G$ essentially preserves the image of $\iota_!$.
	Moreover, since the identification of monads 
	$\mathbb{F} \simeq \iota^{\**} \mathbb{F}_G \iota_{!}$ 
	then yields 
	$\iota_! \mathbb{F} \simeq \mathbb{F}_G \iota!$,
	one can identify the restriction of 
	$\mathbb{F}_G$ to the essential image of $\iota_!$
	with the monad $\mathbb{F}$.
\end{remark}

\begin{remark}\label{MUTMUT2 REM}
The analogue of Remark \ref{MUTMUT1 REM} with $\iota_!$ replaced with $\iota_{\**}$ does not hold. In particular, 
the map
\begin{equation}\label{KEYNONISO EQ}
	\mathbb{F}_G \iota_{\**}
	\xrightarrow{\eta_{\**}}
	\iota_{\**}\iota^{\**} \mathbb{F}_G \iota_{\**}
\end{equation}
is not always an isomorphism 
(a counterexample is discussed at the end of this remark). 

Note that, via the identifications
$\mathbb{F} \simeq 
\iota^{\**} \mathbb{F}_G \iota_! \simeq
\iota^{\**} \mathbb{F}_G \iota_{\**}$,
\eqref{KEYNONISO EQ} is identified with a map
$\mathbb{F}_G \iota_{\**} \to  \iota_{\**} \mathbb{F}$,
so that whenever \eqref{KEYNONISO EQ} is an isomorphism
at $X \in \mathsf{Sym}^G(\mathcal{V})$
one has $\mathbb{F}_G \iota_{\**} X \simeq  \iota_{\**} \mathbb{F}X$.
We note that, informally, 
by \eqref{IOTAFUNSALT EQ} the latter means that
the graph fixed points
$\left(\left(\mathbb{F} X\right)(n)\right)^{\Gamma}$
are computed from the graph fixed points of $X$
via $\mathbb{F}_G$.

The claim that \eqref{KEYNONISO EQ}
\textit{does} become an isomorphism when restricted to \textit{cofibrant} $X \in \mathsf{Sym}^G(\mathcal{V})$
is one of the key ingredients of our proof of the Quillen equivalence between 
$\mathsf{Op}_G(\mathcal{V})$ and
$\mathsf{Op}^G(\mathcal{V})$
given by Theorem \ref{MAINQUILLENEQUIV THM}, 
and will be the subject of \S \ref{COFIB SEC}.

We note that, 
for $\mathcal{O} \in \mathsf{Op}^G(\mathcal{V})$,
the composite
$\mathbb{F}_G \iota_{\**} \mathcal{O}
\to 
\iota_{\**} \mathbb{F} \mathcal{O}
\to 
\iota_{\**} \mathcal{O}$
encodes the compositions of norm maps as in 
\eqref{INTFIXPTCOMP EQ}.

We end this remark with a minimal counterexample to the claim that \eqref{KEYNONISO EQ} 
is an isomorphism.
Let $\mathcal{V}=\mathsf{Set}$, $G=\mathbb{Z}_{/2}$, and 
$Y=\**$ in $\mathsf{Sym}^G = \mathsf{Sym}^G(\mathsf{Set})$ be the singleton.
When evaluating $\mathbb{F}_G \iota_{\**} Y$ at the $G$-fixed stump corolla $G/G \cdot C_0$ 
	(where $C_0 \in \Sigma$ is the $0$-corolla),
	 the two $G$-trees $T_1$ and $T_2$ below
	 (with expanded/orbital representations on the left/right) encode two distinct points 
	 (as $T_1$, $T_2$ are not isomorphic as objects under $G/G\cdot C_0$).
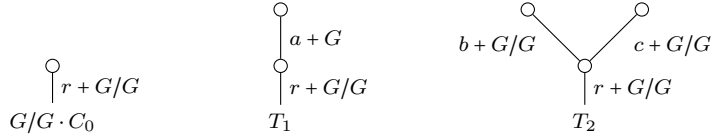
\[
\begin{tikzpicture}[grow=up,auto,
	level distance=2.3em,
	sibling distance=1.5em,
	every node/.style = {font=\footnotesize},dummy/.style={circle,draw,inner sep=0pt,minimum size=1.75mm}]
	\node at (0,0) [font=\normalsize] {$G/G \cdot C_0$}
		child{
			node [dummy] {}
		edge from parent node [swap] {$r$}};
	\node at (1.5,0) [font=\normalsize] {$T_1$}
		child{node [dummy] {}
			child{node [dummy] {}
			edge from parent node [swap,near end] {$a+1$}}
			child{node [dummy] {}
			edge from parent node [near end] {$\phantom{1+}a$}}
		edge from parent node [swap] {$r$}};
	\node at (3.5,0) [font=\normalsize] {$T_2$}
		child{node [dummy] {}
			child{node [dummy] {}
			edge from parent node [swap,near end] {$c\phantom{b}$}}
			child{node [dummy] {}
			edge from parent node [near end] {$b$}}
		edge from parent node [swap] {$r$}};
\begin{scope}[xshift = 19em]
	\node at (0,0) [font=\normalsize] {$G/G \cdot C_0$}
		child{node [dummy] {}
		edge from parent node [swap] {$r+G/G$}};
	\node at (1.75,0) [font=\normalsize] {$T_1$}
		child{node [dummy] {}
			child{node [dummy] {}
			edge from parent node [swap] {$a+G$}}
		edge from parent node [swap] {$r+G/G$}};
	\node at (4.5,0) [font=\normalsize] {$T_2$}
		child{node [dummy] {}
			child{node [dummy] {}
			edge from parent node [swap,near end] {$c+G/G$}}
			child{node [dummy] {}
			edge from parent node [near end] {$b+G/G$}}
		edge from parent node [swap] {$r+G/G$}};
\end{scope}
\end{tikzpicture}
\]
However, when pulling these points back to the $G$-free stump corolla $G \cdot C_0$ one obtains the same point in 
$\mathbb F_G \iota_{\**} Y(G \cdot C_0)$,
namely the point encoded by the $G$-tree $T$ below.
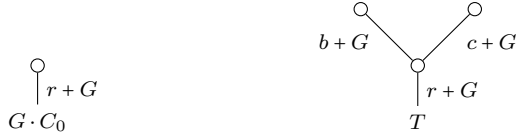
\[
\begin{tikzpicture}[grow=up,auto,
	level distance=2.3em,
	sibling distance=1.5em,
	every node/.style = {font=\footnotesize},dummy/.style={circle,draw,inner sep=0pt,minimum size=1.75mm}]
	\node at (0,0) [font=\normalsize] {}
		child{node [dummy] {}
		edge from parent node [swap] {$r$}};
	\node at (1,0) [font=\normalsize] {}
		child{node [dummy] {}
		edge from parent node [swap] {$r+1$}};
	\node at (3,0) [font=\normalsize] {}
		child{node [dummy] {}
			child{node [dummy] {}
			edge from parent node [swap,near end] {$c\phantom{b}$}}
			child{node [dummy] {}
			edge from parent node [near end] {$b$}}
		edge from parent node [swap] {$r$}};
	\node at (5,0) [font=\normalsize] {}
		child{node [dummy] {}
			child{node [dummy] {}
			edge from parent node [swap,near end] {$c+1$}}
			child{node [dummy] {}
			edge from parent node [near end] {$b+1$}}
		edge from parent node [swap] {$r+1$}};
	\draw[decorate,decoration={brace,amplitude=2.5pt}] (1.1,0.1) -- (-0.1,0.1) node[midway,inner sep=4pt,font = \normalsize]{$G \cdot C_0$};	
	\draw[decorate,decoration={brace,amplitude=2.5pt}] (5.1,0.1) -- (2.9,0.1) node[midway,inner sep=4pt,font = \normalsize]{$T$}; %
\begin{scope}[xshift=25em]
	\node at (0,0) [font=\normalsize] {$G \cdot C_0$}
		child{node [dummy] {}
		edge from parent node [swap] {$r+G$}};
	\node at (3,0) [font=\normalsize] {$T$}
		child{node [dummy] {}
			child{node [dummy] {}
			edge from parent node [swap,near end] {$c+G$}}
			child{node [dummy] {}
			edge from parent node [near end] {$b+G$}}
		edge from parent node [swap] {$r+G$}};
\end{scope}
\end{tikzpicture}
\]
Moreover, it is not hard to modify the example above to produce similar examples when evaluating $\mathbb{F}_GY$ at non-empty corollas. 

However, such counter-examples all require the use of trees with stumps. Indeed, it can be shown that \eqref{KEYNONISO EQ}
is an isomorphism whenever evaluated at a $Y$ such that $Y(C_0)=\emptyset$.
\end{remark}

\begin{remark}\label{SOMEMOREDET REM}
	As in Remark \ref{ALGELMCOL REM},
	let $\mathcal{T}$ 
	be the colored operad $S^C$ in \cite[\S 3.2]{GV12} for $C=\{\**\}$.
	In our notation, 
	colors of $\mathcal{T}$ are corollas $C \in \Sigma$,
	and operations from $C_1,\cdots,C_n$ to $C_0$
	consist of a tree $T\in \Omega$, a permutation 
	$\sigma \in \Sigma_n$
	such that $V(T) = (C_{\sigma(i)})$,
	and a tall map $C_0 \to T$.
	
	Replacing $\Sigma,\Omega,V$ above with 
	$\Sigma_G^{\mathsf{fr}},\Omega_G^{\mathsf{fr}},\boldsymbol{V}_G$
	and demanding $C_0 \to T$ to be a tall \emph{rooted} map
	yields the colored operad $\mathcal{T}^{\mathsf{fr}}_G$
	mentioned in Remark \ref{ALGELMCOL REM}.
	Moreover, $G$ acts on $\mathcal{T}^{\mathsf{fr}}_G$ 
	via 
	$g(C_i)_{i\in I} = (C_{gi})_{i\in I}$
	on objects and 
	$g(T_i)_{i\in I} = (T_{gi})_{i\in I}$ on operations.
	Additionally,
	giving $\mathcal{T}$ the trivial $G$-action,
	one has a $G$-equivariant map
	$\mathcal{T} \to \mathcal{T}^{\mathsf{fr}}_G$
	given by 
	$C \mapsto G \cdot C$ on objects
	and
	$T \mapsto G \cdot T$ on operations.
	After forgetting the $G$-action,
	the functor $\mathcal{T} \to \mathcal{T}^{\mathsf{fr}}_G$
	becomes fully faithful and essentially surjective,
	thus inducing an equivalence of algebra categories,
	so that $\mathcal{T}^{\mathsf{fr}}_G$ algebras are equivalent to 
	$\mathsf{Op}^G(\mathcal{V})$.
	We note that this equivalence is built into 
	Corollary \ref{TWOADJOINTSOP_COR},
	since \eqref{IOTADEF EQ} and the proof of 
	Proposition \ref{MONAD_COMPARISON_PROP}
	secretly use the map $\mathcal{T} \to \mathcal{T}^{\mathsf{fr}}_G$.
	
	However, some care is needed, as the map 
	$\mathcal{T} \to \mathcal{T}^{\mathsf{fr}}_G$
	is not $G$-essentially surjective.
	More precisely, for any $\** \neq H \leq G$
	the fixed point map
	$\mathcal{T}^H \to \left(\mathcal{T}^{\mathsf{fr}}_G\right)^H$
	is \emph{not} essentially surjective.
\end{remark}

\renewcommand{\F}{\mathcal{F}}

\subsection{Indexing systems and partial genuine operads}
\label{INDEXING_SECTION}

As discussed preceding Theorem \ref{MAINEXIST2 THM},
the Elmendorf-Piacenza equivalence
\eqref{COFADJINT EQ} has analogues
\[
\begin{tikzcd}[column sep =4em,row sep=0.3em]
	\mathsf{Top}^{\mathsf{O}_{\mathcal{F}}^{op}}
	\ar[shift left=1]{r}{\iota^{\**}} 
&
	\mathsf{Top}_{\mathcal{F}}^G
	\ar[shift left=1]{l}{\iota_{\**}}
\end{tikzcd}
\]
for each \textit{family} $\mathcal{F}$ of subgroups of $G$
(i.e. a collection closed under conjugation and subgroups).
Here $\mathsf{O}_\mathcal{F} \hookrightarrow \mathsf{O}_G$
consists of those $G/H$ such that $H \in \mathcal{F}$, 
and thus the objects of
$\mathsf{Top}^{\mathsf{O}_{\mathcal{F}}^{op}}$
are partial coefficient systems.
These specialized equivalences provide an alternative approach to universal 
$E \mathcal{F}$-spaces: rather than cofibrantly replacing the object
$\delta_{\mathcal{F}} \in \mathsf{Top}^{\mathsf{O}_G^{op}}$
as in the introduction,
one builds an $E \mathcal{F}$-space by
$\iota^{\**}(C \**) = (C *) (G)$
where now $\** \in \mathsf{Top}^{\mathsf{O}_{\mathcal{F}}^{op}}$
is the terminal object and $C$ the cofibrant replacement in $\mathsf{Top}^{\mathsf{O}_{\mathcal{F}}^{op}}$.

In keeping with the motivation that the Blumberg-Hill $N \mathcal{F}$ operads are the operadic analogues of universal $E \mathcal{F}$ spaces,
we will now show that the closure conditions for 
indexing systems
identified in \cite[Def. 3.22]{BH15}
are (almost exactly) the necessary conditions to define categories 
$\mathsf{Op}_{\mathcal{F}}$
of partial genuine equivariant operads.

We start by recalling that 
a collection $\mathcal F$ of subgroups of $G$
is a family 
if and only if the associated subcategory 
$\mathsf{O}_{\mathcal{F}} \hookrightarrow
\mathsf{O}_G$ is a sieve, as per the following.

\begin{definition}
	A \textit{sieve} of a category $\mathcal{D}$
	is a subcategory $\mathcal{S}$ such that,
	for any arrow $f \colon d \to s$ in $\mathcal{D}$ with 
	$s \in \mathcal{S}$,
	both $d$ and $f$ are also in $\mathcal{S}$. 
	In particular, sieves are full subcategories.
\end{definition}

\begin{definition}\label{FAMILY_COROLLAS_DEF}
      We call a sieve
      $\Sigma_{\mathcal{F}} \hookrightarrow \Sigma_G$
      a \textit{family of $G$-corollas}.
\end{definition}

\begin{remark}\label{FAMILY_COROLLAS_REM}
A family of $G$-corollas $\Sigma_{\mathcal{F}}$
can equivalently be encoded by
a collection $\F = \set{\F_n}_{n \geq 0}$ of 
families $\mathcal F_n$ of \textit{graph subgroups} of $G \times \Sigma_n$, so that there is an equivalence of categories
$\Sigma_\F \simeq \coprod \mathsf{O}_{\F_n}$ (see Lemma \ref{FAMILY_COROLLAS_LEM}).
	As such, we abuse notation and abbreviate either set of data as $\F$. 
\end{remark}

Writing 
$\upgamma \colon 
\Sigma_{\mathcal{F}}
\hookrightarrow
\Sigma_G$\index{key functors!gamma@$\gamma$}
for the inclusion and 
$\mathsf{Sym}_{\mathcal{F}}(\mathcal{V}) = 
\mathcal{V}^{\Sigma_{\mathcal{F}}^{op}}$,
we thus have a pair of adjunctions
\begin{equation}\label{F_TWOADJOINTS_EQ}
\index{categories!of operads/symmetric sequences!SymF@$\mathsf{Sym_{\mathcal{F}}}(\mathcal{F})=\mathcal{V}^{\Sigma_{\mathcal{F}}^{op}}$}
	\begin{tikzcd}[column sep =5em]
		\mathsf{Sym}_\F(\mathcal{V})
		\arrow[r, bend left, "\upgamma_{!}"]
		\arrow[r, bend right, "\upgamma_{\**}"']
	&
		\mathsf{Sym}_{G}(\mathcal{V}) 
		\arrow[l, "\upgamma^{\**}"] 
	\end{tikzcd}
\end{equation}
Our focus will be on the $(\upgamma_!,\upgamma^{\**})$ adjunction.
The requirement that $\Sigma_{\mathcal{F}}$ be a sieve then implies that $\upgamma_!$ simply extends presheaves by the initial object 
$\emptyset \in \mathcal{V}$,
so that $\gamma_!$ identifies 
$\mathsf{Sym}_{\mathcal{F}}(\mathcal{V})$
with a (coreflexive) subcategory of 
$\mathsf{Sym}_G(\mathcal{V})$.
One may then ask for conditions on the family 
of corollas $\mathcal{F}$ such that 
the genuine operad monad $\mathbb{F}_G$
preserves this subcategory.
The answer is almost exactly given by the Blumberg-Hill indexing systems.

\begin{definition}\label{FTREE DEF}
Let $\mathcal{F}$ be a family of $G$-corollas.

We say that a $G$-tree $T$ is a \textit{$\mathcal{F}$-tree}
if all of its $G$-vertices $T_{v}$, $v \in V_G(T)$ are in 
$\Sigma_{\mathcal{F}}$.
We denote by 
$\Omega_\F \hookrightarrow \Omega_G$,
$\Omega_\F^0 \hookrightarrow \Omega_G^0$
the full subcategories spanned by the $\F$-trees.
\end{definition}

\begin{remark}\label{VACUOUSNESS REM}
	By vacuousness, the stick $G$-trees
	$(G/H) \cdot \eta = (\eta)_{G/H}$ are always $\mathcal{F}$-trees.
\end{remark}




\begin{definition}\label{INDEXSYS DEF}
	A family $\mathcal{F}$ of $G$-corollas is called a 
	\textit{weak indexing system} if,
	for any $\mathcal{F}$-tree $T \in \Omega_{\mathcal{F}}^0$,
	we have $\mathsf{lr}(T) \in \Sigma_{\mathcal{F}}$;
	that is, if the leaf-root functor restricts to a functor
	$\mathsf{lr} \colon \Omega_{\mathcal{F}}^0 \to \Sigma_{\mathcal{F}}$.
Moreover, $\mathcal{F}$ is called simply an \textit{indexing system} if all trivial corollas 
$(G/H)\cdot C_n = (C_n)_{G/H}$ are in $\Sigma_{\mathcal{F}}$.
\end{definition}

\begin{remark}
	In light of Remark \ref{VACUOUSNESS REM},
	any weak indexing system must contain the $1$-corollas $(G/H) \cdot C_1 \simeq (C_1)_{G/H}$ for all $H\leq G$.
\end{remark}

\begin{remark}
The notion of indexing system
was first introduced in \cite[Def. 3.22]{BH15}, though packaged quite differently.
Moreover, a third definition of (weak) indexing systems as certain sieves 
$\Omega_{\mathcal{F}} \hookrightarrow \Omega_G$
was presented by the second author in \cite[\S 9]{Pe17}. The equivalence between the definitions in \cite{BH15} and \cite{Pe17} is addressed in 
\cite[Rmk. 9.7]{Pe17}, hence here we address only the easier equivalence between Definition \ref{INDEXSYS DEF} and the sieve definition in \cite[\S 9]{Pe17}.

The existence of canonical maps 
$\mathsf{lr}(T) \to T$ shows that the sieve condition
implies the $\mathsf{lr}$ condition
in Definition \ref{INDEXSYS DEF}. 
Conversely, as discussed immediately preceding \cite[Def. 9.5]{Pe17}, the sieve condition needs only be checked for inner faces and degeneracies, i.e. tall maps, 
and thus follows from Definition \ref{INDEXSYS DEF},
since the subcategory 
$\Omega_{\mathcal{F}}^1 \hookrightarrow \Omega^1_G$ 
of planar tall strings between $\mathcal{F}$-trees
matches the pullback of
$\Omega_{\mathcal{F}}^0 \to
\Fin_s \wr \Sigma_{\mathcal{F}} \leftarrow 
\Fin_s \wr \Omega_{\mathcal{F}}^0
$.
\end{remark}

The connection between weak indexing systems and $\mathbb{F}_G$ is given by the following,
which generalizes 
Proposition \ref{MONAD_COMPARISON_PROP}.

\begin{proposition}\label{F_MONAD_COMPARISON_PROP}
	Let $\F$ be a weak indexing system. Then:
	\begin{itemize}
	\item[(i)] the map 
		$\upgamma^{\**} \mathbb{F}_G
		\xrightarrow{\eta_{\**}}
		\upgamma^{\**} \mathbb{F}_G \upgamma_{\**} \upgamma^{\**}$
		is an isomorphism,
		and thus (cf. Prop. \ref{MONADADJ PROP})
		$\upgamma^{\**} \mathbb{F}_G \upgamma_{\**}$
		is a monad;
	\item[(ii)] the map
		$\upgamma^{\**} \mathbb{F}_G \upgamma_{!}
		\xrightarrow{\beta}
		\upgamma^{\**} \mathbb{F}_G \upgamma_{\**}$ is an isomorphism of monads;
	\item[(iii)] the map
		$\upgamma_{!}\upgamma^{\**} \mathbb{F}_G \upgamma_{!}
		\xrightarrow{\epsilon_!}
		\mathbb{F}_G \upgamma_{!}$ is an isomorphism.
	\end{itemize}
\end{proposition}

\begin{proof}
	This follows just like the analogous parts of
	Proposition \ref{MONAD_COMPARISON_PROP} by
	replacing
	$\mathsf{lr}: \Omega_G^{0,\text{fr}} \to \Sigma_G^{\text{fr}}$
	with
	$\mathsf{lr}: \Omega_{\F}^0 \to \Sigma_{\F}$. 
	For (i), note that if $C \in \Sigma_{\mathcal{F}}$
	there is an identification between
	$C \downarrow_{\mathsf r} \Omega_G^0$
	and
	$C \downarrow_{\mathsf r} \Omega_{\F}^0$,
	so that $\mathbb{F}_G X (C)$
	only depends on the values of $X$ on $\Sigma_{\mathcal{F}}$.
	(ii) is immediate.	
	Lastly, (iii) follows 
	since if $C \nin \Sigma_{\mathcal{F}}$ then
	any tree in $C \downarrow_{\mathsf r} \Omega_G^0$ must contain at least one $G$-vertex not in $\Sigma_{\mathcal{F}}$,
	so that indeed $\mathbb{F}_G \upgamma_{!}Y(C)=\emptyset$.
%
\end{proof}

\begin{notation}
We write 
$\mathbb{F}_{\mathcal{F}} = \upgamma^{\**} \mathbb{F}_G \upgamma_!$ for the induced monad
on
$\mathsf{Sym}_{\mathcal{F}}(\mathcal{V})$,
and $\mathsf{Op}_{\mathcal{F}}(\mathcal{V})$
for the corresponding categories of algebras.
\end{notation}

\begin{corollary}\label{TWOADJOINTSOPF COR}
The adjunctions \eqref{F_TWOADJOINTS_EQ} lift to adjunctions
\begin{equation}\label{TWOADJOINTSOPF EQ}
	\begin{tikzcd}[column sep =5em]
		\mathsf{Op}_\F(\mathcal{V})
			\arrow[r, bend left, "\upgamma_{!}"]
			\arrow[r, bend right, "\upgamma_{\**}"']
		&
		\mathsf{Op}_{G}(\mathcal{V})
		\arrow[l, "\upgamma^{\**}"]
	\end{tikzcd}
\end{equation}
\end{corollary}

\begin{remark}\label{WINDEX_GAMMA_REM}
Part (iii) of Proposition \ref{F_MONAD_COMPARISON_PROP}
states that if $\mathcal{F}$ is a weak indexing system then $\mathbb{F}_G$ essentially preserves the image of $\upgamma_!$ (moreover, the converse is easily seen to also hold).
As such, we will sometimes find it conceptually convenient
to regard $\mathbb{F}_{\mathcal{F}}$ as
``restricting $\mathbb{F}_G$''.
\end{remark}

%
        
\begin{remark}        
	The free corollas of \S	\ref{COMPARISON_REGULAR_SECTION}
	form a weak indexing system
	$\Sigma_G^{\text{fr}} = \Sigma_{\F_{\text{fr}}}$
	and, moreover, there is an equivalence of categories
	$\Op^G \simeq \Op_{\F_{\text{fr}}}$,
	so that Corollary \ref{TWOADJOINTSOP_COR}
	is a particular case of 
	Corollary \ref{TWOADJOINTSOPF COR}.
	However, while our discussion of 
	Corollary \ref{TWOADJOINTSOP_COR}
	focuses on the $(\iota^{\**},\iota_{\**})$-adjunction, due to the fact that the intended model structures on $\mathsf{Op}^G(\mathcal{V})$ in
	Theorem \ref{MAINEXIST1 THM} are defined via fixed point conditions, 
	our discussion of 
	Corollary \ref{TWOADJOINTSOPF COR}
	focuses on the $(\iota_{!},\iota^{\**})$-adjunction, due to the model structures in 
	Theorem \ref{MAINEXIST2 THM} being projective.
\end{remark}

\begin{remark}\label{COMPADJ REM}
	In most cases, the rightmost $(\iota^{\**},\iota_{\**})$-adjunction appearing in Theorem \ref{MAINQUILLENEQUIV THM}
	is induced by an inclusion 
	$\iota \colon \Sigma_G^{\text{fr}} \hookrightarrow \Sigma_{\mathcal{F}}$.
	However, it is possible for  
	$\Sigma_G^{\text{fr}} \nsubset \Sigma_{\mathcal{F}}$ (the most interesting case being that of
	$\Sigma_{\mathcal F} = \Sigma_{G}^{\geq 1}$
	the corollas of arity $\geq 1$, which model non-unital operads).
	In these cases (and compatibly with the 
	$\Sigma_G^{\text{fr}} \hookrightarrow \Sigma_{\mathcal{F}}$ case), we instead use the composite adjunction
\begin{equation}\label{COMPADJ EQ}
\begin{tikzcd}[column sep =5em]
	\mathsf{Op}_\F(\mathcal{V})
	\ar[shift left=1.5]{r}{\upgamma_!} 
&
	\mathsf{Op}_{G}(\mathcal{V}) 
	\arrow[l, shift left=1.5, "\upgamma^{\**}"] 
	\arrow[r, shift left=1.5,swap,"\iota^{\**}"']
&
	\mathsf{Op}^G(\mathcal{V})
	\ar[shift left=1.5]{l}{\iota_{\**}}
\end{tikzcd}
\end{equation}
Note that the right adjoint 
$\gamma^{\**} \iota_{\**}$
is still defined by computing fixed points while the 
left adjoint
$\iota^{\**}\gamma_!$
is still essentially a forgetful functor, with those levels not in $\mathcal{F}$ declared to be $\emptyset$.

In practice, however, the use of the composite adjunction
\eqref{COMPADJ EQ}
is fairly benign, requiring only minor
adjustments to the notation of the proofs in 
\S \ref{MAINTHM_PROOF_SECTION}.
\end{remark}

%


\renewcommand{\F}{\mathbb{F}}

\section{Free extensions and the existence of model structures}
\label{FREE_EXTENSIONS_SECTION}

In order to prove all of our main theorems
we will need to homotopically analyze free extensions 
of genuine equivariant operads,
i.e. pushouts of the form
\begin{equation}
  \label{FREE_FG_EXT_EQ}
  \begin{tikzcd}
    \mathbb{F}_G X \ar{r} \ar{d}[swap]{\mathbb{F}_G u} & \mathcal{P} \ar{d}
    \\
    \mathbb{F}_G Y \ar{r} & \mathcal{P}[u]
  \end{tikzcd}
\end{equation}
in the category $\mathsf{Op}_G(\mathcal{V})$.
As is common in the literature (e.g. \cite{SS00, Spi01, BM03, WY18, Pe16, BB17}),
the key technical ingredient will be the identification of a suitable filtration
\begin{equation}\label{FILTR EQ}
	\mathcal{P}=\mathcal{P}_0 \to 
	\mathcal{P}_1 \to \mathcal{P}_2 \to
	\cdots \to \mathcal{P}_{\infty}=\mathcal{P}[u]
\end{equation}
of the map $\mathcal{P} \to \mathcal{P}[u]$
in the underlying category $\mathsf{Sym}_G(\mathcal{V})$.
To explain how this filtration is obtained
(for a comparison with similar filtrations, 
see Remark \ref{FILTCOMP REM}),
note first that $\mathcal{P}[u]$ is given by a coequalizer
\begin{equation}\label{REFLCOEQ EQ}
\begin{tikzcd}
	\mathcal{P} \mathbin{\check{\amalg}}
	\mathbb{F}_G X \mathbin{\check{\amalg}} \mathbb{F}_G Y
	\ar[shift right=4pt]{r} \ar[shift right=-4pt]{r}
&
	\mathcal{P} \mathbin{\check{\amalg}} \mathbb{F}_G Y 
	\ar[dashed]{l}
\end{tikzcd}
\end{equation}
where $\check{\amalg}$ denotes the algebraic coproduct, 
i.e. the coproduct in $\mathsf{Op}_G(\mathcal{V})$, and, a priori,
the coequalizer is also calculated in $\mathsf{Op}_G(\mathcal{V})$. However, \eqref{REFLCOEQ EQ} is a so called \textit{reflexive coequalizer}, meaning that the maps being coequalized have a common section,
and standard arguments\footnote{
For example, by the proof of 
\cite[Prop. 3.27]{Ha09},
it suffices to show that 
$\mathbb{F}_G$ preserves reflexive coequalizers.
This follows from \eqref{FGXDEF EQ} and the fact that,
if $\otimes$ preserves colimits in each variable,
then $(-)^{\otimes n}$ preserves reflexive coequalizers.}
show that it is hence also an underlying coequalizer in 
$\mathsf{Sym}_G(\mathcal{V})$.

In practice, we will need to enlarge 
\eqref{REFLCOEQ EQ} somewhat.
Firstly, note that \eqref{REFLCOEQ EQ}
corresponds to the two bottom levels of the bar construction
$B_l(\mathcal{P}, \mathbb{F}_G X, \mathbb{F}_G Y)=
\mathcal{P} \mathbin{\check{\amalg}}
(\mathbb{F}_G X)^{\check{\amalg} l} 
\mathbin{\check{\amalg}} \mathbb{F}_G Y$,
whose colimit (over $\Delta^{op}$) is again $\mathcal{P}[u]$.
For technical reasons, we prefer 
the double bar construction\footnote{
More formally,
$B_{\bullet}(\mathcal{P}, \mathbb{F} X, \mathbb{F} X, \mathbb{F} X, \mathbb{F} Y)$
is the diagonal
of the iterated bar construction
$B^{op}_{\bullet}\left(\mathcal{P}, \mathbb{F} X, 
B_{\bullet}(\mathbb{F} X, \mathbb{F} X, \mathbb{F} Y)\right)$,
where the $op$ in $B^{op}_{\bullet}$ indicates that in the outer bar construction we reverse the order of the simplicial operators.
}
(where to increase readability, we 
abbreviate $\mathbb{F}_G$ as $\mathbb{F}$)
\begin{equation}\label{DOUBAR EQ}
	B_l(\mathcal{P}, \mathbb{F} X, \mathbb{F} X, \mathbb{F} X, \mathbb{F} Y)
=
	\mathcal{P} \mathbin{\check{\amalg}}
	(\mathbb{F} X)^{\check{\amalg} l} 
	\mathbin{\check{\amalg}}
	\mathbb{F} X
	\mathbin{\check{\amalg}}
	(\mathbb{F} X)^{\check{\amalg} l} 
	\mathbin{\check{\amalg}} \mathbb{F} Y
=
	\mathcal{P} \mathbin{\check{\amalg}}
	(\mathbb{F} X)^{\check{\amalg} 2l+1} 
	\mathbin{\check{\amalg}} \mathbb{F} Y.
\end{equation}
To actually describe the individual levels of \eqref{DOUBAR EQ},
one further resolves $\mathcal{P}$
so as to obtain the bisimplicial object
(we again abbreviate $\mathbb{F}_G$ as $\mathbb{F}$)
\begin{equation}\label{FURRES EQ}
	B_l(\mathbb{F}^{n+1}\mathcal{P}, \mathbb{F} X, \mathbb{F} X, \mathbb{F} X, \mathbb{F} Y)
=
	\mathbb{F}^{n+1}\mathcal{P} \mathbin{\check{\amalg}}
	(\mathbb{F} X)^{\check{\amalg} 2l+1} 
	\mathbin{\check{\amalg}} \mathbb{F} Y
\simeq
	\mathbb{F}\left(
		\mathbb{F}^{n} \mathcal{P} \amalg
		X^{\amalg 2l+1} \amalg Y
	\right),
\end{equation}
where $\amalg$ denotes the coproduct in $\mathsf{Sym}_G(\mathcal{V})$.
As in Remark \ref{REPACKAGERES REM}, each level of 
\eqref{FURRES EQ}
can then be described as 
\begin{equation}\label{LANLEVELFOR EQ}
 \mathsf{Lan} N (N^{n} \upsilon \mathcal{P} 
\amalg \upsilon X^{\amalg 2l+1} \amalg \upsilon Y)
=
\mathsf{Lan} N^{(\mathcal P, X, Y)}_{n,l},
\end{equation}
for $N$ the span monad (cf. Definition \ref{WSPAN_MONAD_DEFINITION}) and $\amalg$ now the coproduct of spans.
In particular, each level of 
\eqref{FURRES EQ}
is thus a left Kan extension over some category
$\Omega_G^{n,\lambda_l}$, which we explicitly identify in 
\S \ref{LABELSTRI SEC}, giving the first identification below.
\begin{equation}\label{EXTTREEFOR EQ}
	\mathcal{P} \mathbin{\check{\coprod}}_{\mathbb{F}_G X} \mathbb{F}_G Y 
\simeq 
	\colim_{(\Delta \times \Delta)^{op}}
	\left(
	\mathsf{Lan}_{\left( \Omega_{G}^{n,\lambda_l} \to \Sigma_G \right)^{op}}
	N_{n,l}^{(\mathcal{P},X,Y)}
	\right)
\simeq 
	\mathsf{Lan}_{\left( \Omega_{G}^{e} \to \Sigma_G \right)^{op}}
	\tilde{N}^{(\mathcal{P},X,Y)}
\end{equation}
The second identification, 
which reduces the calculation to a single left Kan extension, is an instance of 
Proposition \ref{RANTRANS PROP}, 
a result whose proof is straightforward but lengthy, 
and thus postponed to the appendix.
The category $\Omega_G^e$ of \textit{extension trees}
appearing on the right side
is obtained as a categorical realization
$\Omega_G^e = |\Omega_{G}^{n,\lambda_l}|$,
which we explicitly describe and analyze in 
\S \ref{EXTTREE SEC}.
In particular, we identify a smaller and more convenient
subcategory 
$\widehat{\Omega}_G^e \hookrightarrow \Omega_G^e$
that is suitably initial,
so that $\Omega_G^e$ can be replaced with $\widehat{\Omega}_G^e$
in \eqref{EXTTREEFOR EQ}.

The desired filtration \eqref{FILTR EQ}
then follows from a filtration of the 
category $\widehat{\Omega}_G^e$ itself,
and this discussion is the subject of
\S \ref{FILTRATION_SECTION}.

Lastly, \S \ref{MAINEXIST SEC} concludes this section
by using these filtrations to prove 
Theorems \ref{MAINEXIST1 THM} and \ref{MAINEXIST2 THM}.


\begin{remark}\label{FILTCOMP REM}
	Our approach to the filtration \eqref{FILTR EQ}
	is significantly constrained by Theorems 
	\ref{MAINEXIST1 THM},
	\ref{MAINEXIST2 THM},
	\ref{MAINQUILLENEQUIV THM},
	which present technical challenges not found in similar filtrations in the literature.

	To discuss these challenges, we consider 
	\cite[Prop. 4.3.16]{WY18} (resp. \cite[\S 7]{BB17}),
	which builds filtrations of free extensions
	of algebras over colored operads
	(resp. over polynomial monads).
	The frameworks in \cite{WY18},\cite{BB17} are general enough
	to cover the usual category $\mathsf{Op}(\V)$ of operads,
	so one might hope they would suffice for our purposes.
	However, one runs into two key issues:
\begin{enumerate*}
\item[(i)]
	by Remark \ref{NEED_WREATH_REMARK},
	defining genuine operads $\mathsf{Op}_G(\V)$
	\emph{requires} using diagonal maps in $\V$,
	so that,
	since the monads in \cite{WY18},\cite{BB17}
	do not use diagonals,
	$\mathsf{Op}_G(\V)$
	is not covered by those frameworks;
\item[(ii)]
	\cite{WY18},\cite{BB17}
	are designed to build model structures
	with projective weak equivalences,
	rather than fixed point equivalences as in 	
	\eqref{GENEOPEQMT EQ}.
	Consequently, writing $\widetilde{\mathbb F}$ for the monad  \cite{WY18},\cite{BB17} use to describe
	$\mathsf{Op}^G(\V) = \mathsf{Op}(\V^G)$
	(explicitly, $\widetilde{\mathbb F}$
	is the monad for the composite adjunction 
	\eqref{MAINPFADJAL EQ}
	with 
	$\coprod_{n} \mathcal{V}^{G \times \Sigma_n^{op}}$
	replaced by $\mathcal{V}^{G \times \mathbb{N}_0}$),
	one has that not all the generating (trivial) cofibrations needed for 
	Theorem \ref{MAINEXIST1 THM}
	are in the image of $\widetilde{\mathbb{F}}$
	(more precisely, the left map in \eqref{ANOPUCH EQ}
	is in the image of $\widetilde{\mathbb{F}}$
	iff $K \leq G$;
	compare with \eqref{NAIVEGEINFTY EQ},\eqref{GENGEINFTY EQ}).
\end{enumerate*}

Given these issues, 
rather than modifying all of \cite{WY18},\cite{BB17},
our approach to \eqref{FILTR EQ} adapts the key patterns in
\cite{WY18},\cite{BB17} while focusing on
the $G$-tree perspective for intuition.
As such, the ultimate description of our filtration 
in \eqref{FILTRATION_LAN_LEVEL}
resembles the description of free extensions of operads in \cite[\S 5.11]{BM03}\footnote{We caution that
\cite[\S 3]{BM09}
corrects an issue in \cite{BM03} concerning the treatment of operadic units.},
though we note that our workflow is as in
\cite{WY18},\cite{BB17} rather than \cite{BM03}.
Namely, we start with the abstract pushout $\mathcal{P}[u]$
and work our way to \eqref{FILTRATION_LAN_LEVEL}.
Conversely, 
\cite{BM03} directly uses an analogue
of \eqref{FILTRATION_LAN_LEVEL} to build an object $F_{\infty}$
which must \emph{a posteriori}
be shown to be both an operad and the desired extension $P[u]$ in \cite[Prop. 5.1]{BM03}.
However, we note that, as genuine operads are harder to describe explicitly,
the strategy in \cite{BM03} is ill suited for our context.
\end{remark}

\subsection{Labeled planar strings}\label{LABELSTRI SEC}

In this section we explicitly identify the categories underlying the left Kan extensions in \eqref{LANLEVELFOR EQ}.

In the notation 
of Remark \ref{PRECOMPPOSTCOMP REM},
letting 
$\langle \langle l 
\rangle \rangle = 
\{-\infty,-l,\cdots, -1, 0, 1, \cdots, l, \infty\}$ and writing
$\lambda_{l}$ for the partition
$\lambda_{l,a} = \{-\infty\}$,
$\lambda_{l,i} = 
\langle \langle l \rangle \rangle - \{-\infty\}$,
\eqref{LANLEVELFOR EQ}
can be repackaged as an instance of the functor
$\mathsf{Lan} \circ N \circ \coprod \circ (N^{\times \lambda_l})^{\circ n}\circ \upsilon^{\times \langle \langle l \rangle \rangle}$.
Our goal is thus to understand 
the underlying categories of the spans in the image of the functor
$N \circ \coprod \circ (N^{\times \lambda_l})^{\circ n}$,
though we will find it preferable, and no harder, to tackle the more general case of the functors 
$N^{s+1} \circ \coprod \circ (N^{\times \lambda})^{\circ n-s}$.

\begin{definition}\label{LABMAP DEF}
A \textit{$l$-node labeled $G$-tree} (or just \textit{$l$-labeled $G$-tree}) is a pair $(T,V_G(T) \to \{1,\cdots,l\})$ with $T \in \Omega_G$, which we think of as a $G$-tree together with $G$-vertices labels in $1,\cdots,l$.

Further, a tall map $\varphi \colon T \to S$ between $l$-labeled trees is called a \textit{label map} if, for each $G$-vertex $v_{G e}$ of $T$ with label $j$, all vertices of the subtree $S_{v_{G e}}$ 
(cf. Notation \ref{UEUPEG NOT}) are labeled by $j$.

Lastly, given a subset $J \subseteq \underline{l}$, a planar label map $\varphi \colon T \to S$ is said to be $J$-inert if for every $G$-vertex $v_{G e}$ of $T$ with label $j \in J$, we have $S_{v_{Ge}} = T_{v_{Ge}}$.
\end{definition}

\begin{example}\label{LABELEDTREES EX}
Consider the $2$-labeled trees below (for $G=\**$ the trivial group), with black nodes ($\bullet$) denoting labels by the number $1$ and white nodes ($\circ$) labels by the number $2$.
The planar map $\varphi$ (sending $a_i\mapsto a$, 
$b \mapsto b$, $c \mapsto c$, $d \mapsto d$, $e \mapsto e$) is a label map which is $\{1\}$-inert.
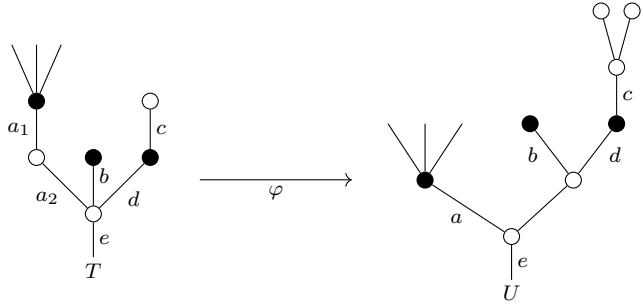
\[
	\begin{tikzpicture}[grow=up,auto,level distance=2.1em,
	every node/.style = {font=\footnotesize,inner sep=2pt},
	dummy/.style={circle,draw,inner sep=0pt,
	minimum size=2.1mm}]
	\begin{scope}[level distance=2.3em]
	\tikzstyle{level 2}=[sibling distance=3.5em]%
	\tikzstyle{level 3}=[sibling distance=2.25em]%
	\tikzstyle{level 4}=[sibling distance=1.25em]%
	\tikzstyle{level 5}=[sibling distance=1.25em]%
		\node at (5.5,0) {$U$}
			child{node [dummy] {}
				child[sibling distance =5em]{node [dummy] {}
					child[sibling distance =3.5em]{node [dummy,fill=black] {}
						child{node [dummy] {}
							child{node [dummy] {}}
							child{node [dummy] {}}
						edge from parent node [swap] {$c$}}
					edge from parent node [swap, near end] {$d$}}
					child[sibling distance =3.5em]{node [dummy,fill=black] {}
					edge from parent node [near end] {$b$}}
				}
				child[sibling distance =7em]{node [dummy,fill=black] {}
					child[sibling distance =1.5em]
					child[sibling distance =1.5em]
					child[sibling distance =1.5em]
				edge from parent node {$a$}}
			edge from parent node [swap] {$e$}};
	\end{scope}
	\begin{scope}[level distance=2.3em]
	\tikzstyle{level 2}=[sibling distance=2.3em]%
	\tikzstyle{level 4}=[sibling distance=1em]%
		\node at (0,0.3) {$T$}
			child{node [dummy] {}
				child{node [dummy,fill=black] {}
					child{node [dummy] {}
					edge from parent node [swap] {$c$}}	
				edge from parent node [swap] {$d$}}
				child{node [dummy,fill=black] {}
				edge from parent node [near end,swap] {$b$}}
				child{node [dummy] {}
					child{node [dummy,fill=black] {}
						child
						child
						child
					edge from parent node {$a_1$}}
				edge from parent node {$a_2$}}
			edge from parent node [swap] {$e$}};
	\end{scope}
	\draw [->] (1.4,1.5) -- node[swap] {$\varphi$} (3.4,1.5);
	\end{tikzpicture}
\]
\end{example}

\begin{definition}\index{categories!of trees!OMNSLAM@$\Omega_G^{n,s,\lambda}$}
Let $-1 \leq s \leq n$ and 
$\lambda = \lambda_a \amalg \lambda_i$ 
be a partition of $\{1,2,\cdots,l\}$.

We define $\Omega_{G}^{n,s,\lambda}$ to have as objects $n$-planar strings (where $T_{-1} = \mathsf{lr}(T_0)$ as in \eqref{STRINGOBJALT EQ})
\begin{equation}\label{NSTRINGLAB EQ}
	T_{-1} \xrightarrow{\varphi_0}
	T_0 \xrightarrow{\varphi_1}
	T_1 \xrightarrow{\varphi_2}
	\cdots \xrightarrow{\varphi_s}
	T_s \xrightarrow{\varphi_{s+1}}
	T_{s+1} \xrightarrow{\varphi_{s+2}}
	\cdots \xrightarrow{\varphi_n}
	T_{n}
\end{equation}
together with
$l$-labelings of $T_s, T_{s+1},\cdots, T_{n}$ such that the $\varphi_r,r>s$ are $\lambda_i$-inert label maps.

Arrows in $\Omega_{G}^{n,s,\lambda}$ are quotients of strings
$(\rho_r \colon T_r \to T'_r)$ such that 
$\rho_r, r\geq s$ are label maps.

Further, for any $s<0$ or $n<s'$ we write
\begin{equation}\label{EXTRACASES EQ}
	\Omega_{G}^{n,s,\lambda} = 
		\Omega_{G}^{n,-1,\lambda},
\qquad
	\Omega_{G}^{n,s',\lambda} = \Omega_{G}^{n}.
\end{equation}
\end{definition}

Intuitively, $\Omega_G^{n,s,\lambda}$ consists of strings that are labeled in the range $s \leq r \leq n$,
with the extra cases \eqref{EXTRACASES EQ} interpreted by infinitely prepending and postpending copies of $T_{-1}$ and $T_n$ to \eqref{NSTRINGLAB EQ}.

The main case of interest is that of $s=0$, which we abbreviate as $\Omega_{G}^{n,\lambda} = \Omega_{G}^{n,0,\lambda}$,
with the remaining
$\Omega_{G}^{n,s,\lambda}$ playing an auxiliary role.
The $s=-1$ case also deserves special attention.

\begin{remark}
	For $s<0$ there are identifications 
\begin{equation}\label{OMEGANMINUSONE EQ}
	\Omega_{G}^{n,s,\lambda} = 
	\Omega_{G}^{n,-1,\lambda} \simeq
		\coprod_{\lambda_a} \Omega_{G}^{n} \amalg
		\coprod_{\lambda_i} \Sigma_G.
\end{equation}
Indeed, since $T_{-1}$ is a $G$-corolla, the label of its unique $G$-vertex determines all other labels.
\end{remark}

\begin{notation}
We will write $(\Omega_G^n)^{\times \lambda}$ to denote the $l$-tuple with 
$(\Omega_G^n)^{\times \lambda}_j = \Omega_G^n$ if 
$j \in \lambda_a$ and
$(\Omega_G^n)^{\times \lambda}_j = \Sigma_G$ if
$j \in \lambda_i$.
As such, \eqref{OMEGANMINUSONE EQ} can be abbreviated as
$\Omega_{G}^{n,-1,\lambda} = \coprod (\Omega_G^n)^{\times \lambda}$.
\end{notation}

The $\Omega_G^{n,s,\lambda}$ categories are related by a number of obvious functors, which we now catalog.

Firstly, if $s \leq s'$ there are forgetful functors
\begin{equation}\label{NKNFGT EQ}
	\Omega_{G}^{n,s,\lambda} \to \Omega_{G}^{n,s',\lambda}
\end{equation}
and the simplicial operators
in Notation \ref{SIMPOPERATORS NOT}
generalize to operators (for $0 \leq i \leq n$, $-1\leq j \leq n$)
\begin{equation}\label{LABSTSIM EQ}
\begin{tikzcd}[row sep =0,column sep =1em]
	d_i \colon 
	\Omega_{G}^{n,s,\lambda} \ar{r} &
	\Omega_{G}^{n-1,s-1,\lambda} &
	i < s & & & &
	s_j \colon 
	\Omega_{G}^{n,s,\lambda} \ar{r} &
	\Omega_{G}^{n+1,s+1,\lambda} &
	j < s
\\
	d_i \colon 
	\Omega_{G}^{n,s,\lambda} \ar{r} &
	\Omega_{G}^{n-1,s,\lambda} &
	s \leq i & & & &
	s_j \colon 
	\Omega_{G}^{n,s,\lambda} \ar{r} &
	\Omega_{G}^{n+1,s,\lambda} &
	s \leq j
\end{tikzcd}
\end{equation}
which are compatible with the forgetful functors in the obvious way.

We will prefer to reorganize 
\eqref{NKNFGT EQ} and \eqref{LABSTSIM EQ} somewhat.
Defining functions 
$d_i \colon \mathbb{Z} \to \mathbb{Z}$
and 
$s_j \colon \mathbb{Z} \to \mathbb{Z}$
by
\begin{equation}\label{INTERMAPDEF EQ}
d_i(s) = 
	\begin{cases}
		s-1, & i<s
	\\
		s, & s \leq i
	\end{cases}
\qquad
s_j(s) = 
	\begin{cases}
		s+1, & j<s
	\\
		s, & s \leq j
	\end{cases}
\end{equation}
\eqref{LABSTSIM EQ} can be rewritten as maps
$
	d_i \colon 
	\Omega_{G}^{n,s,\lambda} \to
	\Omega_{G}^{n-1,d_i(s),\lambda}
$
and 
$
	s_j \colon 
	\Omega_{G}^{n,s,\lambda} \to
	\Omega_{G}^{n+1,s_j(s),\lambda}
$.
Therefore, we henceforth write simply
$\Omega_G^{n,\bullet,\lambda}$ to denote the string of categories $\Omega_G^{n,s,\lambda}$
and forgetful functors, and abbreviate \eqref{LABSTSIM EQ} as
\[
\begin{tikzcd}[row sep =0,column sep =1em]
	d_i \colon 
	\Omega_{G}^{n,\bullet,\lambda} \ar{r} &
	\Omega_{G}^{n-1,\bullet,\lambda} & & & &
	s_j \colon 
	\Omega_{G}^{n,\bullet,\lambda} \ar{r} &
	\Omega_{G}^{n+1,\bullet,\lambda}
\end{tikzcd}
\]

\begin{remark}\label{ORDLABEL REM}
Considering the ordered sets 
$\langle n \rangle =\{0 < 1 < \cdots < n < +\infty\}$, the formulas \eqref{INTERMAPDEF EQ} 
define functions
$d_i \colon \langle n \rangle  \to \langle n-1 \rangle$
,
$s_j \colon \langle n \rangle  \to \langle n+1 \rangle$
which preserve $0$ and $+\infty$, except for 
$s_{-1}$ which preserves only
$+\infty$.
This recovers the description of $\Delta^{op}$
as the category of intervals (i.e. ordered finite sets with a minimum and maximum and maps preserving them).
\end{remark}

Next, the vertex functors $\boldsymbol{V}_G^k$ of
\eqref{VGNISO EQ} generalize to functors
$
	\boldsymbol{V}_G^k \colon
	\Omega_G^{n,s,\lambda} \to
	\Fin_s \wr \Omega_G^{n-k-1,s-k-1,\lambda}
$
given by the same formula
\[
	(T_{k,v_{G e}}\to \cdots \to T_{n,v_{G e}})_{v_{G e} \in V_G(T_k)},
\]
as in \eqref{VGNISO EQ},
except with $T_{m,v_{G e}}$ for $k \leq m \leq n$ inheriting the node labels from $T_m$ (if any).

The diagrams in \eqref{PIIDEFDI EQ}
for $i<k$ and $i>k$ now generalize to diagrams
\begin{equation}\label{PIIDEFDILAB EQ}
\begin{tikzcd}[row sep=1.7em,column sep = 3em]
	\Omega_{G}^{n,\bullet,\lambda} \ar{d}[swap]{d_{i}} \ar{r}{\boldsymbol{V}_G^k} &
	|[alias=F1]|
	\Fin_s \wr \Omega_{G}^{n-k-1,\bullet,\lambda}
	\ar[equal]{d} 
&
	\Omega_{G}^{n,\bullet,\lambda} \ar{d}[swap]{d_{i}} \ar{r}{\boldsymbol{V}_G^k} &
	\Fin_s \wr \Omega_{G}^{n-k-1,\bullet,\lambda}
	\ar{d}{d_{i-k-1}} 
\\
	|[alias=G2]|
	\Omega_{G}^{n-1,\bullet,\lambda} \ar{r}[swap]{\boldsymbol{V}_G^{k-1}}&
	\Fin_s \wr \Omega_{G}^{n-k-1,\bullet,\lambda}  
&
	\Omega_{G}^{n-1,\bullet,\lambda} \ar{r}[swap]{\boldsymbol{V}_G^{k}}&
	\Fin_s \wr \Omega_{G}^{n-k-2,\bullet,\lambda}  
\arrow[Leftrightarrow, from=F1, to=G2,shorten >=0.15cm,shorten <=0.15cm,"\pi_{i}"]
\end{tikzcd}
\end{equation}
while the diagrams in \eqref{PIIDEFDI2 EQ}
for $j<k$ and $j \geq k$ generalize to diagrams
\begin{equation}\label{PIIDEFDI2LAB EQ}
\begin{tikzcd}[row sep=1.7em,column sep = 3em]
	\Omega_{G}^{n,\bullet,\lambda} \ar{d}[swap]{s_{j}} \ar{r}{\boldsymbol{V}_G^k} &
	|[alias=F1]|
	\Fin_s \wr \Omega_{G}^{n-k-1,\bullet,\lambda}
	\ar[equal]{d} 
&
	\Omega_{G}^{n,\bullet,\lambda} \ar{d}[swap]{s_{j}} \ar{r}{\boldsymbol{V}_G^k} &
	\Fin_s \wr \Omega_{G}^{n-k-1,\bullet,\lambda}
	\ar{d}{s_{j-k-1}} 
\\
	|[alias=G2]|
	\Omega_{G}^{n+1,\bullet,\lambda} \ar{r}[swap]{\boldsymbol{V}_G^{k+1}}&
	\Fin_s \wr \Omega_{G}^{n-k-1,\bullet, \lambda}  
&
	\Omega_{G}^{n+1,\bullet, \lambda} \ar{r}[swap]{\boldsymbol{V}_G^{k}}&
	\Fin_s \wr \Omega_{G}^{n-k,\bullet,\lambda}  
\end{tikzcd}
\end{equation}
where we note that in all cases the $s$-index $\bullet$
varies according to \eqref{LABSTSIM EQ}.

Lastly, the $\Omega_G^{n,s,\lambda}$ are also functorial in $\lambda$. Explicitly, given 
$\alpha \colon \{1,\cdots,l\} \to \{1,\cdots,m\}$
and partitions such that 
$\lambda' \leq \alpha^{\**} \lambda$
(i.e. $\lambda'_a \subseteq \alpha^{-1}(\lambda_a)$)
one has forgetful functors
\begin{equation}\label{LAMBINC EQ}
	\Omega_G^{n,s,\lambda'}
\to
	\Omega_G^{n,s,\lambda}
\end{equation}
compatible with the forgetful functors \eqref{NKNFGT EQ},
simplicial operators $d_i$, $s_j$, and isomorphisms
$\pi_i$.

\begin{remark}
	When $\alpha$ is the identity 
and $\lambda' \leq \lambda$ the forgetful functors in
\eqref{LAMBINC EQ} are fully faithful inclusions.
	However, this is not the case for the  forgetful functors in \eqref{NKNFGT EQ}.
	Indeed, regarding the map $T \to U$ in
	Example \ref{LABELEDTREES EX}
	as an object in $\Omega_G^{1,0,\lambda}$
	for $\lambda = 
	\lambda_a \amalg \lambda_i = \{2\} \amalg \{1\}
	=\{\bullet\} \amalg \{\circ\}$,
	changing the label of $a_1 \leq a_2$ to a 
	$\bullet$-label produces a non isomorphic object
	$\bar{T} \to U$ of $\Omega_G^{1,0,\lambda}$
	that forgets to the same object of 
	$\Omega_G^{1,1,\lambda}$.
\end{remark}

We now extend Notation \ref{OMEGAGNA NOT}.

\begin{notation}
Let $(A_j)=(A_j \to \Sigma_G)_{1\leq j \leq l}$ be a $l$-tuple of categories over $\Sigma_G$.
We define 
$\Omega_{G}^{n,s,\lambda} \wr (A_j) $
as the pullback
\begin{equation}\label{OMEGAWRTUP EQ}
\begin{tikzcd}[row sep=1em]
	\Omega_{G}^{n,s,\lambda} \wr (A_j) \ar{r}{\boldsymbol{V}_{G}^{n}} \ar{dd}& 
	\Fin_s \wr \coprod A_j \ar{d}
\\
	& \Fin_s \wr \coprod_l \Sigma_G \ar{d}
\\
	\Omega_{G}^{n,s,\lambda} \ar{r}[swap]{\boldsymbol{V}_{G}^{n}} &
	\Fin_s \wr \Omega_G^{-1,s-n-1,\lambda}
\end{tikzcd}
\end{equation}
\end{notation}

\begin{remark}
To unpack \eqref{OMEGAWRTUP EQ}, note first that,
by \eqref{EXTRACASES EQ},
$\Omega_G^{-1,r,\lambda}$ is simply either 
$\Sigma_G^{\amalg l}$ if $r<0$ or 
$\Sigma_G$ if $r \geq 0$,
while $\Omega_G^{n,s,\lambda} = \coprod (\Omega_{G}^{n})^{\times \lambda}$ if $s<0$.
We can thus break down
\eqref{OMEGAWRTUP EQ}
into the three cases
$s<0$, $0 \leq s \leq n$ and $n < s$,
depicted below.
\begin{equation}
\begin{tikzcd}[column sep =1.25em]
	\Omega_{G}^{n,s,\lambda} \wr (A_j) \ar{r}{\boldsymbol{V}_{G}^{n}} \ar{d}& 
	\Fin_s \wr \coprod_j A_j \ar{d}
&
	\Omega_{G}^{n,s,\lambda} \wr (A_j) \ar{r}{\boldsymbol{V}_{G}^{n}} \ar{d}& 
	\Fin_s \wr \coprod_j A_j \ar{d}
&
	\Omega_{G}^{n,s,\lambda} \wr (A_j) \ar{r}{\boldsymbol{V}_{G}^{n}} \ar{d}& 
	\Fin_s \wr \coprod_j A_j \ar{d}
\\
	\coprod (\Omega_{G}^{n})^{\times \lambda} \ar{r}[swap]{\boldsymbol{V}_{G}^{n}} &
	\Fin_s \wr \coprod_l \Sigma_G
&
	\Omega_{G}^{n,s,\lambda} \ar{r}[swap]{\boldsymbol{V}_{G}^{n}} &
	\Fin_s \wr \coprod_l \Sigma_G
&
	\Omega_{G}^{n} \ar{r}[swap]{\boldsymbol{V}_{G}^{n}} &
	\Fin_s \wr \Sigma_G
\end{tikzcd}
\end{equation}
Therefore, for $s>n$, 
\eqref{OMEGAWRTUP EQ} coincides with 
$\Omega_G^{n} \wr (\coprod_j A_j)$
as defined in Notation \ref{OMEGAGNA NOT}.
Moreover, for $s<0$ both squares in the diagram below
are pullbacks and the bottom composite is $\boldsymbol{V}_G^n$,
\begin{equation}\label{BOTTOM EQ}
\begin{tikzcd}[column sep = 4em]
	\coprod (\Omega_{G}^{n})^{\times \lambda} \wr (A_j) 
	\ar{r}{\coprod (\boldsymbol{V}_{G}^{n})^{\times \lambda}} \ar{d}&
	\coprod \Fin_s \wr A_j \ar{r} \ar{d} & 
	\Fin_s \wr \coprod_j A_j \ar{d}
\\
	\coprod (\Omega_{G}^{n})^{\times \lambda} \ar{r}[swap]{\coprod (\boldsymbol{V}_{G}^{n})^{\times \lambda}} &
	\coprod_l \Fin_s \wr \Sigma_G \ar{r} &
	\Fin_s \wr \coprod_l \Sigma_G
\end{tikzcd}
\end{equation}
so that there is an identification
$\Omega_{G}^{n,s,\lambda} \wr (A_j)\simeq 
\coprod (\Omega_{G}^{n})^{\times \lambda} \wr (A_j)$, 
where in the right side $(\minus)\wr (\minus)$ is computed entry-wise.
\end{remark}

\begin{remark} \label{NATTLABEL REM}
The naturality of
the $\Omega_{G}^{n,s,\lambda} \wr (A_j)$ constructions
with regards to $\lambda$ interacts with the tuple $(A_j)$
in the obvious way, i.e.,
given $\alpha \colon \{1,\cdots,l\} \to \{1,\cdots,m\}$,
$\lambda' \leq \alpha^{\**} \lambda$
and a map $(B_k) \to \alpha^{\**}(A_j)$ one obtains a natural map
\[\Omega_{G}^{n,s,\lambda'} \wr (B_k) \to 
\Omega_{G}^{n,s,\lambda} \wr (A_j).\]
\end{remark}

\begin{proposition}\label{PIIPROPAB PROP}
The analogue 
of Proposition \ref{PIIPROP PROP}
holds for the $\Omega_{G}^{n,s,\lambda}$.
In particular 
(we keep the numbering in 
Proposition \ref{PIIPROP PROP}
and have the $s$-index $\bullet$
vary as in \eqref{LABSTSIM EQ}):
\begin{enumerate}
\item[(a)]
The following composite matches $\boldsymbol{V}_G^{k+l+1}$
\[
\begin{tikzcd}
	\Omega_G^{n,\bullet,\lambda} \ar{r}{\boldsymbol{V}_G^k} &
	\Fin_s \wr \Omega_G^{n-k-1,\bullet,\lambda} \ar{r}{\boldsymbol{V}_G^l} &
	\Fin_s^{\wr 2} \wr \Omega_G^{n-k-l-2,\bullet,\lambda} \ar{r}{\sigma^0} &
	\Fin_s \wr \Omega_G^{n-k-l-2,\bullet,\lambda}
\end{tikzcd}
\]
\item[(d)] The rightmost diagrams in both 
\eqref{PIIDEFDILAB EQ} and 
\eqref{PIIDEFDI2LAB EQ} are pullback diagrams in $\mathsf{Cat}$.

\item[(e)]
For $i < k \leq n$ the composite natural transformation in the diagram below is $\pi_i$.
\begin{equation}\label{INDPI1SL EQ}
\begin{tikzcd}[row sep=1.7em,column sep = 3.5em]
	\Omega_{G}^{n,\bullet,\lambda} \ar{d}[swap]{d_{i}} \ar{r}{\boldsymbol{V}_G^k} &
	|[alias=F1]|
	\Fin_s \wr \Omega_{G}^{n-k-1,\bullet,\lambda} \ar{r}{\Fin_s \wr \boldsymbol{V}_{G}^l} 
	\ar[equal]{d} &
	\Fin_s^{\wr 2} \wr \Omega_G^{n-k-l-2,\bullet,\lambda} \ar[equal]{d} \ar{r}{\sigma^0} &
	\Fin_s \wr \Omega_G^{n-k-l-2,\bullet,\lambda} \ar[equal]{d}
\\
	|[alias=G2]|
	\Omega_{G}^{n-1,\bullet,\lambda} \ar{r}[swap]{\boldsymbol{V}_G^{k-1}}&
	\Fin_s \wr \Omega_{G}^{n-k-1,\bullet,\lambda} \ar{r}[swap]{\Fin_s \wr \boldsymbol{V}_{G}^{l}} &
	\Fin_s^{\wr 2} \wr  \Omega_G^{n-k-l-2,\bullet,\lambda} \ar{r}[swap]{\sigma^0} &
	\Fin_s \wr  \Omega_G^{n-k-l-2,\bullet,\lambda}
\arrow[Leftrightarrow, from=F1, to=G2,shorten >=0.15cm,shorten <=0.15cm,"\pi_{i}"]
\end{tikzcd}
\end{equation}
For $k< i < k+l+1 \leq n$ the composite natural transformation in the diagram below is $\pi_{i}$.
\begin{equation}\label{INDPI2SL EQ}
\begin{tikzcd}[row sep=1.7em,column sep = 3.5em]
	\Omega_{G}^{n,\bullet,\lambda} \ar{d}[swap]{d_{i}} \ar{r}{\boldsymbol{V}_G^k} &
	\Fin_s \wr \Omega_{G}^{n-k-1,\bullet,\lambda} \ar{r}{\Fin_s \wr \boldsymbol{V}_{G}^l} 
	\ar{d}[swap]{\Fin_s \wr d_{i-k-1}} &
	|[alias=F1]|
	\Fin_s^{\wr 2} \wr \Omega_G^{n-k-l-2,\bullet,\lambda} \ar[equal]{d} \ar{r}{\sigma^0} &
	\Fin_s \wr \Omega_G^{n-k-l-2,\bullet,\lambda} \ar[equal]{d}
\\
	\Omega_{G}^{n-1,\bullet,\lambda} \ar{r}[swap]{\boldsymbol{V}_G^k}&
	|[alias=G2]|
	\Fin_s \wr \Omega_{G}^{n-k-2,\bullet,\lambda} \ar{r}[swap]{\Fin_s \wr \boldsymbol{V}_{G}^{l-1}} &
	\Fin_s^{\wr 2} \wr  \Omega_G^{n-k-l-2,\bullet,\lambda} \ar{r}[swap]{\sigma^0} &
	\Fin_s \wr  \Omega_G^{n-k-l-2,\bullet,\lambda}
\arrow[Leftrightarrow, from=F1, to=G2,shorten >=0.15cm,shorten <=0.15cm,"\Fin_s \wr \pi_{i-k-1}"]
\end{tikzcd}
\end{equation}
\end{enumerate}
Moreover, the analogue claim holds for the 
$\Omega_{G}^{n,s,\lambda} \wr (A_j)$ constructions
(with the caveat that we exclude the cases of (d)
that involve $d_n$). 

Additionally, the natural squares  (for $n \geq -1$)
\begin{equation}\label{ADDSQUARE EQ}
\begin{tikzcd}
	\Omega_{G}^{n,n,\lambda}
	\ar{r}{\boldsymbol{V}_{G}^{n}} \ar{d}& 
	\Fin_s \wr \coprod_l \Sigma_G \ar{d}
\\
	\Omega_{G}^{n} \ar{r}[swap]{\boldsymbol{V}_{G}^{n}} &
	\Fin_s \wr \Sigma_G
\end{tikzcd}
\end{equation}
are also pullback squares.
\end{proposition}

\begin{proof}
	Firstly, we note that the $\Omega_{G}^{n,s,\lambda}$
	analogues, as well as the claim for \eqref{ADDSQUARE EQ}, all follow from the previous results
	by keeping track of the labels on the strings, 
	with the only non immediate part
	being the analogue of (d), stating that the right squares in 
	\eqref{PIIDEFDILAB EQ} and
	\eqref{PIIDEFDI2LAB EQ} are pullbacks. Since in these diagrams the $s$-coordinate $\bullet$ is determined by the top left corner, a direct analysis shows that compatible choices of labels for strings on the top right and bottom left corners assemble into the required labels on the top left corner, 
	and the result follows.
		
	For the more general $\Omega_{G}^{n,s,\lambda} \wr (A_j)$ constructions, one can either build the
	general $\boldsymbol{V}_G^k$, $d_i$, $s_j$, $\pi_i$ 
	explicitly, or mimic the argument in Proposition \ref{PIIPROPA PROP}, reducing to the 
	$\Omega_{G}^{n,s,\lambda}$ case.
\end{proof}

\begin{corollary}\label{LABIDEN COR}
For $-1 \leq s \leq n$ there are natural identifications
\begin{equation}\label{MOREIDENT1 EQ}
	\Omega_G^{k} \wr \Omega_G^{n,s,\lambda} \wr (A_j) \simeq
	\Omega_G^{n+k+1,s+k+1,\lambda} \wr (A_j)
 \phantom{||||}
	\Omega_G^{n,s,\lambda} \wr 
	(\Omega_G^k)^{\times \lambda}
	\wr (A_j)
\simeq
	\Omega_G^{n+k+1,s,\lambda} \wr (A_j)
\end{equation}
which identify 
$V^k_G \wr \Omega_G^{n,s,\lambda} \wr (A_j) $ with 
$V^k_G \wr (A_j) $
and 
$\boldsymbol{V}_G^n \wr (\Omega_G^k)^{\times \lambda}\wr (A_j) $
with 
$\boldsymbol{V}_G^n \wr (A_j)$.

Further, these identifications are compatible with each other, associative in the obvious ways, and they induce identifications
\begin{equation}\label{MOREIDENT2 EQ}
\begin{tikzcd}[row sep=0, column sep = 10]
	d_i \wr (\Omega_G^{n})^{\times \lambda} \simeq d_i 
&
	\pi_i \wr (\Omega_G^{n})^{\times \lambda} \simeq \pi_i 
&
	s_j \wr (\Omega_G^{n})^{\times \lambda} \simeq s_j 
\\
	\Omega_G^k \wr (d_i)^{\times \lambda} \simeq d_{i+k+1} 
&
	\Omega_G^k \wr (s_j)^{\times \lambda} \simeq s_{j+k+1}
\end{tikzcd}
\end{equation}
as well as obvious identifications for the forgetful functors
in \eqref{NKNFGT EQ}.
\end{corollary}

\begin{proof}
This is analogous to Corollary \ref{IDEN COR}. For the left identification in
\eqref{MOREIDENT1 EQ}, the case $s \geq 0$ follows
since both sides compute the limit of the solid diagram below, 
where we note that the bottom arrow is
$\boldsymbol{V}_G^k \colon \Omega_G^k \to \Fin_s \wr \Sigma_G$
(this is used to regard the diagram as computing
$\Omega^k_G \wr (-)$).
\[
\begin{tikzcd}[row sep = 1.5em]
	\bullet \ar[dashed]{r} \ar[dashed]{d}&
	\bullet \ar[dashed]{r} \ar[dashed]{d}&
	\Fin_s^{\wr 2} \wr \coprod (A_j) \ar{r}{\sigma^0} \ar{d}&
	\Fin_s \wr \coprod (A_j) \ar{d}
\\
	\Omega_G^{n+k+1,s+k+1,\lambda} \ar{r}[swap]{\boldsymbol{V}_G^k} \ar{d}[swap]{d_{k+1,\cdots,n+k+1}}&
	\Fin_s \wr \Omega^{n,s,\lambda}_G \ar{r}[swap]{\Fin_s \wr \boldsymbol{V}_G^n} \ar{d}{d_{0,\cdots,n}}&
	\Fin_s^{\wr 2} \wr \coprod_l \Sigma_G \ar{r}[swap]{\sigma^0} &
	\Fin_s \wr \coprod_l \Sigma_G &
\\
	\Omega_G^{k,k+1,\lambda} \ar{r}[swap]{\boldsymbol{V}_G^k} &
	\Fin_s \wr \Omega_G^{-1,0,\lambda}
\end{tikzcd}
\]
The $s=-1$ case is similar, but since the bottom arrow is now 
$\boldsymbol{V}_G^k \colon \Omega_G^{k,k,\lambda} \to 
\Fin_s \wr \Omega_G^{-1,-1,\lambda} =
\Fin_s \wr \coprod_l \Sigma_G$,
one first attaches 
the pullback square \eqref{ADDSQUARE EQ}
to the bottom of the diagram above.

The right identification in \eqref{MOREIDENT1 EQ} is analogous, using the pullback of the solid diagram below
(for the 
$\Omega_G^{n+k+1,s,\lambda} \wr (A_j)$ side, 
note that \eqref{BOTTOM EQ} identifies
the composite of the central horizontal arrows as 
$\Fin_s \wr V^k_G
\colon
\Fin_s \wr \Omega_G^{k,s-n-1,\lambda}
\to 
\Fin_s^{\wr 2} \wr \coprod_l \Sigma_G$).
\[
\begin{tikzcd}[row sep = 1.5em,column sep =1.8em]
	\bullet \ar[dashed]{r} \ar[dashed]{d}&
	\bullet \ar[dashed]{rr} \ar[dashed]{d}& &
	\Fin_s \wr \coprod \Fin_s \wr A_j \ar{r} \ar{d}&
	\Fin_s^{\wr 2} \wr \coprod A_j \ar{r}{\sigma^0} \ar{d}&
	\Fin_s \wr \coprod A_j \ar{d}
\\
	\Omega_G^{n+k+1,s,\lambda} \ar{r}[swap]{\boldsymbol{V}_G^n} \ar{d}[swap]{d_{n+1,\cdots,n+k+1}}&
	\Fin_s \wr \coprod (\Omega^{k}_G)^{\times \lambda} \ar{rr}[swap]{\Fin_s \wr \coprod (\boldsymbol{V}_G^k)^{\times \lambda}} \ar{d}{d_{0,\cdots,k}}&&
	\Fin_s \wr \coprod_l \Fin_s \wr \Sigma_G \ar{r} &
	\Fin_s^{\wr 2} \wr \coprod_l \Sigma_G \ar{r}[swap]{\sigma^0}  &
	\Fin_s \wr \coprod_l \Sigma_G &
\\
	\Omega_G^{n,s,\lambda} \ar{r}[swap]{\boldsymbol{V}_G^n} &
	\Fin_s \wr \coprod_l \Sigma_G
\end{tikzcd}
\]
The vertex functor claims are straightforward.
The addition claims in \eqref{MOREIDENT2 EQ} follow from 
\eqref{INDPI1SL EQ},\eqref{INDPI2SL EQ},
and the right side of 
\eqref{PIIDEFDI2LAB EQ}.
\end{proof}

\begin{remark}\label{NPXY_REM}
The identifications in Corollary \ref{LABIDEN COR} do allow for the case $n=-1$,
which is non-trivial due to the existence of
 $\Omega_G^{-1,-1,\lambda} = \coprod_l \Sigma_G$,
 in which case $\Omega_G^{-1,-1,\lambda} \wr (A_j) \simeq \coprod A_j$.
For $-1\leq s \leq n$ the identifications
\[
	\Omega_G^{n,s,\lambda} =
	\Omega_G^{s} \wr \Omega_G^{-1,-1} \wr (\Omega_G^{n-s-1})^{\times \lambda}
\]
then show that 
$\Omega_G^{n,s,\lambda} \wr (\minus)$
encodes (the underlying category of) the functor
$N^{\circ s+1} \coprod (N^{\times \lambda})^{\circ n-s}$.

Next, consider the following diagram
(the right depiction merely unpacks the notation on the left),
where the bottom square is one of the pullback squares
in \eqref{ADDSQUARE EQ}.
\begin{equation}\label{NATCOP EQ}
\begin{tikzcd}[column sep = 3.4em,row sep = 0.8em]
	\Omega^{0,-1,\lambda}_G 
	\ar{r}{\coprod (\boldsymbol{V}_G^0)^{\times \lambda}} \ar{d} &
	\coprod \Fin_s \wr (\Omega_G^{-1})^{\times \lambda} \ar{r} & 
	\Fin_s \wr \Omega_G^{-1,-2,\lambda} \ar[equal]{d} 
&
	\coprod (\Omega^0_G)^{\times \lambda} \ar{r} \ar{d} &
	\coprod \Fin_s \wr \Sigma_G \ar{d}
\\
	\Omega^{0,0,\lambda}_G \ar{rr}{\boldsymbol{V}_G^0} \ar{d} &&
	\Fin_s \wr \Omega_G^{-1,-1,\lambda} \ar{d}
&
	\Omega_G^{0,0,\lambda} \ar{r} \ar{d} & 
	\Fin_s \wr \coprod \Sigma_G \ar{d}
\\
	\Omega^{0,1,\lambda}_G \ar{rr}{\boldsymbol{V}_G^0} &&
	 \Fin_s \wr \Omega_G^{-1,0,\lambda}
&
	\Omega_G^0 \ar{r} &
	 \Fin_s \wr \Sigma_G
\end{tikzcd}
\end{equation}
The two representations of the middle horizontal map
yield an identification
$\left(
\Omega_G^{0,0,\lambda} \wr (A_j)
\right)
	\simeq
\left(
\Omega^0_G \wr \coprod A_j
\right)	
$
so that, since the vertical composites fold coproduct summands,
the forgetful map
$\Omega_G^{0,-1,\lambda} \wr (A_j) \to 
\Omega_G^{0,0,\lambda} \wr (A_j)$
is identified with the natural map
$\coprod (\Omega^0_G)^{\times \lambda} \wr (A_j)
\to 
\Omega^0_G \wr \coprod A_j$,
and thus encodes the natural transformation
$\coprod \circ N^{\times \lambda} \Rightarrow N \circ \coprod $
discussed in Remark \ref{PRECOMPPOSTCOMP REM}.
\end{remark}

\subsection{The category of extension trees}
\label{EXTTREE SEC}

The purpose of this section is to make \eqref{EXTTREEFOR EQ} explicit. We start by discussing 
realizations of simplicial objects in $\mathsf{Cat}$.

Recalling the standard cosimplicial object
$[\bullet] \in \mathsf{Cat}^{\Delta}$ given by 
$[n]=(0 \to 1 \to \cdots \to n)$
yields the following definition.

\begin{definition}\label{REAL DEF}
	The left adjoint below is called the 
	\textit{realization} functor.
	\[
	|\minus|\colon
	\mathsf{Cat}^{\Delta^{op}} 
		\rightleftarrows
	\mathsf{Cat} 
	\colon (\minus)^{[\bullet]}
	\]
\end{definition}

\begin{remark}\label{REALEX REM}
Suppose that $\C \in \mathsf{Cat}$ contains subcategories 
$\C_h$, $\C^v$
such that any arrow of $\C$ factors as 
an arrow of $\C_h$ followed by an arrow of $\C_v$.
Defining 
$\mathcal{C}^{v}_{h,\bullet} \in \mathsf{Cat}^{\Delta^{op}}$
so that the objects of $\mathcal{C}^{v}_{h,n}$ are $n$-strings in $\C_h$ and the arrows are compatible $n$-tuples of
arrows in $\C^v$, it is straightforward to show
that it is
$|\mathcal{C}^{v}_{h,\bullet}| = \C$.

An immediate example is given by the planar strings in Definition \ref{PLANSTR DEF}. Writing 
$\C = \Omega_G^{\mathsf{t}}$ for the category of tall maps,
$\C_h = \Omega_G^{\mathsf{pt}}$ the category of planar tall maps,
and $\C^v = \Omega_G^{0}$ the category of quotients,
one has $\C_{h,\bullet}^{v} = \Omega_G^{\bullet}$ and thus
$|\Omega_G^{\bullet}| = \Omega_G^{\mathsf{t}}$.

Similarly, noting that the $\Omega_G^{n,\lambda} = \Omega_G^{n,0,\lambda}$
categories of \S \ref{LABELSTRI SEC} form a simplicial object, we have that the
$|\Omega_G^{\bullet,\lambda}| = \Omega_G^{\mathsf{t},\lambda}$
is the category of tall label maps between
$l$-labeled trees that induce quotients on 
nodes with $\lambda$-inert labels.
\end{remark}

In the following statement, whose proof is delayed to the appendix, we note that 
it is shown in Lemma \ref{OBJGENREL LEMMA}
that $\text{Ob}(|A_{\bullet}|) \simeq \text{Ob}(A_0)$
and that arrows in $|A_{\bullet}|$ are generated by
the arrows in $A_0$ together with arrows 
$d_1(a) \to d_0(a)$ for each $a \in A_1$.

\begin{proposition}\label{RANTRANS PROP}
Given a simplicial object
$\Sigma_G \leftarrow A_\bullet \xrightarrow{N_{\bullet}} \mathcal{V}^{op}$ 
in $\mathsf{WSpan}^r(\Sigma_G,\mathcal{V}^{op})$
such that the natural transformation components of the differential operators 
$d_i$, $0\leq i < n$ and $s_j$, $0 \leq j \leq n$
are isomorphisms,
there is an identification
\begin{align*}
	\lim_{\Delta}
	\left(
	\mathsf{Ran}_{A_n \to \Sigma_G}
	N_{n}
	\right)
	\simeq 
	\mathsf{Ran}_{ |A_{\bullet}| \to \Sigma_G }
	\tilde{N}
\end{align*}
where $\tilde{N}\colon |A_{\bullet}| \to \mathcal{V}^{op}$
is given by $N_0$ on objects and generating arrows 
in $A_0$, and on generating arrows $d_1(a) \to d_0(a)$
for $a \in A_1$ as the composite
natural transformation arrow
\[
\begin{tikzcd}[column sep =3em]
	|[alias=TA]|
	A_0 \ar{rd} & 
	A_1 \ar{l}[swap]{d_1} \ar{d}[name=T]{}[swap,name=B]{}
	\ar{r}{d_0} &
	|[alias=BA]|
	A_0 \ar{ld}
\\
	& \mathcal{V}^{op}
	\arrow[Rightarrow,from=TA,to=T,shorten <=0.15cm,,shorten >=0.15cm]
	\arrow[Leftrightarrow,from=BA,to=B,shorten <=0.15cm,,shorten >=0.15cm]
\end{tikzcd}
\]
\end{proposition}

Proposition \ref{RANTRANS PROP} applies to both simplicial directions of 
the bisimplicial object
\begin{equation}\label{BISIMP EQ}
	N^{(\mathcal P,X,Y)}_{n,l} =
	N ( N^{\circ n} \upsilon \mathcal{P} \amalg
	\upsilon X^{\amalg 2l+1} \amalg \upsilon Y)
\end{equation}
in \eqref{LANLEVELFOR EQ},
whose underlying categories are 
$\Omega_G^{n,\lambda_l}$
for $\lambda_l$ the partitions described at the beginning of
\S \ref{LABELSTRI SEC}.
Indeed, in the $n$ direction, all $d_i$ with $0 < i < n$
are induced by the multiplication $NN \to N$ defined in 
\eqref{MULTDEFSPAN EQ} while $d_0$
is induced by the composite
$N \circ \coprod \circ N \to N N \circ \coprod \to N \circ \coprod$, with the second map again given by composition
and the first induced
by the natural map 
$\coprod \circ N \to N \circ \coprod$, which is encoded by a strictly commutative diagram of spans,
as seen using the top part of \eqref{NATCOP EQ}
(or, more abstractly, 
it also suffices to note that 
$N$ preserves arrows in $\mathsf{WSpan}^l(\Sigma_G^{op},\mathcal{V})$ given by strictly commutative diagrams).
Degeneracies are similar.
Moreover, that the functor component of $d_n$
matches the functor defined in \eqref{LABSTSIM EQ}
follows from the presence of $\upsilon$ in \eqref{LANLEVELFOR EQ}.

As for the $l$ direction, we note that our convention on 
the double bar construction 
$B_l(\mathcal{P}, \mathbb{F}_G X, \mathbb{F}_G X, \mathbb{F}_G X, \mathbb{F}_G Y)$,
is symmetric, 
with $d_l$ given by combining the maps
$\mathbb{F}_G X \to \mathbb{F}_G Y$ 
and 
$\mathbb{F}_G X \to \mathcal{P}$
and the remaining differentials given by fold maps.
Or, more precisely, the action of the differential operators
on the sets of labels
$\langle \langle l \rangle \rangle = 
\{-\infty,-l, \cdots -1,0,1,\cdots,l,+\infty\}$
is given by extending the functions in 
Remark \ref{ORDLABEL REM} anti-symmetrically.
But then the differential operators 
$d_i$, $s_j$ for $0\leq i<l$ and $0\leq j \leq l$
correspond to instances of the naturality in 
Remark \ref{NATTLABEL REM}
when $(B_k) =\alpha^{\**}(A_j)$,
and are hence given by strictly commutative maps of spans.

Our next task is thus that of identifying the category of extension trees $\Omega_G^e$ appearing
in \eqref{EXTTREEFOR EQ},
i.e. to produce an explicit model for the double realization
$|\Omega_G^{n,\lambda_l}|$.
By Remark \ref{REALEX REM},
we can first perform the realization in the $n$ direction, so as to obtain
$|\Omega_G^{n,\lambda_l}|=|\Omega_G^{\mathsf{t},\lambda_l}|$,
where we recall that 
$\Omega_G^{\mathsf{t},\lambda_l}$
consists of $\langle \langle l \rangle \rangle$-labeled trees
together with tall maps that induce quotients on all nodes not labeled by $-\infty$.

We now identify $\Omega_G^{e}$ directly.

\begin{definition}\label{EXTTREECAT DEF}\index{categories!of trees!OMEXT@$\Omega_G^{e}$}
	The \textit{extension tree category $\Omega_G^e$}
	has as objects $\{\mathcal{P},X,Y\}$-labeled trees
	and as arrows tall maps $\varphi \colon T \to S$ such that:
	\begin{itemize}
		\item[(i)] if $T_{v_{Ge}}$ has a $X$-label, then 
		$S_{v_{Ge}} \in \Sigma_G$ and $S_{v_{Ge}}$ has a $X$-label;
		\item[(ii)] if $T_{v_{Ge}}$ has a $Y$-label, then 
		$S_{v_{Ge}} \in \Sigma_G$ and $S_{v_{Ge}}$ has either a $X$-label or a $Y$-label;
		\item[(iii)] if $T_{v_{Ge}}$ has a $\mathcal{P}$-label, then 
		$S_{v_{Ge}}$ has only $X$ and $\mathcal{P}$-labels.
	\end{itemize}
\end{definition}

\begin{example}\label{REGALTERNMAP EX}
The following  is an example of a planar map in $\Omega_G^e$ for $G=\**$, where black nodes represent $\mathcal{P}$-labeled nodes, grey nodes represent $Y$-labeled nodes,
and white nodes represent $X$-labeled nodes.
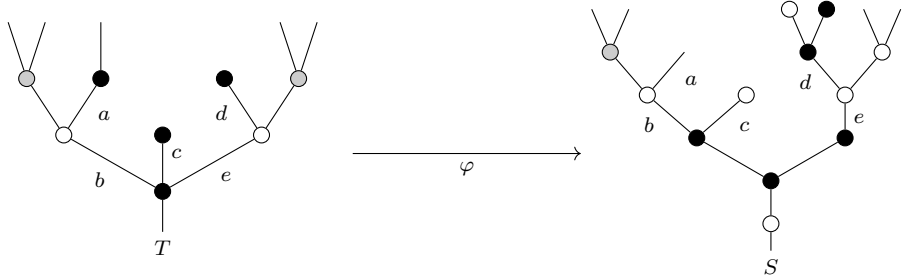
\[
\begin{tikzpicture}[grow=up,auto,level distance=2.3em,
every node/.style = {font=\footnotesize},
dummy/.style={circle,draw,inner sep=0pt,minimum size=2.1mm}]
	\tikzstyle{level 2}=[sibling distance = 4em]
	\tikzstyle{level 3}=[sibling distance = 3em]
	\tikzstyle{level 4}=[sibling distance = 1.5em]
	\node at (0,0.25) {$T$}
		child{node [dummy,fill = black] {}
			child{node [dummy,fill=white] {}
				child{node [dummy,fill = black!20] {}
					child
					child
				}
				child{node [dummy,fill = black] {}
				edge from parent node [near end] {$d$}}
			edge from parent node [swap] {$e$}}
			child{node [dummy,fill=black] {}
			edge from parent node [swap, near end] {$c$}}
			child{node [dummy,fill=white] {}
				child{node [dummy,fill = black] {}
					child
				edge from parent node [swap, near end] {$a\phantom{d}$}}
				child{node [dummy,fill = black!20] {}
					child
					child
				}
			edge from parent node {$b$}}
		};
\begin{scope}[level distance=1.75em]
	\tikzstyle{level 3}=[sibling distance = 6em]
	\tikzstyle{level 4}=[sibling distance = 4em]
	\tikzstyle{level 5}=[sibling distance = 3em]
	\tikzstyle{level 6}=[sibling distance = 1.5em]
	\tikzstyle{level 7}=[sibling distance = 0.75em]
	\node at (8,0) {$S$}
		child{node [dummy,fill = white] {}
			child{node [dummy,fill = black] {}
				child{node [dummy,fill = black] {}
					child{node [dummy,fill=white] {}
						child{node [dummy,fill = white] {}
							child
							child
						}
						child{node [dummy,fill = black] {}
							child{node [dummy,fill=black] {}
						}
							child{node [dummy,fill=white] {}}
						edge from parent node [near end] {$d$}}
					edge from parent node [swap] {$e\phantom{1}$}}
				}
				child{node [dummy,fill = black] {}
					child{node [dummy,fill=white] {}
					edge from parent node [swap, near end] {$c\phantom{1}$}}
					child{node [dummy,fill=white] {}
						child{
						edge from parent node [swap,near end] {$a\phantom{d}$}}
						child{node [dummy,fill = black!20] {}
							child
							child
						}
					edge from parent node [near end] {$\phantom{1}b$}}
				}
			}
		};
\end{scope}
	\draw [->] (2.5,1.5) -- node [swap] {$\varphi$} (5.5,1.5);
\end{tikzpicture}
\]
\end{example}

\begin{remark}
By changing any $X$-labels in $S_{v_{G e}}$ 
into $Y$-labels (resp. $\mathcal{P}$-labels)
whenever $T_{v_{G}}$  has a 
$Y$-label (resp. $\mathcal{P}$-label), one obtains a factorization
\[ T \to \bar{S} \to S \]
such that $T \to \bar{S}$ is a label map 
(cf. Definition \ref{LABMAP DEF})
and $\bar{S} \to S$ is an underlying identity of trees that
merely changes some of the $Y$ and $\mathcal{P}$ labels into 
$X$-labels.
We refer to the latter kind of map as a \textit{relabel map}.
It is clear that the label-relabel factorization 
 is unique.
\end{remark}

\begin{proposition}\label{BISIMP PROP}
There is an identification
$\Omega_G^e \simeq 
|\Omega_{G}^{\mathsf{t},\lambda_l}|$.
\end{proposition}

\begin{proof}
We will show that Remark \ref{REALEX REM} applies to 
$\mathcal{C} = \Omega_G^e$,
with $\mathcal{C}_h$ and $\mathcal{C}^v$ the categories of 
relabel and label maps.
More precisely, we claim that there is an isomorphism 
$\mathcal{C}_{h,\bullet}^{v} \simeq 
\Omega_{G}^{\mathsf{t},\lambda_{\bullet}}$
of objects in $\mathsf{Cat}^{\Delta^{op}}$.
Unpacking notation, one must first show that strings
\begin{equation}\label{RELABSTR EQ}
T_0 \to T_1 \to \cdots \to T_l
\end{equation}
 of relabel arrows in $\Omega_G^e$
 are in bijection with objects of 
 $\Omega_{G}^{\mathsf{t},\lambda_l}$,
 i.e., with trees labeled by
 $\langle \langle l \rangle \rangle =
  \{-\infty, -l, \cdots, -1,0,1,\cdots,l,+ \infty\}$.
Noting that the maps in
\eqref{RELABSTR EQ}
are simply underlying identities on some fixed tree $T$
that convert some of the $\mathcal{P}$, $Y$ labels into $X$ labels,
we label a vertex $T_{v_{Ge}}$ by:
\begin{enumerate*}
\item[(i)]
$j$ such that
$0 < j \leq +\infty$
if the last $j$ labels of $T_{v_{Ge}}$ in 
\eqref{RELABSTR EQ} are $Y$ labels (where $+\infty = l+1$); 
\item[(ii)]
$-j$ such that
$-\infty \leq -j < 0$
if the last $j$ labels of $T_{v_{Ge}}$ in 
\eqref{RELABSTR EQ} are $\mathcal{P}$ labels;
\item[(iii)] $j=0$ if all labels in \eqref{RELABSTR EQ}
are $X$-labels.
\end{enumerate*}
 This process clearly establishes the desired bijection on objects.

The compatibilities with arrows and with the simplicial structure are straightforward.
\end{proof}

\begin{remark}\label{TILNUNPACK REM}
	Regarding \ref{BISIMP EQ} as a 
	bisimplicial object
	$\Sigma_G 
	\leftarrow 
	\Omega_G^{\bullet,\lambda_{\bullet}}
	\xrightarrow{N_{\bullet,\bullet}^{(\mathcal P, X,Y)}}
	\mathcal{V}^{op}$
	in $\mathsf{WSpan}^r(\Sigma_G,\mathcal{V}^{op})$,
	we have now identified the double realization
	$|\Omega_G^{n,\lambda_{l}}|$ as $\Omega^e_G$,
	and thus a double application of Proposition \ref{RANTRANS PROP}
	builds an associated functor
	$\tilde{N}^{(\mathcal{P},X,Y)} \colon 
	\Omega^e_G \to \mathcal{V}^{op}$.
	
	Unpacking the construction in Proposition \ref{RANTRANS PROP},
	this $\tilde{N}^{(\mathcal{P},X,Y)}$
	is described as follows.
	On objects $T \in \Omega^e_G$,
	which are identified with objects
	$T \in \Omega_G^{0,\lambda_{0}}$,
	i.e. $(\mathcal{P},X,Y)$-labeled trees, one has
\begin{equation}\label{NPXY EQ}
	\tilde{N}^{(\mathcal{P},X,Y)}(T) 
\simeq
	\bigotimes\limits_{v \in V_{G}^{\P}(T)}\P(T_v) \otimes
	\bigotimes\limits_{v \in V_{G}^{X}(T)}X(T_v) \otimes
	\bigotimes\limits_{v \in V_{G}^{Y}(T)}Y(T_v).
\end{equation}
As for arrows, note first that (as discussed before Proposition \ref{RANTRANS PROP})
Lemma \ref{OBJGENREL LEMMA} says that
the arrows of 
$|\Omega_G^{n,\lambda_{l}}|$
are generated by three types of arrows:
\begin{enumerate*}
	\item[(i)] arrows in $\Omega_G^{0,\lambda_{0}}$;
	\item[(ii)] arrows determined by an object
	$(T_0 \to T_1) \in \Omega_G^{1,\lambda_{0}}$;
	\item[(iii)]
	arrows determined by an object
	$T \in \Omega_G^{0,\lambda_{1}}$.
\end{enumerate*}
In terms of $\Omega^e_G$, one has that:
type (i) arrows are the (labeled) quotients, 
which act on \eqref{NPXY EQ}
via permutations, diagonal maps, and 
the $G$-symmetric sequence structure maps of
$\mathcal{P},X,Y \in \mathsf{Sym}_G(\mathcal{V})$;
type (ii) arrows are the planar label maps
(necessarily inert on $X,Y$),
which act on \eqref{NPXY EQ}
via the genuine operad structure on $\mathcal{P}$;
type (iii) arrows are the planar relabel maps
(since the proof of Proposition \eqref{BISIMP PROP}
identifies each $T \in \Omega_G^{0,\lambda_{1}}$
with a planar relabel map $T_0 \to T_1$ in $\Omega^e_G$),
which act on \eqref{NPXY EQ} via the given maps
$X \to \mathcal{P}$, $X \to Y$
in \eqref{FREE_FG_EXT_EQ}.
\end{remark}

Proposition \ref{RANTRANS PROP} now yields the following, establishing \eqref{EXTTREEFOR EQ}.

\begin{corollary}\label{ESTABDESC COR}
	$\mathcal P \coprod\limits_{\mathbb F X} \mathbb F Y \simeq \mathsf{Lan}_{(\Omega_G^e \to \Sigma_G)^{op}}\tilde N^{(\mathcal P, X,Y)}$.
\end{corollary}

Our next task is to identify a convenient initial subcategory
$\widehat{\Omega}_G^{e} \hookrightarrow \Omega_G^e$.
We first introduce the auxiliary notion of \emph{alternating trees}.
Recall the notion of input path (cf. Notation \ref{INPUTPATH NOT})
$I(e) = \{f \in T \colon e \leq_d f\}$ for an edge $e \in T$, which naturally extends to $T$ in any of $\Omega, \Phi, \Omega_G, \Phi^G$.

\begin{definition}\label{OMEGAA DEF}\index{categories!of trees!OMEXT@$\Omega_G^{a}$}
A $G$-tree $T \in \Omega_G$ is called \textit{alternating} if, 
for all leafs $l \in T$,
one has that the input path $I(l)$ has an even number of elements.

Further, a vertex $e^{\uparrow} \leq e$ is called \textit{active}
if $|I(e)|$ is odd and \textit{inert} otherwise.

Finally, a tall map $T \xrightarrow{\varphi} S$ between alternating $G$-trees is called a 
\textit{tall alternating map} if,
for any inert vertex $e^{\uparrow} \leq e$ of $T$,
one has that $S_{e^{\uparrow} \leq e}$ is an inert vertex of $S$.

We will denote the category of alternating $G$-trees and tall alternating maps by $\Omega_G^a$.
\end{definition}

\begin{remark}
	A $G$-tree (resp. map) is alternating
	iff each component is.
\end{remark}

\begin{example}
Two alternating trees (for $G=\**$ the trivial group) and a planar tall alternating map between them follow, with active nodes in black ($\bullet$) and white nodes in white ($\circ$).
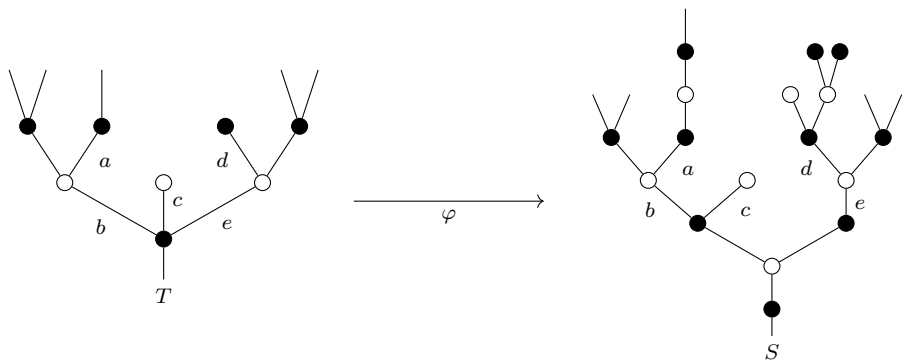
\[
\begin{tikzpicture}[grow=up,auto,level distance=2.3em,every node/.style = {font=\footnotesize},dummy/.style={circle,draw,inner sep=0pt,minimum size=2.1mm}]
	\tikzstyle{level 2}=[sibling distance = 4em]
	\tikzstyle{level 3}=[sibling distance = 3em]
	\tikzstyle{level 4}=[sibling distance = 1.5em]
	\node at (0,0.75) {$T$}
		child{node [dummy,fill = black] {}
			child{node [dummy,fill=white] {}
				child{node [dummy,fill = black] {}
					child
					child
				}
				child{node [dummy,fill = black] {}
				edge from parent node [near end] {$d$}}
			edge from parent node [swap] {$e$}}
			child{node [dummy,fill=white] {}
			edge from parent node [swap, near end] {$c$}}
			child{node [dummy,fill=white] {}
				child{node [dummy,fill = black] {}
					child
				edge from parent node [swap, near end] {$a\phantom{d}$}}
				child{node [dummy,fill = black] {}
					child
					child
				}
			edge from parent node {$b$}}
		};
\begin{scope}[level distance=1.75em]
	\tikzstyle{level 3}=[sibling distance = 6em]
	\tikzstyle{level 4}=[sibling distance = 4em]
	\tikzstyle{level 5}=[sibling distance = 3em]
	\tikzstyle{level 6}=[sibling distance = 1.5em]
	\tikzstyle{level 7}=[sibling distance = 1em]
	\node at (8,0) {$S$}
		child{node [dummy,fill = black] {}
			child{node [dummy,fill = white] {}
				child{node [dummy,fill = black] {}
					child{node [dummy,fill=white] {}
						child{node [dummy,fill = black] {}
							child
							child
						}
						child{node [dummy,fill = black] {}
							child{node [dummy,fill=white] {}
								child{node [dummy,fill = black] {}}
								child{node [dummy,fill = black] {}}
						}
							child{node [dummy,fill=white] {}}
						edge from parent node [near end] {$d$}}
					edge from parent node [swap] {$e\phantom{1}$}}
				}
				child{node [dummy,fill = black] {}
					child{node [dummy,fill=white] {}
					edge from parent node [swap, near end] {$c\phantom{1}$}}
					child{node [dummy,fill=white] {}
						child{node [dummy,fill = black] {}
							child{node [dummy,fill=white] {}
								child{node [dummy,fill = black] {}
									child
								}
							}
						edge from parent node [swap,near end] {$a\phantom{d}$}}
						child{node [dummy,fill = black] {}
							child
							child
						}
					edge from parent node [near end] {$\phantom{1}b$}}
				}
			}
		};
\end{scope}
	\draw [->] (2.5,2) -- node [swap] {$\varphi$} (5,2);
\end{tikzpicture}
\]
The term ``alternating'' reflects the fact that adjacent nodes have different colors, though there is an additional restriction: the ``outer vertices'', i.e. those immediately below a leaf or above the root, are necessarily black/active
(this does not, however, apply to stumps).
\end{example}

\begin{remark}\label{ALTSUB REM}
	If $T \in \Omega$ is alternating, 
	Remark \ref{INPPATH REM} implies that a tall map 
	$\varphi \colon T \to U$ is an alternating map
	iff the corresponding substitution datum 
	in Proposition \ref{SUBDATAUNDERPLAN PROP}
	is given by an isomorphism 
	$U_{e^{\uparrow} \leq e } \simeq T_{e^{\uparrow} \leq e}$
	for inert $e^{\uparrow} \leq e$, 
	and by an alternating tree
	$U_{e^{\uparrow} \leq e }$ for active 
	$e^{\uparrow} \leq e$.
\end{remark}

\begin{definition}\label{HATOMEGAE DEF}
	\index{categories!of trees!OMEXT@$\widehat{\Omega}_G^{e}$}
	$\widehat{\Omega}_G^e \hookrightarrow \Omega_G^e$ is the full subcategory of $(\mathcal{P},X,Y)$-labeled trees
	whose underlying tree is alternating
	and active (resp. inert) nodes are labeled by $\mathcal{P}$
	(by $X$ or $Y$). 
\end{definition}

Note that conditions (i) and (ii) in Definition \ref{EXTTREECAT DEF} 
imply that,
for any map in $\widehat{\Omega}_G^e$,
the underlying map is an alternating map.

The following is the key to establishing the desired initiality of 
$\widehat{\Omega}_G^e$ in $\Omega_G^e$.

\begin{proposition}\label{LXP PROP}
	For each $U \in \Omega_G^e$ there exists a unique 
	$\mathsf{lr}_{\mathcal{P}} (U) \in \widehat{\Omega}_G^e$ together with a unique planar label map in $\Omega_G^e$
\begin{equation}\label{LXP EQ}
	\mathsf{lr}_{\mathcal{P}} (U) \to U.
\end{equation}
	Furthermore, $\mathsf{lr}_{\mathcal{P}}$ extends to a right retraction 
	$\mathsf{lr}_{\mathcal{P}} \colon \Omega_G^e \to \widehat{\Omega}_G^e$.
\end{proposition}

Formally, the map \eqref{LXP EQ} in  Proposition \ref{LXP PROP}
will be built using Proposition \ref{BUILDABLE PROP}(iii),
which loosely says that planar tall maps $T \to U$
are determined by certain collections $\{U_i\}$
of outer faces of $U$,
with $T$ obtained by replacing 
$U_i$ with $\mathsf{lr}(U_i)$
(for the pictorial intuition, see Example \ref{GRAFTSUB EX}).
For the sake of intuition,
we first present an example.

\begin{example}\label{LRP EX}
	The following illustrates the $\mathsf{lr}_{\mathcal{P}}$ construction applied to the map $\varphi$ in
	Example \ref{REGALTERNMAP EX}. 
	Intuitively, 
	for each of the maximal $\mathcal{P}$-labeled outer subtrees
	$T^{\mathcal{P}}_k, S^{\mathcal{P}}_k$
	of $T,S$,
	the functor
	$\mathsf{lr}_{\mathcal{P}}$ replaces 
	$T^{\mathcal{P}}_k, S^{\mathcal{P}}_k$ with the corresponding leaf-root
	$\mathsf{lr}(T^{\mathcal{P}}_k),
	\mathsf{lr}(S^{\mathcal{P}}_k)$, which is again 
	$\mathcal{P}$-labeled.
	Pictorially, this results in the following two effects:
	when
	$T^{\mathcal{P}}_k, S^{\mathcal{P}}_k$ are single edge subtrees of $T,S$ (necessarily not adjacent to a $\mathcal{P}$-vertex)
	one degenerates that edge, adding a new $\mathcal{P}$-vertex of degree $1$;
	when
	$T^{\mathcal{P}}_k, S^{\mathcal{P}}_k$ have vertices,
	so that they are subtrees composed of adjacent 
	$\mathcal{P}$-vertices of $T,S$,
	those vertices are collapsed into a single  
	$\mathcal{P}$-vertex.
	\[
	\begin{tikzpicture}[grow=up,auto,level distance=2.3em,
	every node/.style = {font=\footnotesize,inner sep=2pt},
	dummy/.style={circle,draw,inner sep=0pt,minimum size=2.1mm}]
	\tikzstyle{level 2}=[sibling distance = 6em]
	\tikzstyle{level 3}=[sibling distance = 3em]
	\tikzstyle{level 4}=[sibling distance = 1.5em]
	\node at (0,0.85) {$\mathsf{lr}_{\mathcal{P}}(T)$}
	child{node [dummy,fill = black] {}
		child{node [dummy,fill=white] {}
			child{node [dummy,fill = black] {}
				child{node [dummy,fill = black!20] {}
					child{node [dummy,fill = black] {}
						child}
					child{node [dummy,fill = black] {}
						child}
				}
			}
			child{node [dummy,fill = black] {}
				edge from parent node [near end] {$d$}}
			edge from parent node [swap] {$e$}}
		child{node [dummy,fill=white] {}
			child{node [dummy,fill = black] {}
				child
				edge from parent node [swap, near end] {$a\phantom{d}$}}
			child{node [dummy,fill = black] {}
				child{node [dummy,fill = black!20] {}
					child{node [dummy,fill = black] {}
						child}
					child{node [dummy,fill = black] {}
						child}
				}
			}
			edge from parent node {$b$}}
	};
	\begin{scope}[level distance=1.75em]
	\tikzstyle{level 3}=[sibling distance = 6em]
	\tikzstyle{level 4}=[sibling distance = 4em]
	\tikzstyle{level 5}=[sibling distance = 3em]
	\tikzstyle{level 6}=[sibling distance = 1.5em]
	\tikzstyle{level 7}=[sibling distance = 1.5em]
	\node at (8,0.85) {$\mathsf{lr}_{\mathcal{P}}(S)$}
	child{node [dummy,fill = black] {}
		child{node [dummy,fill = white] {}
			child{node [dummy,fill = black] {}
				child{node [dummy,fill=white] {}
					child{node [dummy,fill = black] {}
						child{node [dummy,fill = white] {}
							child{node [dummy,fill = black] {}
								child}
							child{node [dummy,fill = black] {}
								child}
						}
					}
					child{node [dummy,fill = black] {}
						child{node [dummy,fill=white] {}}
						edge from parent node [near end] {$d$}}
					edge from parent node [swap] {$e\phantom{1}$}}
				child{node [dummy,fill=white] {}
					edge from parent node [swap, near end] {$c\phantom{1}$}}
				child{node [dummy,fill=white] {}
					child{node [dummy,fill=black] {}
						child
						edge from parent node [swap,near end] {$a\phantom{d}$}}
					child{node [dummy,fill=black] {}
						child{node [dummy,fill = black!20] {}
							child{node [dummy,fill=black] {}
								child
							}
							child{node [dummy,fill=black] {}
								child
							}
						}
					}
					edge from parent node {$\phantom{1}b$}}
			}
		}
	};
	\end{scope}
	\draw [->] (2.5,2.5) -- node [swap] {$\mathsf{lr}_{\mathcal{P}}(\varphi)$} (5.5,2.5);
	\end{tikzpicture}
	\]
\end{example}

\begin{proof}[Proof of Proposition \ref{LXP PROP}]
	We first address the non-equivariant case $U \in \Omega^e$.

	To build $\mathsf{lr}_{\mathcal{P}}(U)$, consider the collection of outer faces
	$\{U_i^X\} \amalg \{U_j^Y\} \amalg \{U_k^{\mathcal{P}}\}$
where the $U_i^X$, $U_j^Y$ are simply the $X,Y$-labeled nodes,
and the $\{U_k^{\mathcal{P}}\}$ are the maximal outer subtrees whose nodes have only $\mathcal{P}$-labels (these may possibly be sticks). 
Lemma \ref{OUTERFACEUNION LEM} guarantees that 
each edge and each $\mathcal{P}$-labeled node belong to exactly
one of the $U_k^{\mathcal{P}}$, 
so that applying Proposition \ref{BUILDABLE PROP}(iii)
yields a planar tall map
\begin{equation}\label{LRXDEF EQ}
T = \mathsf{lr}_{\mathcal{P}}(U) \to U
\end{equation}
such that $\{U_{e^{\uparrow} \leq e}\}_{(e^{\uparrow} \leq e) \in V(T)}
 = \{U_i^X\} \amalg \{U_j^Y\} \amalg 
 \{U_k^{\mathcal{P}}\}$. 
 $T$ has an obvious $(\mathcal{P},X,Y)$-labeling making 
\eqref{LRXDEF EQ} into a label map, but we must still check $T \in \widehat{\Omega}^{e}_G$, i.e. that 
$T$ is alternating with active vertices precisely those labeled by $\mathcal{P}$.
But, since the image of each $e \in T$
belongs to precisely one $U_k^{\mathcal{P}}$,
$e$ belongs to precisely one of the $\mathcal{P}$-labeled nodes of $T$, so that any leaf input path
$I(l) = (l = e_n \leq e_{n-1} \leq \cdots \leq e_1 \leq e_0)$
in $T$ must start with, end with, and alternate between 
$\mathcal{P}$-nodes, and thus have even length.

To check uniqueness note that, 
for any other planar label map $S \to U$ with 
$S \in \widehat{\Omega}_G^e$
and $e^{\uparrow} \leq e$ a $\mathcal{P}$ vertex of $S$,
the outer face 
$U_{e^{\uparrow} \leq e}$
must be a maximal 
$\mathcal{P}$-labeled outer face since the vertices adjacent to its root and leaves are labeled by either $X$ or $Y$.
The condition 
$V(U) = \coprod_{V(S)} V(U_{e^{\uparrow} \leq e})$
thus guarantees that the collection of outer faces determined by $S$ matches that determined by $T$
except perhaps in the number of stick faces, so that 
the degeneracy-face factorizations
$S \to F \to U$, $T \to F \to U$
factor through the same planar inner face $F$, with the  unique labeling that makes the inclusion a label map.
$S$, $T$ are thus both trees in $\widehat{\Omega}_G^e$ obtained from $F$ by adding degenerate $\mathcal{P}$ vertices, 
and since this can be done in at most one way, we conclude 
$S=T$. 

To check functoriality,
consider the diagram below, where $T \to U$ is the map \eqref{LRXDEF EQ}
defined above and $\varphi \colon U \to V$ any map in $\Omega_G^e$.
\begin{equation}\label{LRPFUN EQ}
\begin{tikzcd}
	T \ar{r} \ar[dashed]{d} &  U \ar{d}{\varphi}
\\
	S \ar[dashed]{r} & V
\end{tikzcd}
\end{equation}
The composite $T \to V$ is encoded by a substitution datum
$\{T_{e^{\uparrow} \leq e} \to V_{e^{\uparrow} \leq e}\}$
which is given by an isomorphism
if $e^{\uparrow} \leq e$ has label $X$ or $Y$ (possibly changing a $Y$-label to a $X$-label),
and by some $(X,\mathcal{P})$-labeled tree 
$V_{e^{\uparrow} \leq e}$ if $e^{\uparrow} \leq e$
has a $\mathcal{P}$-label.
We now consider the factorization problem 
in \eqref{LRPFUN EQ}, where we want $S \in \widehat{\Omega}_G^e$
and for the map $S \to V$ to the a planar label map.
Combining Remark \ref{ALTSUB REM} with the uniqueness of the
$\mathsf{lr}_{\mathcal{P}}(V_{e^{\uparrow} \leq e})$,
the only possibility is for $S$ to be defined using the 
$T$ substitution datum
that replaces 
$T_{e^{\uparrow} \leq e} \to V_{e^{\uparrow} \leq e}$
with 
$T_{e^{\uparrow} \leq e} \to 
\mathsf{lr}_{\mathcal{P}}(V_{e^{\uparrow} \leq e})$
whenever $e^{\uparrow} \leq e$ has a $\mathcal{P}$-label.
Uniqueness of $\mathsf{lr}_{\mathcal{P}} (V)$ then implies 
$S=\mathsf{lr}_{\mathcal{P}} (V)$, and one sets 
$\mathsf{lr}_{\mathcal{P}} (\varphi)$
to be the map $T\to S$.
Associativity and unitality are automatic from the uniqueness of the factorization of \eqref{LRPFUN EQ}.

For $T = (T_x)_{x \in X}$ in 
$\Omega_G^e$ with $G$ a general group,
one sets
$\mathsf{lr}_{\mathcal{P}}(T) = (\mathsf{lr}_{\mathcal{P}}(T_x))_{x \in X}$.
\end{proof}

\begin{corollary}\label{KANRED COR}
The inclusion 
$\widehat{\Omega}_G^e \hookrightarrow \Omega_G^e$ 
is $\mathsf{Ran}$-initial over $\Sigma_G$.
In other words, for $\C$ any complete category and 
functor $N \colon \Omega_G^e \to \C$ it is
\[
\mathsf{Ran}_{\Omega_G^e \to \Sigma_G} N
	\simeq 
\mathsf{Ran}_{\widehat{\Omega}_G^e \to \Sigma_G} N.
\]
\end{corollary}

\begin{proof}
	Since $\mathsf{lr}_{\mathcal{P}}$ is a right retraction over $\Sigma_G$, the undercategories 
	$C \downarrow \widehat{\Omega}_G^e$ are right retractions of 
	$C \downarrow \Omega_G^e$ for any $C \in \Sigma_G$.
\end{proof}

\renewcommand{\labelenumi}{\theenumi}
\renewcommand{\theenumi}{(\roman{enumi})}%

\subsection{Filtrations of free extensions}
\label{FILTRATION_SECTION}

Summarizing \S \ref{EXTTREE SEC},
Corollary \ref{ESTABDESC COR}
establishes \eqref{EXTTREEFOR EQ}, and hence 
Corollary \ref{KANRED COR} 
gives the alternative formula
(the use of opposite categories turns 
$\mathsf{Ran}$ into $\mathsf{Lan}$)
\begin{equation}\label{ALTFOR EQ}
	\mathcal{P}[u] \simeq
	\mathcal{P} \mathbin{\check{\coprod}}_{\mathbb{F}_G X} \mathbb{F}_G Y 
\simeq 
	\mathsf{Lan}_{\left( \widehat{\Omega}_{G}^{e} \to \Sigma_G \right)^{op}}
	\tilde{N}^{(\mathcal{P},X,Y)},
\end{equation}
which we will now use to 
filter the map
$\mathcal{P} \to \mathcal{P}[u]$
in the underlying category
$\Sym_G(\V)$.

First, given 
$T = (T_i)_{i \in I} \in \Omega_G^e$,
we write $V^X(T_i)$ (resp. $V^Y(T_i)$)
to denote the set of 
$X$-labeled ($Y$-labeled) vertices of $T_i$.
We define \textit{degrees} of $T$ by
\[
|T|_X = |V^X(T_i)|,
	\qquad
|T|_Y = |V^Y(T_i)|,
	\qquad
|T| = |T|_X + |T|_Y,
\]
which we note do not depend on the choice of $i \in I$.

Similarly, for $T = (T_i)_{i \in I} \in \Omega_G^a$,
we write $V^{in}(T_i)$ for the inert vertices and
$|T| = |V^{in}(T_i)|$.

\begin{remark}
	One key property of the degrees $|T|$, $|T|_X$, $|T|_Y$ is that they are invariant under root pullbacks, which are defined
	by generalizing Definition \ref{ROOTPULL DEF}
	in the obvious way.
	
	As such, we caution that, though 
	$|V^{in}(T_i)| = k$
	for each of the $T_i$ that constitute $T=(T_i)_{i \in I}$,
	one can only guarantee $|V_G^{in}(T)| \leq k$.
\end{remark}

\begin{definition}\label{TREE_FILTRATION_PIECES_DEFINITION}
We specify some rooted (i.e. closed under root pullbacks)
full subcategories
 of $\widehat{\Omega}_{G}^e$: 
  \begin{enumerate}
  \item $\widehat{\Omega}_G^e[\leq\! k]$ 
  (resp. $\widehat{\Omega}_G^e[k]$) is the subcategory of $T$ with $|T|\leq k$ ($|T| = k$);
  \item $\widehat{\Omega}_G^e[\leq\! k \mathbin{\backslash} Y]$
  (resp. $\widehat{\Omega}_G^e[k \mathbin{\backslash} Y]$) is the subcategory of $\widehat{\Omega}_G^e[\leq\! k]$ 
  ($\widehat{\Omega}_G^e[k]$) of $T$ with $|T|_{Y}\neq k$.
  \end{enumerate}
Similarly, we define subcategories 
$\Omega_G^a[\leq \! k]$, 
$\Omega_G^a[k]$ of $\Omega_G^a$
by the conditions $|T|\leq k$, $|T|=k$.
\end{definition}

\begin{remark}\label{LIMMOR REM}
  The categories 
  $\widehat{\Omega}_G^e[k]$, $\widehat{\Omega}_G^e[k \mathbin{\backslash} Y]$ and $\Omega_G^a[k]$
  have rather limited morphisms.
  
Indeed, it is clear from Definitions 
\ref{EXTTREECAT DEF} and \ref{OMEGAA DEF} 
that maps never lower degree,
and Remark \ref{ALTSUB REM} further ensures that degree is preserved iff
$\mathcal{P}$-vertices are substituted by $\mathcal{P}$-vertices (rather than larger trees
which would necessarily have inert vertices, thus increasing degree).
Therefore, all maps in $\Omega_G^a[k]$ are quotients while maps in $\widehat{\Omega}_G^e[k]$, $\widehat{\Omega}_G^e[k \mathbin{\backslash} Y]$
are underlying quotients of $G$-trees that 
relabel some $Y$-vertices to $X$-vertices.  
 Moreover, this can be repackaged as saying that 
  the diagonal forgetful functors in
\[
\begin{tikzcd}
  \widehat\Omega_G^e[k \mathbin{\backslash} Y] \arrow[dr] \arrow[rr, hookrightarrow] 
  && \widehat\Omega_G^e[k] \arrow[dl]\\
  & \Omega_G^a[k]
\end{tikzcd}
\]  
 are Grothendieck fibrations whose fibers over 
 $T \in \Omega_G^a[k]$
 are the punctured cube and cube categories
\begin{equation}\label{PUNCUBE EQ}
	(Y \to X)^{\times V_G^{in}(T)} - Y^{\times  V_G^{in}(T)},
\qquad
	(Y \to X)^{\times V_G^{in}(T)}
\end{equation}
for $V_G^{in}(T)$ the set of inert $G$-vertices.
\end{remark}

\begin{lemma}\label{MINUS_LAN_FINAL_LEMMA}
%
The horizontal inclusion below 
\[
\begin{tikzcd}
	\widehat{\Omega}_G^e[\leq\! k-1]
	\arrow{dr}[swap]{\mathsf{lr}} \arrow[rr, hookrightarrow] &&
	\widehat{\Omega}_G^{e}[\leq\! k \mathbin{\backslash} Y] \arrow{dl}{\mathsf{lr}}
\\
	&
	\Sigma_G
\end{tikzcd}
\]  
is $\mathsf{Ran}$-initial (in the sense of Corollary \ref{KANRED COR})
over $\Sigma_G$.  
\end{lemma}

The proof will make use of an additional construction on 
$\Omega_{G}^e$: 
given $T \in \Omega_{G}^e$ we let $T_{\mathcal{P}}$ denote the result of replacing all $X$-labeled nodes of $T$ with $\mathcal{P}$-labeled nodes.

\begin{remark}\label{YINERT REM}
	In contrast to the functor
	$\mathsf{lr}_{\mathcal{P}} \colon
	\Omega_G^e \to \widehat{\Omega}_G^e $
	of Proposition \ref{LXP PROP},
	the $(\minus)_{\mathcal{P}}$ construction 
	is not a full functor
	$\Omega_G^e \to \Omega_G^e$.
	Instead, $(\minus)_{\mathcal{P}}$ is only functorial, and the obvious maps $T_{\mathcal{P}} \to T$ are only natural,
	with respect to the
	maps of $\Omega_G^e$ that preserve $Y$-labels.
\end{remark}

\begin{example}
Combining the $(\minus)_{\mathcal{P}}$ and $\mathsf{lr}_{\mathcal{P}}$ constructions one obtains a construction sending trees in $\widehat{\Omega}^e_G$ to trees in $\widehat{\Omega}^e_G$.
We illustrate this for the tree $T \in \widehat{\Omega}^e$ below (so that $G=\**$), where black nodes are $\P$-labeled, white nodes are $X$-labeled, and grey nodes are $Y$-labeled.
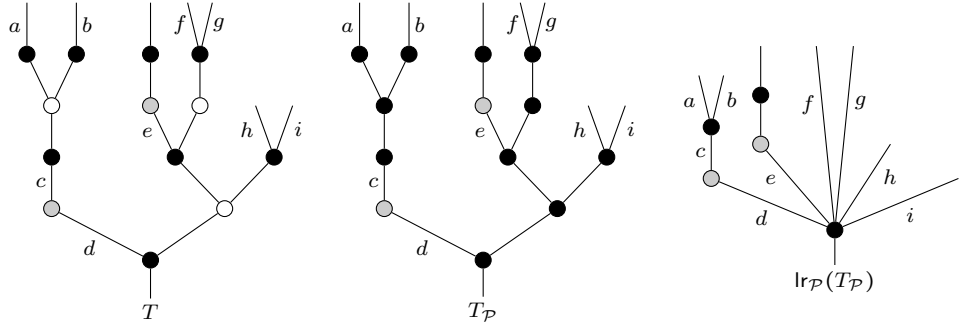
\[
\begin{tikzpicture}
  [grow=up,auto,level distance=2.1em,
  every node/.style = {font=\footnotesize,inner sep=2pt},
  dummy/.style={circle,draw,inner sep=0pt,minimum size=2.1mm}]
  \tikzstyle{level 2}=[sibling distance=6em]
  \tikzstyle{level 3}=[sibling distance=4em]
  \tikzstyle{level 4}=[sibling distance=2em]
  \tikzstyle{level 5}=[sibling distance=2em]
  \tikzstyle{level 6}=[sibling distance=1em]
  \node at (0,0){$T$}
  child{node [dummy, fill=black] {}%
    child{node [dummy, fill=white] {}%
      child{node [dummy,fill=black] {}%
        child[sibling distance = 1.5em]{edge from parent node [swap,near end] {$i$}}
        child[sibling distance = 1.5em]{edge from parent node [near end] {$h$}}
      }
      child{node [dummy, fill=black] {}%
        child{node [dummy, fill=white] {}%
          child{node [dummy, fill=black] {}%
            child{edge from parent node [right]{$g$}} 
            child{edge from parent node [left]{$f$}} 
          }
        }
        child{node [dummy, fill=black!20] {}%
          child{node [dummy, fill=black] {}%
            child{}
          }
          edge from parent node [near end] {$e$}
        }
      }
    }
    child[sibling distance=8em]{node [dummy, fill=black!20] {}%
      child{node [dummy, fill=black] {}%
        child{node [dummy, fill=white] {}%
          child{node [dummy, fill=black] {}%
            child{edge from parent node [swap] {$b$}}
          }
          child{node [dummy, fill=black] {}%
            child{edge from parent node {$\phantom{b}a$}}
          }
        }
        edge from parent node {$c$}
      }
      edge from parent node {$d$}
    }
  };
  \node at (4.375,0){$T_{\mathcal{P}}$}
  child{node [dummy, fill=black] {}%
    child{node [dummy, fill=black] {}%
      child{node [dummy,fill=black] {}%
        child[sibling distance = 1.5em]{edge from parent node [swap,near end] {$i$}}
        child[sibling distance = 1.5em]{edge from parent node [near end] {$h$}}
      }
      child{node [dummy, fill=black] {}%
        child{node [dummy, fill=black] {}%
          child{node [dummy, fill=black] {}%
            child{edge from parent node [right]{$g$}} 
            child{edge from parent node [left]{$f$}} 
          }
        }
        child{node [dummy, fill=black!20] {}%
          child{node [dummy, fill=black] {}%
            child{}
          }
          edge from parent node [near end] {$e$}
        }
      }
    }
    child[sibling distance=8em]{node [dummy, fill=black!20] {}%
      child{node [dummy, fill=black] {}%
        child{node [dummy, fill=black] {}%
          child{node [dummy, fill=black] {}%
            child{edge from parent node [swap] {$b$}}
          }
          child{node [dummy, fill=black] {}%
            child{edge from parent node {$\phantom{b}a$}}
          }
        }
        edge from parent node {$c$}
      }
      edge from parent node {$d$}
    }
  };
  \tikzstyle{level 2}=[sibling distance=2em]
  \tikzstyle{level 4}=[sibling distance=1em]
	\node at (9,0.4){$\mathsf{lr}_{\mathcal{P}}(T_{\mathcal{P}})$}
	child{node [dummy,fill=black] {}%
		child{edge from parent node [swap]{$i$}}
		child[sibling distance =1.5em, level distance =3.5em]{edge from parent node [near end,swap] {$h$}}
	    child[sibling distance=1.5em, level distance = 7.5em]{edge from parent node [swap, near end]{$g$}}
		child[sibling distance=1.5em, level distance = 7.5em]{edge from parent node [near end]{$f$}}
    child[level distance = 3.5em]{node [dummy, fill=black!20] {}%
      child[level distance = 2em]{node [dummy, fill=black] {}%
        child{}
      }
      edge from parent node [near end]{$e$}
    }
    child{node [dummy, fill=black!20] {}%
      child{node [dummy, fill=black] {}%
        child{edge from parent node [swap,near end]{$b$}}
        child{edge from parent node [near end] {$\phantom{b}a$}}
        edge from parent node {$c$}
      }
      edge from parent node {$d$}
    }
  };    
\end{tikzpicture}
\]
\end{example} 

\begin{proof}[Proof of Lemma \ref{MINUS_LAN_FINAL_LEMMA}]

By Proposition \ref{FIBERKANMAP PROP} it suffices to
show that for each $C \in \Sigma_G$
the map of rooted undercategories
\[
C \downarrow_{\mathsf{r}} \widehat{\Omega}_G^e[\leq\! k-1]
	\to 
C \downarrow_{\mathsf{r}} \widehat{\Omega}_G^e[\leq\! k \mathbin{\backslash} Y]
\]
is initial, i.e. 
\cite[IX.3]{McL} that for each
$(S,\pi \colon C \to \mathsf{lr}(S))$ in 
$C \downarrow_{\mathsf{r}} \widehat{\Omega}_G^e[\leq\! k \mathbin{\backslash} Y]$
the overcategory
\begin{equation}\label{UNDERCATPR EQ}
	(C \downarrow_{\mathsf{r}} \widehat{\Omega}_G^e[\leq\! k-1])
		\downarrow
	(S,\pi)  
\end{equation}
is non-empty and connected. 
By definition of rooted undercategory, $\pi$ is the identity on roots and thus an isomorphism in $\Sigma_G$,
so that objects of \eqref{UNDERCATPR EQ}
correspond to maps
$T \to S$
that induce a rooted isomorphism on 
$\mathsf{lr}$, i.e. rooted tall maps.

The case $S\in \widehat{\Omega}_G^e[\leq\! k-1]$ is immediate,
since then the identity $S = S$ is terminal in 
\eqref{UNDERCATPR EQ}.
Otherwise, since $|S|_Y \neq k$ we have
$|\mathsf{lr}_{\mathcal{P}}(S_{\mathcal{P}})|<k$
and the map 
$\mathsf{lr}_{\mathcal{P}}(S_{\mathcal{P}}) \to S$,
which is a rooted tall map, shows that
\eqref{UNDERCATPR EQ} is indeed non-empty.

Next, consider a rooted tall map $T \to S$ with 
$T \in \widehat{\Omega}_G^e[\leq\! k-1]$. One can form a diagram
\begin{equation}\label{K-1LANFINAL EQ}
\begin{tikzcd}
      & S & \mathsf{lr}_{\mathcal{P}}(S_{\mathcal{P}}) \ar{l}
      \\
      T \ar{ur} \ar{r} & T' \ar{u}[swap]{Y-\text{pres}} & \mathsf{lr}_{\mathcal{P}}(T'_{\mathcal{P}}) \ar{l} \ar{u}
\end{tikzcd}
\end{equation}
where $T \to T' \to S$ is the natural factorization such that $ T' \to S$ preserves $Y$-labels, 
i.e., $T'$ is obtained from $T$ by simply relabeling to $X$ those $Y$-labeled vertices of $T$ that become $X$-vertices in $S$.
Note that, by Remark \ref{YINERT REM}, the existence
of the right square relies on 
$T' \to S$ preserving $Y$-labels.
Since all maps in 
\eqref{K-1LANFINAL EQ}
are rooted tall, 
this produces the
necessary zigzag connecting the objects $T \to S$ and 
$\mathsf{lr}_{\mathcal{P}}(S_{\mathcal{P}}) \to S$
in the category \eqref{UNDERCATPR EQ}, finishing the proof.
\end{proof}



In what follows we write $\tilde{N} \colon \widehat{\Omega}_G^{e,op} \to \mathcal{V}$ for the functor in \eqref{ALTFOR EQ}
(see also Remark \ref{TILNUNPACK REM})
and any of its restrictions.
We can now describe the filtration \eqref{FILTR EQ}
of the map $\mathcal{P} \to \mathcal{P}[u]$.

\begin{definition} \label{PK_DEFN}
  Let $\P_k$ denote the left Kan extension
\[
\begin{tikzcd}[column sep =4em]
	|[alias=U]| \widehat{\Omega}_G^e[\leq\! k]^{op} \arrow[r, "\tilde{N}"] \arrow[d, "\mathsf{lr}"'] & \V
\\
	\Sigma_G^{op} \arrow[ur, "\P_k"', ""{name=V}]
	\arrow[Rightarrow, from=U, to=V]
\end{tikzcd}
\]
\end{definition}
Noting that $\widehat{\Omega}_G^e[\leq\! 0] \simeq \Sigma_G$
(since $|T|=0$ only if $T$ is a $G$-corolla with $\mathcal{P}$-labeled vertex)
and that $\widehat{\Omega}_G^e$
is the union of (the nerves of) the 
$\widehat{\Omega}_G^e[\leq\! k]$,
we obtain the desired filtration
\begin{equation}\label{FILT EQ}
	\mathcal{P} = 
	\mathcal{P}_0 \to 
	\mathcal{P}_1 \to
	\mathcal{P}_2 \to
	\cdots \to 
	\colim_k \mathcal{P}_k = \mathcal{P}[u].
\end{equation}

To analyze \eqref{FILT EQ} homotopically we will further need a pushout description of each map 
$\mathcal{P}_{k-1} \to \mathcal{P}_k$. To do so,  note that the diagram of inclusions
\begin{equation}\label{INCDIAG EQ}
\begin{tikzcd}
	\widehat{\Omega}_G^{e}[k \mathbin{\backslash} Y]
	\arrow[d] \arrow[r] &
	\widehat{\Omega}_G^{e}[\leq\! k \mathbin{\backslash} Y]
	\arrow[d]
\\
	\widehat{\Omega}_G^e[k] \arrow[r] &
	\widehat{\Omega}_G^e[\leq\! k]
\end{tikzcd}
\end{equation}
is a pushout of at the level of nerves.
Indeed, this follows since
\[
	\widehat{\Omega}_G^e[k] \cap
	\widehat{\Omega}_G^{e}[\leq\! k \mathbin{\backslash} Y]
	= \widehat{\Omega}_G^{e}[k \mathbin{\backslash} Y],
\qquad
	\widehat{\Omega}_G^e[k] \cup 
	\widehat{\Omega}_G^{e}[\leq\! k\mathbin{\backslash} Y]
	= \widehat{\Omega}_G^{e}[\leq\! k],
\]
and since a map $T \to S$ in 
$\widehat{\Omega}_G^e[\leq\! k]$ is in one of subcategories in \eqref{INCDIAG EQ} if and only if $T$ is.

Since Lemma \ref{MINUS_LAN_FINAL_LEMMA} provides an identification 
$\mathsf{Lan}_{\widehat{\Omega}_{G}^{e}[\leq\! k \mathbin{\backslash} Y]^{op}}\tilde{N} \simeq
\mathsf{Lan}_{\widehat{\Omega}_{G}^{e}[\leq\! k-1]^{op}}\tilde{N} = \mathcal{P}_{k-1}$,
applying left Kan extensions to \eqref{INCDIAG EQ} yields the pushout diagram below.
\begin{equation}\label{FILTRATION_LAN_SQUARE_DIAGRAM}
\begin{tikzcd}
	\mathsf{Lan}_{\widehat{\Omega}_{G}^e[k \mathbin{\backslash} Y]^{op}}\tilde{N} \arrow[d] \arrow[r] & 
	\P_{k-1} \arrow[d]
\\
	\mathsf{Lan}_{\widehat{\Omega}_{G}^e[k]^{op}}\tilde{N} \arrow[r] &
	\P_k
\end{tikzcd}
\end{equation}

We will also make use of an 
explicit levelwise description of  
(\ref{FILTRATION_LAN_SQUARE_DIAGRAM}).

\begin{proposition}\label{FILTINTALT PROP}
For each $G$-corolla $C \in \Sigma_G$,
(\ref{FILTRATION_LAN_SQUARE_DIAGRAM})
is given by the following pushout in $\V^{\mathsf{Aut}(C)}$
\begin{equation}\label{FILTRATION_LAN_LEVEL}
\begin{tikzcd}
	\coprod\limits_{[T] \in \mathsf{Iso}
		\left(C \downarrow_{\mathsf{r}} \Omega_G^a[k]\right)}
	\left(
		\bigotimes\limits_{v \in V_{G}^{ac}(T)}\P(T_v) \otimes
		Q_T^{in}[u]
	\right)
		\mathop{\otimes}\limits_{\mathsf{Aut}(T)} \mathsf{Aut}(C)
	\arrow[r] \arrow[d] &
	\P_{k-1}(C) \arrow[d] 
\\
	\coprod\limits_{[T] \in \mathsf{Iso}
		\left(C \downarrow_{\mathsf{r}} \Omega_G^a[k]\right)}
	\left(
		\bigotimes\limits_{v \in V_{G}^{ac}(T)}\P(T_v) \otimes
		\bigotimes\limits_{v \in V_{G}^{in}(T)} Y(T_v)
	\right)
		\underset{\mathsf{Aut}(T)}{\otimes} \mathsf{Aut}(C)
	\arrow[r] &
	\P_k(C)
\end{tikzcd}
\end{equation}
where $ V_{G}^{ac}(T)$, $V_{G}^{in}(T)$ denote the active and inert vertices of $T \in \Omega_G^a[k]$,
and $Q_T^{in}[u]$ is the domain 
of the iterated pushout product
\[
		\underset{v \in V_G^{in}(T)}
		{\mathlarger{\mathlarger{\mathlarger{\square}}}}u(T_v)
	\colon
		Q_T^{in}[u] \to
		\bigotimes\limits_{v \in V_{G}^{in}(T)} Y(T_v).
\]
\end{proposition}

\begin{proof}
This follows from Remark \ref{LIMMOR REM}.
Explicitly,
consider first the analogue of the leftmost map in 
(\ref{FILTRATION_LAN_SQUARE_DIAGRAM}) 
that left Kan extends to 
$\Omega_G^a[k]$ rather than to $\Sigma_G$, as on the left below.
The (punctured) cube fiber categories in
\eqref{PUNCUBE EQ} and \eqref{FIBERKAN EQ} yield
an identification (of maps)
\[
\left(\mathsf{Lan}_{
	(\widehat{\Omega}_{G}^e[k \mathbin{\backslash} Y]
	\to \Omega^a_G[k])^{op}}\tilde{N}
	\to
\mathsf{Lan}_{
	(\widehat{\Omega}_{G}^e[k]^{op}
	\to \Omega^a_G[k])^{op}
	}\tilde{N}\right)(T)
=
\left(
	\bigotimes\limits_{v \in V_{G}^{ac}(T)}\P(T_v) \otimes
	\underset{v \in V_{G}^{in}(T)}
	{\mathlarger{\mathlarger{\mathlarger{\square}}}}
	u(T_v)
\right)
\]
so that, iterating left Kan extensions,
the leftmost map in 
(\ref{FILTRATION_LAN_SQUARE_DIAGRAM}) is then
\begin{equation}\label{FILTINTALT EQ}
	\mathsf{Lan}_{(\Omega_G^a[k] \to \Sigma_G)^{op}}
	\left(
		\bigotimes\limits_{v \in V_{G}^{ac}(T)}\P(T_v) \otimes
		\underset{v \in V_{G}^{in}(T)}
		{\mathlarger{\mathlarger{\mathlarger{\square}}}}
		u(T_v)
	\right).
\end{equation}
The desired description of the leftmost map in (\ref{FILTRATION_LAN_LEVEL})
now follows by noting that the rooted undercategories
$C \downarrow_{\mathsf{r}} \Omega_G^a[k]$
only depend on the isomorphisms 
of $\Omega_G^a[k],\Sigma_G$
(
cf. \eqref{FGXDEFEXP EQ}).
\end{proof}

\renewcommand{\F}{\mathcal F}

\subsection{Proof of Theorems \ref{MAINEXIST1 THM} and \ref{MAINEXIST2 THM}}
\label{MAINEXIST SEC}

In this section, we use the filtrations just developed to prove our first two main results,
Theorems \ref{MAINEXIST1 THM} and \ref{MAINEXIST2 THM},
concerning the existence of model structures on 
$\mathsf{Op}^G(\mathcal{V})$
and
$\mathsf{Op}_G(\mathcal{V})$.

We begin by recalling the notion of genuine equivariant model structures.

\begin{definition}\label{GENUINEMS_DEF}
      Let $G$ be a group, and $\V$ a model category.
      The \textit{genuine model structure} (if it exists)
      on $\mathcal{V}^{G}$,
      denoted $\mathcal{V}^{G}_{\text{gen}}$,
      has as weak equivalences (resp. fibrations)
      those maps $X \to Y$ such that 
      $X^H \to Y^H$ is a weak equivalence (fibration)
      for all $H \leq G$.

	More generally, for a family 
	$\F$ of subgroups of $G$
	(cf. \S \ref{INDEXING_SECTION}),
      the \textit{$\F$-model structure} (if it exists), denoted $\mathcal V^G_{\F}$,
      has weak equivalences and fibrations defined
      similarly but with $H \in \F$.
\end{definition}

In particular, when $\F$ 
is the family containing only the trivial subgroup $\{e\}$,
the $\F$-model structure is the \textit{projective} model structure, where weak equivalences (resp. fibrations) are those maps which forget to weak equivalences (resp. fibrations) in $\V$.
Note that, as $\F$ increases, 
both weak equivalences and fibrations decrease while
(trivial) cofibrations increase.\\

Our main proof will require some auxiliary results concerning genuine model structures and related hypotheses.
However, since these results are 
instances of subtler results from \S \ref{COFIB SEC}
which will require a far more careful analysis,
we defer their 
proofs to those
of the stronger results in \S \ref{COFIB SEC}.
In particular, 
we postpone the definition of the
\emph{cellular fixed points}
and
\emph{cofibrant symmetric pushout powers}
conditions
(which are (iii) and (iv) in our main theorems)
to Definitions \ref{CELL DEF} and \ref{COFSYMPUSHPOW},
and collect the properties used in this section
in the following remark
(note that, since $\V^G_{\text{gen}}$
has more (trivial) cofibrations than 
$\V^G_{\mathcal{F}}$ for any $\F$,
one can replace all model structures in
Propositions \ref{FGTRIGHT PROP},
\ref{FGTLEFT PROP}, 
\ref{BIQUILLENG PROP}, 
\ref{POWERF PROP}
and \eqref{EXTERINTADJ EQ} with genuine ones).



\begin{remark}\label{GEN_FGTRIGHT_REM}
Suppose $\mathcal{V}$ 
is a closed monoidal model category
which is cofibrantly generated and 
has cellular fixed points (Definition \ref{CELL DEF}).
\begin{enumerate}[label = (\roman*)]
	\item \label{CELL_ITEM}
		By \cite[Prop. 2.6]{Ste16},
		the model structure $\V^G_{\text{gen}}$
		exists and is cofibrantly generated
		with generating (trivial) cofibrations
		the maps 
        $G/H \cdot i$
        for $H\leq G$ 
        and $i$ a generating (trivial) cofibration of $\mathcal{V}$.
        Likewise, for a family $\mathcal{F}$,
        the model structure $\mathcal{V}^{G}_{\mathcal{F}}$
        exists and is cofibrantly generated, 
        with generating (trivial) cofibrations
        described analogously but with $H \in \mathcal{F}$.
        
	\item \label{GROUPHOM_ITEM}
		Propositions \ref{FGTRIGHT PROP} and
		\ref{FGTLEFT PROP} imply that,
		for a group homomorphism $\phi: G \to \bar G$, 
		the functors
		\[
		\begin{tikzcd}
			\bar{G} \cdot_{G} (\minus)
			\colon
			\mathcal{V}_{\text{gen}}^{G}
			\ar{r}
		&
			\mathcal{V}_{\text{gen}}^{\bar{G}}
		&
			\mathsf{res}^{\bar{G}}_{G}
			\colon
			\mathcal{V}_{\text{gen}}^{\bar{G}}
			\ar{r}
		&
			\mathcal{V}_{\text{gen}}^{G}
		\end{tikzcd}
		\]
		are left Quillen functors. 
	\item \label{TENSORLEFT_ITEM}
            \eqref{EXTERINTADJ EQ} implies that the monoidal product on $\mathcal{V}$ lifts to a left Quillen bifunctor
            \[
                  \V^{G}_{\text{gen}} \times \V^{\bar G}_{\text{gen}} 
                  \xrightarrow{\otimes}
                  \V^{G \times \bar G}_{\text{gen}}.
            \]
	\item \label{CSPP_ITEM}
		If, additionally, $\V$ has cofibrant symmetric pushout powers
		(Definition \ref{COFSYMPUSHPOW}),
		then Proposition \ref{POWERF PROP} implies that,
		if $f$ is a (trivial) cofibration in $\V^G_{\text{gen}}$, 
		then $f^{\square n}$ is a (trivial) cofibration in $\V^{\Sigma_n \wr G}_{\text{gen}}$ for any $n \geq 1$.
      \end{enumerate}
\end{remark} 

\begin{remark}\label{ALLCOF REM}
        A skeletal filtration argument shows that all objects in
        $\mathsf{sSet}^{G}_{\text{gen}}$ and
        $\mathsf{sSet}^{G}_{\**,\text{gen}}$
        are cofibrant.
        Moreover, Example \ref{SSET_CSPP_EX} says that $(\mathsf{sSet},\times)$, $(\mathsf{sSet}_{\**},\wedge)$ have
        cofibrant symmetric pushout powers.
\end{remark}

The following lemma is the key to our main proof. 
Here, a map $f$ in 
$\mathsf{Sym}_G(\mathcal{V})$ is called
a \textit{level genuine (trivial) cofibration} if each of the maps
$f(C)$ for $C \in \Sigma_G$ are genuine trivial cofibrations in
$\mathcal{V}^{\mathsf{Aut}(C)}_{\text{gen}}$.

\begin{lemma}\label{EXMAINLEM LEM}
	Suppose $\mathcal{V}$ is a cofibrantly generated closed monoidal model category
	with cellular fixed points and
	with cofibrant symmetric pushout powers.
	
	Let $\mathcal{P} \in \mathsf{Sym}_G(\mathcal{V})$
	be level genuine cofibrant
	and  
	$u: X \to Y$ in $\Sym_G(\V)$ be a level genuine cofibration. 
	Then, for each $T \in \Omega^a_G[k]$, and writing
	$C = \mathsf{lr}(T)$, the map	
\begin{equation}\label{EXMAINLEM EQ}
	\left(
		\bigotimes\limits_{v \in V_{G}^{ac}(T)}\P(T_v) \otimes
		\underset{v \in V_{G}^{in}(T)}
	{\mathlarger{\mathlarger{\mathlarger{\square}}}}
		u(T_v)
	\right) 
	\mathop{\otimes}\limits_{\mathsf{Aut}(T)} \mathsf{Aut}(C).
\end{equation}
	is a genuine cofibration in 
	$\mathcal{V}^{\mathsf{Aut}(C)}_{\text{gen}}$,
	which is trivial if $k \geq 1$ and $u$ is trivial.	
\end{lemma}

\begin{proof}
	Combining the homomorphism $\mathsf{Aut}(T) \to \mathsf{Aut}(C)$ with the leftmost left Quillen functor in 
	Remark \ref{GEN_FGTRIGHT_REM}\ref{GROUPHOM_ITEM},
	it suffices to check that the parenthesized 
	expression in \eqref{EXMAINLEM EQ}
	is a (trivial) genuine 
	$\mathsf{Aut}(T)$-cofibration.

	Furthermore, the homomorphism
	$\mathsf{Aut}(T) \to 
	\mathsf{Aut}\left( (T_v)_{v \in V_G^{ac}(T)}\right) \times 
	\mathsf{Aut}\left( (T_v)_{v \in V_G^{in}(T)}\right)$
	combined with the rightmost left Quillen functor in Remark \ref{GEN_FGTRIGHT_REM}\ref{GROUPHOM_ITEM} and Remark \ref{GEN_FGTRIGHT_REM}\ref{TENSORLEFT_ITEM}
	then yield that it suffices to check that the two maps
\[
\left( \varnothing \to \bigotimes\limits_{v \in V_{G}^{ac}(T)}\P(T_v) \right)
=
\underset{v \in V_{G}^{ac}(T)}{\mathlarger{\mathlarger{\mathlarger{\square}}}}
(\emptyset \to \P)(T_v),
	\qquad
\underset{v \in V_{G}^{in}(T)}{\mathlarger{\mathlarger{\mathlarger{\square}}}}
u(T_v)
\]
are, respectively, 
$\mathsf{Aut}\left( (T_v)_{v \in V_G^{ac}(T)}\right)$ and 
$\mathsf{Aut}\left( (T_v)_{v \in V_G^{in}(T)}\right)$
genuine cofibrations, with the latter trivial if $u$ is. Here, the automorphism groups are taken in the category in $\Fin_s \wr \Sigma_G$,
and thus admit a product description of the form
$
	\Sigma_{|\lambda_1|} \wr 
	\mathsf{Aut}(T_{v_1})
		\times \cdots \times	
	\Sigma_{|\lambda_k|} \wr 
	\mathsf{Aut}(T_{v_k})
$
as in Remark \ref{WREATHFIXED REM}.
A further application of Remark \ref{GEN_FGTRIGHT_REM}\ref{TENSORLEFT_ITEM}
yields that the required conditions need only be checked independently for the
partial pushout product indexed by each $\lambda_i$.
Thus the result follows by Remark \ref{GEN_FGTRIGHT_REM}\ref{CSPP_ITEM}.
\end{proof}

\begin{remark}\label{EXMAINLEM REM}
If 
$T \in \Omega^a[k]$ is a non-equivariant alternating tree, 
$\mathcal{P}$ is level genuine cofibrant in $\mathsf{Sym}^G(\V)$,
and
$u \colon X \to Y$ is a level genuine (trivial) cofibration in $\mathsf{Sym}^G(\V)$,
the previous result applied to
$G \cdot T = (T)_{g \in G}$,
$\iota_{!} \mathcal{P}$,
$\iota_{!} u$,
yields that the analogue of the map
\eqref{EXMAINLEM EQ}
is an $\mathsf{Aut}(G \cdot C_n) 
\simeq G^{op} \times \mathsf{Aut}(C_n) =
G^{op} \times \Sigma_n$ level genuine (trivial) cofibration,
where $C_n = \mathsf{lr}(T)$.
\end{remark}

\begin{proof}
[proof of Theorems \ref{MAINEXIST1 THM} and \ref{MAINEXIST2 THM}]
We first build a seemingly unrelated model structure.
Consider the composite adjunction below, with right adjoints on the bottom, and
where the rightmost right adjoint simply forgets structure and the leftmost right adjoint is given by evaluation.
\begin{equation}\label{MAINPFADJ EQ}
\begin{tikzcd}[column sep =5em]
	\prod_{C \in \Sigma_G}
	\mathcal{V}^{\mathsf{Aut}(C)}_{\text{gen}}
	\ar[shift left=1.5]{r}
&
	\mathsf{Sym}_{G}(\mathcal{V}) 
	\arrow[l, shift left=1.5, "\left(\text{ev}_C(\minus)\right)"] 
	\arrow[r, shift left=1.5,swap,"\mathbb{F}_G"']
&
	\mathsf{Op}_G(\mathcal{V})
	\ar[shift left=1.5]{l}
\end{tikzcd}
\end{equation}
We claim that $\mathsf{Op}_G(\mathcal{V})$ admits a (semi-)model structure with weak equivalences and fibrations defined by the composite right adjoint in 
\eqref{MAINPFADJ EQ}.
Noting that the left adjoint to 
$\left( \text{ev}_C (\minus) \right)$
is given by
$(X_D) \mapsto \coprod_{D \in \Sigma_G}
\mathsf{Hom}_{\Sigma_G} (\minus, D) 
\cdot_{\mathsf{Aut(D)}} X_D$
and using either 
\cite[Thm. 11.3.2]{Hi03} 
in the model structure case
$\mathcal{V} = \mathsf{sSet},\mathsf{sSet}_{\**}$
or
\cite[Thm. 2.2.1]{BW}
in the semi-model structure case,
one must analyze free $\mathbb{F}_G$-extension diagrams of the form
\[ 
\begin{tikzcd} 
	\mathbb{F}_G
	\left(\mathsf{Hom}_{\Sigma_G}(\minus,D)/H \cdot A \right) \arrow[d, "u"'] \arrow[r] 
&
	\P \arrow[d]
\\ 
	\mathbb{F}_G 
	\left(\mathsf{Hom}_{\Sigma_G}(\minus,D)/H \cdot B \right)
	\arrow[r]
&
	\P[u] 
\end{tikzcd} 
\]
where $D \in \Sigma_G$,
$H \leq \mathsf{Aut}(D)$,
and $u \colon A \to B$ is a generating (trivial)
cofibration in $\mathcal{V}$.

The map $\mathcal{P} \to \mathcal{P}[u]$ is then filtered as in \eqref{FILT EQ},
and since
$\mathsf{Hom}_{\Sigma_G}(C,D)/H \cdot u$
is a (trivial) cofibration in 
$\mathcal{V}^{\mathsf{Aut}(C)}_{\text{gen}}$
for all $C \in \Sigma_G$ 
(cf. Remark \ref{GEN_FGTRIGHT_REM}\ref{CELL_ITEM}), 
combining the inductive description of the filtration in (\ref{FILTRATION_LAN_LEVEL})
with Lemma \ref{EXMAINLEM LEM} shows that if
$\mathcal{P}$ is level genuine cofibrant
then 
$\mathcal{P} \to \mathcal{P}[u]$
is a level genuine cofibration, trivial whenever $u$ is.

When
$\mathcal{V} = \mathsf{sSet},\mathsf{sSet}_{\**}$,
Remark \ref{ALLCOF REM}
guarantees that any 
$\mathcal{P}$ is level genuine cofibrant.
Thus in the model (resp. semi-model) structure cases,
the necessary conditions for
\cite[Thm. 11.3.2]{Hi03} (resp. \cite[Thm. 2.2.1]{BW}) are met,
as transfinite composites of trivial cofibrations are again trivial cofibrations,
showing the existence of the (semi-)model structure.

We now turn to showing the existence of the (semi-)model structures appearing in Theorems \ref{MAINEXIST1 THM} and \ref{MAINEXIST2 THM},
which are essentially corollaries 
of the existence of that defined by
\eqref{MAINPFADJ EQ}.

Firstly, consider the projective (semi-)model structure
on $\mathsf{Op}_G(\mathcal{V})$.
This model structure is transferred from the exact same adjunction
\eqref{MAINPFADJ EQ}, except equipping the 
leftmost $\mathcal{V}^{\mathsf{Aut}(C)}$
with their naive model structures, where weak equivalences and fibrations are defined by forgetting the $\mathsf{Aut}(C)$-action,
and ignoring fixed point conditions.
The desired projective model structure thus
has both fewer generating (trivial) cofibrations
and more weak equivalences than the ``genuine projective'' model structure defined by \eqref{MAINPFADJ EQ}.
Therefore, transfinite composites of pushouts of generating projective trivial cofibrations 
are genuine projective equivalences and hence also projective equivalences, showing that the condition in 
\cite[Thm. 11.3.2(2)]{Hi03} (resp. \cite[Thm. 2.2.1]{BW})
holds,
establishing the existence of the projective (semi-)model structure. 

The general case of Theorem \ref{MAINEXIST2 THM}
with $\mathcal{F}$ an arbitrary weak indexing system
slightly refines the argument in the previous paragraph.
Namely, the inclusion
$\upgamma_! \colon 
\mathsf{Op}_{\mathcal{F}}(\mathcal{V})
\to
\mathsf{Op}_{G}(\mathcal{V})$
(which is an extension by $\emptyset$)
has the following key properties:
\begin{enumerate*}
	\item[(i)] it preserves colimits;
	\item[(ii)] it sends the generating (trivial) cofibrations
	of $\mathsf{Op}_{\mathcal{F}}(\mathcal{V})$,
	i.e. the maps
	$\mathbb{F}_{\mathcal{F}}
	\left(
	\mathsf{Hom}_{\Sigma_{\mathcal{F}}}(-,D) \cdot u
	\right)$
	with $D \in \Sigma_{\mathcal{F}}$
	and $u$ a generating (trivial) cofibration in $\mathcal{V}$,
	to generating (trivial) cofibrations
	in the genuine projective model structure
	on $\mathsf{Op}_G(\mathcal{V})$
	defined by \eqref{MAINPFADJ EQ};
	\item[(iii)] maps in $\mathsf{Op}_{\mathcal{F}}(\mathcal{V})$ which become genuine projective 
	weak equivalences in $\mathsf{Op}_{G}(\mathcal{V})$
	are $\mathcal{F}$-projective weak equivalences.
\end{enumerate*}
	Thus, if $f$ is a transfinite 
	composite of pushouts of generating trivial cofibrations
	in $\mathsf{Op}_{\mathcal{F}}(\mathcal{V})$,
	properties (i),(ii) show that 
	$\gamma_!(f)$
	is a genuine projective trivial cofibration in 
	$\mathsf{Op}_{G}(\mathcal{V})$
	and thus (iii) implies
	that $f$ is a $\mathcal{F}$-projective weak equivalence in
	$\mathsf{Op}_{\mathcal{F}}(\mathcal{V})$,
	establishing the 
	condition in 
	\cite[Thm. 11.3.2(2)]{Hi03} (resp. \cite[Thm. 2.2.1]{BW}).
	The existence of the projective (semi-)model structure on $\mathsf{Op}_{\mathcal{F}}(\mathcal{V})$ follows,
	finishing the proof of Theorem \ref{MAINEXIST2 THM}.

We now turn to Theorem \ref{MAINEXIST1 THM}.
Should it be the case that 
$(\mathcal{V},\otimes)$ has diagonals 
(which is not a requirement of Theorem \ref{MAINEXIST1 THM}),
one can simply use the inclusion
$\iota_{!} \colon \mathsf{Op}^G(\mathcal{V})
\to 
\mathsf{Op}_G(\mathcal{V})$ of \eqref{TWOADJOINTS EQ} and repeat the argument in the previous paragraph,
since $\iota_!$ satisfies (i),(ii),(iii) therein
for any choice of 
$\mathcal{F} = \{\mathcal{F}_n\}$
as in Theorem \ref{MAINEXIST1 THM}.
Otherwise, 
one can readily adapt the entire proof with only minor changes required, as follows.
First, one has the following analogue of
\eqref{MAINPFADJ EQ},
where the functors in the leftmost adjunction
are now isomorphisms,
\begin{equation}\label{MAINPFADJAL EQ}
\begin{tikzcd}[column sep =5em]
	\prod_{n \geq 0}
	\mathcal{V}^{G \times \Sigma_n^{op}}_{\text{gen}}
	\ar[shift left=1.5]{r}
&
	\mathsf{Sym}^{G}(\mathcal{V}) 
	\arrow[l, shift left=1.5, "\left(\text{ev}_n(\minus)\right)"] 
	\arrow[r, shift left=1.5,swap,"\mathbb{F}"']
&
	\mathsf{Op}^G(\mathcal{V})
	\ar[shift left=1.5]{l}
\end{tikzcd}
\end{equation}
which we use to induce a 
``genuine projective'' model structure on 
$\mathsf{Op}^G(\mathcal{V})$. This again uses
\cite[Thm. 11.3.2(2)]{Hi03} (resp. \cite[Thm. 2.2.1]{BW}),
with the main step being an analysis of 
free $\mathbb{F}$-extension diagrams in 
$\mathsf{Op}^G(\mathcal{V})$
\begin{equation}\label{ANOPUCH EQ}
	\begin{tikzcd} 
	\mathbb{F}
	\left(
	(G \times \Sigma_n^{op})/K \cdot A 
	\right) \arrow[d, "u"'] \arrow[r] 
&
	\mathcal{O} \arrow[d]
\\ 
	\mathbb{F}
	\left(
	(G \times \Sigma_n^{op})/K \cdot B 
	\right)
	\arrow[r]
&
	\mathcal{O}[u] 
\end{tikzcd} 
\end{equation}
for $K \leq G \times \Sigma_n^{op}$
and $u \colon A \to B$
a generating (trivial) cofibration of $\mathcal{V}$.
Using the identification
$\mathsf{Op}^G(\mathcal{V}) \simeq 
\mathsf{Op}(\mathcal{V}^G)$,
one can apply the filtration
(\ref{FILTRATION_LAN_LEVEL}) when $G = \**$ and 
$\mathcal{V} = \mathcal{V}^G$.
The key fact that the filtration maps 
$\mathcal{O}_{k-1}(n) \to \mathcal{O}_{k}(n)$
are $G\times \Sigma_n^{op}$-genuine cofibrations 
follows by Remark \ref{EXMAINLEM REM}
(replacing the role of Lemma \ref{EXMAINLEM LEM} in the $\mathsf{Op}_G(\mathcal{V})$ argument),
so that \cite[Thm. 11.3.2(2)]{Hi03} (resp. \cite[Thm. 2.2.1]{BW})
applies to establish the 
genuine projective (semi-)model structure on 
$\mathsf{Op}^G(\mathcal{V})$
lifted along \eqref{MAINPFADJAL EQ}.
To finish the argument,
note that, compared to this genuine projective (semi-)model structure,
a choice of 
$\mathcal{F} = \{\mathcal{F}_n\}$
as in Theorem \ref{MAINEXIST1 THM}
decreases generating (trivial) cofibrations and
increases weak equivalences,
so that the argument
in the previous paragraph concerning the
projective (semi-)model structure on
$\mathsf{Op}_{G}(\mathcal{V})$
applies mutatis mutandis.
%
\end{proof}

\section{Cofibrancy and Quillen equivalences}\label{COFIB SEC}

In this final section we prove our main result, Theorem \ref{MAINQUILLENEQUIV THM}. I.e. we show 
that there are Quillen equivalences
\[
\begin{tikzcd}[column sep =4.5em]
	\mathsf{Op}_{G}(\mathcal{V}) 
	\arrow[r, shift left=1.5,swap,"\iota^{\**}"']
&
	\mathsf{Op}^G(\mathcal{V})
	\ar[shift left=1.5]{l}{\iota_{\**}}
&
	\mathsf{Op}_{\mathcal{F}}(\mathcal{V}) 
	\arrow[r, shift left=1.5,swap,"\iota^{\**}"']
&
	\mathsf{Op}^G_{\mathcal{F}}(\mathcal{V})
	\ar[shift left=1.5]{l}{\iota_{\**}}
\end{tikzcd}
\]
In contrast to the existence of model structure results shown in \S \ref{MAINEXIST SEC},
this will require a far more careful analysis of the  
fixed-point model structures
$\mathcal{V}^G_{\mathcal{F}}$
from Definition \ref{GENUINEMS_DEF},
precluding a simple application of \cite[Thm. 2.2.2]{WY18} or \cite[Thm. 2.2.2]{BW}.
This analysis is the subject of \S \ref{FAMILY_SEC} and \S \ref{PUSHPOW SEC}, the results of which are converted 
to the setup of $G$-trees in \S
\ref{G_GRAPH_SECTION},
and culminate in 
the characterization of cofibrant objects in
$\mathsf{Op}_{G}$,
$\mathsf{Op}_{\mathcal{F}}$ given by
Lemma \ref{MAINLEM LEM}
in \S \ref{MAINTHM_PROOF_SECTION},
with this final lemma tantamount to 
Theorem \ref{MAINQUILLENEQUIV THM}.

Lastly, \S \ref{NINFTY_SECTION} discusses our models for the 
$N \mathcal{F}$-operads of Blumberg-Hill.

\subsection{Families of subgroups}
\label{FAMILY_SEC}

Throughout,
$\mathcal{F}$ denotes a \textit{family} of subgroups of a finite group $G$,
i.e. a collection of subgroups closed under conjugation and inclusion, or, equivalently
(cf. \S \ref{INDEXING_SECTION}),
a sieve 
$
\mathsf{O}_{\mathcal{F}}
	\hookrightarrow 
\mathsf{O}_G
$.

\begin{remark}
For fixed $G$, families form a lattice, ordered by inclusion, 
with meet and join given by intersection and union.
\end{remark}

We begin our discussion by recalling
the cellular fixed point condition,
originally from \cite{Gui06} and updated in \cite{Ste16},
that we use to guarantee the existence
of the genuine and $\mathcal{F}$-model structures
$\mathcal{V}^G_{\text{gen}}$,
$\mathcal{V}^G_{\mathcal{F}}$
in Definition \ref{GENUINEMS_DEF}
(see Remark \ref{GEN_FGTRIGHT_REM}\ref{CELL_ITEM} or \cite[Prop. 2.6]{Ste16}).


\begin{definition}\label{CELL DEF}
	A model category $\mathcal{V}$ is said to have 
	\textit{cellular fixed points} if,
	for all finite groups $G$ and subgroups $H,K\leq G$,
	one has that:
\begin{itemize}
	\item[(i)] fixed points $(\minus)^H \colon \mathcal{V}^G \to \mathcal{V}$ preserve direct colimits;
	\item[(ii)] fixed points $(\minus)^H$ preserve pushouts where one leg is $(G/K)\cdot f$, for $f$ a cofibration;
	\item[(iii)] for each object $X \in \mathcal{V}$, the natural map 
	$(G/K)^H \cdot X \to ((G/K) \cdot X)^H$
	is an isomorphism.
\end{itemize}
\end{definition}

This section will establish some useful properties of the $\mathcal{V}^G_{\mathcal{F}}$ model structures.

We start by strengthening the 
cellularity conditions in Definition \ref{CELL DEF}.

\begin{proposition}\label{STRONGCELL PROP}
	Let $\mathcal{V}$ be a cofibrantly generated model category with cellular fixed points. Then:
	\begin{itemize}
		\item[(i)] $(\minus)^H \colon \mathcal{V}^G \to \mathcal{V}$ preserves cofibrations and pushouts where one leg
			 is a genuine cofibration;
		\item[(ii)] if $X$ is genuine cofibrant, the map 
			$(G/K)^H \cdot X^H \to (G \cdot_K X)^H$ is an isomorphism.
	\end{itemize}
\end{proposition}

\begin{proof}
Since both conditions are compatible with retracts, 
we are free to assume each cofibration $f\colon X \to Y$
(or, for $Y$ cofibrant, the map $\emptyset \to Y$)
is a transfinite composition
\begin{equation}\label{TRANSFCOMP EQ}
	X_0 \xrightarrow{f_0} 
	X_1 \xrightarrow{f_1}
	X_2 \xrightarrow{f_2}
	X_3 \xrightarrow{f_3} 
	\cdots
	\to Y = X_{\beta} = \colim_{\alpha < \beta} X_{\alpha}
\end{equation}
where each $f_{\alpha} \colon X_{\alpha} \to X_{\alpha+1}$
is the pushout of a generating cofibration
$(G/H) \cdot i_{\alpha}$. Both (i) and (ii) now follow by transfinite induction on $\alpha$ in the partial composite map
$X_0 \to X_{\alpha}$, with the successor ordinal case following by Def. \ref{CELL DEF}(ii)(iii) and the limit ordinal case by
Def. \ref{CELL DEF}(i). We note that (ii) also includes an obvious base case $X_0=\emptyset$.
\end{proof}

\begin{proposition}\label{FGTRIGHT PROP}
	Let $\phi \colon G \to \bar{G}$ be a homomorphism and $\mathcal{V}$ be cofibrantly generated with cellular fixed points.	
	Then the adjunction
\[
\begin{tikzcd}
	\phi_{!} = \bar{G} \cdot_G(\minus)
	\colon
	\mathcal{V}^{G}_{\mathcal{F}} \ar[shift left=1]{r}
&
	\mathcal{V}^{\bar{G}}_{\bar{\mathcal{F}}}
	\colon \ar[shift left=1]{l}
	\mathsf{res}^{\bar{G}}_G = \phi^{\**}
\end{tikzcd}
\]
is a Quillen adjunction provided that,
for any $H \in \mathcal{F}$,
we have $\phi(H) \in \bar{\mathcal{F}}$.
\end{proposition}

\begin{proof}
Since one has a canonical isomorphism of fixed points
$\left(\mathsf{res}(X)\right)^H \simeq X^{\phi(H)}$,
it is immediate that the right adjoint preserves fibrations and trivial fibrations.
\end{proof}

\begin{proposition}\label{FGTLEFT PROP}
	Let $\phi \colon G \to \bar{G}$ be a homomorphism and $\mathcal{V}$ be cofibrantly generated with cellular fixed points.		
	Then the adjunction
\[
\begin{tikzcd}
	\phi^{\**} = \mathsf{res}^{\bar{G}}_G
	\colon
	\mathcal{V}^{\bar{G}}_{\mathcal{\bar{F}}} \ar[shift left=1]{r}
&
	\mathcal{V}^{G}_{\mathcal{F}}
	\colon \ar[shift left=1]{l}
	\mathsf{Hom}_G(\bar{G},\minus) = \phi_{\**}
\end{tikzcd}
\]
is a Quillen adjunction provided that,
for any $H \in \bar{\mathcal{F}}$, it is 
$\phi^{-1}(H) \in \mathcal{F}$.
\end{proposition}

\begin{proof}
	A choice $\{a\}$ of double coset representatives of 
	$\phi(G)\backslash \bar{G} /H$
	gives $G$-orbit representatives of
	$\bar{G}/H$, yielding the formula
	$\mathsf{res}^{\bar G}_{G}(\bar{G}/H) \simeq 
	\coprod_{[a] \in \phi(G)\backslash \bar{G} /H}
	{\bar G/(\phi(G) \cap H^{a})}$.
	Hence
        \[
              \mathsf{res}\left(\bar{G}/H \cdot f\right)
              \simeq 
              \mathsf{res}\left(\bar{G}/H\right) \cdot f
              \simeq
              \left(
                    \coprod_{[a] \in \phi(G)\backslash \bar{G} /H}
                    {G/\phi^{-1}(H^{a})}
              \right)	\cdot f
        \]
        from which it follows that the left adjoint $\mathsf{res}$ preserves generating (trivial) cofibrations.
\end{proof}

Propositions \ref{FGTRIGHT PROP} and \ref{FGTLEFT PROP}
motivate the following definition.

\begin{definition}
	Let $\phi \colon G \to \bar{G}$ be a homomorphism and $\mathcal{F}$ and $\bar{\mathcal{F}}$ families in $G$
	and $\bar{G}$. We define
\begin{align}\label{PHISTARDEF EQ}
	\phi^{\**}(\bar{\mathcal{F}})
		=&
	\{H \leq G : \phi(H) \in \bar{\mathcal{F}}\}
\\
	\phi_!(\mathcal{F})
		=&
	\{\phi(H)^{\bar{g}}\leq \bar{G} : \bar{g} \in \bar{G}, H \in \mathcal{F}\}
\\ \label{PHISTARDEF3 EQ}
	\phi_{\**}(\mathcal{F})
		=&
	\{\bar{H} \leq \bar{G} : 
	\forall_{\bar{g} \in \bar{G}} 
	\left(
	\phi^{-1}(\bar{H}^{\bar{g}}) \in \mathcal{F}
	\right)\}
\end{align}
\end{definition}

\begin{lemma}\label{REWORFAM LEM}
The $\phi^{\**}(\bar{\mathcal{F}})$, $\phi_{!}(\mathcal{F})$, $\phi_{\**}(\mathcal{F})$ just defined are 
themselves families. Furthermore
\begin{itemize}
\item[(i)] The ``provided that'' condition in Proposition \ref{FGTRIGHT PROP} holds iff 
$\mathcal{F} \subset \phi^{\**} (\bar{\mathcal{F}})$
iff
$\phi_{!}(\mathcal{F}) \subset \bar{\mathcal{F}}$.
\item [(ii)]
The ``provided that'' condition in Proposition \ref{FGTLEFT PROP} holds iff 
$\phi^{\**} (\bar{\mathcal{F}}) \subset \mathcal{F}$
iff
$\bar{\mathcal{F}} \subset \phi_{\**}(\mathcal{F})$.
\end{itemize}
\end{lemma}

\begin{proof}
	Since the result is elementary, we include only the proof of the second iff in (ii), which is the hardest step and illustrates the necessary arguments. This follows by the following equivalences.
\[
	\phi^{\**} (\bar{\mathcal{F}}) \subset \mathcal{F}
\Leftrightarrow
	\left( \underset{ \substack{H \leq G \\ \phi(H) \in \bar{\mathcal{F}} }}{\mathlarger{\mathlarger{\forall}}} 
	H \in \mathcal{F} \right)
\Leftrightarrow
	\left( \underset{\bar{H} \in \bar{\mathcal{F}}}{\mathlarger{\mathlarger{\forall}}}
	\phi^{-1}(\bar{H}) \in \mathcal{F}
	\right)
\Leftrightarrow
	\left( \underset{ \substack{\bar{H} \in \bar{\mathcal{F}}\\\bar{g} \in \bar{G}}}{\mathlarger{\mathlarger{\forall}}}
	\phi^{-1}(\bar{H}^{\bar{g}}) \in \mathcal{F}
	\right)
\Leftrightarrow
        \bar{\mathcal F} \subset \phi_{\**}(\mathcal F)
\]
Here the second equivalence follows since 
$H \leq \phi^{-1}(\phi(H))$ and $\mathcal{F}$ is closed under subgroups while the third equivalence follows since 
$\bar{\mathcal{F}}$ is closed under conjugation. 
\end{proof}

\begin{proposition}\label{BIQUILLENG PROP}
	Suppose that $\mathcal{V}$ is cofibrantly generated, has cellular fixed points, and is also a closed monoidal model category. 	
	Then the bifunctor
\[
	\mathcal{V}^G_{\mathcal{F}}
		\times
	\mathcal{V}^G_{\bar{\mathcal{F}}}
		\xrightarrow{\otimes}
	\mathcal{V}^G_{\mathcal{F} \cap \bar{\mathcal{F}}}
\]
	is a left Quillen bifunctor.
\end{proposition}

\begin{proof}
	A choice $\{a\}$ of double coset representatives
	of $H \backslash G /\bar{H}$
	gives $G$-orbit representatives
	$\left\{([e],[a])\right\}$ of
	$G/H \times G/\bar{H}$,
	yielding the formula
	$G/H \times G/\bar{H}
	\simeq 
	\coprod_{[a]\in H \backslash G /\bar{H}}
	{G/H\cap \bar{H}^a}
	$.
	Hence
\[
	\left(G/H \cdot f\right) \square \left(G/\bar{H} \cdot g\right)
		\simeq
	\left(G/H \times G/\bar{H}\right) \cdot \left(f \square g\right)
		\simeq
	\left(
		\coprod_{[a]\in H \backslash G /\bar{H}}
		{\left(G/H\cap \bar{H}^a\right)} \cdot (f \square g)
	\right)
\]
and the result follows since families are closed under conjugation and subgroups.
\end{proof}

\begin{definition}\label{EXTERINT DEF}
Let $\mathcal{F}$ and $\bar{\mathcal{F}}$ be families of $G$ and $\bar{G}$, respectively.

We define their \textit{external intersection} to be the 
family of $G \times \bar{G}$ given by
\[
	\mathcal{F} \sqcap \bar{\mathcal{F}}
=
	(\pi_{G})^{\**} (\mathcal{F}) 
		\cap
	(\pi_{\bar{G}})^{\**} (\bar{\mathcal{F}})
\]
for 
$\pi_G \colon G \times \bar{G} \to G$,
$\pi_{\bar{G}} \colon G \times \bar{G} \to \bar{G}$
the projections.
\end{definition}

\begin{remark}
	Combining Proposition \ref{FGTLEFT PROP} 
	with Propositon \ref{BIQUILLENG PROP} yields that
	the following composite is a left Quillen bifunctor.
\begin{equation}\label{EXTERINTADJ EQ}
	\mathcal{V}^{G}_{\mathcal{F}}
		\times
	\mathcal{V}^{\bar{G}}_{\bar{\mathcal{F}}}
		\xrightarrow{\mathsf{res}}
	\mathcal{V}^{G \times \bar{G}}_{
	(\pi_G)^{\**}(\mathcal{F})}
		\times
	\mathcal{V}^{G \times \bar{G}}_{
	(\pi_{\bar{G}})^{\**}(\bar{\mathcal{F}})}
		\xrightarrow{\otimes}
	\mathcal{V}^{G \times \bar{G}}_{
	\mathcal{F} \sqcap \bar{\mathcal{F}}}
\end{equation}
\end{remark}

\subsection{Pushout powers}\label{PUSHPOW SEC}

That \eqref{EXTERINTADJ EQ} is a left Quillen bifunctor (and its obvious higher order analogues) is one of the key properties of pushout products of $\mathcal{F}$ cofibrations when those cofibrations (and the group) are allowed to change. However, when those cofibrations (and hence $G$) coincide there is an additional symmetric group action that  we will need to consider.

To handle these actions we will need two new axioms, 
which will concern cofibrancy and fixed point properties. We start by discussing the cofibrancy axiom.

\begin{definition}\label{COFSYMPUSHPOW}
	We say that a symmetric monoidal model category $\mathcal{V}$ has 
	\textit{cofibrant symmetric pushout powers} if,
	for each (trivial) cofibration $f$,
	the pushout product power 
	$f^{\square n}$ is a $\Sigma_n$-genuine 
	(trivial) cofibration.
\end{definition}

\begin{remark}
When $\mathcal{V}$ is cofibrantly generated
the condition in Definition \ref{COFSYMPUSHPOW} needs only be checked for generating cofibrations. 
However, the argument needed is harder than usual
(see, e.g., \cite[Lemma 2.1.20]{Ho98}) due to $(-)^{\square n}$ not preserving composition of maps:
one instead follows the argument in the proof of 
Proposition \ref{POWERF PROP} below when $G=\**$.
\end{remark}

\begin{remark}\label{CSPP_REM}
	The cofibrant symmetric pushout powers condition
	can be viewed as an adaptation of 
	the \textit{power cofibration axiom} of 
	\cite[Def. 4.5.4.2(iii)]{Lur17},
	which asks instead for 
	$f^{\square n}$
	to be a cofibration
	in the \textit{projective} model structure $\V^{\Sigma_n}_{\text{proj}}$.
	Along with some technical conditions,
	the latter axiom suffices for the existence
	of projective (semi-)model structures
	on many categories of algebraic objects,
	such as operads and commutative monoids \cite[Prop. 4.5.4.6]{Lur17} (also \cite[Prop. 6.2.2, Thm. 6.2.3]{WY18}).
	As an aside, we note that the 
	power cofibration axiom
	is quite restrictive (e.g. it fails in $\mathsf{sSet}$),
	and this has lead to the identification of a number of laxer variants
	(e.g. 
	\cite[Def. 3.1]{Wh17},
	\cite[Thm. 6.1.1]{WY18},
	\cite[Def. 2.1]{PS18b}).
	
	
	In practice, the main difference between 
	Definition \ref{COFSYMPUSHPOW} and 
	\cite[Def. 4.5.4.2(iii)]{Lur17}
	is that they are designed for different contexts:
	\cite[Def. 4.5.4.2(iii)]{Lur17}
	(and its variants) serve to build projective model structures;
	Definition \ref{COFSYMPUSHPOW} 
	serves to build fixed point model structures,
	such as the ones in Theorem \ref{MAINEXIST1 THM}.

\end{remark}


\begin{example}
      \label{SSET_CSPP_EX}
	Both $(\mathsf{sSet},\times)$ and 
	$(\mathsf{sSet}_{\**},\wedge)$ have cofibrant symmetric pushout powers.
        To see this, we note first that the case of (non-trivial) cofibrations is immediate since
        genuine cofibrations 
        are precisely the monomorphisms. 
	For the case of $f \colon X \to Y$ a trivial cofibration, it is easier to first show directly that 
	$f^{\otimes n} \colon X^{\otimes n} \to Y^{\otimes n}$
	is a trivial cofibration, 
	and then use the factorizations
	\eqref{COMPNFOLDFACT EQ}
	for $h=f$, $g=(\emptyset \to X)$, 
	in which case $f^{\otimes n} = k_n\cdots k_1$ and 
	$f^{\square n} = k_n$,
	to show by induction on $n$ that 
	$f^{\square n}$ is also a trivial cofibration.
\end{example}

We now turn to describing the symmetric power analogue of 
Definition \ref{EXTERINT DEF}.

We start with notation. Letting 
$\lambda$ be a partition 
$E = \lambda_1 \amalg\cdots \amalg \lambda_k$
of a finite set $E$, 
we write 
 $\Sigma_{\lambda} = \Sigma_{\lambda_1} \times \cdots \times
 \Sigma_{\lambda_k} \leq \Sigma_E$ for the subgroup of permutations preserving $\lambda$. 
 In addition, given $e \in E$ we write
$\lambda_e$ for the partition $E = \{e\} \amalg (E-e)$, so that $\Sigma_{\lambda_e}$ is the isotropy of $e$.

\begin{definition}\label{FLTIMESN DEF}
 Let $\mathcal{F}$ be a family of $G$,
 $E$ a finite set and $e \in E$ any fixed element.
 
We define the \textit{$n$-th semidirect power of $\mathcal{F}$} to be the family of $\Sigma_E \wr G = \Sigma_E \ltimes G^{\times E}$ given by
\[
	\mathcal{F}^{\ltimes E}
		=
	\left(
	\iota_{\Sigma_{\lambda_e} \wr G}
	\right)_{\**}
	\left(
		\left(
		\pi_{G})^{\**}\left(\mathcal{F}\right)
		\right)
	\right),
\]
where $\iota$ is the inclusion 
$\Sigma_{\lambda_e} \wr G
	\to 
\Sigma_E \wr G$
and $\pi$ the projection
$\Sigma_{\lambda_e} \wr G = \Sigma_{\{e\}} \times G \times \Sigma_{E-e} \wr G
\to G$.

More explicitly, since in \eqref{PHISTARDEF3 EQ} one needs only consider conjugates by coset representatives of $\bar{G}/\phi(G)$, when computing 
$\left( \iota_{\Sigma_{\lambda_e} \wr G}\right)_{\**}$
one needs only conjugate by coset representatives of 
$\left(\Sigma_E \wr G\right)/
\left(\Sigma_{\lambda_e} \wr G\right)
\simeq \Sigma_E/\Sigma_{\lambda_e}$, so that
\begin{equation}\label{FLTIMESN2 EQ}
	K \in \mathcal{F}^{\ltimes E} 
	\text{ iff }
	\underset{e \in E}{\forall} \pi_{G}
	\left(
		K \cap \left( \Sigma_{\lambda_e} \wr G \right)
	\right)
	\in \mathcal{F},
\end{equation}
showing that, in particular, $\mathcal{F}^{\ltimes E}$
is independent of the choice of $e \in E$.
\end{definition}

\begin{remark}
The previous definition is likely to seem mysterious at first. Ultimately, the origin of this definition
is best understood by working through this section backwards:
the study of the interactions between equivariant trees and graph families, namely Lemma \ref{KEYLEMMAGECO LEM}, requires the study of the families $\mathcal{F}^{\ltimes_G n}$ in Notation \ref{SEMIDIRG NOT}, which are variants of the $\mathcal{F}^{\ltimes n}$ construction for graph families.
It then suffices, and is notationally far more convenient, to establish the required results first for the $\mathcal{F}^{\ltimes n}$ families,
and then translate them to the $\mathcal{F}^{\ltimes_G n}$ families.
\end{remark}

\begin{proposition}\label{LTIMESPRODINC PROP}
	Writing 
	$\iota \colon \Sigma_E \times \Sigma_{\bar{E}} \to
	\Sigma_{E \amalg \bar{E}}$ for the inclusion, one has 
\[
	\mathcal{F}^{\ltimes E}
		\sqcap
	\mathcal{F}^{\ltimes \bar{E}}
		\subset
	\iota^{\**}\left(\mathcal{F}^{\ltimes E \amalg \bar{E}}\right).
\]
	Hence, the following is a left Quillen bifunctor for $\mathcal V$ as in Proposition \ref{BIQUILLENG PROP}.
\begin{equation}\label{LTIMESPRODQUI EQ}
	\Sigma_{E \amalg \bar{E}} 
	\underset{\Sigma_E \times \Sigma_{\bar{E}}}{\cdot}
	(\minus \otimes \minus)
		\colon
	\mathcal{V}^{\Sigma_E \wr G}
		\times
	\mathcal{V}^{\Sigma_{\bar{E}} \wr G}
		\to
	\mathcal{V}^{\Sigma_{E \amalg \bar{E}} \wr G}
\end{equation}
\end{proposition}

\begin{proof}
	Let 
	$K \in 
	\mathcal{F}^{\ltimes E}
		\sqcap
	\mathcal{F}^{\ltimes \bar{E}}	
	$
	and $e \in E$. 
	We write $\lambda_e$ for the partition of $E \amalg \bar{E}$
	and $\lambda_e^E$ for the partition of $E$.
	One then has
\[
\pi_G
\left(
	K \cap \left( \Sigma_{\lambda_e} \wr G \right) \right)
	=
\pi_G
\left(
	\pi_{\Sigma_E \wr G}(K)
	\cap \left( \Sigma_{\lambda_e^E} \wr G \right)
\right),
\]
where on the right we write
$\pi_{\Sigma_E \wr G} \colon
\Sigma_E \wr G \times \Sigma_{\bar{E}} \wr G
\to 
\Sigma_E \wr G$
and 
$\pi_G \colon \Sigma_{\lambda^E_e} \wr G
=\Sigma_{\{e\}} \times G \times \Sigma_{E-e} \wr G
\to G$. Therefore $K$ 
satisfies \eqref{FLTIMESN2 EQ} for 
$\mathcal{F}^{\ltimes E \amalg \bar{E}}$
since 
$\pi_{\Sigma_E \wr G}(K)$ does so for 
$\mathcal{F}^{\ltimes E}$.
The case of $e \in \bar{E}$ is identical.

\eqref{LTIMESPRODQUI EQ} simply combines 
the left Quillen bifunctor
\eqref{EXTERINTADJ EQ} with 
Proposition \ref{FGTRIGHT PROP}.
\end{proof}

\begin{proposition}\label{POWERF PROP}
	Suppose that $\mathcal{V}$ is a cofibrantly generated closed monoidal model category with cellular fixed points and with cofibrant symmetric pushout powers.
	
	Then, for every $n \geq 1$ and cofibration (resp. trivial cofibration) $f$ of $\mathcal{V}^{G}_{\mathcal{F}}$,
	one has that $f^{\square n}$ is a cofibration (trivial cofibration) of $\mathcal{V}^{\Sigma_n \wr G}_{\mathcal{F}^{\ltimes n}}$.
\end{proposition}

Our proof of Proposition \ref{POWERF PROP} will essentially repeat the main argument in the proof of
\cite[Thm. 1.2]{Pe16}.
However, both for the sake of completeness and to stress that the argument is independent of the (fairly technical) model structures in \cite{Pe16}, we include an abridged version of the proof below, the key ingredient 
of which is that \eqref{LTIMESPRODQUI EQ} is a left Quillen bifunctor.

\begin{proof}
	Consider first the case of a generating
	(trivial) cofibration
	$i = (G/H) \cdot \bar{\imath}$, $H\in \mathcal{F}$, so that 
\begin{equation}\label{GENCOFWR EQ}
	i^{\square n} = 
	(G/H)^{\times n} \cdot \bar{\imath}^{\square n}
	\simeq \Sigma_n \wr G 
	\underset{\Sigma_n \wr H}{\cdot} \bar{\imath}^{\square n},
\end{equation}
	where the action of 
	$\Sigma_n \wr G$
	(resp. $\Sigma_n \wr H$)
	on $\bar{\imath}^{\square n}$
	in the second (resp. third) term
	is given by the projection to $\Sigma_n$.
	To justify the second identification 
	in \eqref{GENCOFWR EQ},
	note that the inclusion 
	$\bar{\imath}^{\square n}
	\to 
	(G/H)^{\times n} \cdot \bar{\imath}^{\square n}$
	onto the 
	$([e],\cdots,[e])$
	component
	is $(\Sigma \wr H)$-equivariant
	and thus induces
	a $(\Sigma \wr G)$-equivariant map 
	$\Sigma_n \wr G 
	\cdot_{\Sigma_n \wr H}
	\bar{\imath}^{\square n}
	\to 
	(G/H)^{\times n} \cdot \bar{\imath}^{\square n}$.
	This latter map is an isomorphism since,
	non-equivariantly,
	both sides are a coproduct of
	$|\Sigma_n \wr G : \Sigma_n \wr H|
	= |G \colon H|^{\times n}$
	copies of $\bar{\imath}^{\square n}$.
	Next, note that 
	$\bar{\imath}^{\square n}$
	is a $\Sigma_n$-genuine (trivial) cofibration
	by the cofibrant symmetric pushout powers assumption
	and thus, by Proposition \ref{FGTLEFT PROP},
	also a 
	$(\Sigma_n \wr H)$-genuine (trivial) cofibration. 
	Thus, since 
	$\Sigma_n \wr H \in \mathcal{F}^{\ltimes n}$,
	Proposition \ref{FGTRIGHT PROP} implies that
	$i^{\square n}$ is a $\mathcal{F}^{\ltimes n}$ 
	(trivial) cofibration, as desired.

	For the general case, we start by making the key observation that,
	for composable arrows 
	$\bullet \xrightarrow{g} \bullet \xrightarrow{h} \bullet$,
	the $n$-fold pushout product $(hg)^{\square n}$
	has a $\Sigma_n$-equivariant factorization
\begin{equation}\label{COMPNFOLDFACT EQ}
	\bullet
		\xrightarrow{k_0}
	\bullet
		\xrightarrow{k_1}
	\cdots
		\xrightarrow{k_n}
	\bullet
\end{equation}
where each $k_r$, $0 \leq r \leq n$, fits into a pushout diagram
\begin{equation}\label{COMPNFOLDFACTPUSH EQ}
\begin{tikzcd}
	\bullet \ar{r} \ar{d}[swap] 
	{\Sigma_n \underset{\Sigma_{n-r} \times \Sigma_r}
	{\cdot}\left( g^{\square n-r} \square h^{\square r} \right)} 
	\ar[dr,phantom, "\ulcorner", near start]
	&
	\bullet \ar{d}{k_r}
\\
	\bullet \ar{r} 
	&
	\bullet.
\end{tikzcd}
\end{equation}
Briefly, \eqref{COMPNFOLDFACT EQ} follows from
a filtration 
$P_0 \subset P_1 \subset \cdots \subset P_n$
of the poset $P_n = (0 \to 1 \to 2)^{\times n}$ 
where $P_0$ consists of ``tuples with at least one $0$-coordinate'' and $P_r$ is obtained from $P_{r-1}$ by adding the ``tuples with $n-r$ $1$-coordinates and $r$ $2$-coordinates''.
Additional details concerning this filtration appear in the proof of \cite[Lemma 4.8]{Pe16}.

The general proof now follows by writing $f$ as a retract of a transfinite composition of pushouts of generating (trivial) cofibrations as in \eqref{TRANSFCOMP EQ}.
As usual, retracts preserve weak equivalences,
and we can hence assume that there is an ordinal $\kappa$
and $X_{\bullet} \colon \kappa \to \mathcal{V}^G$
such that 
\begin{enumerate*}
\item[(i)] 
$f_{\beta} \colon X_{\beta} \to X_{\beta+1}$
is the pushout of a generating (trivial) cofibration $i_{\beta}$;
\item[(ii)] 
$\colim_{\alpha < \beta} X_{\alpha} \xrightarrow{\simeq} X_{\beta}$ for limit ordinals $\beta < \kappa$;
\item[(iii)] setting 
$X_{\kappa} = \colim_{\beta < \kappa} X_{\beta}$, 
$f$ equals the transfinite composite $X_0 \to X_{\kappa}$.
\end{enumerate*}

We argue by transfinite induction on $\kappa$.
Writing $\bar{f}_{\beta} \colon X_0 \to X_{\beta}$ for the partial composites, it suffices to check that the natural transformation of $\kappa$-diagrams (rightmost map not included)
\[
\begin{tikzcd}
		Q^n(\bar{f}_{1}) \ar{d}[swap]{\bar{f}_1^{\square n}} \ar{r}
	&
		Q^n(\bar{f}_{2}) \ar{d}[swap]{\bar{f}_2^{\square n}} \ar{r}
	&
		Q^n(\bar{f}_{3}) \ar{d}[swap]{\bar{f}_3^{\square n}} \ar{r}
	&
		Q^n(\bar{f}_{4}) \ar{d}[swap]{\bar{f}_4^{\square n}} \ar{r}
	&
		\cdots \ar{r}
	&
		Q^n(\bar{f}_{\kappa}) \ar{d}{\bar{f}_{\kappa}^{\square n}
		=\colim_{\beta < \kappa} \bar{f}_{\beta}^{\square n}}
\\
		X_1^{\otimes n} \ar{r}
	&
		X_2^{\otimes n} \ar{r}
	&
		X_3^{\otimes n} \ar{r}
	&
		X_4^{\otimes n} \ar{r}
	&
		\cdots \ar{r}
	&
		X_{\kappa}^{\otimes n},
\end{tikzcd}
\]
is (trivial) $\kappa$-cofibrant, i.e. that the maps 
$Q^n(\bar{f}_{\beta})
\amalg_{\colim_{\alpha < \beta} Q^n(\bar{f}_{\alpha}) }
\colim_{\alpha < \beta} X_{\alpha}^{\otimes n} 
	\to
X_{\beta}^{\otimes n} 
$ are (trivial) cofibrations in 
$\mathcal{V}^{\Sigma_n \wr G}_{\mathcal{F}^{\ltimes n}}$.
Condition (ii) above implies that this map is an isomorphism for $\beta$ a limit ordinal 
while for $\beta+1$ a successor ordinal it is the map
$Q^n(\bar{f}_{\beta+1})
\amalg_{Q^n(\bar{f}_{\beta}) }
X_{\beta}^{\otimes n} 
	\to
X_{\beta+1}^{\otimes n}$.
But since 
$Q^n(\bar{f}_{\beta+1}) \to Q^n(\bar{f}_{\beta+1})
\amalg_{Q^n(\bar{f}_{\beta}) }
X_{\beta}^{\otimes n}$ 
is precisely the map $k_0$ of \eqref{COMPNFOLDFACT EQ} for 
$g=\bar{f}_{\beta}$, $h=f_{\beta}$, this last map is the composite $k_nk_{n-1}\cdots k_1$ so that the result now follows from \eqref{COMPNFOLDFACTPUSH EQ} 
together with the left Quillen bifunctor
\eqref{LTIMESPRODQUI EQ}
since:
\begin{enumerate*}
\item[(i)] the induction hypothesis shows the cofibrancy of  $\bar{f}_{\beta}^{\square n-r}$;
\item[(ii)] the cofibrancy of $i_{\beta}^{\square r}$
together with the fact that 
$f_{\beta}^{\square r}$ is a pushout of $i_{\beta}^{\square r}$
(cf. \cite[Lemma 4.11]{Pe16})
 imply the cofibrancy of $f_{\beta}^{\square r}$.
\end{enumerate*}
\end{proof}

We now turn to discussing the fixed points of
pushout powers $f^{\square n}$.

Firstly, we assume throughout the following discussion that 
$(\mathcal{V},\otimes)$
has diagonal maps, as in
Remark \ref{FINSURJ REM}.
In particular, one has compatible
$\Sigma_n$-equivariant
maps $X \to X^{\otimes n}$.

Consider now a $K$-object
$(X_e)_{e \in E}$ in $(\Fin_s \wr \mathcal{V})^K$ for some finite group $K$.
Explicitly, this consists of an action of $K$ on the indexing set $E$ together with suitably associative and unital isomorphisms
$X_e \to X_{ke}$
for each $(e,k) \in E \times K$.
Moreover, writing $K_e$ for the isotropy of $e \in E$,
note that the induced fixed point isomorphism
$X_e^{K_e} \to X_{k e}^{K_{ke}}$
does not depend on the choice of coset representative $k \in k K_e$,
and we will thus abuse notation
by writing
$X_{[e]}^{K_{[e]}} = X_f^{K_f}$ for an arbitary choice of representative $f \in [e] = Ke$
(more formally, we mean that 
$X_{[e]}^{K_{[e]}} = \left(\coprod_{f\in[e]}
X_f^{K_f}\right)/\Sigma_{[e]}$).

Diagonal maps then induce canonical composites
(generalizing the twisted diagonals discussed following Remark \ref{REFLCOREFL REM})
\[
	X_{[e]}^{K_{[e]}}
\to
	\left( X_{[e]}^{K_{[e]}} \right)^{\otimes [e]}
\simeq
	\bigotimes_{f \in [e]} X_f^{K_f}
\to
	\bigotimes_{f \in [e]} X_f,
\]
leading to the following axiom.

\begin{definition}\label{CARTFIX DEF}
We say that a symmetric monoidal category
with diagonals $\mathcal{V}$ has \textit{cartesian fixed points} if the canonical maps
\begin{equation}\label{CARTFIX EQ}
\begin{tikzcd}
\bigotimes_{[e] \in E/K} X_{[e]}^{K_{[e]}}
	\ar{r}{\simeq} &
\left( \bigotimes_{e \in E} X_e \right)^K
\end{tikzcd}
\end{equation}
are isomorphisms for all $(X_e)_{e \in E}$ in $(\mathsf F_s \wr \V)^K$ for all finite groups $K$.
\end{definition}

\begin{remark}
As the name implies, the condition in the previous definition is automatic for cartesian 
$\mathcal{V}$. Moreover, this condition is easily seen to hold for $\mathcal{V} = \mathsf{sSet}_{\**}$.

The condition \eqref{CARTFIX EQ} naturally breaks down into two conditions.

The first condition, which makes sense in the absence of diagonals, corresponds to the case where $K$ acts trivially on $E$,
and says that $X^K \otimes Y^K \xrightarrow{\simeq} (X \otimes Y)^K$, for $X,Y \in \mathcal{V}^K$.

The second condition, corresponding to the case where $K$ acts transitively, 
concerns the fixed points of what is often called the norm object
$N_{K_e}^K X_e \simeq \bigotimes_{e \in E} X_e$.

These two conditions can roughly
be viewed as multiplicative analogues 
of the two parts of Proposition \ref{STRONGCELL PROP},
though now without cofibrancy requirements.
In fact, if one modifies
Definition \ref{CARTFIX DEF} 
by requiring that
\eqref{CARTFIX EQ}
be an isomorphism only when the
$X_e$ are $K_e$-cofibrant,
it is not hard to show that
this modified condition can be deduced from
the requirement that $\mathcal{V}$ be strongly
cofibrantly generated (i.e. that the domains/codomains of the (trivial) generating cofibrations be cofibrant)
together with
isomorphisms
$X^{\otimes (G/H)^K} \xrightarrow{\simeq} 
\left( X^{\otimes G/H} \right)^K$
for $X \in \mathcal{V}$
(i.e. a power analogue of 
Definition \ref{CELL DEF} (iii)).
\end{remark}

\begin{proposition}\label{FIXEDPUSH PROP}
	Suppose that $\mathcal{V}$ is as in Proposition \ref{POWERF PROP},
	and also has diagonals and cartesian fixed points.
	Let $K \leq \Sigma_n \wr G$ be a subgroup, 
	$f \colon X \to Y$ a map in $\mathcal{V}^G$,
	and consider the natural maps (in the arrow category)
\begin{equation}\label{FIXEDPUSH EQ}
	\underset{[i] \in n/K}{\mathlarger{\mathlarger{\square}}}
	f_{[i]}^{K_{[i]}}
\to
	\left( f^{\square n} \right)^K.
\end{equation}
If $f$ is a genuine cofibration between genuine cofibrant objects then 
\eqref{FIXEDPUSH EQ} is an isomorphism.
\end{proposition}

At first sight, it may seem that 
the desired isomorphism
\eqref{FIXEDPUSH EQ}
should be an immediate consequence of \eqref{CARTFIX EQ}. However, the real content here is that the two pushout products in 
\eqref{FIXEDPUSH EQ} are computed over cubes of different sizes. Namely, while the right hand side is computed using the cube
$(0 \to 1)^{\times n}$,
the left hand side is computed over the fixed point cube
$\left((0 \to 1)^{\times n} \right)^K
\simeq (0 \to 1)^{\times n/K}$,
formed by those tuples
whose coordinates coincide if their indices are in the same coset of $n/K$.



\begin{example}
When $n=3$ and $n/K = \left\{\{1,2\},\{3\}\right\}$ the fixed subposet $(0 \to 1)^{\times n/K}$ is displayed on the right below.
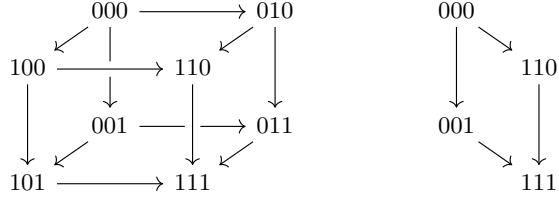
\[
\begin{tikzcd}[column sep=1em, row sep=1em]
	& 000 \ar{rr} \ar{ld} \ar{dd} && 010 \ar{ld} \ar{dd}&
	&&&& 000 \ar{dd} \ar{rd}
\\
	100 \ar[crossing over]{rr} \ar{dd} && 110 &&
	&&&& & 110 \ar{dd}
\\
	& 001 \ar{rr} \ar{ld} && 011 \ar{ld}&
	&&&& 001 \ar{rd}
\\
	101 \ar{rr} && 111 \ar[leftarrow,crossing over]{uu} &&
	&&&& & 111 
\end{tikzcd}
\]
\end{example}


\begin{proof}[proof of Proposition \ref{FIXEDPUSH PROP}]
This will follow by induction on $n$. The base case $n=1$ is obvious.

Moreover, it is clear from \eqref{CARTFIX EQ} that \eqref{FIXEDPUSH EQ}, which is a map of arrows, is an isomorphism on the target objects, hence the real claim is that this map is also an isomorphism on sources.

We now note that, 
by considering \eqref{COMPNFOLDFACT EQ} for
$g= (\emptyset \to X)$, $h=f$,
and removing the last map $k_n$,
one obtains a filtration of the source of $f^{\square n}$.
Applying $(\minus)^K$ to the leftmost map in 
\eqref{COMPNFOLDFACTPUSH EQ}
one has isomorphisms
\begin{align*}
	\left(
	\Sigma_n \underset{\Sigma_{n-i} \times \Sigma_i}
	{\cdot} X^{\otimes n-i} \otimes f^{\square i}
	\right)^K
\simeq &
	\coprod_{\substack{n/K=A/K \amalg B/K \\
	|A|=n-i,|B|=i}}
	\left( X^{\otimes A} \otimes f^{\square B} \right)^K
\simeq
	\coprod_{\substack{n/K=A/K \amalg B/K \\
	|A|=n-i,|B|=i}} 
	\left( X^{\otimes A}\right)^K \otimes \left( f^{\square B} \right)^K
\\
\simeq &
	\coprod_{\substack{n/K=A/K \amalg B/K \\
	|A|=n-i,|B|=i}} 
	\left(
	\bigotimes_{[j]\in A/K} X_{[j]}^{K_{[j]}}
	\right)
\otimes 
	\left(
	\underset{[k] \in B/K}
	{\mathlarger{\mathlarger{\square}}}
	f_{[k]}^{K_{[k]}}
	\right)
\end{align*}
Here the first step is an instance of Proposition \ref{STRONGCELL PROP}(ii),
with the required cofibrancy conditions following from Proposition \ref{POWERF PROP}. The second step follows from \eqref{CARTFIX EQ}.
Lastly, the third step follows by
\eqref{CARTFIX EQ} together with the induction hypothesis, which applies since $|B|=i<n$.

Noting that Proposition \ref{POWERF PROP} guarantees that all required maps are cofibrations
so that fixed points $(\minus)^K$ commute with pushouts by Proposition \ref{STRONGCELL PROP}(i),
we have just shown that 
the leftmost maps in the pushout diagrams \eqref{COMPNFOLDFACTPUSH EQ} for 
$\left( f^{\square n} \right)^K$
are isomorphic to the leftmost maps in the pushout diagrams for the corresponding filtration of 	
$\underset{[i] \in n/K}{\mathlarger{\mathlarger{\square}}}
 f_{[i]}^{K_{[i]}}$.
\end{proof}

\begin{corollary}\label{FIXEDPUSH COR}
	Given a partition $\lambda$ given by
	$\{1,2,\cdots,n\} = \lambda_1 \amalg \cdots \amalg \lambda_k$, cofibrations between cofibrant objects $f_i$ in $\mathcal{V}^{G_i}$, $1\leq i \leq k$,
	and a subgroup
	$K \leq 
	\Sigma_{\lambda_1} \wr G_1
	\times \cdots \times
	\Sigma_{\lambda_k} \wr G_k
	$,
	the natural map
\[
	\underset{1\leq i\leq k}{\mathlarger{\mathlarger{\square}}}
	\phantom{|}
	\underset{[j] \in \lambda_i/K}{\mathlarger{\mathlarger{\square}}}
	f_{i, [j]}^{K_{[j]}}
\to
	\left( 	\underset{1\leq i\leq k}{\mathlarger{\mathlarger{\square}}}f_i^{\square \lambda_i} \right)^K.
\]
is an isomorphism.
\end{corollary}

\begin{proof}
This combines Proposition \ref{FIXEDPUSH PROP}
with the easier isomorphisms
$f^K \square g^K \xrightarrow{\simeq} 
(f \square g)^K$,
which follow by \eqref{CARTFIX EQ}
together with the observation that $(\minus)^K$
commutes with pushouts thanks to the cofibrancy conditions and Proposition \ref{STRONGCELL PROP}(i).
\end{proof}

\subsection{$G$-graph families and $G$-trees}
\label{G_GRAPH_SECTION}

We now convert the results in the previous sections to the context we are truly interested in:
graph families. 
Throughout this section $\Sigma$ will denote a general group,
usually meant to be some type of permutation group.

\begin{definition}
        \label{GRAPH DEF}
A subgroup $\Gamma \leq G \times \Sigma$ is called a
\textit{$G$-graph subgroup} if $\Gamma \cap \Sigma = \**$. 

Further, a family $\mathcal{F}$ of $G \times \Sigma$ is called a \textit{$G$-graph family} if it consists of $G$-graph subgroups.
\end{definition}

\begin{remark}\label{GRAPH REM}
$\Gamma$ is a $G$-graph subgroup iff it can be written as
\[
	\Gamma = 
	\{
	(h,\phi(h)) : h \in H \leq G
	\}
\]
for some partial homomorphism $H \xrightarrow{\phi} \Sigma$
for $H \leq G$, 
thus motivating the terminology.
\end{remark}

\begin{remark}
	The collection of all $G$-graph subgroups is itself a family,
	denoted $\mathcal{F}^{\Gamma}$.
	Indeed, this family  coincides with 
	$(\iota_{\Sigma})_{\**}(\{\**\})$
for the inclusion homomorphism 
$\iota_{\Sigma} \colon \Sigma \to G \times \Sigma$.
\end{remark}

\begin{notation}\label{SEMIDIRG NOT}
Letting $\mathcal{F}$, $\bar{\mathcal{F}}$ be $G$-graph families of $G \times \Sigma$ and $G \times \bar{\Sigma}$ we will write
\[
	\mathcal{F} \sqcap_G \bar{\mathcal{F}} 
	= \Delta^{\**} (\mathcal{F} \sqcap \bar{\mathcal{F}} )
\qquad \qquad
	\mathcal{F}^{\ltimes_G n} = \Delta^{\**} (\mathcal{F}^{\ltimes n})
\]
where $\Delta$ denotes either of the diagonal inclusions
$\Delta \colon 
G \times \Sigma \times \bar{\Sigma} \to 
G \times \Sigma \times G \times \bar{\Sigma}$
or 
$\Delta \colon G \times (\Sigma_n \wr \Sigma)
 \to 
\Sigma_n \wr (G \times \Sigma)$.
\end{notation}

\begin{remark}\label{UNPACKINGSQCAP REM}
	Unpacking Definition \ref{EXTERINT DEF} one has that 
	$\Gamma \in \mathcal{F} \sqcap_G \bar{\mathcal{F}}$ iff
	$\pi_{G \times \Sigma}(\Gamma) \in \mathcal{F}$ and
	$\pi_{G \times \bar{\Sigma}}(\Gamma) \in \bar{\mathcal{F}}$.
\end{remark}

\begin{remark}\label{UNPACKINGLTIMES REM}
Given a finite set $E$,
the image of the inclusion	
$\Delta \colon G \times (\Sigma_E \wr \Sigma)
\to 
\Sigma_E \wr (G \times \Sigma)$
consists of the elements
$(\sigma,(g_e,\tau_e)_{e \in E}),
\sigma \in \Sigma_n,
g_e \in G,
\tau_e \in \Sigma$
such that all $g_e, e \in E$ coincide.
Hence, for fixed $e\in E$,
and when viewed as subgroups of $\Sigma_E \wr (G \times \Sigma)$,
one has an identification
\[
	\left(G \times \Sigma_E \wr \Sigma\right)
	\cap
	\left(
	\Sigma_{\lambda_e} \wr (G \times \Sigma)
	\right)
=
	G \times (\Sigma_{\lambda_e} \wr \Sigma)
\]
(the subgroup $\Sigma_{\lambda_e} \leq \Sigma_E$ is as described prior to Definition \ref{FLTIMESN DEF}).

Thus, unpacking \eqref{FLTIMESN2 EQ} one has  
\[
	K \in \mathcal{F}^{\ltimes_G E} 
	\text{ iff }
	\underset{e \in E}{\forall} \pi_{G \times \Sigma}
	\left(
		K \cap 
		\left(G \times (\Sigma_{\lambda_e} \wr \Sigma) \right)
	\right)
	\in \mathcal{F}.
\]
\end{remark}

Combining either the left Quillen bifunctor \eqref{EXTERINTADJ EQ} or 
Proposition \ref{POWERF PROP}
with Proposition \ref{FGTLEFT PROP} yields the following results.

\begin{proposition}\label{EXTERINTADJG PROP}
Suppose that $\mathcal{V}$ is a cofibrantly generated closed monoidal model category with cellular fixed points.
Let $\mathcal{F}$, $\bar{\mathcal{F}}$ be $G$-graph families of 
$G \times \Sigma$ and $G \times \bar{\Sigma}$. Then the following (with diagonal $G$-action on the image) 
is a left Quillen bifunctor.
\[
	\mathcal{V}^{G \times \Sigma}_{\mathcal{F}}
		\times
	\mathcal{V}^{G \times \bar{\Sigma}}_{\bar{\mathcal{F}}}
		\xrightarrow{\otimes}
		\mathcal{V}^{G \times \Sigma \times \bar{\Sigma}}_{
	\mathcal{F} \sqcap_G \bar{\mathcal{F}}}
\]
\end{proposition}

\begin{proposition}\label{POWERFG PROP}
        Suppose that $\mathcal{V}$ is a cofibrantly generated closed monoidal model category with cellular fixed points and with cofibrant symmetric pushout powers.
	
	Let $\mathcal{F}$ be a $G$-graph family of $G \times \Sigma$. If $f$ is a cofibration (resp. trivial cofibration) in
	$\mathcal{V}^{G \times \Sigma}_{\mathcal{F}}$,
	then so is $f^{\square n}$
        in 
	$\mathcal{V}^{G \times \Sigma_n \wr \Sigma}_{\mathcal{F}^{\ltimes_{G} n}}$.
\end{proposition}

\begin{remark}
        It is straightforward to check that 
        $\mathcal{F} \sqcap_G \bar{\mathcal{F}}$
	is in fact also a $G$-graph family of $G \times \Sigma \times \bar{\Sigma}$.
	However, $\mathcal{F}^{\ltimes_G n}$ is \textit{not}
	a $G$-graph family of $G \times \Sigma_n \wr \Sigma$,
	due to the need to consider the power $\Sigma_n$-action.
\end{remark}

The $G$-graph families we will be interested in
encode the families of $G$-corollas 
 $\Sigma_{\mathcal{F}}$
of Definition \ref{FAMILY_COROLLAS_DEF} and,
more generally, the families of $G$-trees 
$\Omega_{\mathcal{F}}$ 
of Definition \ref{FTREE DEF}. 

First, note that a
homomorphism
$H \to \Sigma_n$ for $H \leq G$
defines an $H$-action on the $n$-corolla $C_n \in \Sigma$.
Thus, by choosing an arbitrary order of 
$G/H$ and coset representatives $g_i$ for $G/H$,
one obtains a $G$-corolla
$(g_i C_n)_{[g_i] \in G/H}$ in $\Sigma_G$.
The following is then elementary.

\begin{lemma}\label{FAMILY_COROLLAS_LEM}
\index{indexing systems!F@$\mathcal F = \set{\mathcal F_n}_{n \geq 0}$}
Writing $\mathcal{F}_{n}^{\Gamma}$ for the family  of $G$-graph subgroups
of $G \times \Sigma_n^{op}$,
there is an equivalence of categories
(for any arbitrary choice of order of the $G/H$ and of coset representatives)
\[\coprod_{n \geq 0} \mathsf{O}_{\F_n^{\Gamma}} \xrightarrow{\simeq} \Sigma_G.\] 
Hence, families of corollas $\Sigma_{\mathcal{F}}$
are in bijection with collections
$\{\mathcal{F}_n\}_{n\geq 0}$
of $G$-graph families 
$\mathcal{F}_n \subseteq \mathcal{F}_{n}^{\Gamma}$.
\end{lemma}

%
%

We will hence abuse notation and use $\F$ to denote either $\{\mathcal{F}_n\}_{n \geq 0}$
or $\Sigma_\F$.

Note that a $G$-corolla $(C_i)_{i \in I}$
is in $\Sigma_{\mathcal{F}}$
iff
for some (and thus all) $i \in I$
the action of the stabilizer $H_i$ of $i$ on $C_i$
is given by a homomorphism
$H_i \to \Sigma_n$ whose graph group is in $\mathcal{F}_n$.

In what follows, given a tree with an $H$-action
$T \in \Omega^H$,
we will abbreviate
$G \cdot_H T = (g_i T)_{[g_i] \in G/H}$
for some arbitrary (and inconsequential for the remaining discussion) choice of order on $G/H$ and of coset representatives.

\begin{proposition}
Let $\mathcal{F}$ be a family of $G$-corollas and $T \in \Omega$ a tree with automorphism group $\Sigma_T$.
	Write $\mathcal{F}_T$ for the collection of $G$-graph subgroups of 
	$G \times \Sigma_T$ encoded by partial homomorphisms
	$H \to \Sigma_T$, for varying $H \leq G$,
	such that the associated $G$-tree
	$G \cdot_H T$ is a $\mathcal{F}$-tree
	(cf. Definition \ref{FTREE DEF}).
	
	Then $\mathcal{F}_T$ is a $G$-graph family.
\end{proposition}

\begin{proof}
	Closure under conjugation follows since conjugate graph subgroups produce isomorphic $G$-trees.
	As for subgroups, they correspond to restrictions $K \leq H \to \Sigma_T$,
	as thus also restrict the stabilizer actions on each vertex $T_{e^{\uparrow} \leq e}$.
\end{proof}

\begin{remark}\label{LRLEFTQUILLEN REM}
The closure condition defining weak indexing systems in Definition \ref{INDEXSYS DEF}
can be translated in terms of families as saying that,
for any tree $T \in \Omega$ with $\mathsf{lr}(T)=C_n$ and 
$\phi \colon \Sigma_T \to \Sigma_n$ 
the natural homomorphism, one has
$(id_G \times \phi)(\Gamma) \in \mathcal{F}_n$
for any $\Gamma \in \mathcal{F}_{T}$. 
Hence, by
Proposition \ref{FGTRIGHT PROP} 
\[
	\phi_{!}
		\colon
	\mathcal{V}^{G \times \Sigma_T}_{\mathcal{F}_T}
		\to
	\mathcal{V}^{G \times \Sigma_n}
	_{\mathcal{F}_{n}}
\]
is a left Quillen functor.
\end{remark}

\begin{remark}\label{UNPACKFTYPE REM}
Unpacking definitions, a partial homomorphism 
$H \to \Sigma_T$ for $H \leq G$
encodes a subgroup in $\mathcal{F}_T$
iff, for each vertex $v= ( e^{\uparrow} \leq e)$ of $T$ with 
$H_e \leq H$ the
$H$-isotropy of the edge $e$, the induced homomorphism
\begin{equation}\label{PARTIALHOMEDGE EQ}
H_e \to \Sigma_{T_{v}} \simeq 
\Sigma_{|v|}
\end{equation}
encodes a subgroup in $\mathcal{F}_{|v|}$, where $|v|=|e^{\uparrow}|$.
\end{remark}

\begin{remark}\label{TREEINDUCDESC REM}
Recall that any tree $T \in \Omega$ other than the stick $\eta$ has an essentially unique grafting decomposition
$T= C_n \amalg_{n \cdot \eta}(T_1 \amalg \cdots \amalg T_n)$ where $C_n$ is the root corolla,
and the leaves of $C_n$ are grafted to the roots of the $T_i$. We now let 
$\lambda$ be the partition 
$\{1,\cdots,n\} = \lambda_1 \amalg\cdots \amalg \lambda_k$
 such that $1 \leq i_1, i_2 \leq n$ are in the same class iff
 $T_{i_1}, T_{i_2} \in \Omega$ are isomorphic.
 
 Writing 
 $\Sigma_{\lambda} = \Sigma_{\lambda_1} \times \cdots \times
 \Sigma_{\lambda_k}$
and picking representatives $i_j \in \lambda_j$ 
one then has isomorphisms
\[
	\Sigma_T \simeq \Sigma_{\lambda} \wr \prod_{i} \Sigma_{T_i}
		\simeq
	\Sigma_{|\lambda_1|} \wr \Sigma_{T_{i_1}}
		\times \cdots \times	
	\Sigma_{|\lambda_k|} \wr \Sigma_{T_{i_k}}
\]
where the second isomorphism, while not canonical 
(it depends on choices of isomorphisms $T_{i_j} \simeq T_l$ for each $i_j \neq l \in \lambda_j$) is nonetheless well-defined up to conjugation.
\end{remark}

The following, which is the key motivation behind the families defined in the last sections,
reinterprets 
Remark \ref{UNPACKFTYPE REM}
in light of the inductive description of trees in
Remark \ref{TREEINDUCDESC REM}.

\begin{lemma}\label{KEYLEMMAGECO LEM}
Let $\Sigma_\mathcal{F}$ be a family of $G$-corollas and 
$T \in \Omega$ a tree other than $\eta$. Then
\begin{equation}\label{KEYLEMMAGECO EQ}
	\mathcal{F}_T =
	\left(\pi_{G \times \Sigma_n}\right)^{\**}(\mathcal{F}_n)
		\cap
	\left(
	\mathcal{F}_{T_{i_1}}^{\ltimes_G |\lambda_1|}
		\sqcap_G \cdots \sqcap_G
	\mathcal{F}_{T_{i_k}}^{\ltimes_G |\lambda_k|}
	\right),
\end{equation}
where $\pi_{G \times \Sigma_n}$ denotes the composite
$G \times \Sigma_T \to G \times \Sigma_{\lambda} \to
 G \times \Sigma_n$.
\end{lemma}

\begin{proof} The argument is by induction on the decomposition
      $T= C_n \amalg_{n \cdot \eta}(T_1 \amalg \cdots \amalg T_n)$
      with the base case, that of a corolla, being immediate.
      
      Consider now a
      homomorphism $H \to \Sigma_T$, with $H \leq G$, encoding a 
      $G$-graph subgroup $\Gamma \leq G \times \Sigma_T$.
	The condition that $\Gamma \in \left(\pi_{G \times \Sigma_n}\right)^{\**}(\mathcal{F}_n)$ states that the composite $H \to \Sigma_T \to \Sigma_n$ is in $\mathcal{F}_n$, 
	and this is precisely the condition \eqref{PARTIALHOMEDGE EQ} in Remark \ref{UNPACKFTYPE REM}
	for $e=r$ the root of $T$.

As for the condition 
	$ \Gamma \in 
	\left(
	\mathcal{F}_{T_{i_1}}^{\ltimes_G |\lambda_1|}
		\sqcap_G \cdots \sqcap_G
	\mathcal{F}_{T_{i_k}}^{\ltimes_G |\lambda_k|}
	\right)	$, by unpacking it by combining 
	Remarks \ref{UNPACKINGSQCAP REM} and 
	\ref{UNPACKINGLTIMES REM},
	this translates to the condition that, for each $i \in \{1,\cdots,k\}$, one has
	\begin{equation}\label{KEYLEMMAGECOR EQ}
	\pi_{G \times \Sigma_{T_i}}
	\left(
		\Gamma \cap 
	\left(
		G \times \Sigma_{\{i\}} \times \Sigma_{T_i}
		\times 
		\Sigma_{\lambda-\{i\}} \wr \prod_{j\neq i} \Sigma_{T_j}
	\right)
	\right)	
	\in \mathcal{F}_{T_i}
	\end{equation}
where $\lambda - \{i\}$ denotes the induced partition of 
$\{1,\cdots,n\} - \{i\}$.
Noting that the intersection subgroup inside $\pi_{G \times \Sigma_{T_i}}$ in \eqref{KEYLEMMAGECOR EQ} can be rewritten as 
$\Gamma \cap \pi_{\Sigma_n}^{-1}
(\Sigma_{\{i\}} \times \Sigma_{\{1,\cdots,n\} - \{i\}})$,
we see that this is the graph subgroup
encoded by the restriction $H_i \to \Sigma_T$,
where $H_i \leq H$ is the isotropy subgroup of the root $r_i$ of $T_i$ (equivalently, this is also the subgroup sending $T_i$ to itself).
But since, for any edge $e \in T_i$, its isotropy $H_e$ 
(cf. \eqref{PARTIALHOMEDGE EQ}) is a subgroup of $H_i$, the induction hypothesis implies that \eqref{KEYLEMMAGECOR EQ}
is equivalent to condition \eqref{PARTIALHOMEDGE EQ} 
across all vertices other than the root vertex.

The previous paragraphs show that 
\eqref{KEYLEMMAGECO EQ}
indeed holds when restricted to $G$-graph subgroups. However, it still remains to show that any group $\Gamma$ in the rightmost family in \eqref{KEYLEMMAGECO EQ} is indeed
a $G$-graph subgroup, i.e. $\Gamma \cap \Sigma_T =\**$.
In other words, we need to show that any element 
$\gamma \in \Gamma \leq
G \times \Sigma_{\lambda} \wr \prod_{i} \Sigma_{T_i}$
whose $G$-coordinate is 
$\gamma_G = e$ is indeed the identity.
But the condition 
$\pi_{G \times \Sigma_n}(\Gamma) \in \mathcal{F}_n$ now implies that for such $\gamma$ the $\Sigma_{\lambda}$-coordinate is $\gamma_{\Sigma_{\lambda}} = e$
and thus \eqref{KEYLEMMAGECOR EQ} in turn implies that the 
$\Sigma_{T_i}$-coordinates are 
$\gamma_{\Sigma_{T_i}} = e$,
finishing the proof.
\end{proof}

In preparation for our discussion of cofibrant objects in $\mathsf{Op}_G(\mathcal{V})$
in the next section, 
we end the current section by applying 
the results in the previous sections to study the leftmost map in the key pushout diagrams
(\ref{FILTRATION_LAN_LEVEL}).
More concretely, 
and writing 
$p(T_v) \colon \emptyset \to \mathcal{P}(T_v)$,
we analyze the cofibrancy of the maps
\[
	\bigotimes\limits_{v \in V_{G}^{ac}(T)}\P(T_v) \otimes
	\underset{v \in V_{G}^{in}(T)}
	{\mathlarger{\mathlarger{\mathlarger{\square}}}}
	u(T_v)
\qquad
\text{or}
\qquad
	\underset{v \in V_{G}^{ac}(T)}
	{\mathlarger{\mathlarger{\mathlarger{\square}}}}
	p(T_v) 
		\square
	\underset{v \in V_{G}^{in}(T)}
	{\mathlarger{\mathlarger{\mathlarger{\square}}}}
	u(T_v)
\]
that constitute the inner part of \eqref{FILTINTALT EQ}, and where we recall that $T \in \Omega_G^a$ is an alternating tree.
This analysis will consist of two parts, to be combined in the next section:
\begin{enumerate*}
\item[(i)] a $\mathcal{F}_{T_e}$-cofibrancy claim when $T=G \cdot T_e$ is free and;
\item[(ii)] a fixed point claim for non free trees, 
as in Remark \ref{REFLCOREFL REM}.
\end{enumerate*}

For both the sake of generality and to simplify notation in the proofs, we will state the following results using the labeled trees 
of Definition \ref{LABMAP DEF},
and write 
$\Omega_G^{\underline{l}}$ for the category of 
$l$-labeled trees and quotients 
(we not need for string categories at this point).
Moreover, $l$-labeled $\mathcal{F}$-trees 
$\Omega_{\mathcal{F}}^{\underline{l}}$ are
defined exactly as in Definition \ref{FTREE DEF},
so that a labeled $G$-tree is a $\mathcal{F}$-tree 
if and only if the underlying $G$-tree is.
Lastly, note that Remarks \ref{UNPACKFTYPE REM}, \ref{TREEINDUCDESC REM}
and Lemma \ref{KEYLEMMAGECO LEM} then extend to the $l$-labeled context, by now writing $\Sigma_T$ for the group of label isomorphisms and defining the partition $\lambda$ 
in Remark \ref{TREEINDUCDESC REM}
by using label isomorphism classes.

\begin{proposition}\label{AUTTCOFPUSH PROP}
	Suppose that $\mathcal{V}$ is a cofibrantly generated closed monoidal model category with cellular fixed points and with cofibrant symmetric pushout powers.

	Let $\mathcal{F}$ be a family of corollas,
	and suppose that 
	$f_s \colon A_s \to B_s$, $1 \leq s \leq l$ are level $\mathcal{F}$-cofibrations (resp. trivial cofibrations)
	in $\mathsf{Sym}^G(\mathcal{V})$, i.e. that 
	$f_s(n) \colon A_s(n) \to B_s(n)$ are cofibrations (trivial cofibrations) in 
	$\mathcal{V}^{G \times \Sigma_n^{op}}_{\mathcal{F}_n}$.
	Then, for any $l$-labeled tree $T \in \Omega^{\underline{l}}$,
	the map
	\begin{equation}\label{FSQVT EQ}
	f^{\square V(T)} = 
		\underset{1\leq s \leq l}{\mathlarger{\mathlarger{\square}}}
		\phantom{!}
		\underset{v \in V_s(T)}{\mathlarger{\mathlarger{\square}}}
	f_s(v)
	\end{equation}
	(where $V_s(T)$ denotes vertices with label $s$) is a cofibration (resp. trivial cofibration) in 
	$\mathcal{V}^{G \times \Sigma_T^{op}}_{\mathcal{F}_T}$.
\end{proposition}

To ease notation, we identify
$\mathcal{V}^{G \times \Sigma_n^{op}}_{\mathcal{F}_n}
\simeq
\mathcal{V}^{G \times \Sigma_n}_{\mathcal{F}_n}$,
$\mathcal{V}^{G \times \Sigma_T^{op}}_{\mathcal{F}_n}
\simeq
\mathcal{V}^{G \times \Sigma_T}_{\mathcal{F}_n}$
throughout the proof.

\begin{proof}
	This follows by induction on the decomposition 
	$T= C_n \amalg_{n \cdot \eta}(T_1 \amalg \cdots \amalg T_n)$, 
	with the base cases of corollas and $\eta$ being immediate. Otherwise, note first that
\[
f^{\square V(T)}
\simeq
f_{s_r}(n) \square
	\underset{1\leq i \leq k}{\mathlarger{\mathlarger{\square}}}
	\left(f^{\square V(T_{i_j})}\right)^{\square \lambda_i}
\]
where we use the notation in
Remark \ref{TREEINDUCDESC REM} and $s_r$ is the root vertex label.

	The description of $\mathcal{F}_T$ in \eqref{KEYLEMMAGECO EQ} combined with the left Quillen functors in 
	Propositions \ref{EXTERINTADJG PROP}, \ref{BIQUILLENG PROP} and \ref{FGTLEFT PROP} then yield that 
\[
\begin{tikzcd}
	\mathcal{V}^{G \times \Sigma_n}_{\mathcal{F}_n}	
		\times
	\mathcal{V}
	^{G \times \Sigma_{|\lambda_1|}\wr \Sigma_{T_{i_1}}}
	_{\mathcal{F}_{T_{i_1}}^{\ltimes_G |\lambda_1|}}
		\times \cdots \times
	\mathcal{V}
	^{G \times \Sigma_{|\lambda_k|}\wr \Sigma_{T_{i_k}}}
	_{\mathcal{F}_{T_{i_k}}^{\ltimes_G |\lambda_k|}}
\ar{r}{\otimes}
&
	\mathcal{V}^{G \times \Sigma_T}_{\mathcal{F}_T}
\end{tikzcd}
\]
is a left Quillen multifunctor.
The result now follows by Proposition \ref{POWERFG PROP} together with the induction hypothesis.
\end{proof}

\begin{remark}\label{WRONGSTRAT REM}
When $G=\**$, Proposition \ref{AUTTCOFPUSH PROP}
matches \cite[Lemma 5.9]{BM08}.
In fact, it is not hard to modify the proof of \cite[Lemma 5.9]{BM08} to show Proposition \ref{AUTTCOFPUSH PROP} for the 
family $\Sigma_G$ of all $G$-corollas.
Indeed, the key to proving Proposition \ref{AUTTCOFPUSH PROP}
is Lemma \ref{KEYLEMMAGECO LEM} and
the last paragraph of our proof of that lemma
is very close to the arguments in 
\cite{BM08}.
However, the case of a general $\Sigma_{\mathcal{F}}$
is intrinsically more subtle,
with the rest of our proof of 
Lemma \ref{KEYLEMMAGECO LEM} 
depending heavily on the 
$\mathcal{F}^{\ltimes_G n}$ families,
which have no analogue in \cite{BM08}.
\end{remark}

By allowing $T \in \Omega^{\underline{l}}$
to vary,
\eqref{FSQVT EQ} defines an arrow
$f^{\square V(-)}$ in $\mathcal{V}^{G \times \Omega^{\underline{l},op}}$.
Our next step is to compare this construction with an analogous construction for $G$-trees.

To do so, and in analogy with the functor
$\iota \colon G^{op} \times \Sigma \to 
\Sigma_G$
in \S \ref{COMPARISON_REGULAR_SECTION},
we likewise define
$\iota \colon G^{op} \times \Omega^{\underline{l}} \to 
\Omega^{\underline{l}}_G$
via $T \mapsto G \cdot T$.
We then write
$\iota_{\**} 
\colon
\mathcal{V}^{G \times \Omega^{\underline{l},op}}
\to 
\mathcal{V}^{\Omega^{\underline{l},op}_G}
$
for the right adjoint to precomposition.
Just as in \eqref{IOTAFUNSALT EQ}, we then have that,
for $Y \in \mathcal{V}^{G \times \Omega^{\underline{l},op}}$
and $T = (T_i)_{i \in I}$
in $\mathcal{V}^{\Omega^{\underline{l}}}_G$, it is
\begin{equation}\label{IOTAFUNSALTBIG EQ}
	\iota_{\**}Y (T)
=
	\left(\prod_{I} Y(T_i)\right)^G
\simeq
	Y(T_1)^H
\end{equation}
where $T_1$ is the first component of $T$
and $H \leq G$ is the isotropy of the first element of $I$.

\begin{proposition}\label{FIXPT PROP}
	Let $\mathcal{V}$ be as in Proposition \ref{AUTTCOFPUSH PROP}, and suppose additionally that 
	$\mathcal{V}$ has diagonal maps and cartesian fixed points.

	Let  
	$f_s \colon A_s \to B_s$, $1\leq s \leq l$ be 
	genuine cofibrations between genuine cofibrant objects in 
	$\mathsf{Sym}^G(\mathcal{V})$.
	Define a map
	$f^{\square V_G(-)}$ in
	$\mathcal{V}^{\Omega^{\underline{l},op}_G}$
	by setting, for each $T \in \Omega_{G}^{\underline{l}}$,
\begin{equation}\label{FSQVTG EQ}
		f^{\square V_G(T)} = 
		\underset{1\leq s \leq l}{\mathlarger{\mathlarger{\square}}}
			\phantom{!}
		\underset{v \in V_{G,s}(T)}{\mathlarger{\mathlarger{\square}}}
		\iota_{\**}f_s(v).
\end{equation}
	One then has a natural identification
	\begin{equation}\label{FIXEDPOINT1 EQ}
		f^{\square V_G(-)} \simeq
		\iota_{\**} \left(f^{\square V(-)}\right).
	\end{equation}
\end{proposition}

\begin{proof}
	For brevity, let us abbreviate \eqref{FSQVT EQ} as
	$\left(f^{\square V(-)}\right) = \underset{v \in V(T)}{\mathlarger{\mathlarger{\square}}}
	f_{\bullet}(v)$,
	leaving the label data implicit in the vertex data,
	and likewise for \eqref{FSQVTG EQ}.
	Letting 
	$T = (T_i)_I$ and $H\leq G$ be as in 
	\eqref{IOTAFUNSALTBIG EQ},
	we then have
\[
	\left(\iota_{\**} \left(f^{\square V(-)}\right)\right)(T)
\simeq
	\left(f^{\square V(T_1)}\right)^H
=
	\left(\underset{v \in V(T)}{\mathlarger{\mathlarger{\square}}}
	f_{\bullet}(v)
	\right)^H
\simeq
	\underset{[v] \in V(T_1)/H}{\mathlarger{\mathlarger{\square}}}
	f_{\bullet}(v)^{H_v}
\simeq
	\underset{[v] \in V_G(T)}{\mathlarger{\mathlarger{\square}}}
\iota_{\**}f_{\bullet}([v])
\]
where the first step is 
\eqref{IOTAFUNSALTBIG EQ},
the second step is \eqref{FSQVT EQ}, 
the third step is
Corollary \ref{FIXEDPUSH COR}
with $H_v$
the $H$-isotropy of $v \in V(T_1)$
(where we simplify the notation
$f_{\bullet}([v])^{H_{[v]}}$
to 
$f_{\bullet}(v)^{H_v}$
by picking the first representative $v$ of $[v]$),
and the final step is
\eqref{IOTAFUNSALT EQ}
together with the observation that
$H_v \leq G$
is also the $G$-isotropy of $v \in V(T)$
and the identification
$V(T_1)/H \simeq V_G(T)$.
Noting that the last term is 
$f^{\square V_G(T)}$
finishes the proof.	
\end{proof}

\subsection{Cofibrancy and the proof of Theorem \ref{MAINQUILLENEQUIV THM}}
\label{MAINTHM_PROOF_SECTION}

Propositions \ref{AUTTCOFPUSH PROP} and \ref{FIXPT PROP} will now allow us to prove Lemma \ref{MAINLEM LEM}, which provides a characterization of cofibrant objects in 
$\mathsf{Op}_{\mathcal{F}}(\mathcal{V})$,
and from which our main 
result Theorem \ref{MAINQUILLENEQUIV THM}
readily follows.
We start by refining the key argument in the proof of
\cite[Thm. 2.10]{Ste16}.

\begin{proposition}\label{COFESSIM PROP}
	Let $\mathcal{V}$ be a cofibrantly generated model category with cellular fixed points, $\mathcal{F}$ a non-empty family of subgroups of $G$,
	and consider the reflexive adjunction
\[
\begin{tikzcd}[column sep =5em]
	\mathcal{V}^{\mathsf{O}_{\mathcal{F}}^{op}}
	\ar[shift left=1.5]{r}{\iota^{\**}} 
&
	\mathcal{V}^G_{\mathcal{F}}
	\ar[shift left=1.5]{l}{\iota_{\**}}.
\end{tikzcd}
\]
Then the cofibrant objects of 
$\mathcal{V}^{\mathsf{O}_{\mathcal{F}}^{op}}$
are precisely the essential image under $\iota_{\**}$
of the cofibrant objects of
$\mathcal{V}^G_{\mathcal{F}}$.
Moreover, the analogous statement for cofibrations between cofibrant objects also holds.
\end{proposition}

\begin{proof}
        Note first that, since $\iota_{\**}$ identifies 
        $\mathcal{V}^{G}$ as a reflexive 
        subcategory of $\mathcal{V}^{\mathsf{O}_{\mathcal{F}}^{op}}$, 
        it is 
        $X \simeq \iota_{\**}Y$ for some 
        $Y \in \mathcal{V}^{G}$
        (i.e. $X \in \mathcal{V}^{\mathsf{O}_{\mathcal{F}}^{op}}$
        is in the essential image of $\iota_{\**}$)
        iff both $\iota^{\**}X \simeq Y$ and the unit map 
        $X \xrightarrow{\simeq} \iota_{\**} \iota^{\**}X$
        is an isomorphism.

        Letting $C_{\mathcal{F}}$ (resp. $C^{\mathcal{F}}$) denote the classes of cofibrant objects in 
        $\mathcal{V}^{\mathsf{O}_{\mathcal{F}}^{op}}$ 
        (resp. $\mathcal{V}^G_{\mathcal{F}}$)
        we need to show 
        $C_{\mathcal{F}} = \iota_{\**}(C^{\mathcal{F}})$,
        where we slightly abuse notation by writing 
        $\iota_{\**}(\minus)$ for the essential image rather than the image.
        Since $C_{\mathcal{F}}$ is characterized as being the smallest class closed under retracts and transfinite composition of cellular extensions
        that contains the initial presheaf $\emptyset$,
        it suffices to show that 
        $\iota_{\**}(C^{\mathcal{F}})$
        satisfies this same characterization.

        It is immediate that $\iota_{\**}(\emptyset) = \emptyset$.
        Further, the characterization in the first paragraph yields that 
        $X \in \iota_{\**}(C^{\mathcal{F}})$ iff $\iota^{\**}(X)\in C^{\mathcal{F}}$ and $X \xrightarrow{\simeq} \iota_{\**} \iota^{\**}X$ is an isomorphism, showing that  
        $\iota_{\**}(C^{\mathcal{F}})$ is closed under retracts.

        The crux of the proof will be to compare 
        cellular extensions in 
        $C_{\mathcal{F}}$ with the images under $\iota_{\**}$ of the cellular extensions in 
        $C^{\mathcal{F}}$.
        Firstly, note that the generating cofibrations in 
        $\mathcal{V}^{\mathsf{O}_{\mathcal{F}}^{op}}$
        have the form $\mathsf{Hom}(\minus,G/H)\cdot f$, 
        and that by the cellularity axiom (iii) in
        Definition \ref{CELL DEF}
        this map is isomorphic to the map
        $\iota_{\**}(G/H \cdot f)$.
        We now claim that the cellular extensions of objects in 
        $\iota_{\**}(C^{\mathcal{F}})$, i.e. pushout diagrams as on the left below
        \begin{equation}\label{TWOCELLEXTEAS EQ}
                \begin{tikzcd}
                        \iota_{\**} X \ar{d}[swap]{\iota_{\**}u} \ar{r} &
                        \iota_{\**} V \ar[dashed]{d} & &
                        X \ar{d}[swap]{u} \ar{r} &
                        V \ar[dashed]{d}
                        \\
                        \iota_{\**} Y  \ar[dashed]{r}&
                        \tilde{W} & &
                        Y \ar[dashed]{r}&
                        W
                \end{tikzcd}
        \end{equation}
        are precisely the essential image under $\iota_{\**}$
        of the cellular extensions of objects in $C^{\mathcal{F}}$, 
        i.e., pushout diagrams as on the right above. That the solid subdiagrams in either side of \eqref{TWOCELLEXTEAS EQ} are indeed in bijection up isomorphism is simply the claim that 
        $\iota^{\**}$ is fully faithful,
        hence the real claim is that $\tilde{W} \simeq \iota_{\**} W$.
	But this follows since,
	by the cellularity axiom (ii) in
	Definition \ref{CELL DEF},
	the map $\iota_{\**}$ preserves the rightmost pushout
	in \eqref{TWOCELLEXTEAS EQ} 
	(recall that $u \colon X \to Y$ is assumed to be a generating cofibration of $\mathcal{V}^G_{\mathcal{F}}$).

        Noting that the cellularity axiom (i) in
        Definition \ref{CELL DEF} implies that
        $\iota_{\**}$ preserves filtered colimits finishes the proof that $C_{\mathcal{F}} = \iota_{\**}(C^{\mathcal{F}})$.

        The additional claim concerning cofibrations between cofibrant objects follows by the same argument.
\end{proof}

\begin{corollary}\label{FINALCOR COR}
Let $\mathcal{V}$ be as above, 
$\phi \colon G \to \bar{G}$
a homomorphism, and 
$\mathcal{F}$, $\bar{\mathcal{F}}$
families of $G$, $\bar{G}$
such that $\phi_{!}\mathcal{F} \subseteq \mathcal{F}$.
Then the diagram
\[
\begin{tikzcd}
	\mathcal{V}^{\mathsf{O}_{\mathcal{F}}^{op}} \ar{d}[swap]{\phi_!} &
	\mathcal{V}^{G}_{\mathcal{F}} \ar{l}[swap]{\iota_{\**}} \ar{d}{\phi_!}&
\\
	\mathcal{V}^{\mathsf{O}_{\mathcal{\bar{F}}}^{op}}  &
	\mathcal{V}^{\bar{G}}_{\mathcal{\bar{F}}} \ar{l}{\iota_{\**}}&
\end{tikzcd}
\]
commutes up to isomorphism when restricted to 
cofibrant objects of $\mathcal{V}^{G}_{\mathcal{F}}$.
\end{corollary}

\begin{proof}
	It is straightforward to check that the left adjoints commute, i.e. that there is a natural isomorphism 
	$\iota^{\**} \phi_{!} \simeq \phi_{!} \iota^{\**}$
which, by adjunction, induces a natural transformation
	$\phi_! \iota_{\**} \to \iota_{\**} \phi_!$.
More explicitly, this natural transformation is the composite
\[\phi_! \iota_{\**} \to 
\iota_{\**} \iota^{\**} \phi_! \iota_{\**} \xrightarrow{\simeq}
\iota_{\**} \phi_! \iota^{\**} \iota_{\**} \xrightarrow{\simeq}
\iota_{\**} \phi_!
\]
where the last two maps are always isomorphisms. But when restricting to cofibrant objects the previous result guarantees both that $\phi_! \iota_{\**}$ lands in cofibrant objects and that cofibrant objects are in the essential image of the bottom $\iota_{\**}$. The result follows.
\end{proof}


The following is the main lemma. We note that the 
operad half of \eqref{FGTFUNC EQ}
was also obtained by Guti\'{e}rrez-White in \cite{GW18}.

\begin{lemma}\label{MAINLEM LEM}
	Let $\mathcal{V}$ be as in 	
	Theorem \ref{MAINQUILLENEQUIV THM}
	and let $\mathcal{F}$ be a weak indexing system.
Then in both of the adjunctions
\begin{equation}\label{COFADJ2 EQ}
\begin{tikzcd}[column sep =5em]
	\mathsf{Op}_{\mathcal{F}}(\mathcal{V}) \ar[shift left=1.5]{r}{\iota^{\**}} 
&
	\mathsf{Op}^G_{\mathcal{F}}(\mathcal{V})
	\ar[shift left=1.5]{l}{\iota_{\**}}
&
	\mathsf{Sym}_{\mathcal{F}}(\mathcal{V}) \ar[shift left=1.5]{r}{\iota^{\**}} 
&
	\mathsf{Sym}^G_{\mathcal{F}}(\mathcal{V})
	\ar[shift left=1.5]{l}{\iota_{\**}}	
\end{tikzcd}
\end{equation}	
the cofibrant objects in the leftmost category are the essential image under $\iota_{\**}$ of the 
cofibrant objects in the rightmost category.
Moreover, both forgetful functors 
\begin{equation}\label{FGTFUNC EQ}
\begin{tikzcd}[column sep =5em]
	\mathsf{Op}_{\mathcal{F}}(\mathcal{V}) \ar{r}{\mathsf{fgt}} 
&
	\mathsf{Sym}_{\mathcal{F}}(\mathcal{V})
&
	\mathsf{Op}^G_{\mathcal{F}}(\mathcal{V})
	 \ar{r}{\mathsf{fgt}}
&
	\mathsf{Sym}^G_{\mathcal{F}}(\mathcal{V})
\end{tikzcd}
\end{equation}
preserve cofibrant objects.
\end{lemma}

Before starting our proof we recall that, as in
Remark \ref{COMPADJ REM},
we do not require that $\mathcal{F}$ contain all free corollas, in which case the adjunctions in 
\eqref{COFADJ2 EQ} are officially composite adjunctions as in 
\eqref{COMPADJ EQ}.
To avoid cumbersome notation, and noting that the inclusions 
$\gamma_! \colon 
\mathsf{Sym}_{\mathcal{F}}(\mathcal{V}) \to 
\mathsf{Sym}_G(\mathcal{V})$,
$\gamma_! \colon 
\mathsf{Op}_{\mathcal{F}}(\mathcal{V}) \to 
\mathsf{Op}_G(\mathcal{V})$
of \S \ref{INDEXING_SECTION}
are compatible with colimits and that 
the monad $\mathbb{F}_{\mathcal{F}}$
is simply a restriction of $\mathbb{F}_G$,
we will simply work in the 
$\mathsf{Sym}_G(\mathcal{V})$,
$\mathsf{Op}_G(\mathcal{V})$ categories throughout,
with the implicit understanding 
that objects lie in the required subcategories.
In particular, $\iota^{\**}$, $\iota_{\**}$
will denote functors from/to 
$\mathsf{Sym}_G(\mathcal{V})$,
$\mathsf{Op}_G(\mathcal{V})$.

\begin{proof}
We first observe that the claim concerning the symmetric sequence adjunction in \eqref{COFADJ2 EQ}
is not really new. Indeed, 
by Lemma \ref{FAMILY_COROLLAS_LEM}
there are equivalences of categories
$
\mathsf{Sym}_{\mathcal{F}}(\mathcal{V})
\simeq \prod_{n \geq 0}
\mathcal{V}^{\mathsf{O}^{op}_{\mathcal{F}_n}}
$,
$
\mathsf{Sym}^G_{\mathcal{F}}(\mathcal{V})
\simeq \prod_{n \geq 0}
\mathcal{V}_{\mathcal{F}_n}^{G \times \Sigma_n^{op}}
$,
compatible with both the model structures and the $(\iota^{\**},\iota_{\**})$ adjunctions,
and hence the symmetric sequence statement merely repackages 
Proposition \ref{COFESSIM PROP}
(with an obvious empty family case if 
$\mathcal{F}_n =\emptyset$ for some $n$).

Moreover, when assuming the claims in \eqref{COFADJ2 EQ},
one has that
the two forgetful functor claims in \eqref{FGTFUNC EQ}
become equivalent,
so we need only establish the
left claim in \eqref{FGTFUNC EQ}.

For the operad adjunction in \eqref{COFADJ2 EQ},
most of the argument in the proof of
Proposition \ref{COFESSIM PROP}
applies mutatis mutandis
except for the claim that 
$\mathbb{F}_{G} (\emptyset) \simeq \iota_{\**} \mathbb{F} (\emptyset)$, 
which is readily checked directly, 
and the comparison of cellular extensions,
which is the key claim.

Further, we will argue the left claim in \eqref{FGTFUNC EQ} in parallel over the same cellular extensions
(the underlying cofibrancy of
$\mathbb{F}(\emptyset)$, 
$\mathbb{F}_G(\emptyset)$
follows from the cofibrancy of the unit 
$I \in \mathcal{V}$).
 
Explicitly, and borrowing the notation
$C_{\mathcal{F}}$ (resp. $C^{\mathcal{F}}$) 
used in the proof of Proposition \ref{COFESSIM PROP} for the 
classes of cofibrant objects in 
$\mathsf{Op}_{\mathcal{F}}(\mathcal{V})$ 
(resp. $\mathsf{Op}_{\mathcal{F}}^G(\mathcal{V})$),
we need to show that cellular extensions of objects in 
$\iota_{\**}(C^{\mathcal{F}})$, such as on the left below
\begin{equation}\label{TWOCELLEXT EQ}
	\begin{tikzcd}
		\mathbb F_G \iota_{\**} X 
		\ar{d}[swap]{\iota_{\**}u} 
		\ar{r} 
	&
		\iota_{\**} \mathcal{O} 
		\ar[dashed]{d} 
	& &
		\mathbb{F} X 
		\ar{d}[swap]{u} 
		\ar{r} 
	&
		\mathcal{O} 
		\ar[dashed]{d}
\\
		\mathbb F_G \iota_{\**} Y
		\ar[dashed]{r}
	&
		(\iota_{\**} \mathcal{O})[\iota_{\**} u]
	& &
		\mathbb{F} Y 
		\ar[dashed]{r}
	&
		\mathcal{O}[u]
	\end{tikzcd}
\end{equation}
are precisely the essential image under $\iota_{\**}$ of cellular extensions of objects in $C^{\mathcal{F}}$, as on the right above.
Moreover, we can assume by induction that
$\iota_{\**} \mathcal{O}$,
$\mathcal{O}$
are underlying cofibrant in 
$\mathsf{Sym}_{\mathcal{F}}(\mathcal{V})$,
$\mathsf{Sym}^G_{\mathcal{F}}(\mathcal{V})$.
%
%
%
%
Now, recalling that Proposition \ref{MONAD_COMPARISON_PROP}(ii)(iv) gives natural isomorphisms 
\[
\iota^{\**} \mathbb{F}_{G} \iota_{\**} \simeq
\iota^{\**} \mathbb{F}_{G} \iota_{!} \simeq
\mathbb{F}
\]
we see that the two solid subdiagrams in 
\eqref{TWOCELLEXT EQ}
are in fact adjoint up to isomorphism, so that there is a bijection between such data. 
We now claim that the leftmost diagram in
\eqref{TWOCELLEXT EQ}
will indeed be the image under $\iota_{\**}$
of the rightmost diagram
provided that all four objects 
are in the essential image of $\iota_{\**}$.
Indeed, 
if that is the case then
\begin{align*}
	\mathbb{F}_G \iota_{\**} Z \simeq 
	\iota_{\**} \iota^{\**} \mathbb{F}_G \iota_{\**} Z \simeq
	\iota_{\**} \mathbb{F} Z
\end{align*}
for $Z=X,Y$
and since $\iota_{\**}$ reflects colimits\footnote{I.e. any diagram that becomes a colimit upon applying $\iota_{\**}$ must have already been a colimit diagram.},
it must indeed be that
$(\iota_{\**} \mathcal{O})[\iota_{\**} u]
\simeq \iota_{\**} (\mathcal{O}[u])$.

To establish the remaining claim that the objects 
in the leftmost diagram in
\eqref{TWOCELLEXT EQ}
are in the essential image of $\iota_{\**}$,
we claim
it suffices to show this for the bottom right corner 
$(\iota_{\**} \mathcal{O})[\iota_{\**} u]$ when $u \colon X \to Y$ is a general cofibration between cofibrant objects in 
$\mathsf{Sym}^{G}_{\mathcal{F}}(\mathcal{V})$.
Indeed, setting  $X=\emptyset$ and $\O=\mathbb{F} (\emptyset)$, 
one has 
$(\iota_{\**} \mathcal{O})[\iota_{\**} u] = 
\mathbb{F}_{G} \iota_{\**} Y$, and similarly for $\mathbb{F}_{G} \iota_{\**} X$.

In the remainder of the proof we write
$\mathcal{P} = \iota_{\**} \mathcal{O}$, so that 
$(\iota_{\**} \mathcal{O})[\iota_{\**} u] = \mathcal{P}[\iota_{\**} u]$.
The previous paragraphs can be summarized as saying that,
to establish the operad half of 
\eqref{COFADJ2 EQ},
it remains only to show that 
$\mathcal{P}[\iota_{\**} u] $
is in the essential image of $\iota_{\**}$.
And since this means that
$\mathcal{P}[\iota_{\**} u] \to \iota_{\**} \iota^{\**} \mathcal{P}[\iota_{\**} u]$
is an isomorphism,
this can be checked by forgetting to $\Sym_G(\V)$.

On the other hand, to establish the left side of 
\eqref{FGTFUNC EQ} it suffices to show that,
under the inductive hypothesis that
$\mathcal{P}$ is cofibrant in $\mathsf{Sym}_{\mathcal{F}}(\mathcal{V})$,
the map
$\mathcal{P} \to \mathcal{P}[\iota_{\**} u]$
is a cofibration in $\mathsf{Sym}_{\mathcal{F}}(\mathcal{V})$.
Moreover, in light of the 
(already established)
symmetric sequence half of \eqref{COFADJ2 EQ},
the claim in the previous sentence suffices to show that 
$\mathcal{P}[\iota_{\**} u] $ is in the essential image of $\iota_{\**}$,
i.e. it suffices 
to establish the remaining claims in both
\eqref{COFADJ2 EQ} and \eqref{FGTFUNC EQ},
and thus to finish the proof.
Hence, using the filtrations in
\eqref{FILT EQ}
it remains only to show,
assuming $\mathcal{P}$ is cofibrant in $\mathsf{Sym}_{\mathcal{F}}(\mathcal{V})$
and arguing by induction on $k \geq 1$,
that the maps 
$\mathcal{P}_{k-1} \to \mathcal{P}_k$ 
are cofibrations between cofibrant objects in
$\mathsf{Sym}_{\mathcal{F}}(\mathcal{V})$.

Using the iterative description 
of the $\mathcal{P}_k$ in
(\ref{FILTRATION_LAN_LEVEL}),
it now suffices
to check that the leftmost map in (\ref{FILTRATION_LAN_LEVEL}) is
a cofibration between cofibrant objects in 
$\mathsf{Sym}_{\mathcal{F}}(\mathcal{V})$.
We now recall that that map can also be described 
(cf. \eqref{FILTINTALT EQ}) as
\begin{equation}\label{FILTINTALTAG EQ}
	\mathsf{Lan}_{(\Omega_{G}^a[k] \to \Sigma_G)^{op}}
	\left(
		\bigotimes\limits_{v \in V_{G}^{ac}(T)}\P(T_v) \otimes
		\underset{v \in V_{G}^{in}(T)}
{\mathlarger{\mathlarger{\mathlarger{\square}}}}
		u(T_v)
	\right).
\end{equation}
Now consider the left square below, which is equivalent to the right square and thus, 
by Corollary \ref{FINALCOR COR},
commutative on cofibrant objects.
\begin{equation}\label{COMCOFOB EQ}
\begin{tikzcd}
	\mathcal{V}^{\Omega_{\mathcal{F}}^a[k]^{op}} \ar{d}[swap]{\phi_!} &
	\mathcal{V}^{G \times \Omega^a[k]^{op}}_{\mathcal{F}} 
	\ar{l}[swap]{\iota_{\**}} \ar{d}{\phi_!} 
&
	\prod_{T \in \mathsf{Iso}(\Omega^a[k])}
\mathcal{V}^{\mathsf{O}^{op}_{\mathcal{F}_T}} \ar{d}[swap]{\phi_!} &
	\prod_{T \in \mathsf{Iso}(\Omega^a[k])}
\mathcal{V}^{G \times \Sigma_T^{op}}_{\mathcal{F}_T}
	\ar{l}[swap]{\iota_{\**}} \ar{d}{\phi_!} &
\\
	\mathcal{V}^{\Sigma_{\mathcal{F}}^{op}}  &
	\mathcal{V}^{G \times \Sigma^{op}}_{\mathcal{F}}
	\ar{l}{\iota_{\**}}
&
	\prod_{n \geq 0}
	\mathcal{V}^{\Sigma_{\mathcal{F}_n}^{op}}  &
	\prod_{n \geq 0}
	\mathcal{V}^{G \times \Sigma_n^{op}}_{\mathcal{F}}
	\ar{l}{\iota_{\**}}&
\end{tikzcd}
\end{equation}
Propositions \ref{AUTTCOFPUSH PROP}
and \ref{FIXPT PROP} 
now show that the inner map inside the left Kan extension in \eqref{FILTINTALTAG EQ}, which can be rewritten as
\[
	\underset{v \in V_{G}^{ac}(T)}
	{\mathlarger{\mathlarger{\mathlarger{\square}}}}
	p(T_v) 
\square
	\underset{v \in V_{G}^{in}(T)}
	{\mathlarger{\mathlarger{\mathlarger{\square}}}}
u(T_v)
\]
for $p(T_v)$ the map
$\emptyset \to \P(T_v)$,
is in the essential image 
of the cofibrations between cofibrant objects
under the top $\iota_{\**}$ map.
But, since \eqref{COMCOFOB EQ} commutes on cofibrant objects
and the $\mathsf{Lan}$ in \eqref{FILTINTALTAG EQ}
is the leftmost $\phi_!$ functor, 
Proposition \ref{COFESSIM PROP} implies that the 
overall map in \eqref{FILTINTALTAG EQ}
is a cofibration between cofibrant objects
in $\mathsf{Sym}_{\mathcal{F}}(\mathcal{V}) = \mathcal{V}^{\Sigma_{\mathcal{F}}^{op}}$, finishing the proof.
\end{proof}

\begin{remark}\label{INFACTEST REM}
The previous proof in fact establishes 
the slightly more general claim that operads 
(in either 
$\mathsf{Op}_{\mathcal{F}}(\mathcal{V})$ or
$\mathsf{Op}_{\mathcal{F}}^G(\mathcal{V})$)
that forget to cofibrant symmetric sequences
(in either 
$\mathsf{Sym}_{\mathcal{F}}(\mathcal{V})$ or
$\mathsf{Sym}_{\mathcal{F}}^G(\mathcal{V})$)
are closed under cellular extensions of operads.

Morever, and as mentioned in Remark \ref{MUTMUT2 REM},
it now follows that \eqref{KEYNONISO EQ}
is an isomorphism when restricted to cofibrant $G$-symmetric sequences.
\end{remark}

\begin{proof}[proof of Theorem \ref{MAINQUILLENEQUIV THM}]
        It suffices to show that both the derived unit and derived counit for the adjunction are given by weak equivalences.

        For the counit, it is immediate from Lemma \ref{MAINLEM LEM} that if $X \in \mathsf{Op}^G(\mathcal{V})$ is bifibrant
        the functor $\iota^{\**} \iota_{\**} X$ is already derived, and hence the derived counit is identified with the counit isomorphism $\iota^{\**} \iota_{\**} X \xrightarrow{\simeq} X$.

        For the unit, note first that it follows from the definitions
        of the model structures in Theorems 
        \ref{MAINEXIST1 THM},\ref{MAINEXIST2 THM} 
        and the formula for $\iota_{\**}$ in \eqref{IOTAFUNSALT EQ}
        that 
        $\iota_{\**} \colon 
        \mathsf{Op}^{G}_{\mathcal{F}}(\mathcal{V})\to  
       \mathsf{Op}_{\mathcal{F}}(\mathcal{V})$
        detects fibrations (as well as weak equivalences)
        and thus, by Lemma \ref{MAINLEM LEM},
        that 
        $Y \in \mathsf{Op}_{\mathcal{F}}(\mathcal{V})$
        is bifibrant iff $Y \simeq \iota_{\**} X$
        for $X \in \mathsf{Op}^{G}_{\mathcal{F}}(\mathcal{V})$ bifibrant.
        But then the functor $\iota_{\**} \iota^{\**} Y$ 
        is also already derived (since $\iota^{\**} Y \simeq \iota^{\**}\iota_{\**} X \simeq X$ is fibrant) and the derived unit is thus the isomorphism
        $Y \xrightarrow{\simeq} \iota_{\**} \iota^{\**} Y$.
\end{proof}

\subsection{Realizing $N_{\infty}$-operads}
\label{NINFTY_SECTION}

We now explain how the $N\mathcal{F}$-operads of 
Blumberg-Hill can be built from the theory of genuine equivariant operads, thus proving Corollary \ref{NINFTY_REAL_COR_MAIN}.

We start with an abstract argument, which has also been used by Guti\'{e}rrez-White in \cite{GW18}. 
Writing 
$\mathcal{I} = \mathbb{F} (\emptyset)$
for the initial equivariant operad in 
$\mathsf{Op}^G(\mathsf{sSet})$,
i.e. the operad consisting of a single operation at level $1$,
consider any 
``cofibration followed by trivial fibration'' factorization
(as given by the
Quillen small object argument)
\begin{equation}\label{OFCONST EQ}
\begin{tikzcd}
	\mathcal{I} \arrow[r, rightarrowtail] &
	\O_{\mathcal{F}}
	\arrow[r,twoheadrightarrow, "\sim"] &
	\mathsf{Com}
\end{tikzcd}
\end{equation}
in the model structure 
$\mathsf{Op}^G_{\mathcal{F}}(\mathsf{sSet})$.
We claim that $\mathcal{O}_{\mathcal{F}}$
is an $N \mathcal{F}$-operad, i.e.
that it has fixed points as described in 
Corollary \ref{NINFTY_REAL_COR_MAIN}.
That $\mathcal{O}_{\mathcal{F}}(n)^{\Gamma} \sim \**$
whenever $\Gamma \in \mathcal{F}_n$
follows from the fact that the map
$\mathcal{O}_{\mathcal{F}} \xrightarrow{\sim} \mathsf{Com}$
is a $\mathcal{F}$-equivalence.
On the other hand, by Lemma \ref{MAINLEM LEM}
the map 
$\mathcal{I} \rightarrowtail \O_{\mathcal{F}}$
is also an underlying cofibration in 
$\mathsf{Sym}^G_{\mathcal{F}}(\mathsf{sSet})$, and thus
$\mathcal{O}_{\mathcal{F}}$ is underlying cofibrant in 
$\mathsf{Sym}^G_{\mathcal{F}}(\mathsf{sSet})$.
The required condition that
$\mathcal{O}_{\mathcal{F}} (n)^{\Gamma}
= \emptyset $
whenever $\Gamma \nin \mathcal{F}_n$
now follows since this holds for any cofibrant object in $\mathsf{Sym}^G_{\mathcal{F}}(\mathsf{sSet})$,
as can readily be checked via a cellular argument.

One drawback of the 
$N \mathcal{F}$-operad $\mathcal{O}_{\mathcal{F}}$
built in \eqref{OFCONST EQ},
however, is that it is not explicit,
due to the need to use the small object argument.
To obtain a more explicit model, we make use of the theory of genuine equivariant operads.

Firstly, any weak indexing system $\mathcal{F}$
gives rise to a genuine equivariant operad
$\delta_{\mathcal{F}}
\in \mathsf{Op}_G(\mathsf{Set})
$
such that
$\delta_{\mathcal{F}}(C) = \**$
if $C \in \Sigma_{\mathcal{F}}$
and 
$\delta_{\mathcal{F}}(C) = \emptyset$
if $C \nin \Sigma_{\mathcal{F}}$.
Alternatively,  
$\delta_{\mathcal{F}}$
can also be regarded as the terminal object of
$\mathsf{Op}_{\mathcal{F}}(\mathsf{Set})
	\hookrightarrow 
\mathsf{Op}_{G}(\mathsf{Set})$.
The characterization of the cofibrant objects
in $\mathsf{Op}_{G}(\mathsf{sSet})$
given by Lemma \ref{MAINLEM LEM}
now shows that the unique map
$\iota_{\**} \mathcal{O}_{\mathcal{F}} 
\xrightarrow{\simeq} \delta_{\mathcal{F}}$
is a cofibrant replacement in
$\mathsf{Op}_G(\mathsf{sSet})$ and, moreover,
it is clear from the argument 
in the previous paragraph that for any
other cofibrant replacement 
$C \delta_{\mathcal{F}} 
\xrightarrow{\simeq} \delta_{\mathcal{F}}$
the equivariant operad
$\iota^{\**} (C \delta_{\mathcal{F}})
\in \mathsf{Op}^G (\mathsf{sSet})$
is an $N \mathcal{F}$-operad.
We will now build an explicit model for such
$C \delta_{\mathcal{F}}$.
We start by considering the following
adjunctions, where both of the right adjoints, 
which we write at the bottom, are forgetful functors.
\begin{equation}\label{MAINPFADJVAR EQ}
\begin{tikzcd}[column sep =9em]
	\mathsf{Set}^{\times \text{Ob}(\Sigma_G)}
	\ar[shift left=1.5]{r}
	{(X_C) \mapsto
	\coprod_C \mathsf{Hom}(\minus,C) \times X_C}
&
	\mathsf{Sym}_{G}(\mathsf{Set}) 
	\arrow[l, shift left=1.5] 
	\arrow[r, shift left=1.5,swap,"\mathbb{F}_G"']
&
	\mathsf{Op}_G(\mathsf{Set})
	\ar[shift left=1.5]{l}
\end{tikzcd}
\end{equation}
We will find it convenient in the following discussion to abuse notation by omitting
occurrences of the forgetful functors.
As such, we write
$\delta_{\mathcal{F}}$ not only for the object in 
$\mathsf{Op}_G(\mathsf{Set})$,
but also for any of the underlying objects in 
$\mathsf{Sym}_G(\mathsf{Set})$, 
$\mathsf{Set}^{\times \text{Ob}(\Sigma_G)}$.
Similarly, $\mathbb{F}_G$
will denote both the functor in 
\eqref{MAINPFADJVAR EQ}
and the monad on 
$\mathsf{Sym}_{G}(\mathsf{Set})$
while 
$\widetilde{\mathbb{F}}_G$
will denote both the top composite functor in 
\eqref{MAINPFADJVAR EQ}
and the composite monad on 
$\mathsf{Set}^{\times \text{Ob}(\Sigma_G)}$.

Since both adjunctions in 
\eqref{MAINPFADJVAR EQ}
restrict to their $\mathcal{F}$ versions,
in which case $\delta_{\mathcal{F}}$ denotes the terminal object of any of the $\mathcal{F}$ analogue categories,
it follows that 
$\delta_{\mathcal{F}} \in \mathsf{Set}^{\times \text{Ob}(\Sigma_G)}$
is a 
$\widetilde{\mathbb{F}}_G$-algebra, 
and we now consider the bar construction
\[B_n(\widetilde{\mathbb{F}}_G,
\widetilde{\mathbb{F}}_G,
\delta_{\mathcal{F}})
= \widetilde{\mathbb{F}}_G \circ
\widetilde{\mathbb{F}}_G ^{\circ n} 
(\delta_{\mathcal{F}}),
\]
where we regard the outer $\widetilde{\mathbb{F}}_G$ as the top composite functor in \eqref{MAINPFADJVAR EQ}.
We thus have
$B_{\bullet}(\widetilde{\mathbb{F}}_G,
\widetilde{\mathbb{F}}_G,
\delta_{\mathcal{F}})
\in
\mathsf{Op}_{\mathcal{F}}(\mathsf{Set})^{\Delta^{op}}
	\hookrightarrow
\mathsf{Op}_G(\mathsf{Set})^{\Delta^{op}}
\simeq \mathsf{Op}_G(\mathsf{sSet})$
and, moreover, the unique genuine operad map
$B_{\bullet}(\widetilde{\mathbb{F}}_G,
\widetilde{\mathbb{F}}_G,
\delta_{\mathcal{F}})
\to 
\delta_{\mathcal{F}}$
is a weak equivalence in 
$\mathsf{Op}_G(\mathsf{sSet})$
thanks to the usual extra degeneracy argument
\cite[\S 4.5]{Ri14}
(which applies after forgetting to 
$\mathsf{Set}^{\times \text{Ob}(\Sigma_G)}$).
Therefore, 
the following result suffices to show that
$B_{\bullet}(\widetilde{\mathbb{F}}_G,
\widetilde{\mathbb{F}}_G,
\delta_{\mathcal{F}})
$
is an $N \mathcal{F}$-operad.

\begin{proposition}\label{BARCOF PROP}
$B_{\bullet}(\widetilde{\mathbb{F}}_G,
\widetilde{\mathbb{F}}_G,
\delta_{\mathcal{F}})
\in \mathsf{Op}_G(\mathsf{sSet})
$
is cofibrant.
\end{proposition}

Proposition \ref{BARCOF PROP}
will follow by analyzing the skeletal filtration of
$B_{\bullet}(\widetilde{\mathbb{F}}_G,
\widetilde{\mathbb{F}}_G,
\delta_{\mathcal{F}})
$ and showing that the corresponding latching maps,
which are built using cubical diagrams, are cofibrations.

Recall that a $n$-\textit{cube} on $\mathsf{sSet}$
is a functor
$\mathcal{X}_{(\minus)} \colon \mathsf{P}_n \to 
\mathsf{sSet}$
for $\mathsf{P}_n$ the poset of subsets of 
$\underline{n} = \{1,\cdots,n\}$.
We call a $n$-cube a \textit{monomorphism $n$-cube}
if the latching maps
\[
\colim_{V \subsetneq U} \mathcal{X}_V = L_U \mathcal{X}
\xrightarrow{l_U \mathcal{X}} 
\mathcal{X}_U
\]
are monomorphisms for all $U \in \mathsf{P}_n$.
Cubes and monomorphism cubes in
$\mathsf{Set}^{\times \text{Ob}(\Sigma_G)}$
are defined identically.

\begin{remark}\label{MONOCUBE REM}
Using model category language, monomorphism $n$-cubes
are the cofibrant objects for the projective model structure on $n$-cubes. As such, they are characterized as the $n$-cubes with the left lifting property against maps of $n$-cubes $\mathcal{Y}_{(\minus)} \to \mathcal{Z}_{(\minus)}$ that are levelwise trivial fibrations.
\end{remark}

\begin{lemma}\label{MONOCUBE LEM}
\begin{itemize}
\item[(a)]
The monad 
$\widetilde{\mathbb{F}}_G \colon 
\mathsf{Set}^{\times \text{Ob}(\Sigma_G)} 
\to \mathsf{Set}^{\times \text{Ob}(\Sigma_G)}$
sends monomorphism $n$-cubes
to monomorphism $n$-cubes.
\item[(b)]
Letting $\eta \colon id \to \widetilde{\mathbb{F}}_G $ denote the unit and
$A \to B$ be a monomorphism in 
$\mathsf{Set}^{\times \text{Ob}(\Sigma_G)}$, the square
\[
\begin{tikzcd}
	A \ar{r} \ar{d}[swap]{f} &
	\widetilde{\mathbb{F}}_G A 
	\ar{d}{\widetilde{\mathbb{F}}_G f}
\\
	B \ar{r} & \widetilde{\mathbb{F}}_G B 
\end{tikzcd}
\]
is a monomorphism square (i.e monomorphism $2$-cube).
\end{itemize}
\end{lemma}

\begin{proof}
	Combining \eqref{FGXDEFEXP EQ} with
	the top left functor in \eqref{MAINPFADJVAR EQ}
	yields the formula
\begin{equation}\label{TILF EQ}
\widetilde{\mathbb{F}}_G X (C) \simeq
\coprod_{T \in 
\mathsf{Iso}(C \downarrow_{\mathsf{r}} \Omega_G^0)}
\left(
\prod_{v \in V_G(T)}
\left(
	\coprod_{D\in \Sigma_G} \mathsf{Hom}(T_v,D) \times X(D)
\right)
\right) 
\cdot_{\mathsf{Aut}(T)} \mathsf{Aut}(C).
\end{equation}
	Distributing the inner $\coprod$ over the 
	$\prod$ 
	shows that 
	$\widetilde{\mathbb{F}}_G f$
	is a coproduct of monomorphisms with the map
	$f \colon A \to B$
	corresponding to the summand with $C=T=D$, and hence (b) follows.
	
	To show (a), note first that there are three types of operations in \eqref{TILF EQ}:
	coproducts, inductions and products.
	Since coproducts and inductions preserve both colimits and monomorphisms, they preserve monomorphism cubes,
	and it thus remains to show that so do products.
	Given monomorphism $n$-cubes
	$\mathcal{Y}_{(\minus)}, \mathcal{Z}_{(\minus)}$
	consider first the $2n$-cube 
	$(\mathcal{Y} \times \mathcal{Z})_{(U,V)} = \mathcal{Y}_U \times \mathcal{Z}_V$.
	It is straightforward to check that this $2n$-cube has latching maps
	$l_{(U,V)} \mathcal{Y} \times \mathcal{Z} =
	 l_U \mathcal{Y} \square l_V \mathcal{Z}$,
	and is thus a monomorphism $2n$-cube.
	It remains to check that
	the diagonal $n$-cube 
	$\Delta^{\**} (\mathcal{Y} \times \mathcal{Z})$
	is a monomorphism $n$-cube.
Considering the adjuntion
	$\Delta^{\**}\colon 
	\mathsf{sSet}^{\mathsf{P}_n \times \mathsf{P}_n}
		\rightleftarrows
	\mathsf{sSet}^{\mathsf{P}_n}
	\colon \Delta_{\**}
	$
and Remark \ref{MONOCUBE REM} it suffices to check that
$\Delta_{\**}$ preserves level trivial fibrations of cubes. But this is obvious from the formula
$(\Delta_{\**} \mathcal{X})_{(U,V)} = \mathcal{X}_{U \cup V}$.	
\end{proof}

\begin{proof}[proof of Proposition \ref{BARCOF PROP}]
We start by analyzing the latching maps
for 
$B_{\bullet}
=
B_{\bullet}(\widetilde{\mathbb{F}}_G,
\widetilde{\mathbb{F}}_G,
\delta_{\mathcal{F}})
$.
To describe the $n$-th latching map, we start with the natural $n$-cube in 
$\mathsf{Set}^{\times \text{Ob}(\Sigma_G)}$ 
given by
$\mathcal{X}^n_U = 
\widetilde{\mathbb{F}}_G^{\circ U}
(\delta_{\mathcal{F}})$
and where maps 
are induced by the unit
$\eta \colon id \to \widetilde{\mathbb{F}}_G$.
For example, in 
$\mathcal{X}^5_{(\minus)}$,
the map 
$\mathcal{X}^5_{\{1,4\}} \to \mathcal{X}^5_{\{1,3,4,5\}}$
is 
\[
\widetilde{\mathbb{F}}_G^{\circ 2}
(\delta_{\mathcal{F}})
\xrightarrow{
\widetilde{\mathbb{F}}_G
\eta
\widetilde{\mathbb{F}}_G
\eta}
\widetilde{\mathbb{F}}_G^{\circ 4}
(\delta_{\mathcal{F}}).
\]
Since degeneracies of $B_{\bullet}$
are also induced by $\eta$,
and recalling the notation $\underline{n}=\{1,\cdots,n\}$ for the maximum in $\mathsf{P}_n$,
one has that the $n$-th latching map of $B_{\bullet}$ is given by
\begin{equation}\label{MONLATCH EQ}
	\check{l}_n B_{\bullet} = 
	\check{l}_{\underline{n}}
	(\widetilde{\mathbb{F}}_G
	\mathcal{X}^n) \simeq
	\widetilde{\mathbb{F}}_G
	(l_{\underline{n}} \mathcal{X}^n)
\end{equation}
where the check decoration on $\check{l}$
for the two leftmost latching maps indicates that the colimits defining those latching maps are taken 
in $\mathsf{Op}_{G}(\mathsf{Set})$, while the rightmost latching map is computed in 
$\mathsf{Set}^{\times \text{Ob}(\Sigma_G)}$.

The key to the proof is the claim that the maps
$l_{\underline{n}} \mathcal{X}^n$
are monomorphisms. This will follow from the stronger claim that the $\mathcal{X}^n$
are monomorphim $n$-cubes, which we argue by induction on $n$.
When $n=0$ there is nothing to show.
Otherwise, for any 
$U \subsetneq \{1,\cdots,n,n+1\}$
the restriction of $\mathcal{X}^{n+1}$ to subsets of
$U$ is isomorphic to the cube $\mathcal{X}^{|U|}$,
so that we need only analyze the top
latching map $l_{\underline{n+1}} \mathcal{X}^{n+1}$.
We now write
$\mathcal{X}^{n+1} = (\mathcal{X}^n \to 
\widetilde{\mathbb{F}}_G \mathcal{X}^n)$, regarding the $(n+1)$-cube as a map of $n$-cubes.
The top latching map $l_{\underline{n_+1}} \mathcal{X}^{n+1}$
is then the latching map of the composite square
(the check decoration
$\check{L}$ again denotes a latching object computed in
$\mathsf{Op}_G(\mathsf{Set})$)
\begin{equation}\label{BARCOFSQ EQ}
\begin{tikzcd}
 	L_{\underline{n}} \mathcal{X}^n \ar{d} \ar[equal]{r}
&
	L_{\underline{n}} \mathcal{X}^n \ar{d}
	\ar{r}
&
	\mathcal{X}^n_{\underline{n}}
	\ar{d}
\\
	\check{L}_{\underline{n}} (\widetilde{\mathbb{F}}_G \mathcal{X}^n)
	\ar{r}
&
	\widetilde{\mathbb{F}}_G(L_{\underline{n}}  \mathcal{X}^n)
	\ar{r}
&
	\widetilde{\mathbb{F}}_G \mathcal{X}^n_{\underline{n}}
\end{tikzcd}
\end{equation}
The latching map in the rightmost square
\eqref{BARCOFSQ EQ}
is a monomorphism since it is an instance of
Lemma \ref{MONOCUBE LEM}(b)
applied to the map
$l_{\underline{n}} \mathcal{X}^n \colon 
L_{\underline{n}} \mathcal{X}^n \to \mathcal{X}^n_{\underline{n}}$, which is a monomorphism by the induction hypothesis.
On the other hand, writing 
$\tilde{\mathcal{X}}^n$ for the cube obtained from 
$\mathcal{X}^n$ by replacing the top level
$\mathcal{X}^n_{\underline{n}}$ with
$L_{\underline{n}} \mathcal{X}^n$,
the left bottom horizontal map in 
\eqref{BARCOFSQ EQ}
can be described as
$
\check{l}_{\underline{n}}
(\widetilde{\mathbb{F}}_G
\tilde{\mathcal{X}}^n) \simeq
\widetilde{\mathbb{F}}_G
(l_{\underline{n}} \tilde{\mathcal{X}}^n)
$
(compare with \eqref{MONLATCH EQ}),
which is a monomorphism by
Lemma \ref{MONOCUBE LEM}(a).
Hence the latching maps in both squares
in \eqref{BARCOFSQ EQ} are monomorphisms, 
and thus so is the latching map of the composite square, showing that $l_{\underline{n_+1}} \mathcal{X}^{n+1}$
is a monomorphism, as desired.

To finish the proof, one now simply notes that the skeletal filtration of $B_{\bullet}$ is 
then iteratively described by the pushouts
in $\mathsf{Op}_{G}(\mathsf{sSet})$
below, where the vertical maps are cofibrations in 
$\mathsf{Op}_{G}(\mathsf{sSet})$
since the maps 
$l_{\underline{n}} \mathcal{X}^n \colon 
L_{\underline{n}} \mathcal{X}^n \to \mathcal{X}^n_{\underline{n}}$
are monomorphisms.
\[
\begin{tikzcd}
	\widetilde{\mathbb{F}}_G 
	(L_{\underline{n}} \mathcal{X}^n \times \Delta^n
	\amalg_{L_{\underline{n}} \mathcal{X}^n \times \delta\Delta^n}
	\mathcal{X}^n_{\underline{n}} \times \delta \Delta^n) 
	\ar{r} \ar{d} &
	\mathsf{sk}_{n-1} B_{\bullet} \ar{d}
\\
	\widetilde{\mathbb{F}}_G 
	(\mathcal{X}^n_{\underline{n}} \times \Delta^n) 
	\ar{r} & \mathsf{sk}_{n} B_{\bullet}
\end{tikzcd}
\]
\end{proof}

\begin{remark}
We now address the ``moreover'' claim in Corollary \ref{NINFTY_REAL_COR_MAIN}. 
For any $\mathcal{O} \in \mathsf{Op}^G(\mathsf{sSet})$
one has $\pi_0 (\iota_{\**} \mathcal{O})\in \mathsf{Op}_G(\mathsf{Set})$. 
Therefore, if $\mathcal{O}$ has fixed points as in \eqref{NFINFTY2 EQ},
then $\pi_0 (\iota_{\**} \mathcal{O}) = \delta_{\mathcal{F}}$
for $\mathcal{F} = \{\mathcal{F}_n\}_{n \geq 0}$
a collection of families of graph subgroups.
But the condition that $\delta_{\mathcal{F}}\in \mathsf{Op}_G(\mathsf{Set})$ simply repackages Definition \ref{INDEXSYS DEF}.
\end{remark}

\begin{remark}\label{LANDINESS REM}
	Regarding
	$\mathsf{Sym}_G(\mathsf{Set})$
	as a subcategory of 
	$\mathsf{Sym}_G(\mathsf{sSet})$,
	the leftmost left adjoint in 
	\eqref{MAINPFADJVAR EQ}
	lands in cofibrant objects of 
	$\mathsf{Sym}_G(\mathsf{sSet})$
	and thus, 
	by Lemma \ref{MAINLEM LEM},
	in the essential image of
	$\iota_{\**} \colon
	\mathsf{Sym}^G(\mathsf{Set})
	\to 
	\mathsf{Sym}_G(\mathsf{Set})$.
	Thus, again by Lemma \ref{MAINLEM LEM},
	the top composite in 
	\eqref{MAINPFADJVAR EQ}
	lands in the essential image of
	$\iota_{\**} \colon
	\mathsf{Op}^G(\mathsf{Set})
	\to 
	\mathsf{Op}_G(\mathsf{Set})$.
\end{remark}

\begin{remark}\label{MAINPFADJVARVAR REM}
If one appends the adjunction
$\iota^{\**} \colon
	\mathsf{Op}_G(\mathsf{Set})
	\rightleftarrows
	\mathsf{Op}^G(\mathsf{Set})
\colon \iota_{\**}$
to \eqref{MAINPFADJVAR EQ} one obtains
an additional composite monad
$\widehat{\mathbb{F}}_G$
on
$\mathsf{Set}^{\times \text{Ob}(\Sigma_G)}$.
Moreover, by Remark \ref{LANDINESS REM} the monads
$\widetilde{\mathbb{F}}_G$ and
$\widehat{\mathbb{F}}_G$
are in fact isomorphic.
This observation now hints at how 
one can build a model for 
$N \mathcal{F}$-operads 
directly in terms of (regular) equivariant operads,
i.e. without making explicit use of genuine equivariant operads.
Namely, consider the adjunctions
\begin{equation}\label{MAINPFADJVARVAR EQ}
\begin{tikzcd}[column sep =3.3em]
	\prod_{n \geq 0}
	\mathsf{Set}^{\times \text{Ob}
	\left(\mathsf{O}^{op}_{\mathcal{F}_n^{\Gamma}}\right)}
	\ar[shift left=1.5]{r}{S}
&
	\prod_{n \geq 0}
	\mathsf{Set}^{\mathsf{O}^{op}_{\mathcal{F}_n^{\Gamma}}} 
	\arrow[l, shift left=1.5]
	\ar[r, shift left=1.5,swap,"\iota^{\**}"']
&
	\mathsf{Sym}^{G}(\mathsf{Set}) 
	\ar[l, shift left=1.5,swap,"\iota_{\**}"']
	\arrow[r, shift left=1.5,swap,"\mathbb{F}"']
&
	\mathsf{Op}^G(\mathsf{Set})
	\ar[shift left=1.5]{l}
\end{tikzcd}
\end{equation}
Abusing notation by again writing 
$\widehat{\mathbb{F}}_G$
for the composite monad and
$\delta_{\mathcal{F}}$
for the obvious object on 
the leftmost category,
it is not hard to
use the equivalence in 
Lemma \ref{FAMILY_COROLLAS_LEM}
to leverage our analysis so as to conclude that
the bar construction
$B_{\bullet}(\widehat{\mathbb{F}}_G,
\widehat{\mathbb{F}}_G,
\delta_{\mathcal{F}})
$
built using \eqref{MAINPFADJVARVAR EQ}
is also a cofibrant $N \mathcal{F}$-operad.

However, we caution that this latter
model is not as simple as it might seem at first.
This is because the task of showing that
$\delta_{\mathcal{F}}$
is a $\widehat{\mathbb{F}}_G$-algebra
is a non-trivial task.
More precisely, while it is clear from the 
definition of $\delta_{\mathcal{F}}$
that there is at most one map
$\widehat{\mathbb{F}}_G
\delta_{\mathcal{F}}
=
\iota_{\**} \mathbb{F} \iota^{\**} S \delta_{\mathcal{F}}
\to 
\delta_{\mathcal{F}}$
(which would necessarily define an algebra structure),
it is unclear if such a map exists at all.
In order to show directly that such a map exists,
one must analyze 
$\iota_{\**} \mathbb{F} \iota^{\**} S \delta_{\mathcal{F}}$,
i.e. (cf. \eqref{IOTAFUNSALT EQ}) one must compute 
the graph fixed points of free operads
$\left(\mathbb{F} \iota^{\**} S \delta_{\mathcal{F}}\right)
(n)^{\Gamma}$,
and we note that the main technical work of 
Rubin in \cite{Rub17} consists precisely of such calculations.
Alternatively, in this paper this fixed point analysis is built into 
Lemma \ref{MAINLEM LEM},
so that, cf. Remarks \ref{INFACTEST REM} and \ref{MUTMUT2 REM},
one has 
$\iota_{\**} \mathbb{F} \iota^{\**} S \delta_{\mathcal{F}}
\simeq 
\mathbb{F}_G \iota_{\**} \iota^{\**} S \delta_{\mathcal{F}}
\simeq
\mathbb{F}_G S \delta_{\mathcal{F}}$
(the second identity follows since
$S \delta_{\mathcal{F}}$
is in the essential image of $\iota_{\**}$,
cf. Remark \ref{LANDINESS REM}).
%
\end{remark}

\appendix

\section{Transferring Kan extensions}
\label{TRANSKAN AP}

The purpose of this appendix is to provide the
somewhat long proof of Proposition \ref{RANTRANS PROP}, 
which is needed when repackaging free extensions of
genuine equivariant operads in \eqref{EXTTREEFOR EQ}.

We start with a more detailed discussion of the realization functor $|\minus|$
defined by the adjunction
	\[
	|\minus|\colon
	\mathsf{Cat}^{\Delta^{op}} 
		\rightleftarrows
	\mathsf{Cat} 
	\colon (\minus)^{[\bullet]}
	\]
in Definition \ref{REAL DEF}.
More explicitly, one has
\begin{equation}\label{REALDEF EQ}
	 |\mathcal{I}_{\bullet}| =
	coeq \left(\coprod_{[n] \to [m]}
	 [n] \times \mathcal{I}_m
	 	\rightrightarrows
	 \coprod_{[n]} [n] \times \mathcal{I}_n
	 \right).
\end{equation}

\begin{example}
Any $\mathcal{I} \in \mathsf{Cat}$ induces objects 
$\mathcal{I},\mathcal{I}_{\bullet},\mathcal{I}^{[\bullet]} \in \mathsf{Cat}^{\Delta^{op}}$ 
where $\mathcal{I}$ is the constant simplicial object and $\mathcal{I}_{\bullet}$ is the nerve $N \mathcal{I}$ with each level regarded as a discrete category.
It is straightforward to check that 
$|\mathcal{I}|\simeq |\mathcal{I}_{\bullet}| \simeq
|\mathcal{I}^{[\bullet]}| \simeq \mathcal{I}$.
\end{example}

\begin{lemma}\label{OBJGENREL LEMMA}
	Given $\mathcal{I}_{\bullet} \in \mathsf{Cat}^{\Delta^{op}}$ one has an identification
	$\text{Ob}(|\mathcal{I}_{\bullet}|) \simeq \text{Ob}(\mathcal{I}_0)$.
	Furthermore, the arrows of $|\mathcal{I}_{\bullet}|$ are generated by the image of the arrows in $\mathcal{I}_0 \simeq \mathcal{I}_0 \times [0]$ and the image of the arrows in 
	$[1] \times \text{Ob}(\mathcal{I}_1)$.
\end{lemma}

For each $i_1 \in \mathcal{I}_1$, we will denote the arrow of 
$|\mathcal{I}_{\bullet}|$ induced by the arrow in $[1] \times \{i_1\}$ by
\[d_1(i_1) \xrightarrow{i_1} d_0(i_1).\]

\begin{proof}
	We write $d_{\hat{k}}$, $d_{\hat{k},\hat{l}}$ for the simplicial operators induced by the maps 
	$[0]\xrightarrow{0 \mapsto k} [n]$,
	$[1]\xrightarrow{0 \mapsto k,1 \mapsto l} [n]$
	which can informally be thought of as the ``composite of all faces other than $d_k$, $d_l$''.
Using \eqref{REALDEF EQ} one has equivalence relations
between the objects  
$(k,i_n) \in [n] \times \mathcal{I}_n$
and 
$(0,d_{\hat{k}}(i_n))
\in [0] \times \mathcal{I}_0$
and,
since for any generating relation $(k,i_n)\sim (l,i'_m)$
it is $d_{\hat{k}}(i_n) = d_{\hat{l}}(i'_m)$,
the identification 
$\text{Ob}(|\mathcal{I}_{\bullet}|) \simeq \text{Ob}(\mathcal{I}_0)$
follows.

To verify the claim about generating arrows, note that any arrow of $[n]\times \mathcal{I}_n$ factors as 
\begin{equation}\label{FACTORIZATIONREAL EQ}
(k,i_n) \to (l,i_n)  \xrightarrow{I_n} (l,i'_n)
\end{equation}
for $I_n \colon i_n \to i'_n$
an arrow of $\mathcal{I}_n$. 
The $d_{\hat{l}}$ relation identifies the right arrow in 
\eqref{FACTORIZATIONREAL EQ}
with
$(0,d_{\hat{l}}(i_n))
	\xrightarrow{d_{\hat{l}}(I_n)}
(0,d_{\hat{l}}(i'_n))
$
in $[0]\times \mathcal{I}_0$
while (if $k<l$) the $d_{\hat{k},\hat{l}}$ relation identifies the left arrow with 
$(0,d_{\hat{k},\hat{l}}(i_n)) \to (1,d_{\hat{k},\hat{l}}(i_n))$
in $[1]\times \mathcal{I}_1$. The result follows.
\end{proof}

\begin{remark}
	Given $\mathcal{I}_{\bullet} \in \mathsf{Cat}^{\Delta^{op}}$, $\mathcal{C} \in \mathsf{Cat}$, the isomorphisms
	\[
	\mathsf{Hom}_{\mathsf{Cat}}\left(|\mathcal{I}_{\bullet}|,\mathcal{C}\right)
		\simeq
	\mathsf{Hom}_{\mathsf{Cat}^{\Delta^{op}}}\left(\mathcal{I}_{\bullet},\mathcal{C}^{[\bullet]}\right)
	\]
	together with the fact that $\mathcal{C}^{[\bullet]}$ is  $2$-coskeletal show that $|\mathcal{I}_{\bullet}|$
	is determined by the categories 
	$\mathcal{I}_0,\mathcal{I}_1,\mathcal{I}_2$
	and maps between them, i.e. by the truncation of
	formula $\eqref{REALDEF EQ}$ for $n,m \leq 2$.

Indeed, one can show that a sufficient set of generating relations for $|\mathcal{I}_{\bullet}|$ is given by:
\begin{enumerate*}
\item[(i)]
 the relations in $\mathcal{I}_0$
(including relations stating that identities of  
$\mathcal{I}_0$ are identities of $|\mathcal{I}_{\bullet}|$);
\item[(ii)] relations stating that for each $i_0 \in \mathcal{I}_0$ the arrow 
$i_0 = d_1(s_0(i_0)) \xrightarrow{s_0(i_0)} d_1(s_0(i_0)) = i_0$
is an identity;
\item[(iii)] for each arrow $I_1\colon i_1 \to i'_1$ in $\mathcal{I}_1$ the relation that the square below commutes
\end{enumerate*}
\[
\begin{tikzcd}
	d_1(i_1) \ar{r}{i_1} \ar{d}[swap]{d_1(I_1)} & 
	d_0(i_1) \ar{d}{d_0(I_1)}
\\
	d_1(i'_1) \ar{r}{i'_1} &
	d_0(i'_1)
\end{tikzcd}
\]
and;
\begin{enumerate*}
\item[(iv)] for each object $i_2 \in \mathcal{I}_2$ the relation that the following triangle commutes.
\end{enumerate*}
\[
\begin{tikzcd}[row sep = 0.5em]
	d_{1,2}(i_2) \ar{rr}{d_1(i_2)} \ar{rd}[swap]{d_2(i_2)} & & d_{0,1}(i_2) \\
	& d_{0,2}(i_2) \ar{ru}[swap]{d_0(i_2)}
\end{tikzcd}
\]
\end{remark}



We now relate diagrams in 
the span categories of \S \ref{WSPAN DEF} with 
the Grothendieck constructions 
of Definition \ref{GROTHCONS DEF}.

\begin{lemma}\label{SIMPSPANREIN LEMMA}
Functors $F \colon \mathcal{D} \ltimes \mathcal{I}_{\bullet} \to \C$ are in bijection with lifts
\[
\begin{tikzcd}
    & \mathsf{WSpan}^l(\**,\C) \ar{d}{\mathsf{fgt}} \\
\mathcal{D} \ar{r}[swap]{\mathcal{I}_{\bullet}} \ar[dashed]{ru}{\mathcal{I}_{\bullet}^F} & \mathsf{Cat}.
\end{tikzcd}
\]
where $\mathsf{fgt}$ is the functor forgetting the maps to $\**$ and $\C$.
\end{lemma}

\begin{proof}
	This is a matter of unpacking notation. The restrictions 
	$F|_{\mathcal{I}_d}$ to the fibers 
	$\mathcal{I}_d \hookrightarrow \mathcal{D} \ltimes \mathcal{I}_{\bullet}$
	are precisely the functors 
	$\mathcal{I}^F_d \colon \mathcal{I}_d \to \C$ describing $\mathcal{I}_{\bullet}^F(d)$.
	
	Furthermore, the images
	$F \left( (d,i) \to (d',f_{\**}(i)) \right)$	
	of the pushout arrows over a fixed arrow $f \colon d \to d'$ of $\mathcal{D}$
assemble to a natural transformation 
\[
	\begin{tikzcd}[row sep=0.4em]
		\mathcal{I}_d 
		\ar{dr}[name=F1]{I_d^F} \ar{dd}[swap]{f_{\**}} &
	\\
 & \C 
	\\
|[alias=G2]| \mathcal{I}_{d'}  \ar{ur}[swap]{I_{d'}^F} & 
		\arrow[Rightarrow, from=F1, to=G2,shorten >=0.15cm,shorten <=0.20cm]
	\end{tikzcd}
\]
which describes $\mathcal{I}_{\bullet}^F(f)$. 
One readily checks that the associativity and unitality conditions coincide.
\end{proof}

In the cases of interest we have $\mathcal{D}=\Delta^{op}$.
The following is the key result in this section.

\begin{proposition}\label{SOURCEFINAL PROP}
	Let $\mcI_{\bullet} \in \mathsf{Cat}^{\Delta^{op}}$.
	Then there is a natural functor
\[
\begin{tikzcd}
	\Delta^{op} \ltimes \mcI_{\bullet}
	\ar{r}{s} &
	\left| \mcI_{\bullet} \right|.
\end{tikzcd}
\]
Further, $s$ is final.
\end{proposition}

\begin{remark}
	The $s$ in the result above stands for \textit{source}. 
	This is because, for $\mcI \in \mathsf{Cat}$, the map
	$\Delta^{op} \ltimes \mcI^{[\bullet]}
	\to \left| \mcI^{[\bullet]} \right|
	\simeq \mcI$ is given by $s(i_0\to \cdots \to i_n) = i_0$.
\end{remark}

\begin{proof}
Recall that $|\mcI_{\bullet}|$ is the coequalizer \eqref{REALDEF EQ}. Given $(k,g_m) \in [n] \times \mcI_m$, we write 
$[k,g_m]$ for the corresponding object in $|\mcI_{\bullet}|$.
To simplify notation, we write objects of $\mcI_n$ as $i_n$
and implicitly assume that $[k,i_n]$ refers to the class of the object $(k,i_n) \in [n] \times \mcI_n$.

We define $s$ on objects by 
$s([n],i_n)=[0,i_n]$ and on an arrow 
$(\phi,I_m)\colon (n,i_n) \to (m,i'_m)$ as the composite
(note that $\phi\colon [m] \to [n]$ and $I_m\colon \phi^{\**}i_n\to i_m$)
\begin{equation}\label{TARGETDEFINITON EQ}
	[0,i_n] \to [\phi(0),i_n] =
	[0,\phi^{\**}i_n]	
	 \xrightarrow{I_m} 
	[0,i'_m].
\end{equation}
To check compatibility with composition,
the cases of a pair of either two fiber arrows (i.e. arrows where $\phi$ is the identity) or two pushforward arrows (i.e. arrows where $I_m$ is the identity) are immediate from \eqref{TARGETDEFINITON EQ}, 
hence we are left with the case 
$([n],i_n) \xrightarrow{I_n} ([n],i'_n) \to 
([m],\phi^{\**} i'_n)$
 of a fiber arrow followed by a pushforward arrow. 
 Noting that in $\Delta^{op} \ltimes \mcI_{\bullet}$
this composite can be rewritten as
$([n],i_n) \to ([m],\phi^{\**} i_n)
\xrightarrow{\phi^{\**}I_n} 
([m],\phi^{\**}i'_n)$
 this amounts to checking that
\[
\begin{tikzcd}
\left[0,i_n\right] \ar{r} \ar{d}[swap]{I_n} &
\left[\phi(0),i_n) \right] \ar[equal]{r} \ar{d}[swap]{I_n} &
\left[0,\phi^{\**} i_n \right] \ar{d}{\phi^{\**} I_n}
	\\
\left[0,i'_n\right] \ar{r} &
\left[\phi(0),i'_n\right] \ar[equal]{r} &
\left[0,\phi^{\**} i_n \right]
\end{tikzcd}
\]
commutes in $|\mcI_{\bullet}|$,
which is the case since the left square is encoded by a square in $[n]\times \mcI_n$
and the right square is encoded by an arrow in $[m]\times \mathcal{I}_n$.

We now show that $s$ is final.
Fix $h \in \mcI_0$. We must check that 
$[0,h] \downarrow \Delta^{op} \ltimes \mcI_{\bullet}$ is connected.
By Lemma \ref{OBJGENREL LEMMA},
any object in this undercategory has a description (not necessarily unique) as a pair
\begin{equation}\label{UNDERCATOB EQ}
\left(\left([n],i_n\right), [0,h] \xrightarrow{f_1} \cdots \xrightarrow{f_r} s([n],i_n) \right)
\end{equation}
where each $f_i$ is a generating arrow of $|\mcI_{\bullet}|$
induced by either an arrow $I_0$ of $\mcI_0$ or object $i_1\in \mcI_1$.
 We will connect \eqref{UNDERCATOB EQ} to the canonical object 
 $\left(([0],h),[0,h]=[0,h]\right)$, arguing by induction on $r$. 
If $n \neq 0$, the map 
$d_{\hat{0}} \colon ([n],i_n) \to ([0],d_{\hat{0}}^{\**}(i_n))$
 and the fact that 
$s \left(d_{\hat{0}}^{\**}\right) = id_{[0, d_{\hat{0}}^{\**}(i_n)]}$ provides an arrow to an object with $n=0$ without changing $r$.
If $n=0$, one can apply the induction hypothesis by lifting $f_r$ to $\Delta^{op} \ltimes \mcI_{\bullet}$ according to one of two cases:
\begin{enumerate*}
	\item[(i)] if $f_r$ is induced by an arrow $I_0$ of $\mcI_0$, the lift of $f_r$ is simply  
	$([0],i'_0) \xrightarrow {I_0} ([0],i_0)$;
	\item[(ii)] if $f_r$ is induced by $i_1\in \mcI_1$ the lift is provided by the map
	$([1],i_1) \to ([0],d_0(i_1))$.
\end{enumerate*}
\end{proof}

\begin{remark}\label{DUALRESULTS REM}
	The involution
	\[\Delta \xrightarrow{\tau} \Delta\]
	which sends $[n]$ to itself and $d_i,s_i$ to $d_{n-i},s_{n-i}$
	induces vertical isomorphisms
\[
\begin{tikzcd}
	\Delta^{op} \ltimes \left(\mathcal{I}_{\bullet} \circ \tau \right) \ar{r}{s} \ar{d}[swap]{\simeq} &
	\left|\mathcal{I}_{\bullet} \circ \tau \right|
	\ar{d}{\simeq}
\\
	\Delta^{op} \ltimes \mathcal{I}_{\bullet} \ar{r}[swap]{t} &
	\left| \mathcal{I}_{\bullet}^{op} \right|^{op}
\end{tikzcd}
\]
which reinterpret the ``source'' functor as what one might call the ``target'' functor, with $t([n],i_n)= [n,i_n]$ rather than 
$s([n],i_n)= [0,i_n]$.
The target functor is thus also final.

Moreover, the source/target formulations of 
all the results that follow are equivalent.
\end{remark}

In practice, we will need to know that the source $s$ and target $t$ satisfy a stronger finality condition with respect to left Kan extensions.

\begin{lemma}\label{UNDERLEFTADJ LEM}
Let $\mathcal{J} \in \mathsf{Cat}$ be a small category and 
$j \in \mathcal{J}$. 
Then the under and over category functors
\[
	\mathsf{Cat} \downarrow \mathcal{J}
		\xrightarrow{(\minus) \downarrow j}
	\mathsf{Cat},
\qquad
	\mathsf{Cat} \downarrow \mathcal{J}
		\xrightarrow{j \downarrow (\minus)}
	\mathsf{Cat}
\]
are left adjoints, and hence preserve colimits.
\end{lemma}

\begin{proof}
The right adjoint to $(\minus) \downarrow j$,
which we denote
$(\minus)^{\downarrow j} \colon
\mathsf{Cat}
	\to 
\mathsf{Cat} \downarrow \mathcal{J}
$,
is given on a category $\mathcal{C} \in \mathsf{Cat}$
by the Grothendieck construction
$\mathcal{C}^{\downarrow j}
= \mathcal{J} \ltimes \mathcal{C}^{\times \mathcal{J}(\minus,j)}$
for the functor
\[
\begin{tikzcd}[row sep = 0em]
	\mathcal{J} \ar{r} & \mathsf{Cat} \\
	k \ar[mapsto]{r} & \C^{\times \mathcal{J}(k,j)}.
\end{tikzcd}
\]
Given 
$(\mcI \xrightarrow{\pi} \mathcal{J})
\in (\mathsf{Cat} \downarrow \mathcal{J})$ and 
$\C \in \mathsf{Cat}$, we will show that functors 
$F \colon (\mathcal{I} \downarrow j) \to\C$
are in bijection with functors
$\hat{F} \colon \mathcal{I} \to \C^{\downarrow j}$ over $\mathcal{J}$.
Given $F$, we now describe the corresponding 
$\hat{F}$.

First, $F$ associates to each object
$(i,J\colon \pi(i) \to j)$ of
$\mathcal{I} \downarrow j$
an object
$F(i,J)\in \mathcal{C}$.
Write
$F_i \in \C^{\times \mathcal{J}(\pi(i),j)}$
for the assignment
$J \mapsto F(i,J)$,
i.e. $F_i(J) = F(i,J)$.
On objects $i \in \mathcal{I}$
one then sets
$\hat{F}(i) = (\pi(i), F_i)$.

Next, 
recall that arrows in 
$\mathcal{I} \downarrow j$
have the form
$(i',J \circ \pi(I)) \to (i,J)$
for some arrow $I \colon i' \to i$
in $\mathcal{I}$.
To each such arrow,
$F$ associates an arrow
$F_{i'}(J \circ \pi(I)) 
\to
F_i(J)$.
Fixing $I$ and allowing 
$J \in \mathcal{J}(\pi(i),j)$ to vary,
these arrows form a natural transformation
$F_I \colon F_{i'} \circ \pi(I)^{\**} 
\Rightarrow
F_i$,
where 
$\pi(I)^{\**} \colon \mathcal{J}(\pi(i),j) 
\to 
\mathcal{J}(\pi(i'),j)$
denotes precomposition with $\pi(I)$.
On arrows
$I \colon i' \to i$
one now sets
$\hat{F}(I) \colon 
(\pi(i'), F_{i'}) \to  (\pi(i), F_i)$
to be 
$(\pi(I) \colon \pi(i') \to \pi(i),
F_I \colon F_{i'} \circ \pi(I)^{\**} 
\Rightarrow
F_i)$.

It is clear that the procedures above relating the values of 
$F,\hat{F}$ on objects and arrows are invertible. 
One can readily check that
the functoriality requirements on 
$F,\hat{F}$ match.

Noting that 
$j \downarrow (-)$ is the composite
$
\mathsf{Cat} \downarrow \mathcal{J}
\xrightarrow{(-)^{op}}
\mathsf{Cat} \downarrow \mathcal{J}^{op}
\xrightarrow{(-) \downarrow j}
\mathsf{Cat}
\xrightarrow{(-)^{op}}
\mathsf{Cat}
$
yields that
its right adjoint
is the composite
$
\mathsf{Cat}
\xrightarrow{(-)^{op}}
\mathsf{Cat}
\xrightarrow{(-)^{\downarrow j}}
\mathsf{Cat} \downarrow \mathcal{J}^{op}
\xrightarrow{(-)^{op}}
\mathsf{Cat} \downarrow \mathcal{J}
$.
\end{proof}


\begin{corollary}\label{SOURCELANFINAL COR}
	Consider a map
	$\mcI_{\bullet} \to \mathcal{J}$
	between $\mcI_{\bullet} \in \mathsf{Cat}^{\Delta^{op}}$
	and a constant object
	$\mathcal{J} = \mathcal{J}_{\bullet} \in \mathsf{Cat}^{\Delta^{op}}$. Then the source and target maps 
\[
	\begin{tikzcd}
	\Delta^{op} \ltimes \mcI_{\bullet} \ar{rr}{s}  \ar{rd}&& \left|\mcI_{\bullet} \right|\ar{dl}
& &
	\Delta^{op} \ltimes \mcI_{\bullet} \ar{rr}{t}  \ar{rd}&& \left|\mcI_{\bullet}^{op} \right|^{op}\ar{dl}
\\
	& \mathcal{J} &
& &
	& \mathcal{J} &
	\end{tikzcd}	
\]
are $\Lan$-final over $\mathcal{J}$, i.e. the functors 
$s \downarrow j \colon (\Delta^{op} \ltimes \mcI_{\bullet})\downarrow j \to |\mcI_{\bullet}|\downarrow j$ are final for all $j\in \mathcal{J}$,
and similarly for $t$.
\end{corollary}

\begin{proof}
It is clear that $(\Delta^{op} \ltimes \mcI_{\bullet})\downarrow j \simeq \Delta^{op} \ltimes ( \mcI_{\bullet}\downarrow j)$
while Lemma \ref{UNDERLEFTADJ LEM}
guarantees that, since $(\minus) \downarrow j$ is a left adjoint, $|\mcI_{\bullet}|\downarrow j \simeq |\mcI_{\bullet}\downarrow j |$. One thus reduces to Proposition \ref{SOURCEFINAL PROP}.
\end{proof}

We will require two additional straightforward lemmas.

\begin{lemma}\label{TWISTING LEMMA}
	Let $\mcI_{\bullet}^F \in \mathsf{WSpan}^l(\**,\C)^{\Delta^{op}}$ be such that the diagrams
	\begin{equation}\label{IDENTSIMPRELSISO EQ}
	\begin{tikzcd}[row sep=0.4em,column sep = 5em]
		\mathcal{I}_n
		\ar{dr}[name=F1]{F_n} \ar{dd}[swap]{d_i} & &
		\mathcal{I}_n
		\ar{dr}[name=F2]{F_n} \ar{dd}[swap]{s_j} & &
	\\
 & \C & & \C &
	\\
|[alias=G2]| \mathcal{I}_{n-1}  \ar{ur}[swap]{F_{n-1}} & & 
|[alias=G3]| \mathcal{I}_{n+1}  \ar{ur}[swap]{F_{n+1}} & &
		\arrow[Leftrightarrow, from=F1, to=G2,shorten >=0.25cm,shorten <=0.25cm,"\delta_{i}"]
		\arrow[Leftrightarrow, from=F2, to=G3,shorten >=0.25cm,shorten <=0.25cm,"\sigma_{j}"]
	\end{tikzcd}
\end{equation}
are given by natural isomorphisms for $0 < i \leq n$, $0 \leq j \leq n$.
Then the functors $\tilde{F}_n \colon \mcI_n \to \C$ given by the composites
\[
\mcI_n \xrightarrow{d_{1,\cdots,n}} 
\mcI_0 \xrightarrow{F_0}
\C
\]
assemble to an object 
$\mcI_{\bullet}^{\tilde{F}} \in \mathsf{WSpan}^l(\**,\C)^{\Delta^{op}}$ which is isomorphic to $\mcI_{\bullet}^F$ and such that:
\begin{enumerate*}
\item[(i)] $\mcI_{\bullet}^{\tilde{F}}$
has the same operators $d_i,s_j$;
\item[(ii)] in $\mcI_{\bullet}^{\tilde{F}}$ the 
diagrams \eqref{IDENTSIMPRELSISO EQ} for $0 < i \leq n$, $0 \leq j \leq n$ are strictly commutative;
in $\mcI_{\bullet}^{\tilde{F}}$
the natural transformation associated to $d_0$
is the composite
\end{enumerate*}
\begin{equation}
\begin{tikzcd}[column sep = 40pt]
	\mathcal{I}_n \ar{r}{d_{2,\cdots,n}}
	\ar{d}[swap]{d_0} &
	\mathcal{I}_1 \ar{d}[swap]{d_0} \ar{rr}{d_1} 
	\ar{rrd}{F_1 }[near end,name=F1]{}
	[swap,near start,name=F2]{}
&&
	|[alias=Vt2]|
	\mathcal{I}_0 \ar{d}{F_0} &
\\
	\mathcal{I}_{n-1} \ar{r}{d_{1,\cdots,n-1}} &
	|[alias=FV2]|
	\mathcal{I}_0 \ar{rr}[swap]{F_0} &&
\mathcal{V}^{op} 
\arrow[Leftrightarrow, from=Vt2, to=F1, shorten <=0.10cm,shorten >=0.10cm,"\delta_1"]
\arrow[Rightarrow, from=F2, to=FV2, shorten <=0.10cm,shorten >=0.10cm,"\delta_0"]
\end{tikzcd}
\end{equation}

Dually, if \eqref{IDENTSIMPRELSISO EQ}
are natural isomorphisms for
$0\leq i <n$ and $0\leq j \leq n$,
one can form 
$\mcI_{\bullet}^{\tilde{F}} \in \mathsf{WSpan}^l(\**,\C)^{\Delta^{op}}$ 
such that the corresponding diagrams are strictly commutative.
\end{lemma}

\begin{proof}
This follows by a straightforward verification.
\end{proof}

\begin{lemma}\label{SOURCEFACT LEM}
	A (necessarily unique) factorization
\begin{equation}\label{SOURCEFACT EQ}
	\begin{tikzcd}[row sep = 0.5em]
	\Delta^{op} \ltimes \mcI_{\bullet} \ar{rr}{F_{\bullet}} \ar{rd}[swap]{s}& & \C \\
	& \left|\mcI_{\bullet}\right| \ar[dashed]{ru}[swap]{F}
	\end{tikzcd}
\end{equation}
	exists iff for the associated object 
	$\mcI_{\bullet} \in \mathsf{WSpan}^l(\**,\C)^{\Delta^{op}}$
	(cf. Lemma \ref{SIMPSPANREIN LEMMA})
	all faces $d_i$ for $0<i\leq n$ and degeneracies $s_j$ for $0\leq j \leq n$ are strictly commutative, i.e. they are given by diagrams
\begin{equation}\label{IDENTSIMPRELS EQ}
	\begin{tikzcd}[row sep=0.4em,column sep = 3.5em]
		\mathcal{I}_n
		\ar{dr}[name=F1]{F_n} \ar{dd}[swap]{d_0} & &
		\mathcal{I}_n
		\ar{dr}{F_n} \ar{dd}[swap]{d_i} & &
		\mathcal{I}_n
		\ar{dr}{F_n} \ar{dd}[swap]{s_j} &
	\\
 & \C & & \C & & \C
	\\
|[alias=G2]| \mathcal{I}_{n-1}  \ar{ur}[swap]{F_{n-1}} & & 
 \mathcal{I}_{n-1}  \ar{ur}[swap]{F_{n-1}} & &
 \mathcal{I}_{n+1}  \ar{ur}[swap]{F_{n+1}} &
		\arrow[Rightarrow, from=F1, to=G2,shorten >=0.25cm,shorten <=0.25cm,"\varphi_n"]
	\end{tikzcd}
\end{equation}
Dually, a factorization through the target 
$t \colon \Delta^{op} \ltimes \mathcal{I}_{\bullet}
\to |\mathcal{I}_{\bullet}^{op}|^{op}$
exists iff the faces $d_i$ and degeneracies 
$s_j$ are strictly commutative for
$0\leq i <n$, $0\leq j \leq n$.
\end{lemma}

\begin{proof}
For the ``only if'' direction, it suffices to note that $s$ sends all pushout arrows of $\Delta^{op} \ltimes \mcI_{\bullet}$ for faces $d_i$, $0<i\leq n$ and degeneracies
$s_j$, $0\leq j \leq n$ to identities,
yielding the required commutative diagrams in \eqref{IDENTSIMPRELS EQ}.

For the ``if''  direction, this will follow by building a 
functor
$\mcI_{\bullet} \xrightarrow{\bar{F}_{\bullet}} \C^{[\bullet]}$ together with the naturality of the source map $s$ (recall that $|\C^{[\bullet]}|\simeq \C)$. We define
$\bar{F}_n|_{k \to k+1}$ as the map
\begin{equation}\label{EQUIVALENCEDEF EQ}
F_{n-k} d_{0,\cdots,k-1}
	\xrightarrow{\varphi_{n-k} d_{0,\cdots,k-1}}
F_{n-k-1} d_{0,\cdots,k}.
\end{equation}
The claim that $s \circ (\Delta^{op} \ltimes \bar{F})$ recovers the horizontal map in \eqref{SOURCEFACT EQ} is straightforward, hence the real task is to prove that \eqref{EQUIVALENCEDEF EQ} defines a map of simplicial objects.
First, functoriality of the original $F_{\bullet}$
yields identities
\[
	\varphi_{n-1}d_i = \varphi_n,\phantom{1}1<i
		\qquad
	\varphi_{n-1}d_1 = (\varphi_{n-1}d_0) \circ \varphi_n,
		\qquad
	\varphi_{n+1} s_i = \varphi_{n},\phantom{1}0<i,
		\qquad
	\varphi_{n+1} s_{0} =id_{F_{n}}
\]
Next, note that there is no ambiguity in writing simply 
$\varphi_{n-k} d_{0,\cdots,k-1}$
to denote the map \eqref{EQUIVALENCEDEF EQ}.
We now check that $\bar{F}_{n-1} d_i = d_i \bar{F}_n$, $0 \leq i \leq n$, which must be verified after restricting to each $k \to k+1$, $0\leq k \leq n-2$. There are three cases, depending on $i$ and $k$:
\begin{itemize}
	\item[($i <k+1$)] 
	$\varphi_{n-k-1} d_{0,\cdots,k-1} d_i =
	\varphi_{n-k-1} d_{0,\cdots,k}$;
	\item[($i=k+1$)] 
	$\varphi_{n-k-1} d_{0,\cdots,k-1} d_i =
	\varphi_{n-k-1} d_1 d_{0,\cdots,k-1}=
	(\varphi_{n-k-1} d_0 \circ \varphi_{n-k})d_{0,\cdots,k-1}=
	(\varphi_{n-k-1}d_{0,\cdots,k})\circ(\varphi_{n-k}d_{0,\cdots,k-1})
	$;
	\item[($i>k+1$)] 
	$\varphi_{n-k-1} d_{0,\cdots,k-1} d_i =
	\varphi_{n-k-1} d_{i-k} d_{0,\cdots,k-1} =
	\varphi_{n-k}d_{0,\cdots,k-1}$.
\end{itemize}
The case of degeneracies is similar.
\end{proof}

\begin{proof}[proof of Proposition \ref{RANTRANS PROP}]
The result follows from the following string of identifications.
\begin{align*}
	\lim_{\Delta}
	\left(
	\mathsf{Ran}_{A_n \to \Sigma_G}
	N_{n}
	\right)
	\simeq &
	\mathsf{Ran}_{\Delta \times \Sigma_G \to \Sigma_G}
	\left(
	\mathsf{Ran}_{A_n \to \Sigma_G}
	N_n
	\right) \simeq
\\
	\simeq &
	\mathsf{Ran}_{\Delta \times \Sigma_G \to \Sigma_G}
	\left(
	\mathsf{Ran}_{(\Delta^{op} \ltimes A^{op}_{\bullet})^{op} \to
	\Delta \times \Sigma_G}
	N_{\bullet}
	\right) \simeq
\\
	\simeq &
	\mathsf{Ran}_{
	(\Delta^{op} \ltimes A^{op}_{\bullet})^{op} \to
	\Sigma_G}
	N_{\bullet}
	\simeq 
	\mathsf{Ran}_{
	(\Delta^{op} \ltimes A^{op}_{\bullet})^{op} \to \Sigma_G}
	\tilde{N}_{\bullet}
	\simeq
	\mathsf{Ran}_{
	|A_{\bullet}| \to \Sigma_G}
	\tilde{N}
\end{align*}
The first step simply rewrites 
$\lim_{\Delta}$. 
The second step 
follows from Proposition \ref{FIBERKANMAP PROP} applied to the map 
$(\Delta^{op} \ltimes A^{op}_{\bullet})^{op} \to
\Delta \times \Sigma_G$
of Grothendieck fibrations over $\Delta$ since,
for each $(n,C) \in \Delta \times \Sigma_G$,
one has a natural identification
between
$(n,C) \downarrow_{\Delta} (\Delta^{op} \ltimes A^{op}_{\bullet})^{op}$
and
$C \downarrow A_n$.
The third step follows since iterated Kan extensions are again Kan extensions.
The fourth step twists $N_{\bullet}$
as in Lemma \ref{TWISTING LEMMA}
to obtain $\tilde{N}_{\bullet}$
such that the $d_i$, $s_j$ are given by strictly commutative diagrams for
$0\leq i < n$, $0\leq j \leq n$.
Lastly, the final step uses Lemma \ref{SOURCEFACT LEM}
to conclude that 
$\tilde{N}_{\bullet}$ factors through the
target functor $t$, obtaining $\tilde{N}$, 
and then uses 
Corollary \ref{SOURCELANFINAL COR}
to conclude that the Kan extensions indeed coincide.
\end{proof}

\section{The nerve theorem}\label{NERVE AP}

Our goal in this appendix is to prove the nerve theorem below,
adapting \cite[Prop. 5.3, Thm. 6.1]{MW08}.
Throughout we assume that the monoidal structure on $\mathcal{V}$
is the cartesian product.

\begin{theorem}\label{NERVE THM}
	There is a fully faithful nerve functor
	$\mathcal{N} \colon \mathsf{Op}_G(\mathcal{V})
	\to 
	\mathcal{V}^{\Omega_G^{op}}$
	whose essential image 
	consists of the pointed strict Segal objects,
	i.e. those $X \in \mathcal{V}^{\Omega_G^{op}}$
	such that the natural maps
	\begin{equation}\label{SEGCOND EQ}
	X(T) \xrightarrow{\simeq}
	\prod_{v \in V_G(T)} X(T_v)
	\end{equation}
	are isomorphisms for all $T \in \Omega_G$.
\end{theorem}

We will prove Theorem \ref{NERVE THM} by building 
$\mathcal{N}$ in \eqref{FULLNER EQ},
describing its partial inverse in \eqref{XMINUSMULT EQ},
then finishing the argument at the end of the appendix.

\begin{remark}
	In \cite{MW08}, which sets 
	$\mathcal{V} = \mathsf{Set}$, $G = \**$
	and works with \emph{colored} operads 
	$\mathsf{Op}_{\bullet}$ (of sets),
	the nerve functor
	$\mathcal{N} \colon 
	\mathsf{Op}_{\bullet} \to \mathsf{Set}^{\Omega^{op}}$
	is defined by
\begin{equation}\label{NEREASY EQ}
	(\mathcal{N} \mathcal{O})(T)
=
	\mathsf{Op}_{\bullet}\left(
	\Omega(T),\mathcal{O}
	\right)
\end{equation}
where $\Omega(T)$ for $T \in \Omega$
is the colored operad described in
\cite[\S 3]{MW07} (or after \eqref{COMPEX EQ}).

However, since this paper does not
discuss \emph{colored} genuine operads 
(due to $\mathsf{Op}_G(\mathcal{V})$
being the single colored case),
we can not obtain
Theorem \ref{NERVE THM} by 
directly adapting \eqref{NEREASY EQ}.
\end{remark}

\begin{remark}
	The term ``pointed'' in Theorem \ref{NERVE THM}
	is motivated by the fact that if 
	$X$ satisfies \eqref{SEGCOND EQ}
	then it is $X(G/H \cdot \eta) = \**$ for all $H \leq G$,
	due to $V_G(G/H \cdot \eta) = ()$ being the empty tuple.
	
	This pointedness reflects the fact that
	$\mathsf{Op}_G(\mathcal{V})$
	includes only \emph{single colored} genuine operads. 
	In the multiple color setting,
	the Segal condition \eqref{SEGCOND EQ} needs to be modified
	\cite[Def. 3.35]{BP_edss}.
\end{remark}

Our description of 
$\mathcal{N} \colon \mathsf{Op}_G(\mathcal{V})
\to 
\mathcal{V}^{\Omega_G^{op}}$
will make use of the monad 
$N$ on
$\mathsf{WSpan}^r(\Sigma_G,\mathcal{V}^{op})$
in Definition \ref{WSPAN_MONAD_DEFINITION}.
Given $\mathcal{P} \in \mathsf{Op}_G(\mathcal{V})$,
so that $\upsilon \mathcal{P}$ is a $N$-algebra (cf. Remark \ref{REPACKAGERES REM}; recall that 
$\upsilon \colon \mathsf{Fun}(\Sigma_G,\mathcal{V}^{op})
\to 
\mathsf{WSpan}^r(\Sigma_G,\mathcal{V}^{op})$
is the inclusion functor sending
$\Sigma_G \xrightarrow{\mathcal{P}}\mathcal{V}^{op}$
to the span 
$\Sigma_G \overset{=}{\leftarrow} \Sigma_G \xrightarrow{\mathcal{P}}\mathcal{V}^{op}$, as discussed in \S\ref{FGMON SEC}),
consider the bar construction
$B_{\bullet} = 
B_{\bullet}(N,N,\upsilon \mathcal{P})
= N^{\bullet+1} \upsilon \mathcal{P}$,
whose levels we denote as
\begin{equation}\label{BARLEVELS EQ}
	\Sigma_G \leftarrow 
	\Omega_G^n \xrightarrow{N_n^{\mathcal{P}}}
	\mathcal{V}^{op}.
\end{equation}
Ignoring the map to $\Sigma_G$,
\eqref{BARLEVELS EQ} determines 
a simplicial object in 
$\mathsf{WSpan}^r(\**,\mathcal{V}^{op})$.
Moreover, since the face maps $d_i$ with $i < n$
are given by the multiplication
$NN \Rightarrow N$ in \eqref{MULTDEFSPAN EQ},
the opposite of this simplicial object 
(in $\mathsf{WSpan}^l(\**,\mathcal{V})$)
satisfies the dual case conditions in 
Lemma \ref{TWISTING LEMMA}.
Thus, Lemma \ref{TWISTING LEMMA}
provides an isomorphic simplicial object
$\tilde{N}_n^{\mathcal{P}} \colon 
\Omega_G^n \to \mathcal{V}^{op}$
satisfying the dual of the conditions in
\eqref{IDENTSIMPRELS EQ}.
Hence, by Lemma \ref{SOURCEFACT LEM} and Remark \ref{REALEX REM},
upon realization this induces a functor
\begin{equation}\label{TALLNER EQ}
	|\Omega^n_G| = \Omega_G^{\mathsf{t}}
	\xrightarrow{\mathcal{N} \mathcal{P}}
	\mathcal{V}^{op}
\end{equation}
where $\Omega_G^{\mathsf{t}} \subset \Omega_G$
is the subcategory of tall maps.
To define the nerve $\mathcal{N}$ in Theorem \ref{NERVE THM},
we must extend \eqref{TALLNER EQ}
to the entire category $\Omega_G$.
To do so, we enlarge the string categories
$\Omega_G^n$.

\begin{definition}\label{PLANSTRO DEF}
	Let $n \geq 0$. The category $\overline{\Omega}^n_G$
	has objects the planar tall strings
	$(T_0 \to \cdots \to T_n) \in \Omega^n_G$
	and arrows diagrams
	$(\rho_i \colon T_i \to T'_i)$
	as in \eqref{PTNARROW EQ}
	where the $\rho_i$ are outer maps
	in each tree component.
\end{definition}

\begin{remark}
	In contrasting Definitions 
	\ref{PLANSTR DEF} and \ref{PLANSTRO DEF},
	recall that quotients are the maps
	which are isomorphisms in each tree component, so that
	$\Omega^n_G \subseteq \overline{\Omega}^n_G$.
\end{remark}

Clearly the 
$\overline{\Omega}^{n}_G$ still form a simplicial object, 
i.e. one has operators
$d_i \colon \overline{\Omega}^{n}_G \to \overline{\Omega}^{n-1}_G$
for $0 \leq i \leq n+1$
and 
$s_j \colon \Omega^{n}_G \to 
\overline{\Omega}^{n+1}_G$
for $0 \leq j \leq n$,
though we caution that the $\overline{\Omega}^{n}_G$
have no augmentation to $\Sigma_G$ nor extra degeneracies $s_{-1}$.
Moreover, it is clear that
$|\overline{\Omega}_G| = \Omega_G$.

More importantly,
since maps that are outer in each tree component send vertices to vertices,
one has that the formula in Notation \ref{VGDEF NOT} extends to define a functor
\begin{equation}\label{VGDEF2 EQ}
	\overline{\Omega}^n_G 
	\xrightarrow{\boldsymbol{V}_G}
	\Fin \wr \Omega^{n-1}_G
\end{equation}
Note that \eqref{VGDEF2 EQ} 
requires the full category $\Fin$ of finite sets
rather than the subcategory $\Fin_s$ of surjections.
By construction of $N$ in Definition \ref{WSPAN_MONAD_DEFINITION}
one has that the functors in \eqref{BARLEVELS EQ}
extend to functors
$N^{\mathcal{P}}_n \colon \overline{\Omega}^n_G \to \mathcal{V}^{op}$.
Moreover, the natural transformations for the associated simplicial object in $\mathsf{WSpan}^r(\**,\mathcal{V}^{op})$
all factor through one of the diagrams below,
\begin{equation}
\begin{tikzcd}[column sep = 18pt]
	\overline{\Omega}^{n}_G \ar{r}{\boldsymbol{V}_G} 
	\ar{d}[swap]{d_0} &
	\Fin \wr \Omega^{n-1}_G \ar{r}{\Fin \wr \boldsymbol{V}_G} &
	|[alias=Vt]|
	\Fin^{\wr 2} \wr \Omega^{n-2}_G \ar{d}[swap]{\sigma^0}
&
	\overline{\Omega}^{n}_G \ar{r}{\boldsymbol{V}_G} \ar{d}[swap]{d_{i+1}} &
	\Fin \wr \Omega^{n-1}_G \ar{d}{\Fin \wr d_i}
&
	\overline{\Omega}^{n}_G \ar{r}{\boldsymbol{V}_G} \ar{d}[swap]{s_{j+1}} &
	\Fin \wr \Omega^{n-1}_G \ar{d}{\Fin \wr s_j}
\\
	|[alias=FV]|
	\overline{\Omega}^{n-1}_G \ar{rr}[swap]{\boldsymbol{V}_G} &&
	\Fin \wr \Omega^{n-2}_G  
&
	\overline{\Omega}^{n-1}_G \ar{r}[swap]{\boldsymbol{V}_G} &
	\Fin \wr \Omega^{n-2}_G
&
	\overline{\Omega}^{n+1}_G \ar{r}[swap]{\boldsymbol{V}_G} &
	\Fin \wr \Omega^{n}_G
\arrow[Leftrightarrow, from=Vt, to=FV, shorten <=0.10cm,shorten >=0.10cm,"\pi"]
\end{tikzcd}
\end{equation}
so that
$N^{\mathcal{P}}_n \colon \overline{\Omega}^n_G \to \mathcal{V}^{op}$
extends \eqref{BARLEVELS EQ}
as a simplicial object in $\mathsf{WSpan}^r(\**,\mathcal{V}^{op})$.
Thus, by Lemmas 
\ref{TWISTING LEMMA} and \ref{SOURCEFACT LEM},
one again obtains an isomorphic simplicial object
$\tilde{N}^{\mathcal{P}}_n \colon \overline{\Omega}^n_G \to \mathcal{V}^{op}$ which, upon realization,
extends \eqref{TALLNER EQ} to obtain the desired nerve
\begin{equation}\label{FULLNER EQ}
|\overline{\Omega}^n_G| = \Omega_G
\xrightarrow{\mathcal{N} \mathcal{P}}
\mathcal{V}^{op}.
\end{equation}

We next describe the partial inverse to 
$\mathcal{N} \colon 
\mathsf{Op}_G(\mathcal{V}) \to \mathcal{V}^{\Omega_G^{op}}$.
Choose 
$X \colon \Omega_G \to \mathcal{V}^{op}$
whose opposite satisfies the Segal condition \eqref{SEGCOND EQ}.
Letting $\mathcal{P}_X$ be the composite
$\Sigma_G \to \Omega_G \to \mathcal{V}^{op}$,
we will show that $\mathcal{P}_X$ is a genuine operad
or, equivalently, that $\upsilon \mathcal{P}_X$ is a $N$-algebra.  
Throughout, we write
$\overline{\Omega}^n_G \to \Omega_G$
for the target functor
$(T_0 \to \cdots \to T_n) \mapsto T_n$ and denote by
\begin{equation}\label{SIMP EQ}
\begin{tikzcd}[row sep=1.2em, column sep = 6em]
	\overline{\Omega}^n_G
	\ar{r}[name=F1]{} 
	\ar{d}[swap]{d_n} &
	\Omega_G 
\\
	|[alias=G2]|
	\overline{\Omega}^{n-1}_G 
	\ar{ur}[swap]{} & 
\arrow[Leftarrow, from=F1, to=G2,shorten >=0.15cm,shorten <=0.20cm,
"\varphi_n"]
\end{tikzcd}
\end{equation}
the natural transformation induced by $T_{n-1} \to T_n$.
We now define spans $X_n$ for $n \geq -1$ by
\begin{equation}\label{XNSPANS EQ}
X_n =
\left(
\Sigma_G \leftarrow
\Omega_G^n \to
\overline{\Omega}^n_G \to \Omega_G 
\xrightarrow{X} \mathcal{V}^{op} 
\right).
\end{equation}
Note that $X_{-1} = \upsilon \mathcal{P}_X$.
Moreover, the transformations \eqref{SIMP EQ}
make the $X_n$ into a simplicial object in
$\mathsf{WSpan}^r(\Sigma_G,\mathcal{V}^{op})$.
Next, note that one has natural transformations $\rho_n$
\begin{equation}\label{THERHON EQ}
	\begin{tikzcd}
	\overline{\Omega}^n_G \ar[equal]{d} \ar{r} &
	\Omega_G \ar{r}{\delta^0} &
|[alias=Vt]|
	\Fin \wr \Omega_G \ar[equal]{d}
\\
	|[alias=FV]|
	\overline{\Omega}^n_G \ar{r}[swap]{\boldsymbol{V}_G} &
	\Fin \wr \Omega_G^{n-1} \ar{r} &
	\Fin \wr \Omega_G 
\arrow[Leftarrow, from=Vt, to=FV, shorten <=0.10cm,shorten >=0.10cm,"\rho_n"]
\end{tikzcd}
\end{equation}
which,
on $(T_0 \to \cdots \to T_n) \in \Omega^n_G$,
are given by the tuple map
$(T_{n,v})_{v \in V_G(T_0)} \to (T_n)$
determined by the inclusions
$T_{n,v} \to T_n$.
Note that, by whiskering with the map
$\Fin \wr \Omega_G \to 
\Fin \wr \mathcal{V}^{op} 
\xrightarrow{\prod} \mathcal{V}^{op}$,
$\rho_n$
determines a map of spans which we likewise denote 
$\rho_n \colon X_n \to N X_{n-1}$.

\begin{remark}\label{SEGALCOND REM}
	The Segal condition \eqref{SEGCOND EQ} holds iff
	$\rho_0 \colon X_0 \to N X_{-1}$ is an isomorphism
	and iff the 
	$\rho_n \colon X_n \to N X_{n-1}$ are isomorphisms 
	for all $n\geq 0$.
\end{remark}

\begin{proposition}
	Suppose the opposite of $X\colon \Omega_G \to \mathcal{V}^{op}$
	satisfies the Segal condition \eqref{SEGCOND EQ}. 
	
	Then the span $X_{-1}$ in \eqref{XNSPANS EQ} is a $N$-algebra with multiplication
\begin{equation}\label{XMINUSMULT EQ}
	\begin{tikzcd}
	N X_{-1} &
	X_0 \ar{l}[swap]{\rho_0}{\simeq} \ar{r}{d_0} &
	X_{-1}.
	\end{tikzcd}
\end{equation}
\end{proposition}

\begin{proof}
The required associativity and unitality conditions for 
\eqref{XMINUSMULT EQ}
say that the outer paths in the diagrams below coincide
($\mu,\eta$ are the multiplication and unit of $N$,
cf. Definition \ref{WSPAN_MONAD_DEFINITION}).
\begin{equation}\label{ALGCHECK EQ}
\begin{tikzcd}[column sep = 25pt]
	NNX_{-1} \ar{rr}{\mu}&&
	NX_{-1}
&
	X_{-1} \ar{r}{\eta} \ar{rd}{s_{-1}} \ar[equal]{rdd} &
	NX_{-1}
\\
	NX_0 \ar{u}{N \rho_0}[swap]{\simeq} \ar{d}[swap]{Nd_0} &
	X_1 \ar{l}[swap]{\rho_1}{\simeq}  \ar{d}{d_1} \ar{r}{d_0} &
	X_0 \ar{d}{d_0} \ar{u}[swap]{\rho_0}{\simeq}
&
	&
	X_0 \ar{u}[swap]{\rho_0}{\simeq} \ar{d}{d_0}
\\
	N X_{-1} &
	X_0 \ar{r}{d_0} \ar{l}[swap]{\rho_0}{\simeq} &
	X_{-1} 
&
	&
	X_{-1}
\end{tikzcd}
\end{equation}
One readily checks that the unitality diagram commutes.
It remains to check that all squares in the associativity diagram commute. 
The case of the bottom right square is tautological.
For the bottom left square note that, up to whiskering with
$\Fin \wr \Omega_G \to \mathcal{V}^{op}$,
the composites 
$X_1 
\xrightarrow{d_1} X_0 
\xrightarrow{\rho_0} N X_{-1}$
and
$X_1  
\xrightarrow{\rho_1} N X_0 
\xrightarrow{N d_0} N X_{-1}$
are induced by the diagrams below
\begin{equation}
\begin{tikzcd}
	\overline{\Omega}_G^1 \ar{d}[swap]{d_1} \ar{r} &
	|[alias=Vtt]|
	\Omega_G \ar{r}{\delta^0} \ar[equal]{d} &
	\Fin \wr \Omega_G \ar[equal]{d} 
&
	\overline{\Omega}_G^1 \ar[equal]{d} \ar{rr} &&
	|[alias=Vtt2]|
	\Fin \wr \Omega_G \ar[equal]{d} 
\\
	|[alias=FVt]|
	\overline{\Omega}_G^0 
	\ar{r} \ar[equal]{d} &
	\Omega_G \ar{r}{\delta^0}&
	|[alias=Vt]|
	\Fin \wr \Omega_G \ar[equal]{d} 
&
	|[alias=FFV]|	
	\overline{\Omega}_G^1 
	\ar{r} \ar{d}[swap]{d_1} &
	\Fin \wr \Omega_G^0 \ar{r} \ar{d}[swap]{d_0} &
	|[alias=Vt2]|
	\Fin \wr \Omega_G \ar[equal]{d}
\\
	|[alias=FV]|
	\overline{\Omega}_G^0 \ar{r}[swap]{\boldsymbol{V}_G} &
	\Fin \wr \Sigma_G \ar{r} &
	\Fin \wr \Omega_G 
&
	\overline{\Omega}_G^0 \ar{r}[swap]{\boldsymbol{V}_G} &
	|[alias=FV2]|
	\Fin \wr \Sigma_G \ar{r} &
	\Fin \wr \Omega_G 
\arrow[Leftarrow, from=Vt, to=FV,shorten <=0.10cm,shorten >=0.10cm,"\rho_0"]
\arrow[Leftarrow, from=Vtt, to=FVt,shorten <=0.10cm,shorten >=0.10cm,"\varphi_1"]
\arrow[Leftarrow, from=Vt2, to=FV2,shorten <=0.10cm,shorten >=0.10cm,"\varphi_0"]
\arrow[Leftarrow, from=Vtt2, to=FFV, shorten <=0.10cm,shorten >=0.10cm,"\rho_1"]
\end{tikzcd}
\end{equation}
That these composite natural transformations coincide is the observation that,
on $(T_0 \to T_1) \in \Omega^1_G$,
both compute the map
$(T_{0,v})_{v \in V_G(T_0)} \to (T_1)$
given by the maps $T_{0,v} \to T_1$.

To show that the top square in \eqref{ALGCHECK EQ} commutes,
we first consider the composite
$N X_0 \xrightarrow{N \rho_0}
NN X_{-1} \xrightarrow{\mu} N X_{-1}$,
which is given by the composite diagram below.
\begin{equation}\label{INTERM EQ}
\begin{tikzcd}[column sep=1.05em]
\overline{\Omega}_G^1 \ar[equal]{d} \ar{r} &
\Fin \wr \Omega_G^0 \ar{r} \ar[equal]{d}&
\Fin \wr \Omega_G \ar{r}{\delta^1} &
|[alias=dog2]|
\Fin^{\wr 2} \wr \Omega_G \ar{r} \ar[equal]{d} &
\Fin^{\wr 2} \wr \mathcal{V}^{op} \ar{r}{\prod} \ar[equal]{d} &
\Fin \wr \mathcal{V}^{op} \ar{r}{\prod} \ar[equal]{d}&
\mathcal{V}^{op} \ar[equal]{d} &
\\
\overline{\Omega}_G^{1} \ar{r} \ar{d}[swap]{d_{0}} &
|[alias=cat2]|
\Fin \wr \Omega_G^0 \ar{r} &
|[alias=FFOmega]|
\Fin^{\wr 2} \wr \Sigma_G \ar{r} \ar{d}{\sigma^0} &
\Fin^{\wr 2} \wr \Omega_G \ar{d}{\sigma^0} \ar{r} &
|[alias=cat3]|
\Fin^{\wr 2} \wr \mathcal{V}^{op} \ar{d}{\sigma^0} 
\ar{r}{\prod} &
\Fin \wr \mathcal{V}^{op} \ar{r}{\prod} &
|[alias=dog]|
\mathcal{V}^{op} \ar[equal]{d}
\\
|[alias=Omega]|
\overline{\Omega}_G^{0} \ar{rr} &&
\Fin \wr \Sigma_G \ar{r} &
\Fin \wr \Omega_G \ar{r} &
|[alias=cat]|
\Fin \wr \mathcal{V}^{op} \ar{rr}[swap]{\prod} &&
\mathcal{V}^{op}
\arrow[Leftrightarrow, from=FFOmega, to=Omega,shorten <=0.15cm,,shorten >=0.15cm,"\pi_0"]
\arrow[Leftrightarrow, from=dog, to=cat,shorten <=0.15cm,,shorten >=0.15cm,"\alpha"]
\arrow[Leftarrow, from=dog2, to=cat2,shorten <=0.15cm,,shorten >=0.15cm,"\rho_0"]
\end{tikzcd}
\end{equation}
We now consider the following, where all terms (other than the 
additional $\Fin \wr \mathcal{V}^{op}$ term on the top row)
retain their relative positions in \eqref{INTERM EQ}.
\begin{equation}
\begin{tikzcd}[column sep=1.05em]
\Fin \wr \Omega_G \ar{rr} \ar{dr}[swap]{\delta^1} &&
\Fin \wr \mathcal{V}^{op} \ar{d}{\delta^1} &&&
\\
&
\Fin^{\wr 2} \wr \Omega_G \ar{d}{\sigma^0} \ar{r} &
|[alias=cat3]|
\Fin^{\wr 2} \wr \mathcal{V}^{op} \ar{d}{\sigma^0} 
\ar{r}{\prod} &
\Fin \wr \mathcal{V}^{op} \ar{r}{\prod} &
|[alias=dog]|
\mathcal{V}^{op} \ar[equal]{d}
\\
&
\Fin \wr \Omega_G \ar{r} &
|[alias=cat]|
\Fin \wr \mathcal{V}^{op} \ar{rr}[swap]{\prod} &&
\mathcal{V}^{op}
\arrow[Leftrightarrow, from=dog, to=cat,shorten <=0.15cm,,shorten >=0.15cm,"\alpha"]
\end{tikzcd}
\end{equation}
By \eqref{COHER3 EQ},
the diagram above is the identity for
the functor
$\Fin \wr \Omega_G \to 
\Fin \wr \mathcal{V}^{op}
\to \mathcal{V}^{op}$.
As such, the composite
natural transformation in \eqref{INTERM EQ}
can be described 
by whiskering its 4 leftmost columns
(equivalently, the 3 bottom rows of the left diagram in 
\eqref{INTERM2 EQ})
with 
$\Fin \wr \Omega_G \to \mathcal{V}^{op}$.
It now follows that the composites
$X_1 \xrightarrow{\rho_1} 
N X_0 \xrightarrow{N \rho_0}
NN X_{-1} \xrightarrow{\mu} N X_{-1}$
and
$X_1 \xrightarrow{d_0} X_0 \xrightarrow{\rho_0} N X_{-1}$
are obtained by whiskering the diagrams below with 
$\Fin \wr \Omega_G \to \mathcal{V}^{op}$.
\begin{equation}\label{INTERM2 EQ}
\begin{tikzcd}[column sep=1.05em,row sep = 1em]
\overline{\Omega}_G^1 \ar[equal]{d} \ar{r} &
\Omega_G \ar{r}{\delta^0} &
|[alias=dog]|
\Fin \wr \Omega_G \ar{d} \ar{r}{\delta^1} &
\Fin^{\wr 2} \wr \Omega_G \ar[equal]{d}
&&
\overline{\Omega}_G^1 \ar{r} \ar{d}[swap]{d_0} &
\Omega_G \ar{r}{\delta^0} \ar[equal]{d} &
\Fin \wr \Omega_G \ar[equal]{d}
\\
|[alias=cat]|
\overline{\Omega}_G^1 \ar[equal]{d} \ar{r} &
\Fin \wr \Omega_G^0 \ar{r} \ar[equal]{d}&
\Fin \wr \Omega_G \ar{r}{\delta^1} &
|[alias=dog2]|
\Fin^{\wr 2} \wr \Omega_G \ar[equal]{d} 
&&
\overline{\Omega}_G^0 \ar[equal]{d} \ar{r} &
\Omega_G \ar{r}{\delta^0} &
|[alias=Vt]|
\Fin \wr \Omega_G \ar[equal]{d}
\\
\overline{\Omega}_G^{1} \ar{r} \ar{d}[swap]{d_{0}} &
|[alias=cat2]|
\Fin \wr \Omega_G^0 \ar{r} &
|[alias=FFOmega]|
\Fin^{\wr 2} \wr \Sigma_G \ar{r} \ar{d}{\sigma^0} &
\Fin^{\wr 2} \wr \Omega_G \ar{d}{\sigma^0} 
&&
|[alias=FV]|
\overline{\Omega}_G^0 \ar{r} &
\Fin \wr \Sigma_G \ar{r} &
\Fin \wr \Omega_G 
\\
|[alias=Omega]|
\overline{\Omega}_G^{0} \ar{rr} &&	
\Fin \wr \Sigma_G \ar{r} &
\Fin \wr \Omega_G 
\arrow[Leftrightarrow, from=FFOmega, to=Omega,shorten <=0.15cm, shorten >=0.15cm, near start, "\pi_0"]
\arrow[Leftarrow, from=dog2, to=cat2,shorten <=0.15cm, shorten >=0.15cm, near start, "\rho_0"]
\arrow[Leftarrow, from=dog, to=cat,shorten <=0.15cm, shorten >=0.15cm, near start,"\rho_1"]
\arrow[Leftarrow, from=Vt, to=FV, shorten <=0.10cm, shorten >=0.10cm, near start, "\rho_0"]
\end{tikzcd}
\end{equation}
That the composites in 
\eqref{INTERM2 EQ} coincide follows since,
on $(T_0 \to T_1) \in \Omega_G^1$,
both compute the map
$(T_{1,v})_{v \in V_G(T_1)} \to (T_1)$
whose components are given by the inclusions
$T_{1,v} \to T_1$.
\end{proof}

\begin{proof}[Proof of Theorem \ref{NERVE THM}]
	Let $\mathcal{P} \in \mathsf{Op}_G(\mathcal{V})$ and
	$\mathcal{N} \mathcal{P} \colon 
	\Omega_G \to \mathcal{V}^{op}$ be as in 
	\eqref{FULLNER EQ}.
	By the construction in 
	Lemmas \ref{TWISTING LEMMA} and \ref{SOURCEFACT LEM},
	the composites below coincide,
	where $(d_i,\nu_i)$ denote the simplicial operators of the simplicial object 
	$N_n^{\mathcal{P}}$ in $\mathsf{WSpan}^r(\**,\mathcal{V}^{op})$
	discussed above \eqref{FULLNER EQ}.
\begin{equation}\label{EXPREAL EQ}
\begin{tikzcd}[column sep = 30pt]
	\overline{\Omega}^1_G \ar{d}[swap]{d_1} \ar{r}{d_0} &
	|[alias=Vt]|
	\overline{\Omega}_G^0 \ar{rd}{N_0^{\mathcal{P}}} \ar{d} &
&
	\overline{\Omega}^1_G \ar{d}[swap]{d_1} \ar{rr}{d_0} 
	\ar{rrd}{N_1^{\mathcal{P}}}[near end,name=F1]{}
	[swap,near start,name=F2]{}
&&
	|[alias=Vt2]|
	\overline{\Omega}_G^0 \ar{d}{N_0^{\mathcal{P}}} &
\\
	|[alias=FV]|
	\overline{\Omega}^0_G \ar{r} &
	\overline{\Omega}_G \ar{r}[swap]{\mathcal{N} \mathcal{P}} &
	\mathcal{V}^{op} 
&
	|[alias=FV2]|
	\overline{\Omega}^0_G \ar{rr}[swap]{N_0^{\mathcal{P}}} &&
	\mathcal{V}^{op} 
	\arrow[Leftarrow, from=Vt, to=FV, shorten <=0.10cm,shorten >=0.10cm,"\varphi_1"]
	\arrow[Leftrightarrow, from=Vt2, to=F1, shorten <=0.10cm,shorten >=0.10cm,"\nu_0"]
	\arrow[Leftarrow, from=F2, to=FV2, shorten <=0.10cm,shorten >=0.10cm,"\nu_1"]
\end{tikzcd}
\end{equation}
	Setting $X= \mathcal{N} \mathcal{P}$,
	the fact that the left triangle above commutes shows that
	$X_0 \xrightarrow{\rho_0} N X_{-1}$ is the identity 
	(cf. \eqref{BARLEVELS EQ} and
	Definition \ref{WSPAN_MONAD_DEFINITION})
	so that, by Remark \ref{SEGALCOND REM},
	$\mathcal{N} \mathcal{P}$ satisfies the Segal condition
	\eqref{SEGCOND EQ}.
	Moreover, $d_0 \colon X_0 \to X_{-1}$
	is thus computed by the left diagram below,
	whose composite is the right diagram
	by the simplicial identities in $N^{\mathcal{P}}_{n}$.
\begin{equation}\label{BLABLEID EQ}
\begin{tikzcd}[column sep = 30pt]
	\Omega^0_G \ar{r}{s_{-1}} \ar{d}[swap]{d_0} &
	\overline{\Omega}^1_G \ar{d}[swap]{d_1} \ar{rr}{d_0} 
	\ar{rrd}{N_1^{\mathcal{P}}}[near end,name=F1]{}
	[swap,near start,name=F2]{}
	&&
	|[alias=Vt2]|
	\overline{\Omega}_G^0 \ar{d}{N_0^{\mathcal{P}}}
&
	\Omega^0_G \ar{d}[swap]{d_0}
	\ar{rd}{N^{\mathcal{P}}_{0}}[swap,name=F3]{}
\\
	\Sigma_G \ar{r}[swap]{s_{-1}} &
	|[alias=FV2]|
	\overline{\Omega}^0_G \ar{rr}[swap]{N_0^{\mathcal{P}}} &&
	\mathcal{V}^{op} 
&
	|[alias=FV3]|
	\Sigma_G \ar{r}[swap]{N^{\mathcal{P}}_{-1}} & 
	\mathcal{V}^{op}
	\arrow[Leftrightarrow, from=Vt2, to=F1, shorten <=0.10cm,shorten >=0.10cm,"\nu_0"]
	\arrow[Leftarrow, from=F2, to=FV2, shorten <=0.10cm,shorten >=0.10cm,"\nu_1"]
	\arrow[Leftarrow, from=F3, to=FV3, shorten <=0.10cm,shorten >=0.10cm,"\nu_0"]
\end{tikzcd}
\end{equation}
But, by definition, this right diagram is simply the $N$-algebra multiplication of $\upsilon \mathcal{P}$,
showing that \eqref{XMINUSMULT EQ} indeed inverts $\mathcal{NP}$ 
by recovering $\mathcal{P}$ 
with its genuine operad structure.

For the reverse claim characterizing the essential image,
suppose the opposite of
$X \colon \Omega_G \to \mathcal{V}^{op}$
satisfies the Segal condition \eqref{SEGCOND EQ},
let $X_{-1}$ be the $N$-algebra in \eqref{XMINUSMULT EQ}
and $\mathcal{P}_X \in \mathsf{Op}_G(\mathcal{V})$
be so that $X_{-1} = \upsilon \mathcal{P}_X$.
It remains to show that
$X \simeq \mathcal{N} \mathcal{P}_X$.
But now recall that $\mathcal{N} \mathcal{P}_X$
is built by realizing the simplicial object
$\tilde{N}_n^{\mathcal{P}_X} \colon 
\overline{\Omega}^n_G \to \mathcal{V}^{op}$
in $\mathsf{WSpan}^r(\**,\mathcal{V}^{op})$
built from $N_n^{\mathcal{P}_X}$ via Lemma \ref{TWISTING LEMMA},
so that 
$\tilde{N}_n^{\mathcal{P}_X}$
is as below, where the right side expands $N_0^{\mathcal{P}_X}$.
\[
	\overline{\Omega}^{n}_G 
	\xrightarrow{d_{1,\cdots,n}}
	\overline{\Omega}^{0}_G \xrightarrow{N_0^{\mathcal{P}_X}}
	\mathcal{V}^{op}
\qquad
	\overline{\Omega}^{n}_G 
	\xrightarrow{d_{1,\cdots,n}}
	\overline{\Omega}^{0}_G \xrightarrow{\boldsymbol{V}_G}
	\Fin \wr \Sigma_{G} \to 
	\Fin \wr \Omega_{G} \xrightarrow{X}
	\Fin \wr \mathcal{V}^{op} \xrightarrow{\prod}
	\mathcal{V}^{op}
\]
Further writing $X_n$ for the simplicial object 
$\overline{\Omega}^{n}_G 
\xrightarrow{d_{1,\cdots,n}}
\overline{\Omega}^{0}_G \to \Omega_G \xrightarrow{X}
\mathcal{V}^{op}$
in $\mathsf{WSpan}^r(*,\mathcal{V}^{op})$ 
(this simply forgets structure in \eqref{XNSPANS EQ}),
the natural transformations $\rho_0$ in \eqref{THERHON EQ}
define an isomorphism of simplicial objects
$\rho_0 \colon X_n \xrightarrow{\simeq} \tilde{N}_n^{\mathcal{P}_X}$.
Indeed, the non-trivial claim is that $\rho_0$ respects 
the natural transformation components of the differentials
$d_1\colon X_1 \to X_0$
and
$d_1\colon \tilde{N}_1^{\mathcal{P}_X} 
\to \tilde{N}_0^{\mathcal{P}_X}$.
But the latter is computed by 
\eqref{EXPREAL EQ}, and is thus the 
natural transformation component of the composite
$NX_{-1} \overset{\mu}{\leftarrow}
NNX_{-1} \overset{N \rho_0}{\leftarrow}
NX_0 \xrightarrow{Nd_0}
NX_{-1}$
(as $\nu_0, \nu_1$ in \eqref{EXPREAL EQ}
are induced by $\mu$ and 
\eqref{XMINUSMULT EQ}),
which by \eqref{ALGCHECK EQ}
has the same natural transformation component as
$NX_{-1} \overset{\rho_0}{\leftarrow}
X_0 \overset{d_0}{\leftarrow}
X_1 \xrightarrow{d_1}
X_0 \xrightarrow{\rho_0}
N X_{-1}$ 
(note that the natural transformation for $d_0$ is the identity).
Thus $\rho_0 \colon X_n \xrightarrow{\simeq} \tilde{N}_n^{\mathcal{P}_X}$
is indeed an isomorphism of simplicial objects in 
$\mathsf{WSpan}^r(\**,\mathcal{V}^{op})$ so that, 
upon realization,
we obtain the desired isomorphism
$X \simeq \mathcal{N} \mathcal{P}_X$.
\end{proof}

\newpage

\printindex

\providecommand{\bysame}{\leavevmode\hbox to3em{\hrulefill}\thinspace}
\providecommand{\MR}{\relax\ifhmode\unskip\space\fi MR }
\providecommand{\MRhref}[2]{%
	\href{http://www.ams.org/mathscinet-getitem?mr=#1}{#2}
}
\providecommand{\doi}[1]{%
	doi:\href{https://dx.doi.org/#1}{#1}}
\providecommand{\arxiv}[1]{%
	arXiv:\href{https://arxiv.org/abs/#1}{#1}}
\providecommand{\href}[2]{#2}


\end{document}
